\documentclass[12pt, letterpaper]{article}
\usepackage{iftex}
\ifPDFTeX\pdfoutput=1\fi
\usepackage[authoryear, round]{natbib}
\usepackage{xcolor}
\usepackage[colorlinks=true]{hyperref}
\hypersetup{
  citecolor=blue,
  linkcolor=blue,
  urlcolor=blue
}
\usepackage{amsmath,amssymb,amsfonts,amsthm}
\providecommand{\overunderset}[3]{\overset{#1}{\underset{#2}{#3}}}
\usepackage{subcaption}
\usepackage{url}
\usepackage[colorlinks=true]{hyperref}
\usepackage{enumerate}
\usepackage{csquotes}
\usepackage{graphicx}

\newcommand{\EE}{\mathbb{E}}
\newcommand{\GG}{\mathbb{G}}

\newcommand{\II}{\mathbb{I}}

\newcommand{\PP}{\mathbb{P}}

\newcommand{\RR}{\mathbb{R}}

\def\dd{\mathrm{d}}

\renewcommand{\hat}{\widehat}

\def \heps     {\hat{\heps}}

\DeclareMathOperator{\sgn}{sgn}

\DeclareMathOperator{\Var}{Var}

\def\BP{\mbox{BP}}

\newcommand{\distK}{\mathrm{d_K}}

\def \se    {\mbox{se}}
\def \hatse {\hat{\mbox{se}}}

\def \sgn   {\mbox{sgn}}

\newtheoremstyle{mytheoremstyle}
    {\topsep}
    {\topsep}
    {\normalfont}
    {}
    {\bfseries}
    {.}
    {.5em}
    {}
\theoremstyle{mytheoremstyle}
\newtheorem{theorem}{Theorem}

\newtheorem{lemma}{Lemma}
\newtheorem{proposition}{Proposition}
\newtheorem{remark}{Remark}
\newtheorem{corollary}{Corollary}
\newtheorem{definition}{Definition}

\newcommand{\fitdisplay}[2][\linewidth]{%
  \begingroup
  \setbox0=\hbox{$\displaystyle #2$}%
  \ifdim\wd0>#1\relax
    \resizebox{#1}{!}{\box0}%
  \else\box0\fi
  \endgroup}
\makeatletter
\def\singlespace{\def\baselinestretch{2}\@normalsize}

\makeatletter
\def\singlespace{\def\baselinestretch{1}\@normalsize}

\renewcommand{\baselinestretch}{1.7}
\title{\begin{center}\textsc{The  threshold breakdown point} \end{center}}
\author{Tianjun Ke\footnote{Department of Statistics, Columbia University} \and Marco Avella Medina\footnotemark[1]}
\begin{document}
\maketitle
\begingroup
\renewcommand{\baselinestretch}{1}
\begin{abstract}
We introduce a novel approach to finite sample robustness that avoids the pessimism of traditional breakdown analyses. We define the \emph{threshold breakdown point}, the smallest contamination fraction needed to induce a prescribed deviation, and the finite sample \emph{$m$-sensitivity}, the worst-case deviation that an estimator can incur after $m$ observations are contaminated. We derive these measures for commonly used $M$-estimators, their standard errors and related test statistics. This allows us to extend the decision breakdown point of \citet{zhang1996} to obtain general breakdown characterizations for hypothesis testing, and show how these notions connect to finite sample counterparts of the power and level breakdown functions of \citet{he:simpson:portnoy1990}. We complement our work with an inferential framework for the threshold breakdown and  $m$-sensitivity that yields consistency and asymptotic normality results, as well as a valid multiplier bootstrap for uncertainty quantification. We illustrate the practical utility of our methods in various numerical examples and an application to a two sample testing problem for a blood pressure dataset. \footnote{Code is available at \texttt{https://github.com/keanson/Threshold-Breakdown}.}
\end{abstract}
\endgroup
\section{Introduction}
The breakdown point is one of the main tools used to assess the robustness of statistical procedures \citep{huberandronchetti2009,hampeletal1986,maronna2019robust}. It does so by quantifying the fraction of contamination that could make an estimator take arbitrarily aberrant values \citep{hampel1968,hampel1971,donoho:huber1983}. While this notion is very intuitive, it is also by construction very pessimistic because it essentially identifies worst case contamination scenarios. In contrast, our work introduces a more refined version of the breakdown point that answers the  following question
\begin{displayquote}
    \it{How many observations need to be arbitrarily changed in  order to  move a statistic of interest by a fixed prescribed amount?}
\end{displayquote}
Questions of this type arise naturally in the context of differential privacy \citep{nissim:raskhodnikova:smith2007,dworkandlei2009,avellamedina:brunel2020} and seem quite natural in the context of inference.  For example, if one computes a test statistic and rejects the null hypothesis based on the value obtained, how sensitive was the conclusion to small fractions of the observed data? Our work addresses this question for M-estimators, and, in doing so, generalizes the notion of finite sample breakdown point across a broad range of problems. More precisely, our main contributions can be grouped as follows:
\begin{enumerate}[(a)]
    \item
\emph{Threshold breakdown point and sensitivity of M-estimators:} we provide tools to compute the threshold breakdown point of convex M-estimators as well as its ``dual'' notion, the \emph{$m$-sensitivity}. The latter evaluates the largest shift that a statistic can make given that $m$ out of $n$ observations are changed arbitrarily. These notions complement the classical finite sample breakdown point of \cite{donoho:huber1983} in the context of location-scale estimation with M-estimators.
    \item
    \emph{Breakdown of tests:}
    we extend our threshold breakdown to hypothesis testing, with a general treatment of tests built from M-estimators and their standard errors.
        We apply our framework to the rejection and acceptance breakdown point of tests \citep{zhang1996}, and obtain a test-agnostic characterization with exact formulas and computable bounds. This extends the test-specific results of \cite{zhang1996} both to general tests and beyond the fixed critical regions analysis considered there. Our results are based on computing the $m$-sensitivity of the location and standard error estimator used by the tests.
       {Furthermore,  we find the population limit of the decision breakdown point and connect it  to the power and level breakdown function of \citet{he:simpson:portnoy1990}.}
    \item
    \emph{Bootstrap inference for the threshold breakdown:} since our threshold breakdown and $m$-sensitivity are themselves statistics, we introduce a bootstrap method to assess their variability in a nonparametric way. For this, we introduce a statistical framework to analyze the performance of a multiplier bootstrap which computes the threshold breakdown of weighted M-estimators. The validity of our bootstrap scheme is established using a general  M-estimation framework for the natural population counterparts of the threshold breakdown and $m$-sensitivity.
\end{enumerate}
\subsection{Related Literature}
The robustness properties of  $M$–estimators are well studied in the literature, and have traditionally been largely focused on population functionals under $\varepsilon$–contamination neighborhoods \citep{huber1964robust,hampel1974influence,donoho:huber1983,huberandronchetti2009,rieder1994,maronna2019robust}.
In this context, the maxbias curve was developed as a refined characterization of asymptotic robustness, quantifying the worst–case deviation of an estimator under a given contamination \citep{martin:zamar1989,martin:yohai:zamar1989,he:simpson1993,berrendero:zamar2001}. This quantity can be seen as a population counterpart of our $m$-sensitivity for $m/n\to\varepsilon$, as will be discussed in detail in Section \ref{sec:stat_framework}.
Maxbias curves, when specialized to finite samples in the literature, have been evaluated through ad hoc calculations or simulations for particular estimators \citep{ruckdeschel2010robustness,sinova2019empirical,de2021robustness}, rather than through a general framework that produces data–dependent contamination thresholds or sensitivities as presented in our work.
We note that similar ideas have also been explored more recently in the machine learning literature, essentially exploring approximations of the empirical maxbias curve using the influence function and variants thereof \citep{kohetal2019, broderick:giordano:meager2020,fisheretal2023,huetal2024}.
The breakdown point is a very intuitive and  popular tool for quantifying  robustness in the robust statistics literature. Our threshold breakdown point recovers as a special case the  finite sample breakdown point of \cite{donoho:huber1983}. The latter was introduced in the context of univariate location estimators and was quickly extended to various multivariate settings in the robust statistics literature \citep{rousseeuw:yohai1984,rousseeuw1984,davies1987,yohai1987}. More recently, the breakdown point has been a subject of great interest in the context of multivariate quantiles \citep{konen:paindavein2025,konen:paindaveine2026,paindaveine:passeggeri2026,avellamedina:gonzalezsanz2026,gonzalezsanz:avellamedina2026}. The breakdown point was originally introduced in an asymptotic form by \cite{hampel1968,hampel1971}. In Section \ref{sec:stat_framework} we will also see how our finite sample threshold breakdown and $m$-sensitivity converge to their natural population counterparts. We will exploit this result in order to construct a consistent bootstrap inference procedure  to provide valid confidence intervals for our finite sample threshold breakdown and perhaps more interesting in practice, for the path of the $m$-sensitivity curve for $m\in \{1,2,\dots,\lfloor n/2\rfloor\}$.
Robustness notions for tests of hypothesis  mirror those for estimators, but the literature is comparatively sparse. The breakdown point of tests is particularly relevant for our work. \citet{he:simpson:portnoy1990} introduced the breakdown robustness of tests, describing the smallest contamination fraction that can nullify the test. This formalism was subsequently extended to sample–level notions of breakdown for tests \citep{zhang1996} and was applied to specific robust test constructions. However, existing work  treats each testing problem separately with no general finite–sample, data–dependent results. We note that a very different approach to robust testing is rooted in shrinking contamination neighborhoods and the influence function \citep{ronchetti1982,heritier:ronchetti1994,rieder1994}.
In differential privacy, robustness typically plays an indirect role by means of the global sensitivity, which is a popular quantity used to  control the amount of noise required for private data release \citep{dworkandroth2014,nissim:raskhodnikova:smith2007}. The seminal work of \citet{dworkandlei2009} explicitly connected differential privacy to robust statistics, exploiting the low local sensitivity of robust estimators to calibrate the privacy inducing noise required to make their randomized  counterparts private. More recent work has systematically studied the interaction between robust statistics and differential privacy, also focusing on various versions of sensitivity or resilience bounds \citep{chaudhuri:hsu2012,avellamedina2020,kamath2020private,liu2022differential,avellamedina:brunel2020,avellamedina2021,georgiev2022privacy,avella2023differentially,li:berrett:yu2023,alabietal2023,hopkinsetal2023,loh2025theoretical}. In fact, the notion of threshold breakdown point that we study in this paper has been essentially defined as part of the design of the propose-test-release procedure of \cite{avellamedina:brunel2020}. It also corresponds essentially to the length function defined in the work of \cite{asi:duchi2020} used by their inverse sensitivity mechanism and appeared again as the score function used by the exponential mechanism of \cite{hopkinsetal2023}.
\subsection{Organization of the Paper}
Section \ref{sec:thresholdBP} introduces the threshold breakdown point and $m$-sensitivity for general statistics. Section \ref{sec:convex_M} derives explicit formulas for these quantities for convex location M-estimators. Section \ref{sec:other_M} extends the result to scale M-estimators, two–stage estimators, and plug-in estimators of the standard error. Section \ref{sec:testing} introduces the breakdown point of tests, connects it to the breakdown function of \citet{he:simpson:portnoy1990}, and gives computable approximations. Section \ref{sec:stat_framework} reformulates $m$-sensitivity and threshold breakdown as (approximate) roots of coupled estimating equations. Within this framework, we establish consistency and asymptotic normality and provide a valid multiplier bootstrap inference procedure. Section \ref{sec:empirical_experiment} gives an empirical illustration based on blood-pressure data. We conclude in Section \ref{sec:conclusion} with a short discussion. All proofs and additional results are presented in the Appendix.
\section{The Threshold Breakdown Point}
\label{sec:thresholdBP}
Let $x^{(n)}=(x_1,\dots,x_n)\in\mathcal{X}^n$ be a sample of size $n$, taking values in some sample space $\mathcal{X}\subset\mathbb{R}$.
We use $\hat \theta=\hat\theta(x^{(n)})$ to denote a fixed deterministic estimate of some unknown parameter $\theta$.
For  $x^{(n)}\in\RR^n$ and $m\geq 0$, let
$
B_{\textsf{H}}(x^{(n)},m)=\{y^{(n)}\in\RR^n:\sum_{i=1}^n \mathbb{I}\{x_i\neq y_i\}\leq m\}.
$
Thus $B_{\textsf{H}}(x^{(n)},m)$ is the ``Hamming ball'' of  all datasets obtained by contaminating arbitrarily up to $m$ observations of  $x^{(n)}$.
 Throughout the paper, we use $[n]$ to denote the set $\{1, 2, \dots, n\}$.
\begin{definition}
For $\eta>0$,  the finite sample threshold breakdown point of $\hat\theta$ at $x^{(n)}$ is
\begin{align}
 \BP_{\eta}(\hat\theta,x^{(n)}):=\frac{1}{n} \min\{m \in [n]: \exists y^{(n)}\in B_\textsf{H}(x^{(n)},m),  |\hat\theta(y^{(n)})-\hat\theta(x^{(n)})| \ge \eta\}. \label{BP}
\end{align}
\end{definition}
Our definition of $\BP_{\eta}(\hat\theta,x^{(n)})$  uses  the classical $\varepsilon$-replacement contamination framework of \citet{donoho:huber1983}, which studies worst-case effect of  replacing an $\varepsilon$ fraction of the sample by arbitrary contaminating observations. Their classical finite sample  breakdown point is recovered by letting $\eta\to\infty$, i.e., it computes the minimum contamination level at which the estimator ``blows up.'' In contrast, the threshold version keeps track of potentially large but finite deviations, which are often the object of interest in applications.
Viewing the threshold breakdown points as a function of $\eta$, it is natural to also look at the inverse viewpoint: for a prescribed contamination level $m/n$, what is the largest shift that can be induced in an estimate  by altering at most $m$ data points? This leads to the following finite sample analogue of a maxbias, which we will refer to as \emph{sensitivity}, consistent with Hampel's sensitivity curve \citep{hampel1974influence} and the notion of local sensitivity in the differential privacy literature \citep{dworkandroth2014}.
\begin{definition}
\label{def:eta}
    For a data set $x^{(n)}$ {and $m\in[n]$}, the $m$-sensitivity to $n$ points is defined as
    \begin{align*}
    {\eta_{m/n}(\hat\theta,x^{(n)}) := \sup\Bigl\{\eta>0:\BP_\eta(\hat\theta,x^{(n)})\le\frac mn\Bigr\}.}
    \end{align*}
\end{definition}
Thus, $\eta_{m/n}(\hat\theta,x^{(n)})$ quantifies the largest error that can be forced by contaminating at most $m$ out of $n$ observations. In this sense, the map $m/n \mapsto \eta_{m/n}(\hat\theta,x^{(n)})$ plays the role of a finite sample maxbias curve associated with the estimator $\hat\theta$ at the observed data $x^{(n)}$.
\section{Convex M-estimators}
\label{sec:convex_M}
We study the important class of convex M-estimators of location and find a simple explicit  characterization of the threshold breakdown point in terms of the sample order statistics and the score function $\psi$.
An M-estimate of location $\hat\theta$ is defined as
\begin{equation}
\label{M-est}
\hat\theta=\arg\min_{\theta\in\mathbb{R}} \sum_{i=1}^n\rho(x_i-\theta) \iff \sum_{i=1}^n\psi(x_i-\hat\theta)=0,
\end{equation}
where $\rho$ is a differentiable and  convex loss function and $\psi$ denotes its derivative. Let $x_{(1)}, \ldots, x_{(n)}$ be order statistics of  $x^{(n)}$ i.e.,
$
 \min_{1\leq i\leq n} x_i=x_{(1)}\leq \ldots\leq x_{(n)}=\max_{1\leq i\leq n} x_i.
$
We adapt the arguments of Theorem 3.6 in \cite{huberandronchetti2009} to our threshold breakdown point.
{Since we assume $\psi$ is  nondecreasing  there could be multiple roots, so we follow the convention of choosing the worst possible choice for $\hat \theta$ throughout the paper as in \citet[Section~3.2.3]{huberandronchetti2009}.}%
\begin{theorem}
\label{thm:loc_main}
    Assume that $\psi$ in \eqref{M-est} is non-decreasing and passes through 0. Then,
    \begin{equation*}
     \BP_{\eta}(\hat\theta, x^{(n)}) =   \frac{1}{n}\min \biggl\{ \!\!
     \resizebox{0.74\linewidth}{!}{
     $ \displaystyle
     m \in [n]: m \ge \min\bigg\{\frac{\sum_{i\le n-m} \psi(x_{(i)} - (\hat \theta - \eta))}{-\psi(-\infty)},\frac{\sum_{i> m} \psi(x_{(i)} - (\hat \theta + \eta))}{-\psi(\infty)} \biggr\}
     $ }
     \!\!
     \biggr \} .
    \end{equation*}
    \end{theorem}
For a given threshold $\eta>0$, the quantities inside the inner minimum describe the minimal number of observations that must be moved to drive the estimator up or down by {at least $\eta$}.
The overall threshold breakdown point is the more vulnerable direction.
The next corollary inverts this relationship and expresses the $m$-sensitivity as a function of the contamination level $m/n$.
\begin{corollary}
\label{cor:opt_attack_m_est_main}
Assume that $\psi$ in \eqref{M-est} is non-decreasing and passes through 0. Then,
\begin{align*}
\eta_{m/n}(\hat\theta,x^{(n)})
=
\max\Biggl\{
\eta \ge 0:
\max\Biggl\{
&\sum_{i>m}\psi\bigl(x_{(i)}-(\hat\theta+\eta)\bigr)+m\psi(\infty),\, \\
&-\sum_{i\le n-m}\psi\bigl(x_{(i)}-(\hat\theta-\eta)\bigr)-m\psi(-\infty)
\Biggr\}
\ge 0
\Biggr\}.
\end{align*}
\end{corollary}
For each contamination fraction $m/n$, the $m$-sensitivity is the largest shift $\eta$ for which the corresponding score equation remains solvable after modifying at most $m$ observations.
\section{Threshold Breakdown for Scale, Two-Stage Location, and Variance Estimators}
\label{sec:other_M}
We extend our threshold breakdown and $m$-sensitivity analysis beyond pure location to three settings that are central to robust statistics in practice: scale M-estimators, two-stage location-scale M-estimators, and variance estimators. For scale M-estimators, we obtain a closed-form expression for the threshold breakdown point. For two-stage procedures, where both location and scale are estimated, the exact breakdown point is typically intractable. Nevertheless, we derive computable upper and lower bounds that are tight enough to be informative and that connect directly back to the location case analyzed in Theorem~\ref{thm:loc_main}. We also obtain either explicit solutions or computable approximations for variance estimators.
\subsection{Scale M-estimators}
An M-estimate of scale $\hat \sigma$ is defined as the root of the equation
\begin{equation}
\label{eq:scale}
    \sum_{i=1}^n \chi \left ( \frac{x_i}{\hat \sigma} \right ) = 0,
\end{equation}
where $\chi:\mathbb{R}\to\mathbb{R}$ is the  score function for the scale parameter. A common choice is $\chi(u)=\rho(u)-b$, where $\rho$ is an even loss function and $b$ is a Fisher consistency constant. For example, when $X_1,\dots,X_n\overset{iid}{\sim}N(0,\sigma^2)$  Huber's scale estimator \citep{huberandronchetti2009} takes
$$
\chi_\delta(t)= \psi_\delta(t)^2 = t^2 \II\{ |t|\le \delta\} + \delta^2\II\{ |t|> \delta\}
$$
and defines $\hat\sigma_H$ as the positive solution of \eqref{eq:scale} with  $\chi(t) = \chi_\delta(t) - b$.
One can easily adapt Theorem~\ref{thm:loc_main} to obtain the threshold breakdown point of scale M-estimators in terms of the ordered absolute values $|x|_{(1)} \le \dots \le |x|_{(n)}$. Indeed, consider the log-scale reparametrization
$$
\tilde x_i=\log |x_i|,\qquad \tilde \sigma=\log \sigma,\qquad \tilde\chi(u)=\chi(e^u),
$$
with $\tilde x_i=-\infty$ when $x_i=0$. Then
$
\chi({x_i}/{\sigma})=\tilde\chi(\tilde x_i-\tilde \sigma),
$
and $\tilde\chi$ satisfies the assumptions of Theorem~\ref{thm:loc_main} and we obtain the following corollary.
\begin{corollary}
\label{cor:scale_main}
Assume that $\chi(x)$ is even, non-decreasing in $|x|$, and passes through $0$.
Writing $|x^{(n)}|=(|x_1|,\ldots,|x_n|)$, suppose $\hat\sigma(x^{(n)})>0$. 
{For $\eta>0$,
\begin{align*}
\BP_\eta(\hat\sigma,x^{(n)}) =\frac1n\min\Biggl\{m\in[n]: & \ \sum_{i>m}\chi\!\left(\frac{|x^{(n)}|_{(i)}}{\hat\sigma+\eta}\right) +m\chi(\infty)\ge0 \quad\text{or} \\
& \sum_{i\le n-m}\chi\!\left(\frac{|x^{(n)}|_{(i)}}{\hat\sigma-\eta}\right) +m\chi(0)\le0 \quad (\eta\le\hat\sigma)\Biggr\}.
\end{align*}
The downward expression is used only for $0\le\eta\le\hat\sigma$;
at $\eta=\hat\sigma$ it denotes its limit from below.
Moreover,
\begin{align*}
\eta_{m/n}(\hat\sigma,x^{(n)}) =\max\Biggl\{ & \sup\left\{\eta\ge0: \sum_{i>m}\chi\!\left(\frac{|x^{(n)}|_{(i)}}{\hat\sigma+\eta}\right) +m\chi(\infty)\ge0\right\}, \\
& \sup\left\{0\le\eta\le\hat\sigma: \sum_{i\le n-m}\chi\!\left(\frac{|x^{(n)}|_{(i)}}{\hat\sigma-\eta}\right) +m\chi(0)\le0\right\} \Biggr\}.
\end{align*}
}
\end{corollary}
We see that the threshold breakdown point of scale M-estimators is driven by the extreme observations. {The two score conditions correspond to inflating the scale}
and deflating it. As in the location case, the resulting breakdown point is fully explicit and can be evaluated directly from the data and the score function $\chi$.
\subsection{Two-stage Location M-estimators}
\label{sec:two_stage}
A two-stage M-estimator of location $\hat\theta$ is defined as follows. Given a sample $x^{(n)}=(x_1,\ldots,x_n)$,
\begin{enumerate}
\item[(i)] Compute an M-estimate of scale $\hat \sigma(x^{(n)})$.
\item[(ii)] Obtain $\hat\theta$ as any solution of
\begin{equation}
\label{M-est-two}
\sum_{i=1}^n \psi\left(\frac{x_i-\theta}{\hat \sigma(x^{(n)})}\right)=0 .
\end{equation}
\end{enumerate}
Because the contamination simultaneously alters both the numerator and the scale in \eqref{M-est-two}, the exact threshold breakdown point of $\hat\theta$ is not computable in closed form. Therefore, we derive informative upper and lower bounds that quantify how far one can move the two-stage estimator by perturbing $m$ out of $n$ observations. For this, we introduce additional notation. Define the sets of left- and right-oriented contamination maps  $\mathsf{C}_{m,c}^L,\mathsf{C}_{m,c}^R:\mathbb R^n\to \overline{\mathbb R}^n$ as
$$
\bigl(\mathsf{C}_{m,c}^L(x^{(n)})\bigr)_{(i)}
=
\begin{cases}
c, & 1\le i\le m,\\
x_{(i)}, & m<i\le n,
\end{cases}
\qquad
\bigl(\mathsf{C}_{m,c}^R(x^{(n)})\bigr)_{(i)}
=
\begin{cases}
x_{(i)}, & 1\le i\le n-m,\\
c, & n-m<i\le n.
\end{cases}
$$
Letting $\mathcal C \subset \mathbb{R}$, we further write the collection of contaminated samples as
$
\mathsf C^L_{m, \mathcal C_L} (x^{(n)}):= \{\mathsf{C}_{m,c}^L(x^{(n)}): c\in \mathcal C \Big\}$ and $ \mathsf C^R_{m, \mathcal C} := \{\mathsf{C}_{m,c}^R(x^{(n)}): c \in \mathcal C \}.
$
Proposition~\ref{prop:two_stage_main} shows that, even though we cannot write down the exact threshold breakdown point in closed form, we can upper bound it in terms of a small number of explicit contamination patterns. These bounds are easy to compute and  capture how sensitive a two-stage procedure is to extreme alterations of a subset of the data.
\begin{proposition}
\label{prop:two_stage_main}
Define
$$
\mathsf T_\eta(\tilde x^{(n)}) :=
\begin{cases}
-\dfrac{\sum_{i\le n-m}\psi\big((\tilde x_{(i)}-\hat\theta (x^{(n)})-\eta)/\hat\sigma(\tilde x^{(n)})\big)}
{\max\!\left\{\psi\big((\tilde x_{(n)}-\hat\theta (x^{(n)})-\eta)/\hat\sigma(\tilde x^{(n)})\big),\,0\right\}},
& \tilde x^{(n)} \in \mathsf C^L_{m,\{\infty,\,x_{(n)}\}}(x^{(n)}),\\[2.2ex]
-\dfrac{\sum_{i>m}\psi\big((\tilde x_{(i)}-\hat\theta (x^{(n)})+\eta)/\hat\sigma(\tilde x^{(n)})\big)}
{\min\!\left\{\psi\big((\tilde x_{(1)}-\hat\theta (x^{(n)})+\eta)/\hat\sigma(\tilde x^{(n)})\big),\,0\right\}},
& \tilde x^{(n)} \in \mathsf C^R_{m,\{-\infty,\,x_{(1)}\}}(x^{(n)}).
\end{cases}
$$
Assume that $\psi$ in \eqref{M-est-two} is bounded, non-decreasing,  and passes through $0$. Then,
$$
\BP_{\eta}(\hat\theta,x^{(n)})
\le
\frac1n \min\Bigl\{m \in [n] : m \ge \min_{\tilde x^{(n)} \in \mathsf C^L_{m,\{\infty,\,x_{(n)}\}}\cup \mathsf C^R_{m,\{-\infty,\,x_{(1)}\}}} \mathsf T_\eta(\tilde x^{(n)})\Bigr\}.
$$
Moreover, let $\tilde\theta(\tilde x^{(n)})$ denote the largest solution to \eqref{M-est-two} when $\tilde x^{(n)} \in \mathsf C^L_{m,\{\infty,\,x_{(n)}\}}(x^{(n)})$ and the smallest solution when $\tilde x^{(n)} \in \mathsf C^R_{m,\{-\infty,\,x_{(1)}\}}(x^{(n)})$. Then
$$
\eta_{m/n}(\hat\theta,x^{(n)})
\ge
\max\Biggl\{
\max_{\tilde x^{(n)} \in \mathsf C^L_{m,\{\infty,\,x_{(n)}\}}} \bigl(\tilde\theta(\tilde x^{(n)})-\hat\theta (x^{(n)})\bigr),\,
\max_{\tilde x^{(n)} \in \mathsf C^R_{m,\{-\infty,\,x_{(1)}\}}} \bigl(\hat\theta (x^{(n)})-\tilde\theta(\tilde x^{(n)})\bigr)
\Biggr\}.
$$
\end{proposition}
\begin{remark}
Proposition~\ref{prop:two_stage_main} uses the strategy of moving points to infinity or to the first/last order statistic for simplicity. One can construct more sophisticated contamination classes than $\mathsf C^L_{m,\{\infty,\,x_{(n)}\}}\cup \mathsf C^R_{m,\{-\infty,\,x_{(1)}\}}$ to achieve tighter bounds, see Appendix \ref{sec:con} for an example.
\end{remark}
The next proposition provides complementary lower bounds on the threshold breakdown point leveraging the scale invariance of the M-estimator.
\begin{proposition}
\label{prop:two_stage_lb_main}
Assume that $\psi$ is non-decreasing, odd, and passes through $0$. Let $\hat\theta(x^{(n)})$ denote the two-stage estimator defined in \eqref{M-est-two}. For $m\in[n]$, set
$
\underline{\sigma}_{m/n}
:=
\max\bigl\{\hat \sigma(x^{(n)}) - \eta_{m/n}(\hat \sigma, x^{(n)}), 0 \bigr\},
\overline{\sigma}_{m/n}
:=
\hat \sigma(x^{(n)}) + \eta_{m/n}(\hat \sigma, x^{(n)}).
$
Define the sign-split rescalings
$$
r_i
=
\frac{x_i}{\underline{\sigma}_{m/n}} \II\{x_i \ge 0\}
+
\frac{x_i}{\overline{\sigma}_{m/n}} \II\{x_i < 0\},
\qquad
r_i'
=
\frac{x_i}{\underline{\sigma}_{m/n}} \II\{x_i < 0\}
+
\frac{x_i}{\overline{\sigma}_{m/n}} \II\{x_i \ge 0\},
$$
and we adopt the convention $r_i=\infty$ and $r_i'=\infty$  if either $\underline{\sigma}_{m/n}=0$ or $\overline{\sigma}_{m/n}=0$. Let $\hat\theta_{\mathrm{loc}}(x^{(n)})$ denote the solution to \eqref{M-est}. Denote
\begin{align*}
\sigma_{m/n+}
&:=
\underline{\sigma}_{m/n}
\II\bigl\{\hat\theta_{\rm loc}(r^{(n)}) + \eta_{m/n}(\hat\theta_{\rm loc}, r^{(n)}) \le 0\bigr\}
+
\overline{\sigma}_{m/n}
\II\bigl\{\hat\theta_{\rm loc}(r^{(n)}) + \eta_{m/n}(\hat\theta_{\rm loc}, r^{(n)}) > 0\bigr\},
\\
\sigma_{m/n-}
&:=
\underline{\sigma}_{m/n}
\II\bigl\{\hat\theta_{\rm loc}(r^{(n)'}) - \eta_{m/n}(\hat\theta_{\rm loc}, r^{(n)'}) \ge 0\bigr\}
+
\overline{\sigma}_{m/n}
\II\bigl\{\hat\theta_{\rm loc}(r^{(n)'}) - \eta_{m/n}(\hat\theta_{\rm loc}, r^{(n)'}) < 0\bigr\}.
\end{align*}
Then the $m$-sensitivity satisfies
\begin{align*}
\eta_{m/n}(\hat\theta,x^{(n)})
\le\overline{\eta}_{m/n}(\hat\theta,x^{(n)}) := \max \Bigl\{
&\sigma_{m/n+}
\Bigl(\hat\theta_{\rm loc}(r^{(n)}) + \eta_{m/n}(\hat\theta_{\rm loc}, r^{(n)})\Bigr)
-
\hat\theta(x^{(n)}), \\
&~\!\hat\theta(x^{(n)})
-
\sigma_{m/n-}
\Bigl(\hat\theta_{\rm loc}(r^{(n)'}) - \eta_{m/n}(\hat\theta_{\rm loc}, r^{(n)'})\Bigr)
\Bigr \}.
\end{align*}
Furthermore,
$$
\BP_{\eta}(\hat\theta,x^{(n)})
\ge
\frac{1}{n}
\min\Bigl\{
m\in[n]:
\overline{\eta}_{m/n}(\hat\theta,x^{(n)}) \ge \eta
\Bigr\}.
$$
\end{proposition}
\begin{remark}
    Proposition~\ref{prop:two_stage_lb_main} shows that the $m$-sensitivity of a two-stage M-estimator can be controlled by the $m$-sensitivity of a corresponding location M-estimator applied to rescaled data. Notice that $\eta_{m/n}(\hat \sigma, x^{(n)})$ can be computed by Corollary \ref{cor:scale_main} for scale M-estimators.
\end{remark}
We illustrate Proposition \ref{prop:two_stage_main} and Proposition \ref{prop:two_stage_lb_main} by reporting the resulting upper/lower bounds for the threshold breakdown point and the corresponding $m$-sensitivity curves for the two-stage estimator across different thresholds $\eta$ or different contamination sizes $m$. We follow the common simulation setup in Section \ref{sec:NI_location}.
\begin{figure}[htbp]
  \centering
  \begin{subfigure}[b]{0.45\textwidth}
    \centering
    \includegraphics[width=\linewidth]{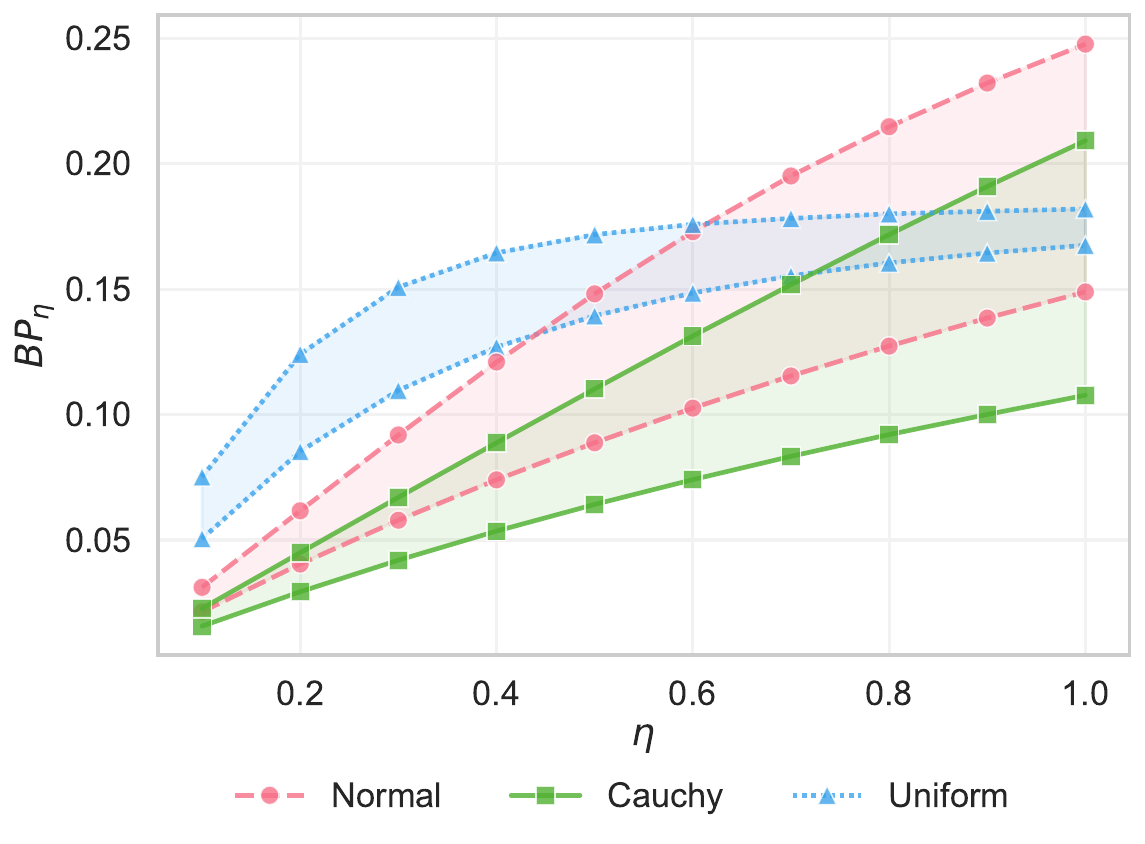}
    \caption{The threshold breakdown point.}
  \end{subfigure}
  \hfill
  \begin{subfigure}[b]{0.45\textwidth}
    \centering
    \includegraphics[width=\linewidth]{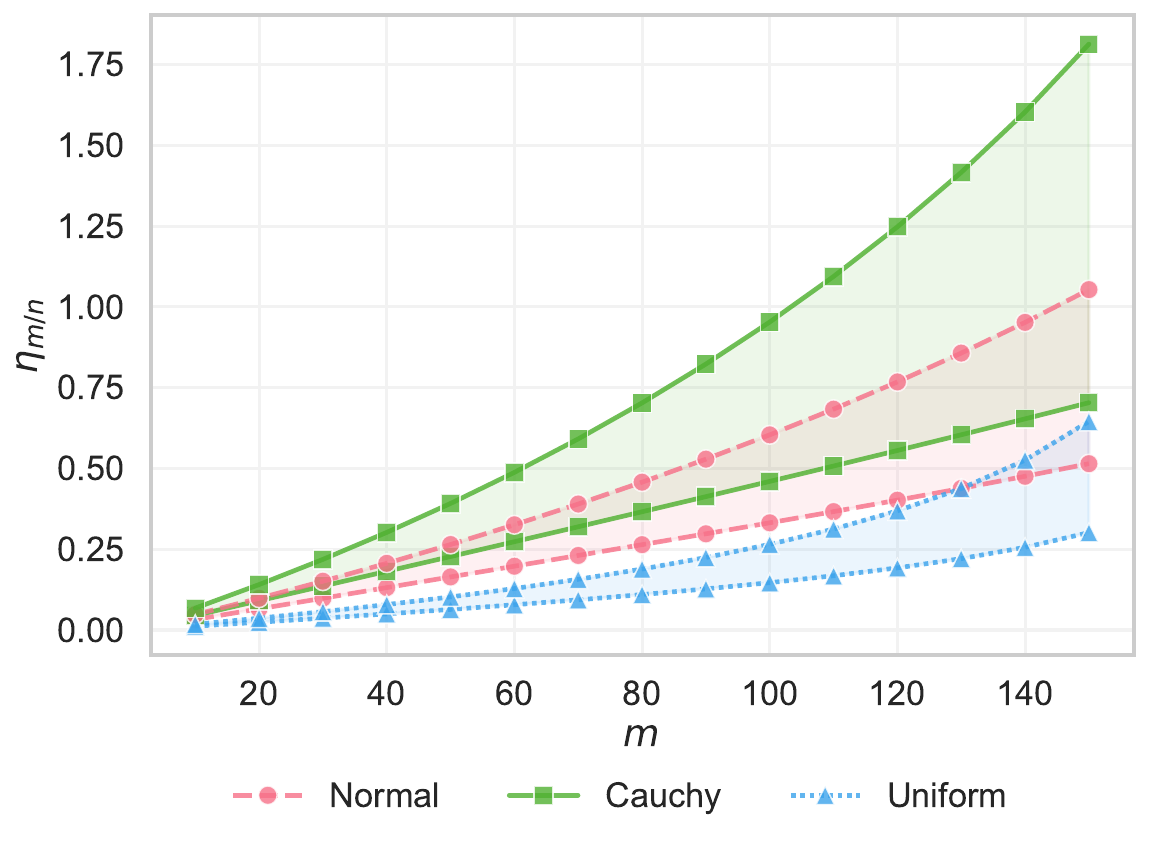}
    \caption{Sensitivity curve.}
  \end{subfigure}
  \caption{ \small The threshold breakdown points and $m$-sensitivity for the two-stage M-estimator under $\mathcal N(0,1)$, $\mathrm{Cauchy}(0,1)$, and $\mathrm{Unif}(0,1)$ distributions with $n=1000$, $m \in \{10, 20, \dots, 150\}$, and $\eta \in \{0.1, 0.2, \dots, 1\}$, using Huber's loss for location and Huber's proposal 2 for scale \citep{huber1964robust}. The figure suggests a tail-dependent finite sample effect: when $m$ is sufficiently large for contamination to interact with the sample extremes and the estimated scale, heavier-tailed samples tend to yield larger sensitivity and smaller threshold breakdown points.}
  \label{fig:bp_location_two_stage}
\end{figure}
Figure \ref{fig:bp_location_two_stage} shows that the threshold breakdown point and $m$-sensitivity show a clear distributional dependence. In particular, heavier-tailed samples tend to produce larger estimator shifts and smaller threshold breakdown points at the same contamination level.
This is to be contrasted with the classical breakdown point which does not depend on the distribution of the data.
\subsection{Variance Estimation Threshold BP}
\label{sec:varianceBP}
Variance estimators are essential for both uncertainty quantification and hypothesis testing. When using M-estimators, it is common practice  to estimate their standard error via a plug-in formula based on the score function $\psi$. Specifically, the plug-in estimate of the standard error takes the form
\begin{equation}
    \label{eq:se}
\hatse(\hat \theta) = \sqrt{\frac{\sum_{i=1}^n \psi(x_i - \hat \theta)^2}{(\sum_{i = 1}^n \psi'(x_i - \hat \theta))^2}},
\end{equation}
where $\psi$ is differentiable almost everywhere, and $\psi'$ denotes a well-defined function on $\RR$ agreeing with the a.e. derivative of $\psi$ almost everywhere: for instance, one can take $\psi_\delta'(x) = \II\{|x| \le \delta\}$ for Huber's loss with parameter $\delta$. This expression appears in the asymptotic variance formulas for M-estimators and underlies Wald-type tests.
Alternatively, in the context of hypothesis testing, under $H_0:\theta=\theta_0$, one may consider the restricted plug-in estimate of the standard error
$\hatse(\theta_0)$
for a fixed $\theta_0$. Evaluating the standard error at $\theta_0$ rather than at $\hat\theta$ is sometimes advocated to mitigate the Hauck–Donner effect\footnote{This occurs when Wald test statistic fails to increase  monotonically as a function of its distance from the null value.} \citep{hauck:donner1977,yee2022}. In this subsection, we analyze the threshold breakdown point of $\hatse(\hat\theta)$ and $\hatse(\theta_0)$, and relate it to the $m$-sensitivity of the underlying location estimator.
\begin{lemma}
\label{lem:se_at_theta_0_main}
Assume that $\psi$ is odd, bounded, differentiable a.e., passes through $0$  and such that  $\psi'(x)$ is non-increasing in $|x|$. Further assume that the data are centered at $\theta_0$. Let $\pi$ be a permutation such that
$
|x_{\pi_1}| \le |x_{\pi_2}| \le \cdots \le |x_{\pi_n}|,
$
and define
$
\psi_{\max}=\max\{|\psi(-\infty)|,|\psi(\infty)|\}.
$
Then the threshold breakdown point of $\hatse(\theta_0)$ is
$$
\BP_{\eta}\bigl(\hatse(\theta_0),x^{(n)}\bigr)
=
\frac{1}{n}\min\left\{
m\in[n]:
\eta_{m/n}\bigl(\hatse(\theta_0),x^{(n)}\bigr)\mathrel{{\ge}}\eta
\right\},
$$
and its $m$-sensitivity is
\begin{align*}
\eta_{m/n}\bigl(\hatse(\theta_0),x^{(n)}\bigr) =
\resizebox{0.77\linewidth}{!}{$\displaystyle\max\Biggl\{
\sqrt{
\frac{
m\psi_{\max}^2+\sum_{i>m}\psi\bigl(x_{\pi_i}\bigr)^2
}{
\bigl(\sum_{i>m}\psi'\bigl(x_{\pi_i}\bigr)\bigr)^2
}
}
-\hatse(\theta_0),
\;
\hatse(\theta_0)
-
\sqrt{
\frac{
\sum_{i\le n-m}\psi\bigl(x_{\pi_i}\bigr)^2
}{
\bigl(\sum_{i\le n-m}\psi'\bigl(x_{\pi_i}\bigr)+m\psi'(0)\bigr)^2
}
}
\Biggr\}.
$}
\end{align*}
\end{lemma}
Lemma~\ref{lem:se_at_theta_0_main} shows that, for the plug-in variance estimator, the threshold breakdown point can be expressed exactly in terms of ordered absolute residuals. In particular, the map $m/n \mapsto \eta_{m/n}(\hatse(\theta_0),x^{(n)})$ can be viewed as a finite sample maxbias curve for the standard error.
The  lemma characterizes the threshold breakdown behavior when the location is fixed at some $\theta_0$. This regime is relevant when the standard error is evaluated at the null value, or when we assume one cannot access or manipulate the plug-in estimate $\hat\theta$ used to form the standard error. Allowing the contamination to act on $\hatse(\hat\theta)$ leads to a substantially more complicated problem since any perturbation of the data $x$ simultaneously affects both $\hat\theta$ and the standard error computed at that value. We consider this more complicated setting and  provide computable upper and lower bounds on the sensitivities $\eta_{m/n}(\hatse(\hat\theta),x^{(n)})$ in Appendix~\ref{sec:opt_attack_m_est_se}.
\section{Hypothesis Testing Threshold Breakdown Point}
\label{sec:testing}
Once an M-estimator $\hat\theta$ has been obtained, a natural next step is to conduct hypothesis testing about the underlying $\theta$. We consider tests induced by data-dependent intervals, i.e.,
$$
\phi_n(x^{(n)}) := \II\{\theta_0 \notin C_n(x^{(n)})\}.
$$
A basic example is confidence intervals based on asymptotic normality of the form
\begin{align}
\label{eq:wald}
    C_n(x^{(n)})
=
\bigl[\hat\theta - z_{1-\alpha/2}\hatse,\ \hat\theta + z_{1-\alpha/2}\hatse\bigr],
\end{align}
where $\hatse$ is a standard error estimate, possibly evaluated at $\hat\theta$ or $\theta_0$. This includes the usual and restricted Wald tests as special cases. In the next section, we extend threshold breakdown to such interval-based procedures.
\subsection{Breakdown Point of Tests}
\label{sec:BPoftests}
We now study how many observations must be modified to reverse the outcome of such tests. Specifically, we seek the smallest number of contaminations needed to change the decision from reject to accept (fail to reject), or vice versa. This decision-level notion was also considered by \citet{zhang1996}, who defines the sample breakdown point of a test as the smallest proportion of arbitrary contamination capable of flipping the test decision, thereby tying the breakdown directly to the realized sample and the critical region. \citet{zhang1996} computed this quantity for several commonly used one- and two-sided tests, including the $t$-test, Hotelling's $T^2$, and various $M$- and score tests, and also studied their asymptotic behavior. We take that definition as our starting point, but pursue a different goal. Rather than analyzing individual tests case by case, we derive a general characterization in terms of the $m$-sensitivity from Definition~\ref{def:eta}, leading to computable bounds, and in several settings exact formulas, for broad classes of interval-based tests. Moreover, while \citet{zhang1996} assumes a fixed critical region, we also allow the data-dependent critical region to be modified simultaneously.
\begin{definition}
    The finite sample threshold breakdown point of a deterministic test $\phi:\mathbb{R}^n\to \{0,1\}$ based on $x^{(n)}$ is
    $$
        \BP(\phi,x^{(n)}) := \frac{1}{n} \min\{m \in [n]: \exists y^{(n)}\in B_\textsf{H}(x^{(n)},m),\ |\phi(x^{(n)}) - \phi(y^{(n)})| = 1\}.
    $$
    We also define the one-sided finite sample test breakdown points
    \begin{align*}
        \BP_{\text{accept}}(\phi,x^{(n)}) &:= \frac{1}{n} \min\{m\in [n]: \exists y^{(n)}\in B_\textsf{H}(x^{(n)},m),\phi(x^{(n)})=0, \ \phi(y^{(n)}) = 1\},  \\
        \BP_{\text{reject}}(\phi,x^{(n)}) &:= \frac{1}{n} \min\{m\in [n]: \exists y^{(n)}\in B_\textsf{H}(x^{(n)},m), \, \phi(x^{(n)})=1,\ \phi(y^{(n)}) = 0\}.
    \end{align*}
\end{definition}
Clearly, for interval-based tests of the form
$$
\phi(x^{(n)})=\II\Bigl\{\theta_0\notin \bigl[\hat\theta - c_{\alpha} \cdot \hatse,\
  \hat\theta + c'_{\alpha} \cdot \hatse\bigr]\Bigr\}, \qquad c_\alpha, c'_\alpha \in \RR,
$$
the breakdown is determined by the smallest contamination level that can push a relevant interval bound across $\theta_0$. Consequently, the hypothesis testing breakdown point can be bounded by means of the $m$-sensitivities of $\hat\theta$ and $\hatse$.
This viewpoint is especially useful in cases where the $m$-sensitivities admit closed-form expressions or sharp upper bounds, as in Corollary \ref{cor:opt_attack_m_est_main}, Corollary \ref{cor:scale_main}, Proposition \ref{prop:two_stage_lb_main}, Lemma \ref{lem:se_at_theta_0_main}, and Lemma \ref{lem:opt_attack_m_est_se}. In such settings, the corresponding test breakdown points can be computed explicitly. In particular, if the test uses a fixed null standard deviation $\sigma_0$ rather than a plug-in standard error $\hatse(\cdot)$, then the standard-error $m$-sensitivity vanishes, and the resulting breakdown formulas become exact. The precise technical bounds, together with application to common robust tests, including one-sample Wald-type tests, two-sample tests, score-type tests, and two-stage M-estimate tests, are given in Appendix~\ref{sec:tests}.
\subsection{Numerical Illustration: Wald-type and Score-type Tests}
\label{sec:NI_wald}
We follow the common simulation setup in Section \ref{sec:NI_location} for the choice of losses and tuning.
We study one-sample two-sided tests of $H_0:\theta=0$ at level $\alpha$ based on (i) the regular Wald statistic, (ii) the restricted Wald statistic, (iii) the regular score test, and (iv) the restricted score test, as specified in Appendix~\ref{sec:tests} (see also Appendix~\ref{sec:score} for score test details).
To evaluate $\BP_{\mathrm{reject}}(\phi,x^{(n)})$, we generate data from $\mathcal N(\theta,1)$ over a grid of values of $\theta\in[-2,2]$.
For each replication, we repeatedly draw  a sample
until the realized dataset satisfies $\phi(x^{(n)})=1$ (rejects), and then compute the upper and lower bounds on $\BP_{\mathrm{reject}}(\phi,x^{(n)})$ given by Theorem~\ref{thm:meta} (and its specialization in Appendix~\ref{sec:tests}).
We use $n\in\{500,1000,2000\}$.
\begin{figure}[htbp]
  \centering
  \begin{subfigure}[b]{0.45\textwidth}
    \centering
    \includegraphics[width=\linewidth]{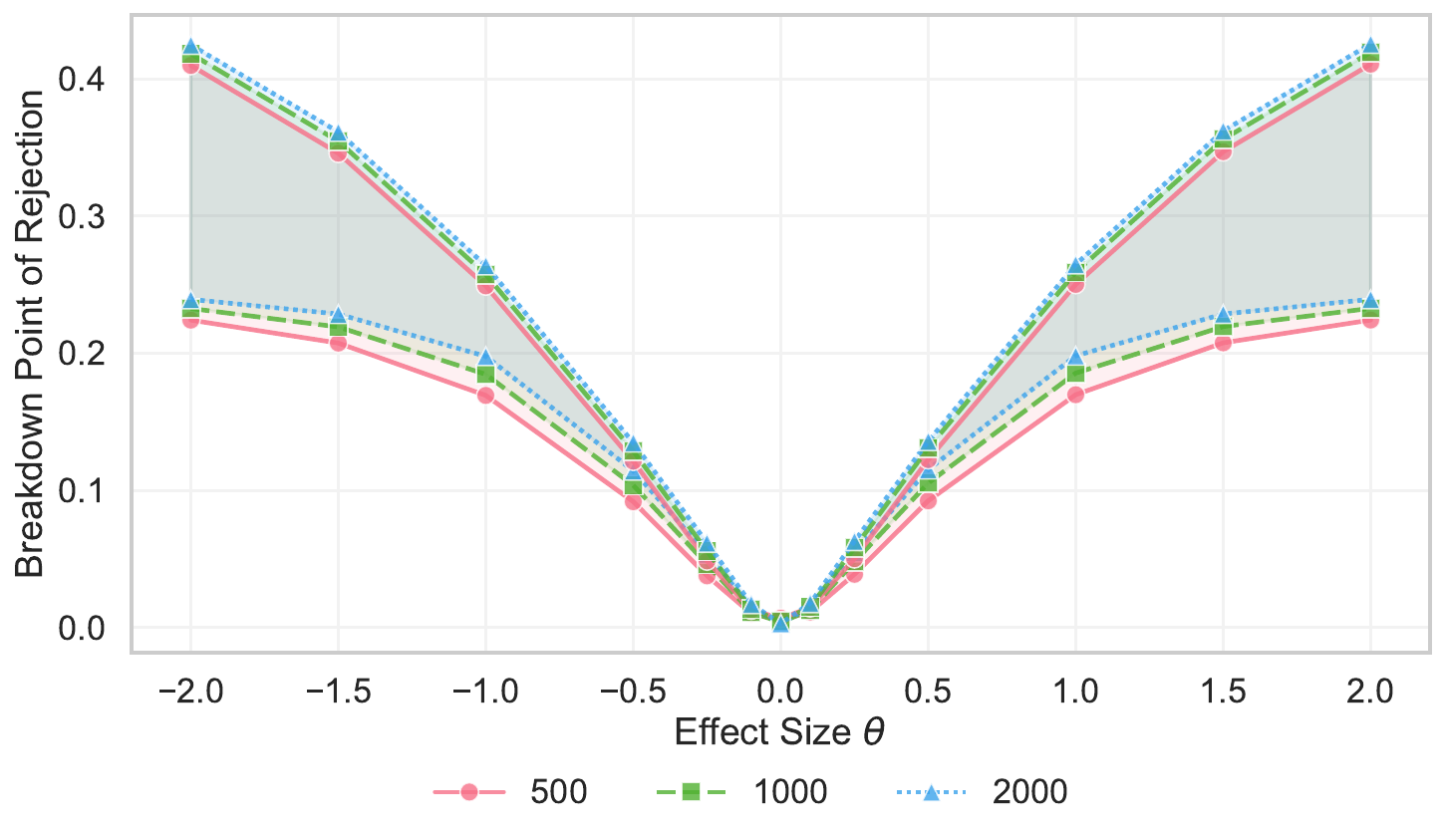}
    \caption{Regular Wald test.}
    \label{fig:huber_reg_ratio}
  \end{subfigure}
  \hfill
  \begin{subfigure}[b]{0.45\textwidth}
    \centering
    \includegraphics[width=\linewidth]{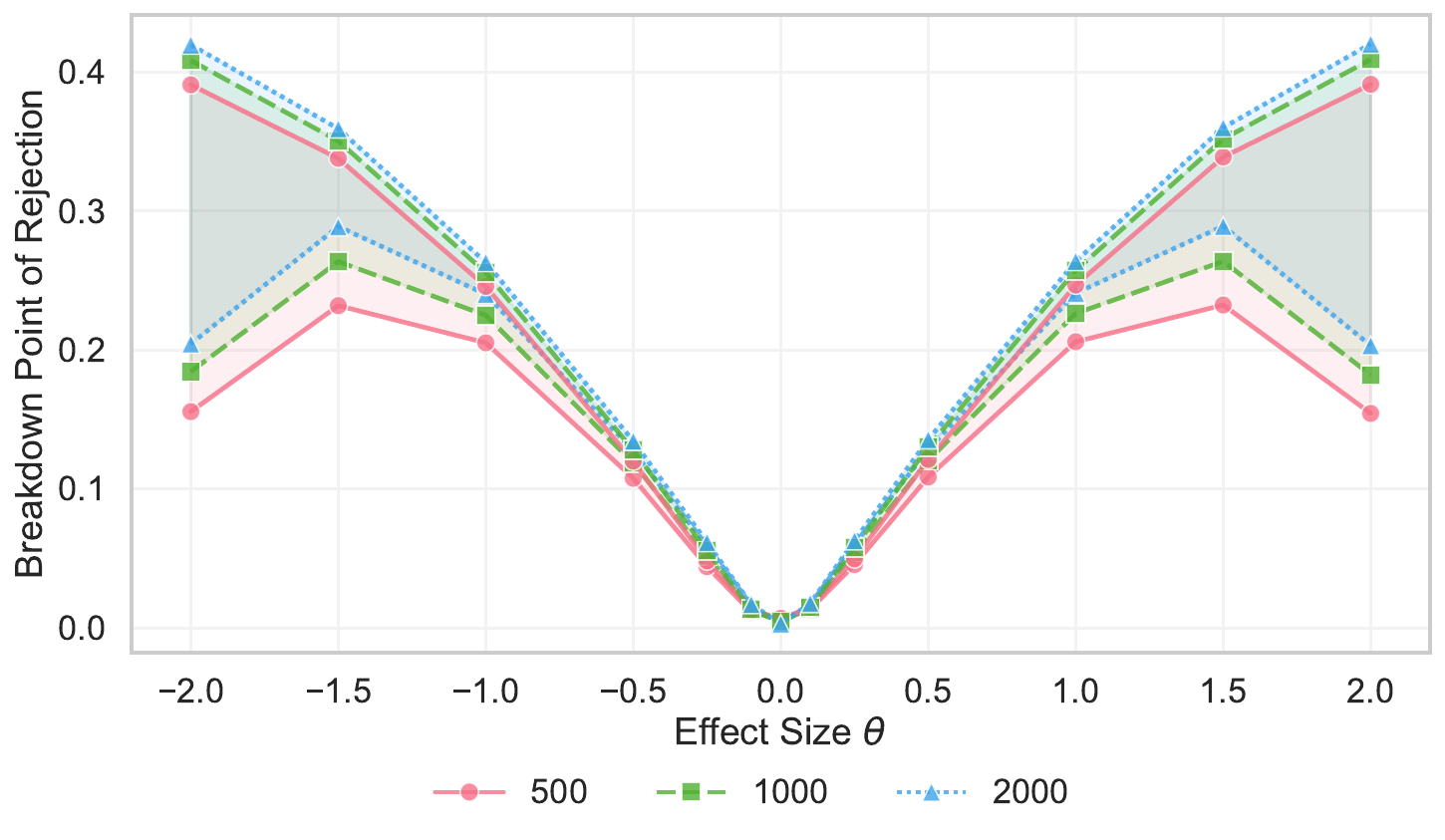}
    \caption{Restricted Wald test.}
    \label{fig:huber_restricted_ratio}
  \end{subfigure}
  \vspace{1em}
  \begin{subfigure}[b]{0.45\textwidth}
    \centering
    \includegraphics[width=\linewidth]{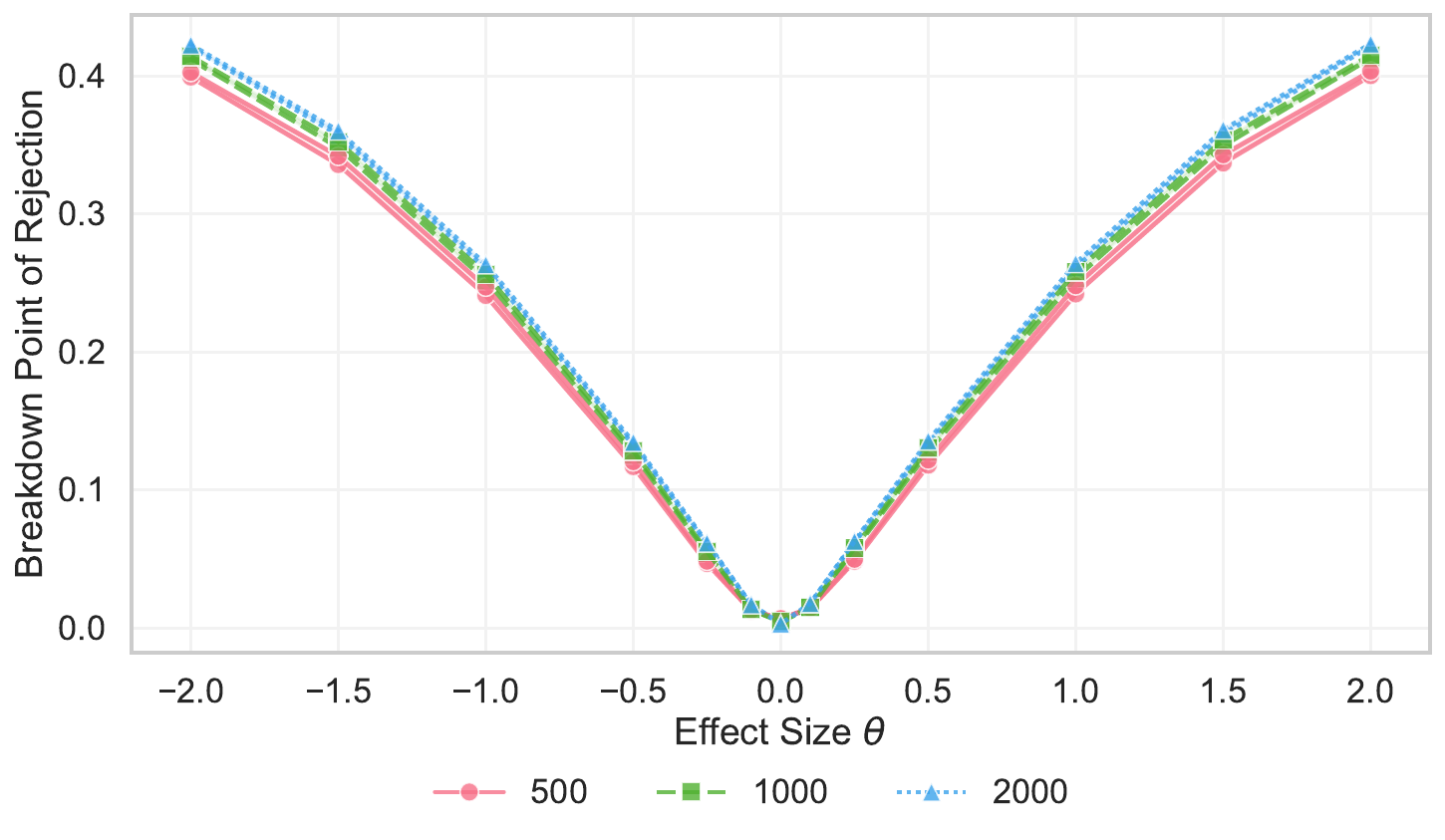}
    \caption{ \small Regular score test.}
    \label{fig:score_ratio}
  \end{subfigure}
  \hfill
  \begin{subfigure}[b]{0.45\textwidth}
    \centering
    \includegraphics[width=\linewidth]{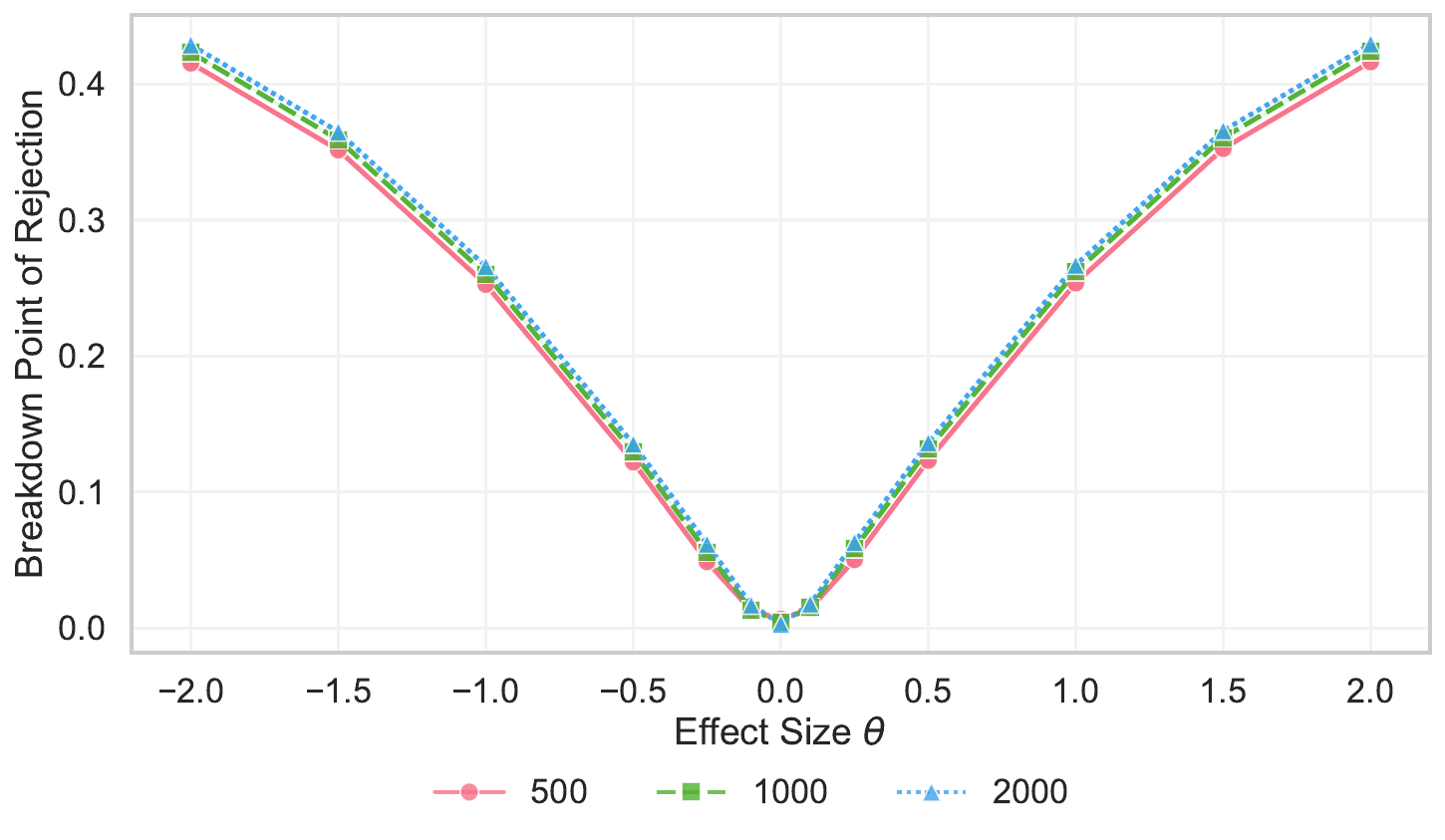}
    \caption{ \small Restricted score test.}
    \label{fig:score_restricted_ratio}
  \end{subfigure}
  \caption{ \small Upper and lower bounds for the rejection breakdown point $\BP_{\mathrm{reject}}(\phi,x^{(n)})$ as a function of the effect size $\theta$ under $\mathcal N(\theta,1)$.
  Line types distinguish $n\in\{500,1000,2000\}$; the shaded band indicates the interval between the lower and upper bounds.
  }
  \label{fig:bp_test_ratio}
\end{figure}
Across all four procedures, Figure~\ref{fig:bp_test_ratio} shows a pronounced V-shape in $\theta$: the rejection breakdown point is smallest near $\theta=0$ and increases with $|\theta|$, reflecting that stronger signals are harder to overturn by a fixed contamination level.
Increasing $n$ slightly raises $\BP_{\mathrm{reject}}(\phi,x^{(n)})$ while preserving the overall shape, indicating improved robustness at larger samples.
Finally, the bound gap is smallest around weak signals, where the inequalities in Theorem~\ref{thm:meta} are empirically tightest for these procedures. The contrast between the score and Wald tests is informative. For the score test, the band is very narrow, indicating that the upper and lower controls from our sensitivity analysis are closely aligned and thus provide a sharp characterization of finite sample rejection breakdown. For the Wald test, the band is visibly wider. This is consistent with the fact that contamination affects the procedure through both the point estimate and the standard error. Furthermore,  our current analysis controls these two effects separately and is therefore a little crude. Figure~\ref{fig:bp_test_ratio} suggests that our sensitivity-based approach is especially sharp for the score test, and somewhat more conservative for the Wald test. Further numerical results in Appendix \ref{sec:NI_wald_appendix} show that in fact, for the Wald test, this gap decreases as the sample size grows.

\subsection{Connection to Power and Level Breakdown Functions}
\label{sec:pbp-connection}

\citet{he:simpson:portnoy1990} introduced population power and level breakdown
functions, while \citet{zhang1996} studied finite sample breakdown points
defined through the observed test decision. These perspectives motivate a
population interpretation of the rejection and acceptance breakdown points in Section~\ref{sec:BPoftests}. We show that, under uniform approximation of the interval endpoints and
strict crossing conditions, the sample breakdown points converge almost
surely to the smallest population replacement fractions that reverse the
decision. This gives a common sufficient criterion for interval-based tests
and accounts for contamination of both the center and the width of the
interval. 
Appendix~\ref{sec:def_pbp} proves this connection between the sample and population problems.

Let \(X_1,X_2,\ldots\overset{\mathrm{iid}}{\sim}F\), fix a null value
\(\theta_0\) and \(0<\bar\varepsilon<\rho<1\), and let \(\mathcal P\)
denote all probability measures on \(\RR\). Consider the intervals and tests
\begin{align}
 C_n(y^{(n)})&=[L_n(y^{(n)}),U_n(y^{(n)})],
 &\phi_n(y^{(n)})&=\II\{\theta_0\notin C_n(y^{(n)})\},
 \notag\\
 C(H)&=[L(H),U(H)],
 &\phi(H)&=\II\{\theta_0\notin C(H)\},\qquad H\in\mathcal P.
 \label{eq:population_interval}
\end{align}
All intervals are closed, nonempty, and have finite endpoints.
For an indexed sample \(y^{(n)}=(y_1,\ldots,y_n)\), let
\(\widehat F_y=n^{-1}\sum_{i=1}^n\delta_{y_i}\) denote its empirical law.
Define
\begin{equation}\label{eq:replacement_neighborhood}
 \mathcal R_\varepsilon(F)=\{H\in\mathcal P:d_{\mathrm{TV}}(F,H)\le\varepsilon\},
 \qquad
 d_{\mathrm{TV}}(P,Q)=\sup_{A\text{ Borel}}|P(A)-Q(A)|.
\end{equation}
The population decision thresholds are
\begin{align}
 \varepsilon_{\mathrm{reject}}(F)
 &=\inf\left\{\varepsilon\in[0,1]:
       \exists H\in\mathcal R_\varepsilon(F),
       \ \theta_0\in C(H)\right\},\label{eq:population_rejection}\\
 \varepsilon_{\mathrm{accept}}(F)
 &=\inf\left\{\varepsilon\in[0,1]:
       \exists H\in\mathcal R_\varepsilon(F),
       \ \theta_0\notin C(H)\right\}.
       \label{eq:population_acceptance}
\end{align}

Consider the following two conditions.
\begin{enumerate}
 \item[\textup{(C1)}] 
 There is a finite, nonnegative, nondecreasing function
 \(\omega:[0,\infty)\to[0,\infty)\), with
 \(\omega(t)\to0\) as \(t\downarrow0\), such that
 \begin{equation}\label{eq:population_regularity}
 \max\{|L(P)-L(Q)|,\ |U(P)-U(Q)|\}
 \le\omega(\distK(P,Q))
 \end{equation}
 whenever \(\distK(P,F)<\rho\) and \(\distK(Q,F)<\rho\), where \(\distK(P,Q)=\|P-Q\|_\infty\).
 \item[\textup{(C2)}] 
 Assume
 \begin{equation}\label{eq:sample_approximation}
 \sup_{y^{(n)}\in B_{\textsf{H}}(X^{(n)},\lfloor n\bar \varepsilon\rfloor)}
 \max\left\{
 |L_n(y^{(n)})-L(\widehat F_y)|,
 |U_n(y^{(n)})-U(\widehat F_y)|
 \right\}\xrightarrow{\mathrm{a.s.}}0.
 \end{equation}
\end{enumerate}
All functions and random extrema used below are assumed measurable.

\begin{theorem}\label{thm:conv}
In the setting above, suppose that (C1) and (C2) hold. For \(0\le\varepsilon\le\bar\varepsilon\), set
\begin{align*}
 \underline L_\varepsilon(F)
 =\inf_{H\in\mathcal R_\varepsilon(F)}L(H),&
 \qquad
 \overline L_\varepsilon(F)
 =\sup_{H\in\mathcal R_\varepsilon(F)}L(H),\\
\underline U_\varepsilon(F)
 =\inf_{H\in\mathcal R_\varepsilon(F)}U(H),&
 \qquad
 \overline U_\varepsilon(F)
 =\sup_{H\in\mathcal R_\varepsilon(F)}U(H).
 \end{align*}

\medskip\noindent\textbf{(i) Rejection breakdown.}
Suppose \(\theta_0\notin C(F)\), and assume there is a continuous statistic
\(\widehat\theta_n:\RR^n\to\RR\) satisfying
\begin{equation}\label{eq:continuous_center}
 \widehat\theta_n(y^{(n)})\in C_n(y^{(n)}).
\end{equation}
Suppose some \(r_{\mathrm{reject}}\in(0,\bar\varepsilon)\) satisfies the
following crossing conditions:
\begin{equation}\label{eq:rejection_crossing_right}
\resizebox{0.92\linewidth}{!}{
$ \displaystyle
 \begin{cases}
 \underline L_\varepsilon(F)>\theta_0,
       &0\le\varepsilon<r_{\mathrm{reject}},\\
 \underline L_\varepsilon(F)<\theta_0,
       &r_{\mathrm{reject}}<\varepsilon\le\bar\varepsilon,
 \end{cases} 
 \text{\quad if \(L(F)>\theta_0\), and \quad}
  \begin{cases}
 \overline U_\varepsilon(F)<\theta_0,
       &0\le\varepsilon<r_{\mathrm{reject}},\\
 \overline U_\varepsilon(F)>\theta_0,
       &r_{\mathrm{reject}}<\varepsilon\le\bar\varepsilon,
 \end{cases}
 \text{\quad if \(U(F)<\theta_0\)}.
 $
}
\end{equation}
Then the clean sample rejects eventually almost surely, and
\begin{equation*}
 \BP_{\mathrm{reject}}(\phi_n,X^{(n)})
 \xrightarrow{\mathrm{a.s.}}
 \varepsilon_{\mathrm{reject}}(F)=r_{\mathrm{reject}}.
\end{equation*}

\medskip\noindent\textbf{(ii) Acceptance breakdown.}
Suppose \(\theta_0\in C(F)\), and suppose some
\(r_{\mathrm{accept}}\in(0,\bar\varepsilon)\) satisfies
\begin{equation}\label{eq:acceptance_crossing}
 \begin{cases}
 \overline L_\varepsilon(F)<\theta_0<\underline U_\varepsilon(F),
       &0\le\varepsilon<r_{\mathrm{accept}},\\
 \overline L_\varepsilon(F)>\theta_0\ \text{or}\ \underline U_\varepsilon(F)<\theta_0,
       &r_{\mathrm{accept}}<\varepsilon\le\bar\varepsilon.
 \end{cases}
\end{equation}
Then the clean sample accepts eventually almost surely, and
\begin{equation*}
 \BP_{\mathrm{accept}}(\phi_n,X^{(n)})
 \xrightarrow{\mathrm{a.s.}}
 \varepsilon_{\mathrm{accept}}(F)=r_{\mathrm{accept}}.
\end{equation*}
\end{theorem}

For rejection, the crossing condition says that below the population
threshold all intervals remain on the original side of the null; above it,
the endpoint facing the null can cross it. For acceptance,
\eqref{eq:acceptance_crossing} says that below the threshold every interval
still contains the null in its interior; above it, some interval excludes
the null. 

\begin{remark}
For the usual Wald interval \eqref{eq:wald}, the population endpoints equal
the population location. With fixed tuning and
\(0<\bar\varepsilon<\rho<1/2\), the logcosh and self-concordant losses
used in Appendix \ref{sec:simulation_setting} satisfy (C1) and (C2). Their bounded scores keep the
possible locations bounded. Their positive derivatives keep the average
score slope uniformly away from zero. This gives uniform continuity of the
population location for (C1). It also gives uniformly \(O(n^{-1/2})\)
standard errors in \eqref{eq:se} for (C2). Huber's loss requires additional
control of the mass retained in its quadratic region to maintain this
slope bound.
\end{remark}

\section{Asymptotic Theory and Bootstrap Inference}
\label{sec:stat_framework}
In previous sections, we established that threshold breakdown points and sensitivities provide a useful and finite sample way to quantify robustness of estimators and tests. These quantities look intrinsically combinatorial and dataset-specific as they are defined through worst-case contamination of $m$ data points.  This raises three basic questions: (i) under a probabilistic model, what population objects do they approximate? (ii) Can we perform classical inference on them? And (iii) can we assess their variability using  bootstrap procedures?
We address these questions by placing our finite sample robustness measures in an M-estimation framework. For this we need to  introduce one-sided versions of the threshold breakdown point and the $m$-sensitivity.
\begin{definition}
    \label{def:BP_directional}
    For $\eta>0$ and a data set $x^{(n)}$, the one-sided finite sample threshold breakdown point of $\hat\theta$ is
    \begin{align*}
     & \BP_{\eta+}(\hat\theta,x^{(n)}):= \frac{1}{n} \min\{m \in [n]: \exists y^{(n)}\in B_\textsf{H}(x^{(n)},m), \hat \theta(y^{(n)})-\hat\theta(x^{(n)})\ge\eta\}, \\
     & \BP_{\eta-}(\hat\theta,x^{(n)}):= \frac{1}{n} \min\{m \in [n]: \exists y^{(n)}\in B_\textsf{H}(x^{(n)},m), \hat\theta(x^{(n)})- \hat \theta(y^{(n)})\ge\eta\}.
    \end{align*}
\end{definition}
\begin{definition}
\label{def:eta_directional}
For $m\in[n]$, the one-sided $m$-sensitivities of $\hat\theta$ at $x^{(n)}$ are defined as
\begin{gather*}
    {\eta_{m/n\pm}(\hat\theta,x^{(n)}) := \sup\Bigl\{\eta>0:\BP_{\eta\pm}(\hat\theta,x^{(n)})\le\frac mn\Bigr\}.}
\end{gather*}
\end{definition}
We first identify their population counterparts and relate them to existing robustness notions in the literature. The key step is to introduce a coupled system of estimating equations tracking the location parameter, the $m$-sensitivity, a critical trimming quantile, and the breakdown fraction, so that  $(\eta_{m/n\pm},\BP_{\eta\pm})$ can be characterized by (approximate) solutions to this generalized system of estimating equations. This characterization allows us to invoke standard M-estimation theory to obtain consistency and asymptotic normality, and to justify a multiplier bootstrap that consistently approximates the resulting limiting distribution. As a consequence, $\eta_{m/n\pm}$ and $\BP_{\eta\pm}$ are no longer purely descriptive, dataset-specific diagnostics: they can be treated as estimators of well-defined population targets with  intrinsic sampling variability.
\subsection{Assumptions}
\label{sec:assumptions}
We will assume that $X_1,\dots,X_n \overset{\text{i.i.d.}}{\sim} F$ for some distribution $F$ and require the following assumptions in our theoretical analysis.
\begin{enumerate}
    \item[(A1)] Fix $\varepsilon \in (0,0.5)$ and set $m \in \{\lceil \varepsilon n\rceil,\, \lfloor \varepsilon n\rfloor\}$ with $m \le 0.5n$.
    \item[(A2)] The score $\psi:\RR\to\RR$ is odd, nondecreasing, {$L$-Lipschitz}, and $\|\psi\|_\infty  = B \in \mathbb{R}_{\geq 0}$.
\end{enumerate}
Assumption (A1) fixes the asymptotic contamination level as $m\approx \varepsilon n$  with $\varepsilon\in(0,0.5]$ so that a limiting population quantity is well-defined.
Denote  the $\varepsilon$-quantile of $F$ as
$
q_\varepsilon := F^{-1}(\varepsilon)
$ and let $\theta_0$ be the zero for $\EE_F[\psi(X-\theta)]$, which exists under (A2).
We require the following technical condition to avoid ill-posed population $m$-sensitivities.
\begin{enumerate}
    \item[(A3)] $F$ has a positive density $f$ at a small neighborhood around $q_\varepsilon$ and $q_{1-\varepsilon}$.
\end{enumerate}
Equipped with (A3), we can  define
\begin{align}
\label{eq:population_eta+}
    H_{+}(\eta)
    &= (1-\varepsilon)\,\EE_F[\psi\big(X-(\theta_0+\eta)\big)\,\big|\,X>q_\varepsilon] + \varepsilon \|\psi\|_\infty, \\ \label{eq:population_eta-}
    H_{-}(\eta)
    &= (1-\varepsilon)\,\EE_F[\psi\big(X-(\theta_0-\eta)\big)\,\big|\,X<q_{1-\varepsilon}] - \varepsilon\|\psi\|_\infty.
\end{align}
Because $H_{\pm}(\eta)$ have different signs when $\eta$ goes to $\pm \infty$, under (A2), there exist $\eta_{\varepsilon\pm}\in\RR$ such that $H_{\pm}(\eta_{\varepsilon\pm})=0$. The uniqueness and regularity at these zeros is guaranteed by our last condition.
\begin{enumerate}
    \item[(A4)] $\EE_F[\psi'(X-\theta_0)]>0$, $H'_+(\eta_{\varepsilon+})<0$ and $H'_-(\eta_{\varepsilon-})>0$, where $\psi'$ is an a.e.\ derivative.
\end{enumerate}
\subsection{Population Sensitivity and the Maxbias Curve}
\label{sec:maxbias}
The next result summarizes the regularity of the induced population maps $\varepsilon\mapsto\eta_{\varepsilon\pm}$.
\begin{proposition}
\label{prop:sensitivity_map}
    Suppose (A2)-(A4) is true for $\varepsilon \in (0, 0.5)$, then the solution map for \eqref{eq:population_eta+} and \eqref{eq:population_eta-} $\varepsilon \mapsto \eta_{\varepsilon\pm}$ is $C^1((0, 0.5), \mathbb{R}_{\geq 0})$ and strictly increasing.
    Specifically, the solution maps $\eta_{\varepsilon\pm}  = \eta_\pm(\varepsilon)$ for $\varepsilon \in [0, 0.5)$ are characterized by the following ordinary differential equations (ODEs):
    \begin{equation*}
  \eta_+(\varepsilon) = \int_{0}^\varepsilon \frac{ -\psi(q_{t} - \theta_0 - \eta_+(t)) + \|\psi\|_\infty}{\int_{q_{t}}^{\infty} \psi'(x - \theta_0 - \eta_+(t)) \dd F(x)} \dd t, \quad \eta_-(\varepsilon) = \int_{0}^\varepsilon \frac{ \psi(q_{1-t} - \theta_0 + \eta_-(t)) + \|\psi\|_\infty}{\int_{-\infty}^{q_{1-t}} \psi'\big(x-\theta_0+\eta_-(t)\big) \dd F(x)} \dd t.
    \end{equation*}
\end{proposition}
The ODEs arise by differentiating the population equations $H_{\pm}(\eta_{\varepsilon\pm})=0$ in $\varepsilon$ and applying the implicit function theorem, using (A4) to ensure that the denominators are strictly positive. The monotonicity reflects the fact that, at the population level, increasing the contamination fraction $\varepsilon$ can only increase the worst-case displacement.
Proposition \ref{prop:sensitivity_map} shows that the population $m$-sensitivity and contamination ratio are linked by a one-to-one, smooth transformation. This is the same object that appears in the classical robust statistics literature under the name of the \emph{maxbias curve}: it describes the growth of the largest population bias as contamination increases, and equivalently determines the minimal contamination level required to reach a given displacement threshold \citep{martin:yohai:zamar1989,martin:zamar1989, rousseeuw1999maxbias, croux2002maxbias}.
Existing analyses typically derive this curve under parametric assumptions on $F$, such as normality or a symmetric location-scale family. Our formulation, however, works for general $F$ and characterizes the maxbias curve for M-estimators directly through the distribution $F$ and the score $\psi$.
\subsection{Coupled M-estimation for Sensitivity and Breakdown}
\label{sec:coupled Z}
We will show that the M-estimator of interest, its sensitivity and breakdown point can be expressed in a single system of equations. To do this we embed the location $\theta$, $m$-sensitivity $\eta$, trimming quantile $q$, and breakdown fraction $\varepsilon$ into the parameter
$
\vartheta = (\theta,\eta,q,\varepsilon)
$
and define the following score functions $\Psi_\pm:\mathcal{X}\times (\Theta\times \RR_{>0} \times \mathbb{R} \times (0, 0.5])\to\mathbb{R}^3$
\begin{align}
\label{eq:Z-system}
\resizebox{0.92\linewidth}{!}{
     $ \displaystyle
        \Psi_{+}(x; \vartheta)
        =
        \begin{pmatrix}
        \psi(x-\theta)\\
        \psi(x-\theta-\eta)\II\{x > q\} + \varepsilon B \\
        \II\{x\le q\}- \varepsilon
        \end{pmatrix},
        \Psi_{-}(x;\vartheta )
        =
        \begin{pmatrix}
        \psi(x-\theta)\\
        \psi(x-\theta+\eta)\II\{x < q\} - \varepsilon B \\
        \II\{x\le q\}- (1- \varepsilon)
        \end{pmatrix}.
    $}
\end{align}
At the population level, we consider the estimating equation functionals
$
\EE[\Psi_{\pm}(X; \vartheta)] = 0,
$
whose solutions $(\theta,\eta,q,\varepsilon)$ encode, in particular, the population location $\theta_0$, the trimmed sensitivities $\eta_{\varepsilon\pm}$ in \eqref{eq:population_eta+}–\eqref{eq:population_eta-}, the corresponding quantiles, and the breakdown $\varepsilon$.
Note that the parameter $\vartheta$ is four-dimensional while $\Psi_{\pm}(x;\vartheta)$ has three components. This deliberate over-parameterization encodes $m$-sensitivity and breakdown point in a single coupled system. More specifically, the next two propositions give the deterministic root characterization underlying the later random-sample analysis. For a fixed dataset $x^{(n)}$, we show that these two quantities can be represented as exact or approximate solutions of the estimating equations \eqref{eq:Z-system}. Throughout, write $F_n(t) := \tfrac{1}{n} \sum_{i=1}^n \II\{x_i \le t\}$. For $m \in [n]$, let $\hat q_{m+} \in [x_{(m)}, x_{(m+1)})$, and {$\hat q_{m-}\in(x_{(n-m)},x_{(n-m+1)})$} with the convention of $x_{(0)} = -\infty$ and $x_{(n+1)} = \infty$.
\begin{proposition}
\label{prop:finite_identification_sensitivity}
Fix $m \in [n]$ with $m \le 0.5n$. Let $\eta_{m/n\pm}(\hat \theta, x^{(n)})$ be  as in Definition~\ref{def:eta}, where $\hat \theta$ uniquely solves \eqref{M-est}. Assume $x_{(m)} < x_{(m+1)}$, $x_{(n-m)} < x_{(n-m+1)}$, and that (A2) holds. Define
$
\hat {\vartheta}_{\pm}^{\mathrm{sen}}
:=
\bigl(\hat \theta, \eta_{m/n\pm}(\hat \theta, x^{(n)}), \hat q_{m\pm}, m/n\bigr).
$
Then,
$
\EE_{F_n}\bigl[\Psi_{\pm}(X; \hat {\vartheta}_{\pm}^{\mathrm{sen}})\bigr] = 0.
$
\end{proposition}
\begin{proposition}
\label{prop:finite_identification_BP}
Fix $\eta>0$, let $\BP_{\eta \pm}(\hat \theta, x^{(n)})$ be  as in Definition~\ref{def:BP_directional}, where $\hat \theta$ uniquely solves \eqref{M-est}. Define $k_{\eta\pm} := n \cdot \BP_{\eta\pm}(\hat\theta, x^{(n)})$ and assume $x_{(k_{\eta+})} < x_{(k_{\eta+}+1)}$, $x_{(n-k_{\eta-})} < x_{(n-k_{\eta-}+1)}$, and that (A2) holds. Define
$
\hat {\vartheta}_{\pm}^{\mathrm{BP}}
:=
\bigl(\hat \theta, \eta, \hat q_{k_{\eta\pm}\pm}, \BP_{\eta \pm}(\hat \theta, x^{(n)})\bigr).
$
Then, there exists $r_{n\pm} \in [-2B/n, 2B/n]$ such that
$
\EE_{F_n}\bigl[\Psi_{\pm}(X; \hat {\vartheta}_{\pm}^{\mathrm{BP}})\bigr]
=
(0, r_{n\pm}, 0).
$
\end{proposition}
Naturally, if we move from a deterministic $x^{(n)}$ to a random dataset $X^{(n)}$ generated by $F$, we can invoke standard M-estimation machinery to analyze $m$-sensitivity and the threshold breakdown point. We formalize this intuition in Section \ref{sec:asymptotics}.
\subsection{Multiplier Bootstrap}
\label{sec:multiplier_bootstrap}
To approximate sampling variability for estimation procedures, we use the multiplier bootstrap of \citep{jin2001simple}.
Let $W_1,\ldots,W_n$ be i.i.d.\ nonnegative weights, independent of $X^{(n)}$, such that
\begin{align}
\label{eq:boostrap_condition}
    \EE[W_1]=1,\qquad \Var(W_1)=1,\qquad \EE[W_1^{2+\kappa}]<\infty,
\end{align}
for some $\kappa>0$.
Write $\bar W := n^{-1}\sum_{i=1}^n W_i$, and define
$
\tilde F_{n}(x):= \frac{1}{n}\sum_{i=1}^n \frac{W_i}{\bar W}\,\II\{X_i \le x\}
$.
For the location M-estimator, a bootstrap replicate $\hat\theta_b$ is defined as any solution of
\begin{equation}
\EE_{\tilde F_n}\bigl[\psi(X-\hat\theta_b)\bigr] =\sum_{i=1}^n W_i\,\psi(X_i-\hat\theta_b)=0.
\label{eq:bootstrap}
\end{equation}
We use the same estimating equations introduced in Section~\ref{sec:coupled Z} to define bootstrap versions of $\hat {\vartheta}_{\pm}^{\mathrm{BP}}$ or $\hat {\vartheta}_{\pm}^{\mathrm{sen}}$  by replacing
$F_n$ with $\tilde F_n$
in the expectation.
Since the resulting multiplier bootstrap draws produce weighted empirical measures rather than samples obtained by replacing $m$ observations, the  substitution of $F_n$ with $\tilde F_n$ leads to awkward finite sample  threshold breakdown points and $m$-sensitivities. To address this, Appendix~\ref{sec:generalization_bootstrap} extends the definitions of threshold breakdown point and $m$-sensitivity to weighted empirical measures via total-variation neighborhoods which allows to split mass of the multiplier bootstrap distributions $\tilde F_n$ in order to keep comparable definitions of $m$-sensitivities. In particular, Appendix~\ref{sec:generalization_bootstrap} shows that the resulting bootstrap quantities induced by
 $\tilde F_n$
do admit the intended robustness interpretation, i.e., they lead to bootstrap threshold breakdown point taking values on the grid $\{1/n,\dots,\lceil n/2 \rceil/n\}$, just as the original breakdown point in the unweighted case.  We show that these extensions satisfy analogues of Proposition~\ref{prop:finite_identification_sensitivity}, Proposition~\ref{prop:finite_identification_BP}, and Proposition~\ref{prop:sensitivity_map_finite_main}; see Appendix~\ref{sec:generalization_bootstrap} and Appendix~\ref{sec:identification_proof}-\ref{sec:normality_proof} for details.
\subsection{Consistency, Asymptotic Normality, and Bootstrap Inference}
\label{sec:asymptotics}
We use our generalized M-estimation formulation  to give asymptotic results for both the $m$-sensitivity $\eta_{m/n\pm}$ and the threshold breakdown point $\BP_{\eta\pm}$.
Assumptions (A2) and (A4) ensure that for each $\varepsilon\in(0,0.5]$
there exist unique solutions to $\EE_F[\Psi_{\pm}(X; \vartheta)]=0$ given by
$$
\vartheta_{0+} = (\theta_0,\eta_{\varepsilon+},q_\varepsilon,\varepsilon),
\qquad
\vartheta_{0-} = (\theta_0,\eta_{\varepsilon-},q_{1-\varepsilon},\varepsilon),
$$
where $\theta_0$ and $\eta_{\varepsilon\pm}$ are the same as in Section~\ref{sec:assumptions}.  We first establish a consistency result.
\begin{theorem}
\label{thm:eta_consistency_inP}
Let $X_1,\dots,X_n \overset{\text{i.i.d.}}{\sim} F$ and suppose (A1)-(A4) hold.\footnotemark{} Let $\hat \theta$ be defined by \eqref{M-est} and remind $\eta_{m/n\pm}(\hat \theta, X^{(n)})$ is defined in Definition \ref{def:eta}. Then,
$
\eta_{m/n\pm}(\hat\theta,X^{(n)}) \overset{p}{\to} \eta_{\varepsilon\pm},
$
where $\eta_{\varepsilon \pm}$ is the unique zero to \eqref{eq:population_eta+} and \eqref{eq:population_eta-}.
\end{theorem}
\footnotetext{A weaker set of sufficient conditions for consistency of M-estimators can be found in \citet[Theorem 5.9]{van2000asymptotic}. We retain (A1)-(A4) here because they are convenient for the later arguments.}
Since for a fixed $\eta$ the one-sided threshold breakdown point is essentially an empirical quantile of the one-sided $m$-sensitivity, Theorem \ref{thm:eta_consistency_inP} also suggests that the threshold breakdown point converges to the corresponding population contamination level $\varepsilon^*$ solving $H_{\pm}(\eta_{\varepsilon^*\pm})=0$. This is formalized in the next result.
\begin{corollary}
\label{cor:BP_consistency}
    Let $X_1,\dots,X_n \overset{\text{i.i.d.}}{\sim} F$ and suppose a given $\eta$ is the solution to \eqref{eq:population_eta+} or \eqref{eq:population_eta-} for some $\varepsilon^* \in (0, 0.5]$, which we denote as $H_{\pm}(\eta) := H_{\pm}(\eta_{\varepsilon^*\pm}) = 0$. Suppose (A2)-(A4) hold for $\varepsilon^*$. Then,
    $
    \BP_{\eta\pm}(\hat \theta, X^{(n)})\overset{p}{\to} \varepsilon^*.
    $
\end{corollary}
We strengthen these consistency results by establishing asymptotic normality and consistency of the multiplier bootstrap. We denote by $\tilde {\vartheta}_{\pm}^{\mathrm{sen}} = (\tilde \theta_\pm,\tilde \eta_\pm, \tilde q_\pm, m/n)$ the estimators obtained by solving the expected estimating equations \eqref{eq:Z-system} with respect to $\tilde F_n$, {up to an error of $o_p(n^{-1/2})$}.
\begin{theorem}
\label{thm:bootstrap_normality_main}
    Let $X_1,\dots,X_n \overset{\text{i.i.d.}}{\sim} F$  and suppose (A1)-(A4) hold.
    Then,
   $$
\sqrt{n}\bigl(\eta_{m/n\pm}(\hat \theta, X^{(n)}) - \eta_{\varepsilon \pm} \bigr)
    {\ \rightsquigarrow \ }
    Z \sim
    \mathcal{N}\left(0,
    V_\pm
    \right)
    \text{ and }
    \sqrt{n}\bigl((\tilde \eta_{\pm} - \eta_{m/n\pm}(\hat \theta,  X^{(n)})\bigr)
    \overunderset{p}{W}{\rightsquigarrow}
    Z,
    $$
    where $V_\pm$ is given in \eqref{eq:var_eta+}-\eqref{eq:var_eta-}, and ${\rightsquigarrow}$ and $\overunderset{p}{W}{\rightsquigarrow}$ denote weak and  conditional weak convergence.
\end{theorem}
Similar to Theorem~\ref{thm:bootstrap_normality_main}, we obtain the asymptotic normality and bootstrap validity for the threshold breakdown point itself.
We denote by $\tilde {\vartheta}_{\pm}^{\mathrm{BP}} = (\tilde \theta_\pm, \eta, \tilde q_\pm, \tilde \varepsilon_\pm)$ the estimators obtained by solving the expected estimating equations \eqref{eq:Z-system} with respect to $\tilde F_n$, {up to an error of $o_p(n^{-1/2})$}.
\begin{theorem}
\label{thm:BP_normality_main}
    Let $X_1,\dots,X_n \overset{\text{i.i.d.}}{\sim} F$  and suppose a given $\eta$ is a solution to \eqref{eq:population_eta+} or \eqref{eq:population_eta-} for some $\varepsilon^* \in (0, 0.5)$, which we denote as $H_{\pm}(\eta) := H_{\pm}(\eta_{\varepsilon^*\pm}) = 0$. Suppose (A2)-(A4) hold for $\varepsilon^*$.
    Then,
    $$
    \sqrt{n}\bigl(\BP_{\eta \pm}(\hat \theta, X^{(n)}) - \varepsilon^* \bigr)
    {\ \rightsquigarrow \ }
    Z \sim
    \mathcal{N}\left(0,
    \sigma^2_{\pm}
    \right)
    \text{ and }
    \sqrt{n}\bigl(\tilde \varepsilon_\pm - \BP_{\eta \pm}(\hat \theta, X^{(n)})\bigr)
    \overunderset{p}{W}{\rightsquigarrow}
    Z,
    $$
    where
    \begin{align*}
        \sigma^2_+ :=
        \resizebox{0.41\linewidth}{!}{$ \displaystyle  \left (\frac{\int_{q_{\varepsilon^*}}^{\infty} \psi'(x - (\theta_0 + \eta_{\varepsilon^*+})) \dd F(x)}{ -\psi(q_{\varepsilon^*} - (\theta_0 + \eta_{\varepsilon^*+})) + \|\psi\|_\infty} \right)^2 V_+, $}
        \quad
        \sigma^2_- :=
        \resizebox{0.41\linewidth}{!}{$
        \displaystyle
        \left (\frac{-\int_{-\infty}^{q_{1-\varepsilon^*}} \psi'\big(x-(\theta_0-\eta_{\varepsilon^*-})\big) \dd F(x)}{ -\psi(q_{1-\varepsilon^*} - (\theta_0 - \eta_{\varepsilon^*-})) - \|\psi\|_\infty} \right)^2 V_- $}
    \end{align*}
    for $V_\pm$ specified in \eqref{eq:var_eta+}-\eqref{eq:var_eta-}.
\end{theorem}
A natural follow-up question is to characterize the relationship between $m$-sensitivity and threshold breakdown point, as they arise from a coupled system of M-estimating equations. This is addressed by the following proposition.
\begin{proposition}
\label{prop:sensitivity_map_finite_main}
Suppose (A1)-(A4) hold for a fixed $\varepsilon^* \in (0, 0.5)$ with $\eta^*$ being zeros of \eqref{eq:population_eta+} or \eqref{eq:population_eta-}.  Further assume that
    $
    F\!\left(\left\{x\in\mathbb R:\psi' \text{ is discontinuous at } x-(\theta_0 \pm \eta_{\varepsilon^*\pm})\right\}\right)=0.
    $
    The threshold breakdown point of $\hat \theta$ at $X^{(n)}=(X_1,\dots,X_n)$ can be expressed as
    $$
        \BP_{\eta^*\pm}(\hat \theta, X^{(n)})- \varepsilon^* = -\frac{\dd\varepsilon}{\dd\eta_{\varepsilon\pm}} \Big |_{\varepsilon = \varepsilon^*} \cdot (\eta_{m/n\pm}(\hat \theta, X^{(n)}) - \eta^*) + o_p \bigg (\frac{1}{\sqrt{n}} \bigg),
    $$
    where the derivative satisfies the analytical expression found in equation  \eqref{eq:derivative}.
\end{proposition}
\begin{remark}
    Proposition~\ref{prop:sensitivity_map_finite_main} formalizes the intuition that $m$-sensitivity and threshold breakdown are two sides of the same coin: both inherit their first-order fluctuations from the same empirical process, and differ only by a deterministic rescaling given by the slope of the population maxbias curve.
\end{remark}
In the next section, we will illustrate how our asymptotic analysis can be used to assess the robustness of our statistical inference procedures in practical settings.
\section{Empirical Experiments}
\label{sec:empirical_experiment}
We analyze a dataset available in The Data and Story Library (\url{http://lib.stat.cmu.edu/DASL/Datafiles/Calcium.html}). It describes the blood pressure measurements for 21 African–American men with $10$ on calcium and $11$ on placebo. The goal is to test whether increasing calcium intake reduces blood pressure. We revisit the calcium trial, where $x^{(n_x)}$ denotes the calcium group and $y^{(n_y)}$ the placebo group, and $\hat\theta_{x^{(n_x)}}, \hat\theta_{y^{(n_y)}}$ are location estimates under the loss shown in each panel. We normalize each group separately using MAD estimator of scale tuned to achieve consistency at the normal model,
and take the tuning parameter $\delta$ as described in Appendix \ref{sec:simulation_setting} for different robust losses.
We want to test whether the location parameters are equal in both populations. In other words, we want to test $H_0:\theta_x=\theta_y$, where $\theta_x$ and $\theta_y$ denote the location parameters of the calcium and placebo groups respectively. For each discrete level $m\in\{0,1,\dots, \lceil\frac{1}{2} \min(n_x,n_y) \rceil\}$ we plot the $m$-sensitivity $\eta_{m/n}(\hat\theta_{x^{(n_x)}}-\hat\theta_{y^{(n_y)}})$ on the $y$-axis against $m/\min(n_x,n_y)$ on the $x$-axis in Figure \ref{fig:bp_two_sample_ratio_real_data}. Uncertainty bands are obtained from the bootstrap introduced in Section \ref{sec:multiplier_bootstrap}. To tighten the upper bounds for the contaminated standard error $\hatse$, in the real-data analysis we take
$
\mathcal A=\{a:\psi(a\delta)\in\{-0.1\delta,0,0.1\delta,\ldots,\delta\}\}
$
for the score function under consideration, and use this choice in the contamination schemes \eqref{eq:contamination_left_loc} and \eqref{eq:contamination_right_loc}.
The gray lines represent one-sided rejection boundaries computed by holding the pooled variance $\hatse$ fixed; explicitly, the boundary at level $\alpha$ is
$
z_{1-\alpha}\,\hatse (\hat \theta_{x^{(n_x)}}, \hat\theta_{y^{(n_y)}})-(\hat\theta_{x^{(n_x)}}-\hat\theta_{y^{(n_y)}}),
$
where $\hatse (\hat \theta_{x^{(n_x)}}, \hat\theta_{y^{(n_y)}})$ is the pooled standard error estimate. The original data do not reject at $m=0$. As the level increases, whenever the red $m$-sensitivity curve first crosses a gray line, the one-sided test at that $\alpha$ would change its decision. The corresponding $m/\min(n_x,n_y)$ is the test’s breakdown point. Across three robust losses, curvature reflects the robustness of the two-sample Wald-type tests, while all panels saturate at $m/\min(n_x,n_y)=0.5$, where complete breakdown is inevitable. As a concrete interpretation, consider Figure \ref{fig:two_sample_huber_ratio_real_data}. At $m=1$, equivalently $m/\min(n_x,n_y)=0.1$, the red $m$-sensitivity curve already crosses the $\alpha=0.05$ decision threshold, so contaminating a single observation will flip the one-sided test decision. By contrast, for the more stringent level $\alpha=0.001$, the upper end of the 95\% bootstrap confidence interval at $m=1$ still remains below the corresponding threshold, indicating that even after accounting for sampling variability, a single contaminated observation can not flip the test decision.
\begin{figure}[htbp]
  \centering
  \begin{subfigure}[b]{0.32\textwidth}
    \centering
    \includegraphics[width=\linewidth]{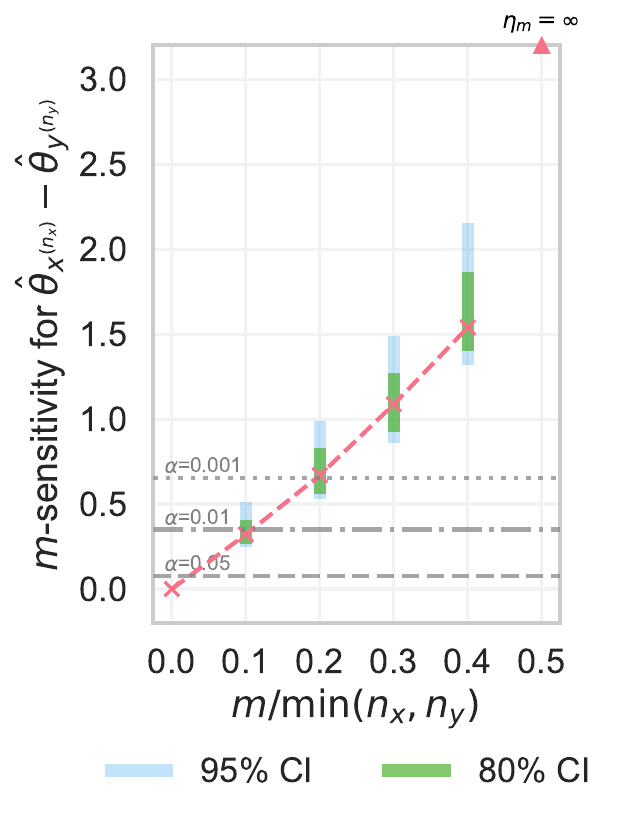}
    \caption{Huber's loss.}
    \label{fig:two_sample_huber_ratio_real_data}
  \end{subfigure}
  \hfill
  \begin{subfigure}[b]{0.32\textwidth}
    \centering
    \includegraphics[width=\linewidth]{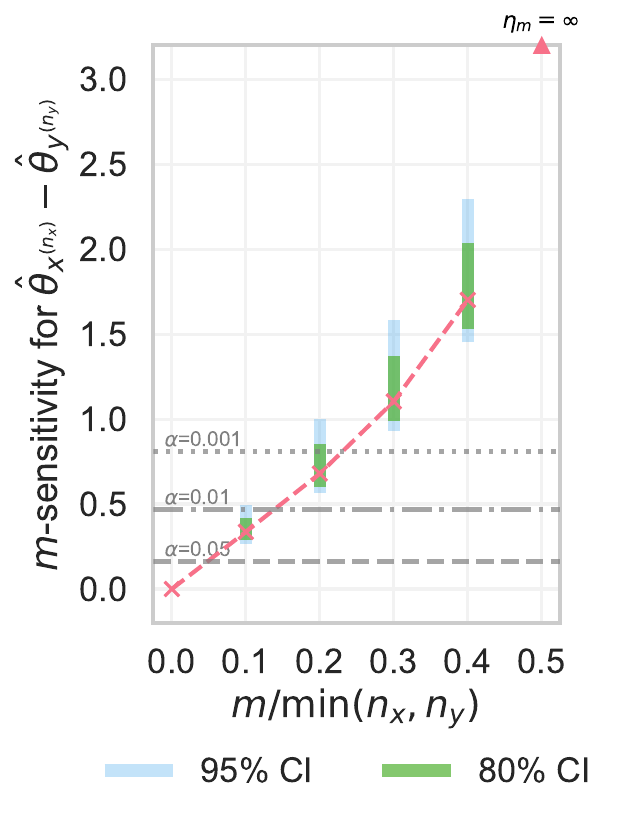}
    \caption{Logcosh loss.}
    \label{fig:two_sample_logcosh_ratio_real_data}
  \end{subfigure}
  \hfill
  \begin{subfigure}[b]{0.32\textwidth}
    \centering
    \includegraphics[width=\linewidth]{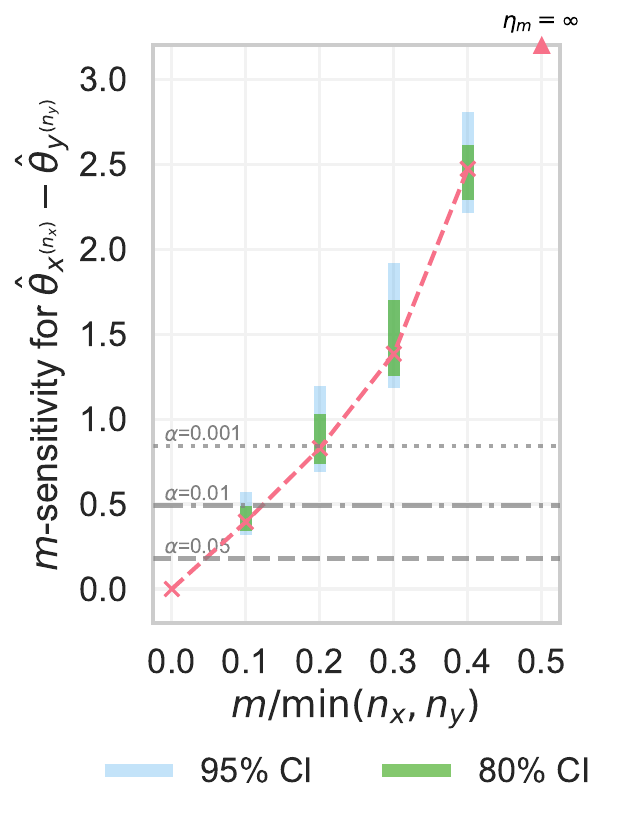}
    \caption{Self-concordant loss.}
    \label{fig:two_sample_concordant_ratio_real_data}
  \end{subfigure}
  \caption{\small Sensitivity curves for the calcium–placebo blood pressure data. Panels (a)–(c) report the $m$-sensitivity $\eta_{m/n}(\hat\theta_{x^{(n_x)}}-\hat\theta_{y^{(n_y)}})$ for Huber's, logcosh, and self-concordant losses, respectively. The horizontal axis is the discrete level $m/\min(n_x,n_y)$; at $m/\min(n_x,n_y)=0.5$ complete breakdown occurs, matching the classical finite sample regime. Red markers (and the dashed line connecting them) are the computed sensitivities on this dataset. Vertical green and blue bands are 80\% and 95\% bootstrap confidence intervals. Gray dashed horizontals are one-sided decision thresholds computed under a fixed pooled-variance estimate, equal to $z_{1-\alpha}\,\hatse(\hat \theta_{x^{(n_x)}}, \hat\theta_{y^{(n_y)}})-(\hat\theta_{x^{(n_x)}}-\hat\theta_{y^{(n_y)}})$; a crossing indicates the $\alpha$-level test would flip its decision.}
  \label{fig:bp_two_sample_ratio_real_data}
\end{figure}
\begin{figure}[htbp]
  \centering
  \begin{subfigure}[b]{0.32\textwidth}
    \centering
    \includegraphics[width=\linewidth]{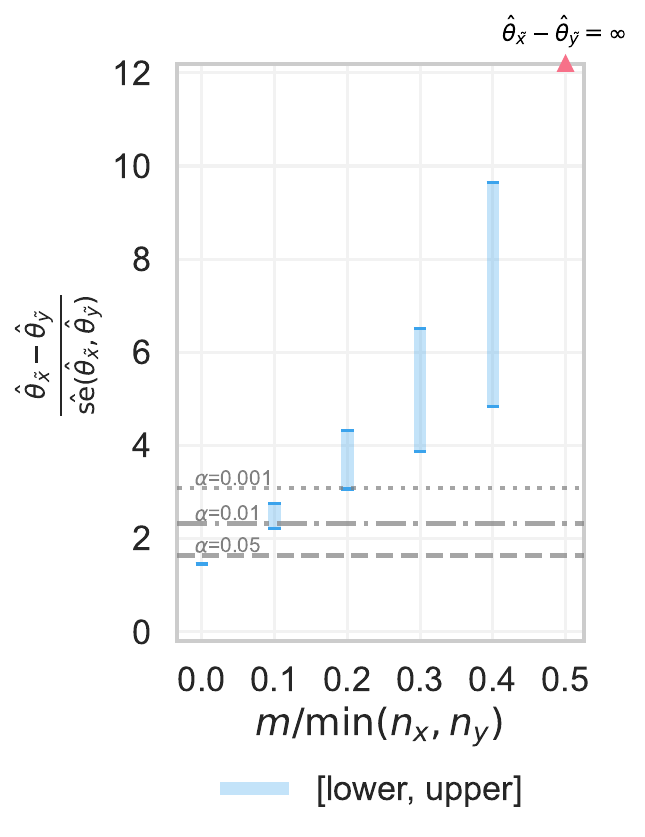}
    \caption{Huber's loss.}
    \label{fig:two_sample_test_huber_ratio}
  \end{subfigure}
  \hfill
  \begin{subfigure}[b]{0.32\textwidth}
    \centering
    \includegraphics[width=\linewidth]{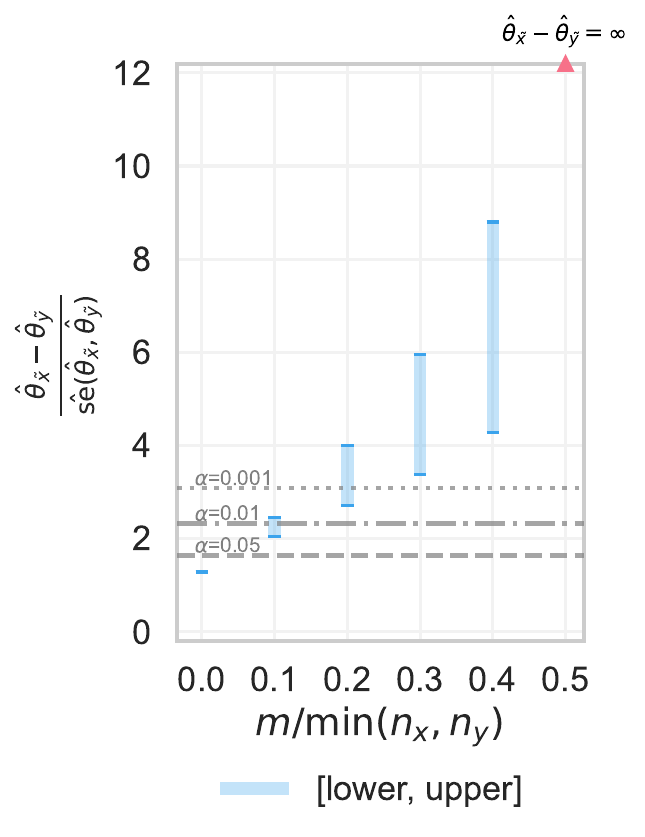}
    \caption{Logcosh loss.}
    \label{fig:two_sample_test_logcosh_ratio}
  \end{subfigure}
  \hfill
  \begin{subfigure}[b]{0.32\textwidth}
    \centering
    \includegraphics[width=\linewidth]{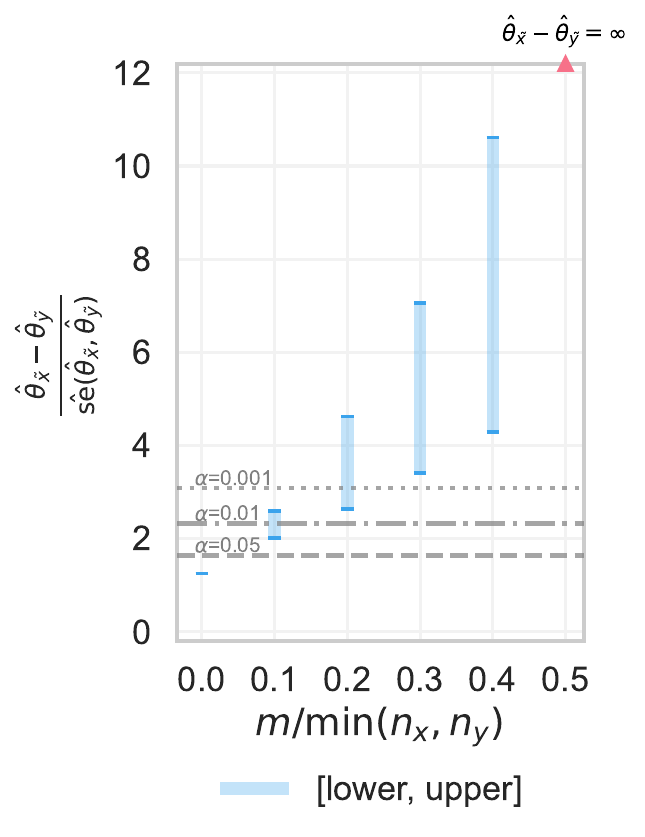}
    \caption{Self-concordant loss.}
    \label{fig:two_sample_test_concordant}
  \end{subfigure}
  \caption{ \small Standardized test statistic bands for the calcium-placebo blood pressure data. For each loss, we report with the light-blue bars, at level $m/\min(n_x,n_y)$, the lower bound and upper bound given by Theorem \ref{thm:meta} for
    $
    ({\hat\theta_{\tilde x}-\hat\theta_{\tilde y}})/{\hatse(\hat\theta_{\tilde x},\hat\theta_{\tilde y})},
    $
    where $\tilde x$ and $\tilde y$ denote the contaminated datasets at each level. Gray dashed lines are again one-sided thresholds $z_{1-\alpha}$. A bar lying entirely above a gray line certifies a flip of the $\alpha$-level test; a bar entirely below certifies no flip; straddling leaves the status unresolved. }
  \label{fig:test_two_sample_ratio_real_data}
\end{figure}
Complementing the $m$-sensitivity curves, Figure \ref{fig:test_two_sample_ratio_real_data} plots the standardized test statistic that allows contamination on both the location estimate and the pooled scale for the two samples. It is not possible to obtain an exact solution for the optimal contamination at each level, and we use the upper and lower bounds derived in Theorem \ref{thm:meta} (or more specifically, Corollary \ref{cor:two_sample_test}) to characterize the range of the test statistic under contamination. Note that the lower endpoint is attainable by construction, but the upper endpoint is a conservative bound. Gray dashed lines are again one-sided thresholds $z_{1-\alpha}$. A bar lying entirely above the line certifies a flip of the $\alpha$-level test; a bar entirely below certifies no flip; straddling leaves the status unresolved. As a concrete interpretation, also consider Figure \ref{fig:two_sample_test_huber_ratio}. At $m/\min(n_x,n_y)=0.1$, the lower end of the blue band already crosses the $\alpha=0.05$ decision threshold, so contaminating a single observation will flip the one-sided test decision. By contrast, for the more stringent level $\alpha=0.001$, the upper end of the blue band at $m=1$ still remains below the corresponding threshold, indicating a single contaminated observation is not possible to flip the test decision.

\section{Concluding Remarks}
\label{sec:conclusion}
We proposed a threshold breakdown point to quantify the amount of contamination required to move an M-estimator by a predetermined fixed amount. We extended this definition to location-scale estimators and test statistics based on simple M-estimators, and proposed a statistical framework  for carrying out a complete diagnostic of the robustness of a statistical inference procedure based on M-estimators.
Several extensions appear natural. Many modern estimators in causal inference and missing-data problems, such as inverse probability weighted and augmented IPW estimators, can be written as functionals of a weighted empirical measure; our weighted results provide a starting point for formal outlier robustness analyses in those contexts. It would also be interesting to consider extensions to the Bayesian setting following the M-posterior framework of \cite{marusic:rush_avellamedina2025}. Extensions to multivariate M-estimators are also natural and an intuitive starting point could be to consider robust linear regression M-estimation or  multivariate location–scatter M-estimators. We hope that the tools developed in this paper could be extended to such multivariate settings with suitable geometric perturbation refinements of our breakdown analysis.

{
\begingroup
\renewcommand{\baselinestretch}{1}\selectfont
\small
\bibliographystyle{plainnat}
\bibliography{Biblio}
\endgroup
}

\appendix
\newpage
{
\begingroup
\renewcommand{\baselinestretch}{1}\selectfont
\section{Additional Theoretical Results}
\subsection{Additional Result for Variance Estimation BP}
\label{sec:opt_attack_m_est_se}
This section complements the main text by analyzing the robustness of plug-in standard errors $\hatse(\hat\theta)$ in \eqref{eq:se}. Unlike the restricted regime in which the standard error is evaluated at a fixed value $\theta_0$ and treated as non-manipulable, the mapping $x^{(n)} \mapsto \hatse(\hat\theta(x^{(n)}))$ is more involved. Any perturbation of the sample $x^{(n)}$ simultaneously changes the estimator $\hat\theta(x^{(n)})$ and the empirical quantities entering the standard-error formula. This  prevents a direct application of the threshold breakdown characterization used for $\hatse(\theta_0)$. Our goal here is therefore more modest and operational: we derive computable upper and lower bounds on the one-sided sensitivities $\eta_{m/n\pm}(\hatse(\hat\theta),x^{(n)})$. The resulting bounds are then used in Appendix~\ref{sec:tests} to instantiate Theorem~\ref{thm:meta} for testing procedures that employ $\hatse(\hat\theta)$.
The next lemma connects the $m$-sensitivity of the plug-in standard error to the $m$-sensitivity of the underlying location estimator. It provides finite sample upper bounds on $\eta_{m/n\pm}(\hatse(\hat\theta),x^{(n)})$ in terms of $\eta_{m/n}(\hat\theta,x^{(n)})$, together with additional curvature terms that only depend on $\psi$.
\begin{lemma}
\label{lem:opt_attack_m_est_se}
Assume that $\psi$ is differentiable a.e., bounded, non-decreasing and passes through $0$. Further assume that $\psi$ is odd and $\psi'(x)$ is non-increasing in $|x|$.
Define
\begin{align*}
    \Delta(t)
    &:= \max_{x} \bigl|\psi'(x+t) - \psi'(x)\bigr|, \\
    S_{m+}
    &:= \max_{\substack{I \subseteq [n] \\ |I| = n - m}}
        \max_{\tilde \theta \in [\hat \theta - \eta_{m/n-}(\hat \theta, x^{(n)}), \,
                                 \hat \theta + \eta_{m/n+}(\hat \theta, x^{(n)})]}
        \sum_{i \in I} \II\{ |x_i - \hat \theta | < |x_i - \tilde \theta| \}, \\
    S_{m-}
    &:= \max_{\substack{I \subseteq [n] \\ |I| = n - m}}
        \max_{\tilde \theta \in [\hat \theta - \eta_{m/n-}(\hat \theta, x^{(n)}), \,
                                 \hat \theta + \eta_{m/n+}(\hat \theta, x^{(n)})]}
        \sum_{i \in I} \II\{ |x_i - \hat \theta | > |x_i - \tilde \theta| \}.
\end{align*}
Let $\pi$ be a permutation such that {$|x_{\pi_1}-\hat\theta|\le|x_{\pi_2}-\hat\theta|\le\cdots\le|x_{\pi_n}-\hat\theta|$}.
Then
\begin{align*}
    \eta_{m/n+}(\hatse(\hat \theta), x^{(n)})
    \le{} &\;
    \sqrt{
      \frac{
        5m \psi_{\max}^2
        + \sum_{i>m} \psi(x_{\pi_i} - \hat \theta)^2
      }{
        \max \Bigl\{
          \sum_{i>m} \psi'(x_{\pi_i} - \hat \theta) - S_{m+}\Delta(\eta_{m/n}(\hat \theta, x^{(n)})),\, 0
        \Bigr\}^2
      }
    }
    - \hatse(\hat \theta),
    \\
    \eta_{m/n-}(\hatse(\hat \theta), x^{(n)})
    \le{} &\;
      \hatse(\hat \theta)
      - \sqrt{
        \frac{
          \max\bigl\{
            \sum_{i\le n-m} \psi(x_{\pi_i} - \hat \theta)^2
            - 4m \psi_{\max}^2,\, 0
          \bigr\}
        }{
          \bigl(
            m\psi'(0) + \sum_{i\le n-m} \psi'(x_{\pi_i} - \hat \theta)
            + S_{m-}\Delta(\eta_{m/n}(\hat \theta, x^{(n)}))
          \bigr)^2
        }
      }.
\end{align*}
Then
\begin{align*}
S_{m+}=\max\Bigl\{& \II\{\eta_{m/n+}(\hat\theta,x^{(n)})>0\}\sum_{i=1}^{n-m}\II\{\hat\theta+\eta_{m/n+}(\hat\theta,x^{(n)})>2x_{(i)}-\hat\theta\}, \\ 
& \II\{\eta_{m/n-}(\hat\theta,x^{(n)})>0\}\sum_{i=m+1}^{n}\II\{\hat\theta-\eta_{m/n-}(\hat\theta,x^{(n)})<2x_{(i)}-\hat\theta\}\Bigr\}, \\
S_{m-}=\min\Bigl\{& n-m,\ \max\bigl\{\II\{\eta_{m/n+}(\hat\theta,x^{(n)})>0\}\sum_{i=1}^n\II\{x_i>\hat\theta\},\ \II\{\eta_{m/n-}(\hat\theta,x^{(n)})>0\}\sum_{i=1}^n\II\{x_i<\hat\theta\}\bigr\}\Bigr\}.
\end{align*}
As a special case, for Huber's loss with parameter $\delta$ one takes $\psi'(x) = \II\{|x| \le \delta\}$. Defining
\begin{align*}
    q_{m+}
    &:= \min_{\tilde \theta \in [\hat \theta - \eta_{m/n-}(\hat \theta, x^{(n)}), \,
                                 \hat \theta + \eta_{m/n+}(\hat \theta, x^{(n)})]}
        \sum_{i=1}^n
        \bigl(
          \II\{x_i - \tilde \theta \in [-\delta, \delta]\}
          - \II\{x_i - \hat \theta \in [-\delta, \delta]\}
        \bigr), \\
    q_{m-}
    &:= \max_{\tilde \theta \in [\hat \theta - \eta_{m/n-}(\hat \theta, x^{(n)}), \,
                                 \hat \theta + \eta_{m/n+}(\hat \theta, x^{(n)})]}
        \sum_{i=1}^n
        \bigl(
          \II\{x_i - \tilde \theta \in [-\delta, \delta]\}
          - \II\{x_i - \hat \theta \in [-\delta, \delta]\}
        \bigr),
\end{align*}
we have that
\begin{align*}
    \eta_{m/n+}(\hatse(\hat \theta), x^{(n)})
    \le{} &\;
    \sqrt{
      \frac{
        5 m \psi_{\max}^2
        + \sum_{i>m} \psi(x_{\pi_i} - \hat \theta)^2
      }{
        \max \Bigl\{
          \sum_{i>m} \psi'(x_{\pi_i} - \hat \theta)
          + {(q_{m+}-m)},\, 0
        \Bigr\}^2
      }
    }
    - \hatse(\hat \theta),
    \\
    \eta_{m/n-}(\hatse(\hat \theta), x^{(n)})
    \le{} &\;
      \hatse(\hat \theta)
      - \sqrt{
        \frac{
          \max \Bigl\{
            \sum_{i\le n-m} \psi(x_{\pi_i} - \hat \theta)^2
            - 4m \psi_{\max}^2,\, 0
          \Bigr\}
        }{
          \Bigl(
            m\psi'(0) + \sum_{i\le n-m} \psi'(x_{\pi_i} - \hat \theta)
            + {(q_{m-}+m)}
          \Bigr)^2
        }
      }.
\end{align*}
\end{lemma}
Lemma~\ref{lem:opt_attack_m_est_se} makes explicit how the robustness of the variance estimator is driven by two components: the location $m$-sensitivity $\eta_{m/n}(\hat\theta,x^{(n)})$, and the way in which contamination changes which observations fall into the high-weight region of $\psi'$. The specialisation to Huber's loss shows that, in this case, the dependence is entirely through changes in the set $\{i\in[n]: |x_i-\hat\theta|\le\delta\}$.
{
For small sample sizes, one can sharpen the bound by explicitly optimizing over the admissible shift values.
\begin{corollary}
\label{lem:opt_attack_m_est_se_envelope}
Assume the same conditions and notation as in Lemma~\ref{lem:opt_attack_m_est_se}. Define
$$
I_m
:=
\Bigl[
-\eta_{m/n-}(\hat\theta,x^{(n)}),\,
\eta_{m/n+}(\hat\theta,x^{(n)})
\Bigr].
$$
For each $t\in I_m$, let
$$
h_i(t):=\psi(x_i-\hat\theta-t)^2,\qquad i=1,\dots,n,
$$
and let
$$
h_{(1)}(t)\le \cdots \le h_{(n)}(t)
$$
be the order statistics of $\{h_i(t)\}_{i=1}^n$. Define
\begin{align*}
\overline N_m
&:=
\sup_{t\in I_m}
\left\{
m\psi_{\max}^2+\sum_{i=m+1}^n h_{(i)}(t)
\right\},\\
\underline N_m
&:=
\inf_{t\in I_m}
\sum_{i=1}^{n-m} h_{(i)}(t).
\end{align*}
Then
\begin{align*}
    \eta_{m/n+}(\hatse(\hat \theta), x^{(n)})
    \le{} &\;
    \sqrt{
      \frac{
        \overline N_m
      }{
        \max \Bigl\{
          \sum_{i>m} \psi'(x_{\pi_i} - \hat \theta) - S_{m+}\Delta(\eta_{m/n}(\hat \theta, x^{(n)})),\, 0
        \Bigr\}^2
      }
    }
    - \hatse(\hat \theta),
    \\
    \eta_{m/n-}(\hatse(\hat \theta), x^{(n)})
    \le{} &\;
      \hatse(\hat \theta)
      - \sqrt{
        \frac{
          \underline N_m
        }{
          \bigl(
            m\psi'(0) + \sum_{i\le n-m} \psi'(x_{\pi_i} - \hat \theta)
            + S_{m-}\Delta(\eta_{m/n}(\hat \theta, x^{(n)}))
          \bigr)^2
        }
      }.
\end{align*}
\end{corollary}
\begin{remark}
    Moreover, if each $h_i$ is continuous and piecewise $C^1$, then $\overline N_m$ and $\underline N_m$ can be computed by checking only:
    $$
    t=-\eta_{m/n-}(\hat\theta,x^{(n)}),\qquad
    t=\eta_{m/n+}(\hat\theta,x^{(n)}),
    $$
    all non-smooth points of the functions $h_i$, all switching points where $h_i(t)=h_j(t)$ for some $i\neq j$, and all interior critical points of the smooth pieces. On any interval where the active set is fixed and all $h_i$ are differentiable,
    $$
    h_i'(t)=-2\psi(x_i-\hat\theta-t)\psi'(x_i-\hat\theta-t).
    $$
    Hence the relevant critical points are those satisfying
    $$
    \sum_{i\in A} h_i'(t)=0,
    $$
    where $A$ is the active index set corresponding to the smallest $n-m$ terms for $\underline N_m$, or the largest $n-m$ terms for $\overline N_m$.
\end{remark}
}
{%

\subsection{Generalization of Threshold Breakdown Point and Multiplier Bootstrap}
\label{sec:generalization_bootstrap}
While the multiplier bootstrap targets uncertainty quantification for $\hat\theta$, our focus here is robustness variability. We therefore seek a finite threshold breakdown notion that is intrinsically compatible with multiplicative weighting.
In the unweighted case, the finite sample threshold breakdown point is the smallest fraction of observations that must be replaced in order to move the estimator by {at least} $\eta$. {Let $F_{x^{(n)}}=n^{-1}\sum_{i=1}^n\delta_{x_i}$ denote the empirical law associated with $x^{(n)}$.}
Replacing at most $m$ observations gives $d_{\mathrm{TV}}(F_{x^{(n)}},F_{y^{(n)}})\le m/n$. %
}

Under multiplicative weighting, however, each replicate corresponds to a weighted empirical distribution, and the phrase ``$m$ points'' is no longer literal. Moving a few high-weight observations can have the same effect as moving many low-weight ones. A natural replacement is therefore to measure the amount of moved mass through total variation. At the same time, to preserve interpretability and to align with the lattice-valued breakdown point in the unweighted setting, we continue to index contamination levels on the grid $\{0,1/n,\dots,1\}$.

{For probability laws, use $d_{\mathrm{TV}}(F,H)=\sup_A|F(A)-H(A)|$.}
{For weights $w_i\ge0$, $\sum_iw_i=1$, put $F_{x^{(n)}}^w=\sum_iw_i\delta_{x_i}$ and define}
\[
\fitdisplay{{B_{\mathrm{TV}}(F_{x^{(n)}}^w,m/n) =\{H:d_{\mathrm{TV}}(F_{x^{(n)}}^w,H)\le m/n\}.}}
\]
This leads to the following generalized threshold breakdown definition.
\begin{definition}
\label{def:BP_extended}
For $\eta>0$, define
\begin{equation}\label{BP_generalization}
\fitdisplay[\dimexpr\linewidth-3em\relax]{\BP_\eta(\hat\theta,F_{x^{(n)}}^w) =\frac1n\min\left\{{m\in\{0,\ldots,n\}}: {\exists\,H\in B_{\mathrm{TV}}(F_{x^{(n)}}^w,m/n)}, |{\hat\theta(H)}-\hat\theta(F_{x^{(n)}}^w)|\ge\eta\right\}.}
\end{equation}
{Here $\hat\theta(H)$ denotes the location functional determined by the score equation.}
The one-sided definitions replace the absolute displacement by its
corresponding signed displacement.
\end{definition}
\begin{definition}
\label{def:eta_extended}
For a weighted empirical distribution function $F_{x^{(n)}}^w$ {and $m\in[n]$}, the $m$-sensitivity at level $m/n$ is defined by
\begin{equation}
{\eta_{m/n}(\hat\theta,F_{x^{(n)}}^w) := \sup\Bigl\{\eta>0:\BP_\eta(\hat\theta,F_{x^{(n)}}^w)\le\frac mn\Bigr\}.}
\end{equation}
\end{definition}
\begin{remark}
{For the convex location estimators characterized in
Theorem~\ref{thm:loc_bootstrap}, equal weights recover the unweighted
breakdown and sensitivity formulas: the extremal operation removes
exactly $m$ atoms and inserts $m/n$ mass at an extreme.
}
The grid $\{0,1/n,\ldots,1\}$ is retained, so $m$ is a mass-equivalent count.
Normalized and unnormalized weights give the same score root.
\end{remark}

Under this framework, we have the following theorem characterizing the exact threshold breakdown point and $m$-sensitivity of convex location $M$-estimators.
\begin{theorem}
\label{thm:loc_bootstrap}
Assume that $\psi$ is non-decreasing. For a weighted empirical distribution
\[
F_{x^{(n)}}^w(t)=\sum_{i=1}^n w_i\II\{x_i\le t\},
\qquad \sum_{i=1}^n w_i=1,
\]
write $W_k=\sum_{i=1}^k w_i$, with $W_0=0$, and suppose
$x_1\le\cdots\le x_n$. For {$m\in[n]$}, choose $i_m,i'_m\in[n]$
such that
\[
W_{i_m-1}<m/n\le W_{i_m},
\qquad W_{i'_m-1}\le1-m/n<W_{i'_m}.
\]
Then,
\begin{align*}
\BP_{\eta+}(\hat\theta,F_{x^{(n)}}^w)=\frac1n\min\Biggl\{m\in[n]:\ &\sum_{i>i_m}w_i\psi(x_i-(\hat\theta+\eta))\\
&+(W_{i_m}-m/n)\psi(x_{i_m}-(\hat\theta+\eta))+(m/n)\psi(\infty)\ge0\Biggr\},\\
\BP_{\eta-}(\hat\theta,F_{x^{(n)}}^w)=\frac1n\min\Biggl\{m\in[n]:\ &\sum_{i<i'_m}w_i\psi(x_i-(\hat\theta-\eta))\\
&+(1-m/n-W_{i'_m-1})\psi(x_{i'_m}-(\hat\theta-\eta))+(m/n)\psi(-\infty)\le0\Biggr\}.
\end{align*}
These score inequalities also apply when the score differences used as
denominators in the quotient form are zero.
Moreover,
\begin{align*}
\eta_{m/n+}(\hat\theta,F_{x^{(n)}}^w)=\max\Biggl\{\eta:\ &\sum_{i>i_m}w_i\psi(x_i-(\hat\theta+\eta))\\
&+(W_{i_m}-m/n)\psi(x_{i_m}-(\hat\theta+\eta))+(m/n)\psi(\infty)\ge0\Biggr\},\\
\eta_{m/n-}(\hat\theta,F_{x^{(n)}}^w)=\max\Biggl\{\eta:\ &\sum_{{i<i'_m}}w_i\psi(x_i-(\hat\theta-\eta))\\
&+(1-m/n-W_{i'_m-1})\psi(x_{i'_m}-(\hat\theta-\eta))+(m/n)\psi(-\infty)\le0\Biggr\}.
\end{align*}
\end{theorem}
With the help of Theorem \ref{thm:loc_bootstrap}, we obtain analogous asymptotic results for the bootstrap threshold breakdown point and $m$-sensitivity. We only state the asymptotic normality results in Theorem \ref{thm:bootstrap_normality} and Theorem \ref{thm:BP_normality}, since consistency follows by similar arguments under slightly weaker assumptions. Moreover, we can extend the relationship in Proposition \ref{prop:sensitivity_map_finite_main} to the bootstrap quantities as follows. Let
$$
F_n(t) := \frac{1}{n}\sum_{i=1}^n \II\{X_i \le t\},
\qquad
\tilde F_n(t) := \frac{1}{n}\sum_{i=1}^n \frac{W_i}{\bar W}\II\{X_i \le t\},
$$
with the same standing assumptions \eqref{eq:boostrap_condition} from Section \ref{sec:multiplier_bootstrap}. Following \citet{cheng2010bootstrap}, given a real-valued function $\Delta_n$, we write $\Delta_n = o_p^W(1)$ if for any $\varepsilon,\delta>0$,
\begin{equation*}
\PP_X\Bigl\{
\PP_{W\mid X}\bigl(|\Delta_n|>\varepsilon\bigr)>\delta
\Bigr\}
\to 0
\qquad \text{as } n\to\infty.
\end{equation*}
Notice that if $\Delta_n$ is independent of $W$, then
$$
\PP_{W\mid X}\bigl(|\Delta_n|>\varepsilon\bigr)
=
\II\{|\Delta_n|>\varepsilon\},
$$
so that
$$
\PP_X\Bigl\{
\PP_{W\mid X}\bigl(|\Delta_n|>\varepsilon\bigr)>\delta
\Bigr\}
=
\PP_X\bigl(|\Delta_n|>\varepsilon\bigr).
$$
Hence, for such $\Delta_n$,
\begin{equation}
\label{eq:op1_equivalence}
\Delta_n=o_p(1)
\iff
\Delta_n=o_p^W(1).
\end{equation}
\begin{proposition}
\label{prop:sensitivity_map_finite}
Suppose (A1)--(A4) hold for a fixed $\varepsilon^* \in (0,0.5)$, with $\eta^*$ solving \eqref{eq:population_eta+} or \eqref{eq:population_eta-}.
{
Then
\begin{align*}
   & \BP_{\eta^*+}(\hat\theta,F_n)-\varepsilon^*
=
-\frac{\int_{q_{\varepsilon^*}}^{\infty} \psi'(x - (\theta_0 + \eta_{\varepsilon^*+})) \dd F(x)}
{ -\psi(q_{\varepsilon^*} - (\theta_0 + \eta_{\varepsilon^*+})) + \|\psi\|_\infty}
\cdot
\bigl(
\eta_{m/n+}(\hat\theta,F_n)-\eta^*
\bigr)
+
o_p\left(\frac{1}{\sqrt{n}}\right), \\
&\BP_{\eta^*-}(\hat\theta,F_n)-\varepsilon^*
=
-\frac{-\int_{-\infty}^{q_{1-\varepsilon^*}} \psi'\bigl(x-(\theta_0-\eta_{\varepsilon^*-})\bigr) \dd F(x)}
{-\psi(q_{1-\varepsilon^*} - (\theta_0 - \eta_{\varepsilon^*-})) - \|\psi\|_\infty}
\cdot
\bigl(
\eta_{m/n-}(\hat\theta,F_n)-\eta^*
\bigr)
+
o_p\left(\frac{1}{\sqrt{n}}\right).
\end{align*}
We also have
\begin{align*}
    &~ \BP_{\eta^*+}(\hat\theta_b,\tilde F_n)
-
\BP_{\eta^*+}(\hat\theta,F_n)\\
   = &~ -\frac{\int_{q_{\varepsilon^*}}^{\infty} \psi'(x - (\theta_0 + \eta_{\varepsilon^*+})) \dd F(x)}
{ -\psi(q_{\varepsilon^*} - (\theta_0 + \eta_{\varepsilon^*+})) + \|\psi\|_\infty}
\cdot
\bigl(
\eta_{m/n+}(\hat\theta_b,\tilde F_n)
-
\eta_{m/n+}(\hat\theta,F_n)
\bigr)
+
o_p^W\left(\frac{1}{\sqrt{n}}\right), \\
&~ \BP_{\eta^*-}(\hat\theta_b,\tilde F_n)
-
\BP_{\eta^*-}(\hat\theta,F_n)\\
   = &~ -\frac{-\int_{-\infty}^{q_{1-\varepsilon^*}} \psi'\bigl(x-(\theta_0-\eta_{\varepsilon^*-})\bigr) \dd F(x)}
{-\psi(q_{1-\varepsilon^*} - (\theta_0 - \eta_{\varepsilon^*-})) - \|\psi\|_\infty}
\cdot
\bigl(
\eta_{m/n-}(\hat\theta_b,\tilde F_n)
-
\eta_{m/n-}(\hat\theta,F_n)
\bigr)
+
o_p^W\left(\frac{1}{\sqrt{n}}\right).
\end{align*}
If we further assume that
$$
F\!\left(\left\{x\in\mathbb R:\psi' \text{ is discontinuous at } x-(\theta_0 \pm \eta_{\varepsilon^*\pm})\right\}\right)=0.
$$
Then Lemma \ref{lem:sens-bias_property} implies that
$$
\BP_{\eta^*\pm}(\hat\theta,F_n)-\varepsilon^*
=
-\frac{\dd\varepsilon}{\dd\eta_{\varepsilon\pm}} \Big |_{\varepsilon = \varepsilon^*}
\cdot
\bigl(
\eta_{m/n\pm}(\hat\theta,F_n)-\eta^*
\bigr)
+
o_p\left(\frac{1}{\sqrt{n}}\right),
$$
and
$$
\BP_{\eta^*\pm}(\hat\theta_b,\tilde F_n)
-
\BP_{\eta^*\pm}(\hat\theta,F_n)
=
-\frac{\dd\varepsilon}{\dd\eta_{\varepsilon\pm}} \Big |_{\varepsilon = \varepsilon^*}
\cdot
\bigl(
\eta_{m/n\pm}(\hat\theta_b,\tilde F_n)
-
\eta_{m/n\pm}(\hat\theta,F_n)
\bigr)
+
o_p^W\left(\frac{1}{\sqrt{n}}\right).
$$
}
\end{proposition}
\section{Additional Numerical Experiments}
The code used to produce all simulations, empirical results, and figures in this paper is available at \url{https://github.com/keanson/Threshold-Breakdown}.
\subsection{Simulation Setting}
\label{sec:simulation_setting}
Unless stated otherwise, all numerical illustrations follow the protocol described in this section.
We consider three common convex M-estimation losses, including Huber's loss, logcosh loss, and the self-concordant loss. Huber's loss is given by
\begin{equation*}
    \rho_\delta (t) = \frac{1}{2} t^2 \II\{|t| < \delta\} + \delta (|t| - \frac{1}{2} \delta) \II\{|t| \ge \delta\}.
\end{equation*}
The logcosh loss is often used as a smoothed version of Huber's loss
\begin{equation*}
    \rho(t) = \log(\cosh(t)).
\end{equation*}
\citet{ostrovskii2021} introduced the following self-concordant pseudo-Huber loss to achieve a tighter finite sample guarantees for M-estimators
\begin{equation*}
    \rho(t) = \frac{1}{2} \left [ \sqrt{1+4t^2} - 1 + \log \left ( \frac{\sqrt{1 + 4t^2} - 1}{2t^2} \right ) \right ].
\end{equation*}
The latter two smoothed Huber losses are equipped with a robustness parameter $\delta$ through the rescaling
\begin{align*}
\rho_\delta(t) \;=\; \delta^2 \rho(t/\delta),
\qquad
\psi_\delta(t) \;=\; \rho_\delta'(t) \;=\; \delta\,\psi(t/\delta).
\end{align*}
To make comparisons across losses meaningful, we choose $\delta$ so that the resulting M-estimator has $95\%$ asymptotic efficiency under the model $\mathcal N(0,1)$.
Numerically, this gives $\delta=1.345$ (Huber), $\delta=1.2047$ (logcosh), and $\delta=1.4811$ (self-concordant).
When we are referring to normal, Cauchy, and uniform distributions, we imply the standard normal distribution $\mathcal{N}(0,1)$, the standard Cauchy distribution $\mathrm{Cauchy}(0, 1)$ (located at $0$ with a scale of 1), and the standard uniform distribution $\mathrm{Unif}(0,1)$. Each configuration is repeated over $100$ Monte Carlo replications.
\begin{figure}[htbp]
  \centering
  \includegraphics[width=\linewidth]{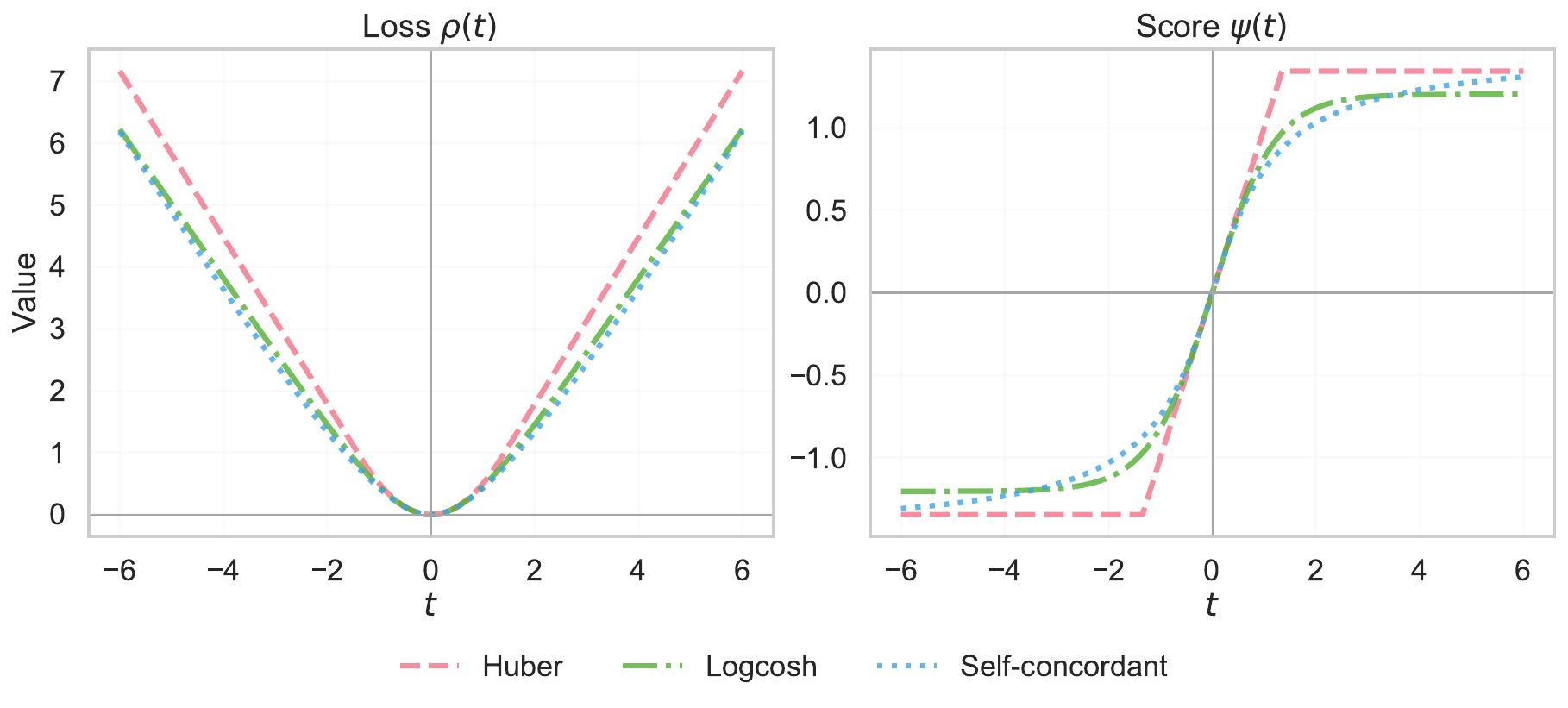}
  \caption{ \small Loss functions (left) and the corresponding score functions (right) for Huber's, logcosh, and self-concordant losses. The tuning parameters are set to $\delta = 1.345$ for Huber, $\delta = 1.2047$ for logcosh, and $\delta = 1.4811$ for the self-concordant loss.
}
  \label{fig:losses}
\end{figure}
\subsection{Numerical Illustration: Location Estimators}
\label{sec:NI_location}
\begin{figure}[htbp]
  \centering
  \begin{subfigure}[b]{0.3\textwidth}
    \centering
    \includegraphics[width=\linewidth]{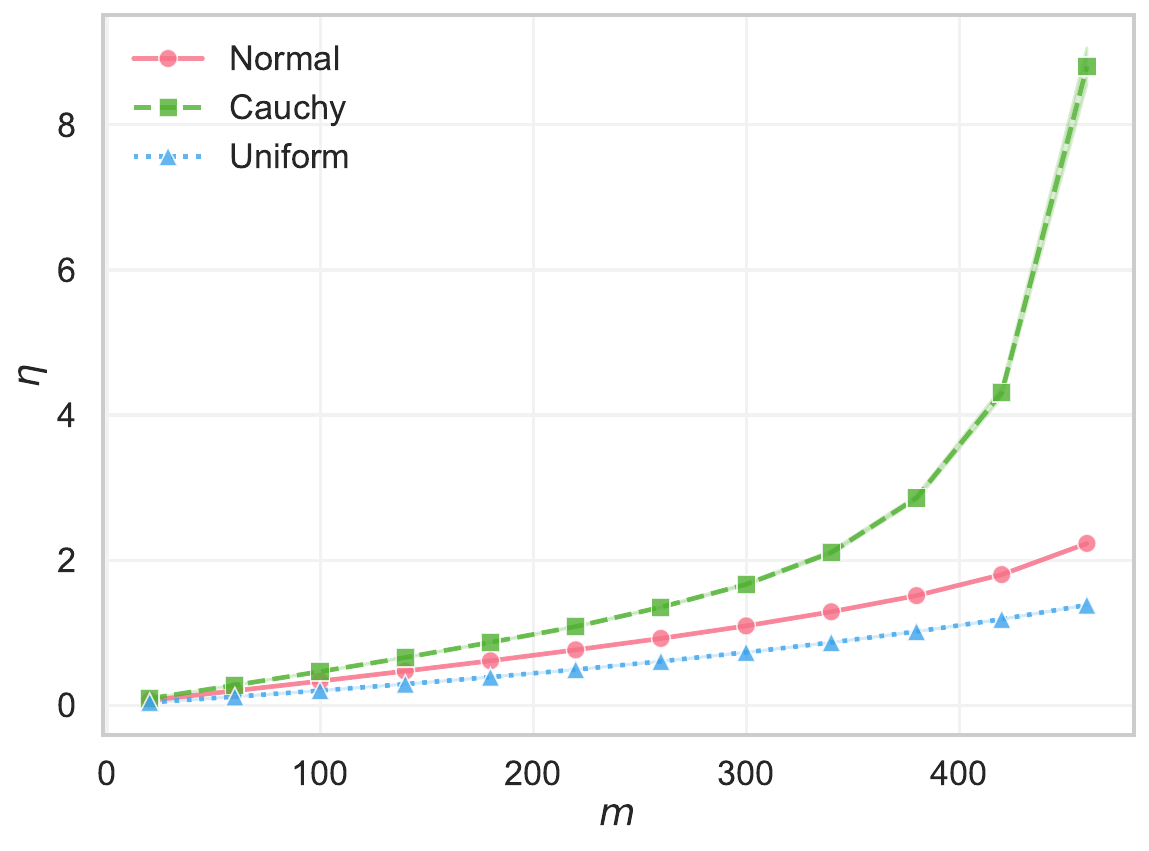}
    \caption{Huber's loss.}
    \label{fig:eta_huber}
  \end{subfigure}
  \hfill
  \begin{subfigure}[b]{0.3\textwidth}
    \centering
    \includegraphics[width=\linewidth]{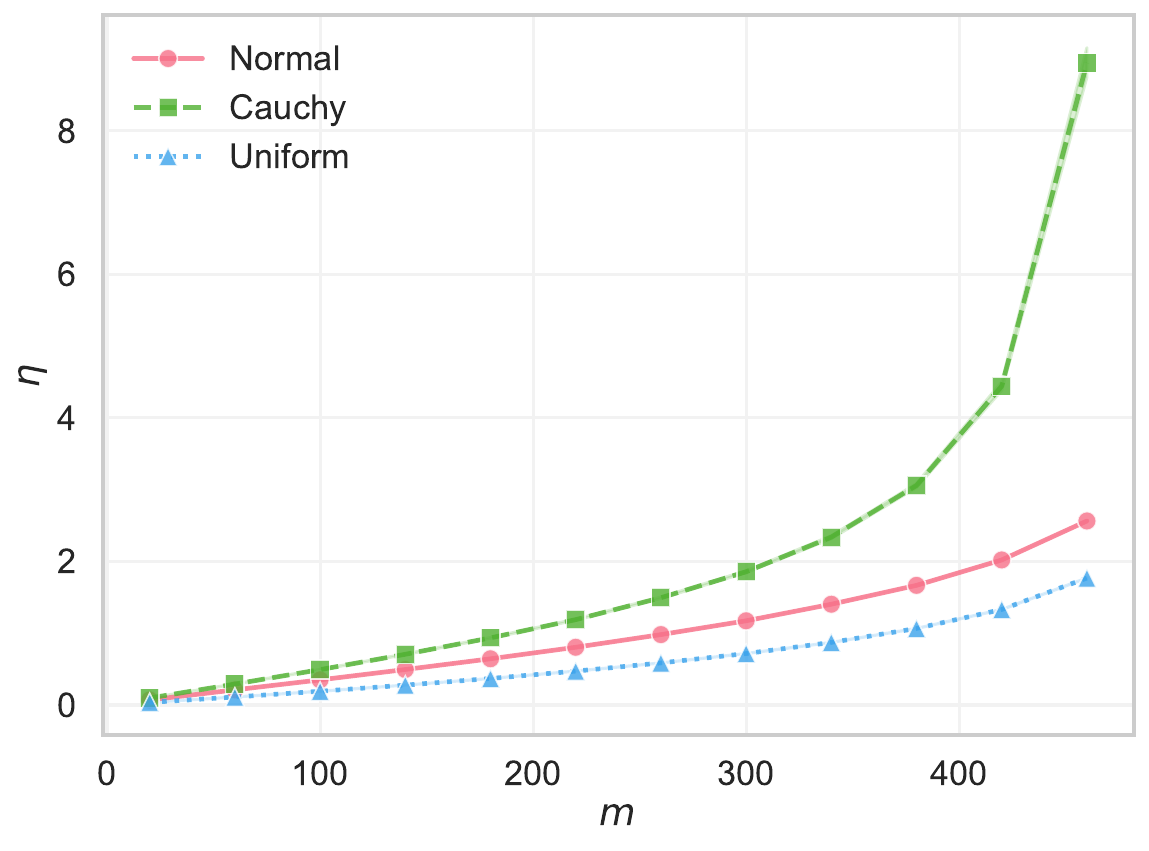}
    \caption{Logcosh loss.}
    \label{fig:eta_logcosh}
  \end{subfigure}
  \hfill
  \begin{subfigure}[b]{0.3\textwidth}
    \centering
    \includegraphics[width=\linewidth]{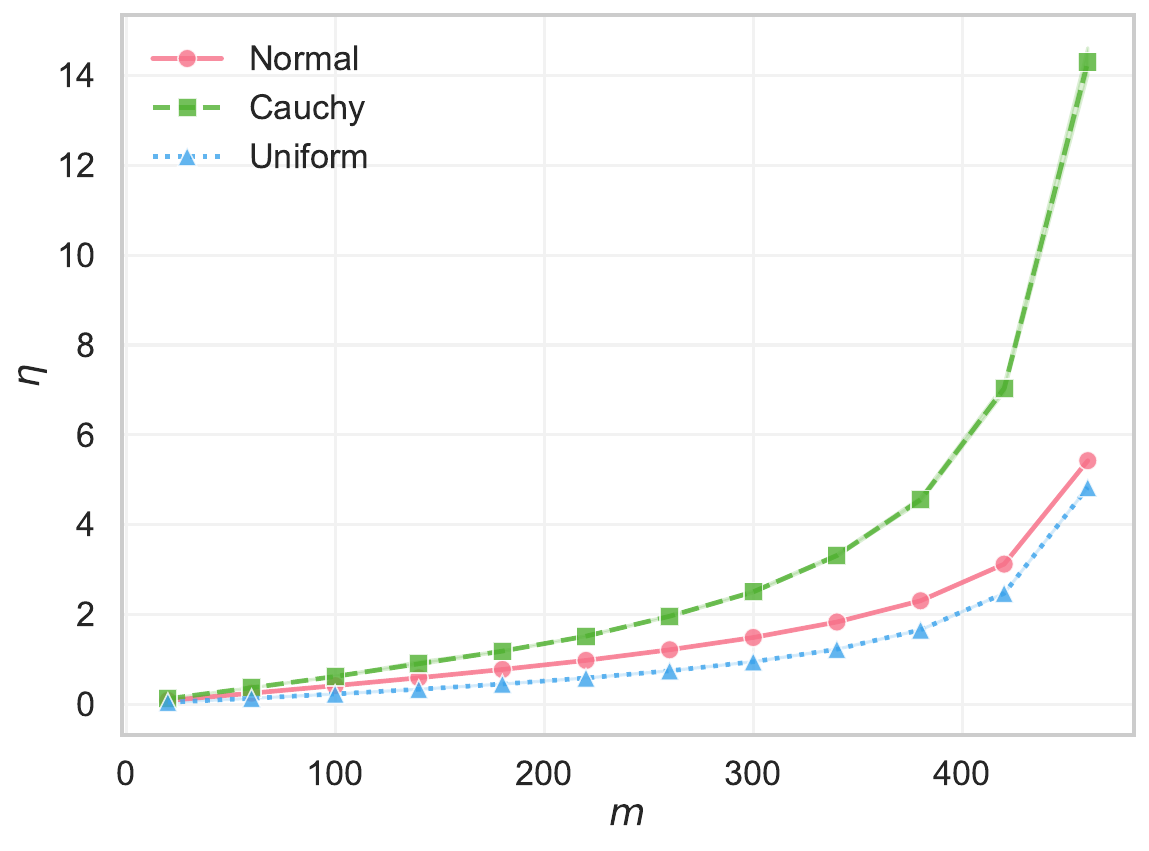}
    \caption{Self-concordant loss.}
    \label{fig:eta_con}
  \end{subfigure}
  \vspace{1em}
  \begin{subfigure}[b]{0.3\textwidth}
    \centering
    \includegraphics[width=\linewidth]{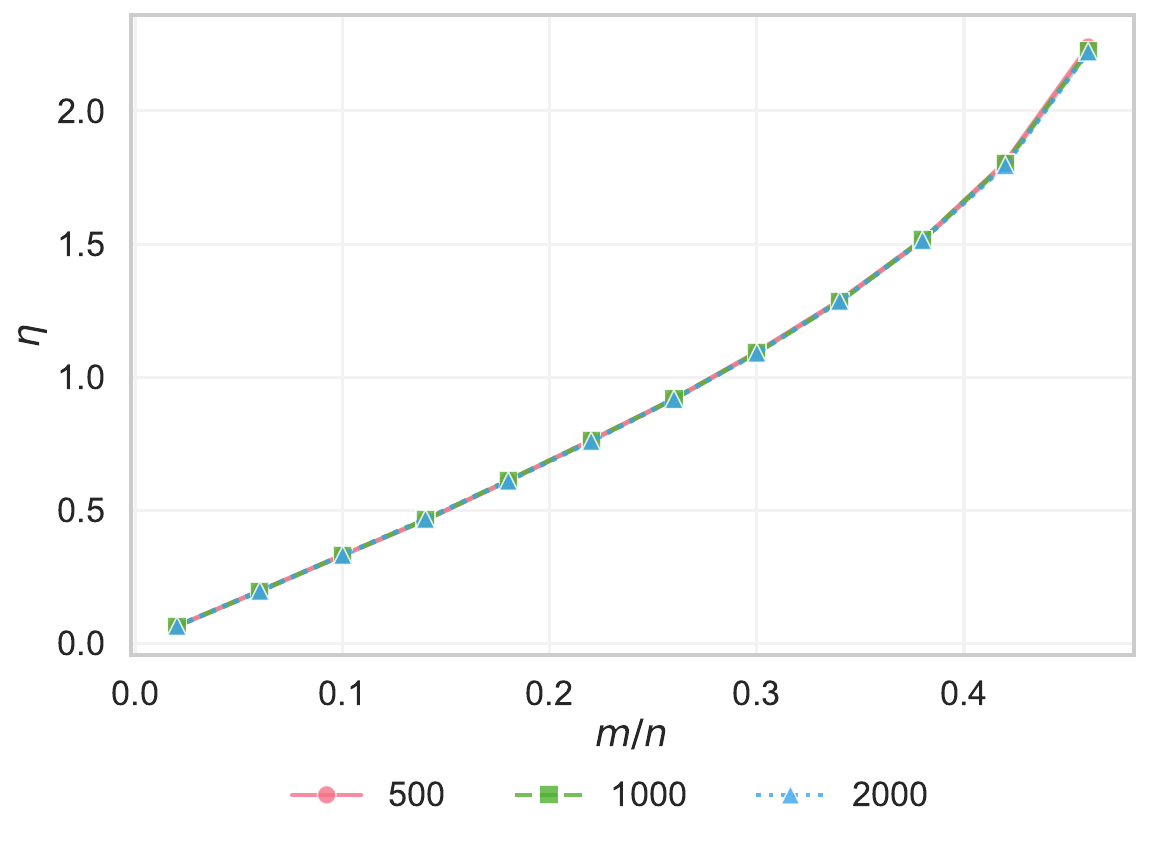}
    \caption{Huber's loss.}
    \label{fig:eta_n_huber}
  \end{subfigure}
  \hfill
  \begin{subfigure}[b]{0.3\textwidth}
    \centering
    \includegraphics[width=\linewidth]{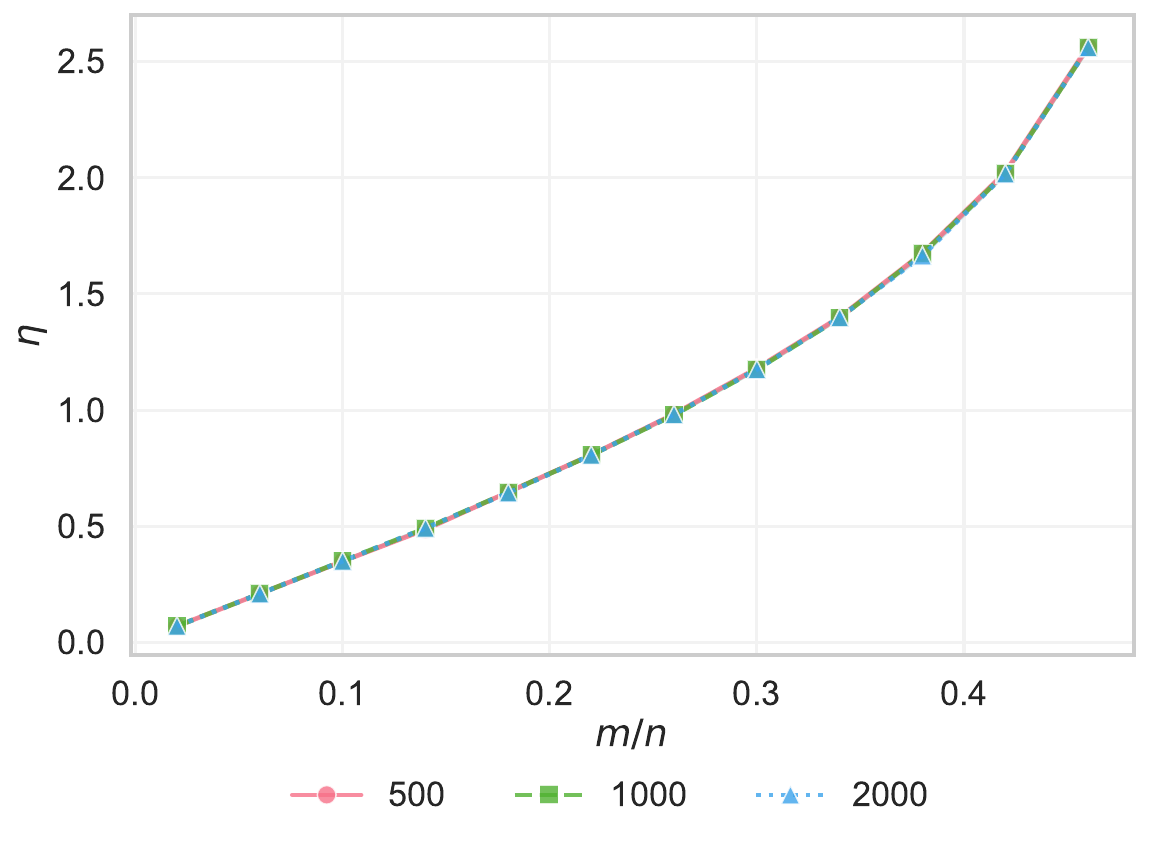}
    \caption{Logcosh loss.}
    \label{fig:eta_n_logcosh}
  \end{subfigure}
  \hfill
  \begin{subfigure}[b]{0.3\textwidth}
    \centering
    \includegraphics[width=\linewidth]{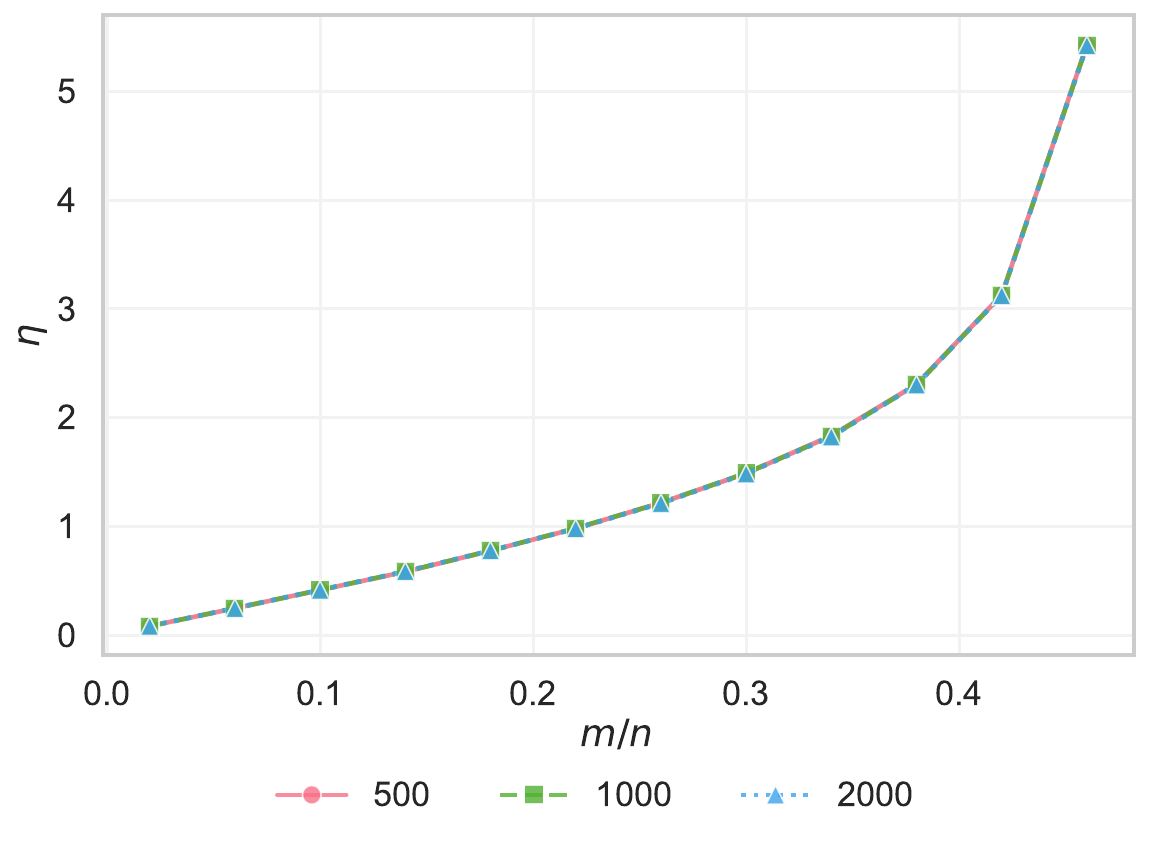}
    \caption{Self-concordant loss.}
    \label{fig:eta_n_con}
  \end{subfigure}
  \caption{ \small $m$-sensitivity $\eta_{m/n}(\hat\theta,x^{(n)})$ for location M-estimators. Top panels: $n=1000$ with $m\in\{20,60,\ldots,460\}$, and line types distinguish data generating distributions. Bottom panels: $m/n \in\{0.02,0.06,\ldots,0.46\}$, and line types distinguish sample sizes under $\mathcal N(0,1)$.
}
  \label{fig:bp_location}
\end{figure}
Now we illustrate Theorem \ref{thm:loc_main} and Corollary \ref{cor:opt_attack_m_est_main} by reporting the $m$-sensitivity curves for the location M-estimators with different data-generating mechanisms and different sample sizes in Figure \ref{fig:bp_location}. First, for fixed $m/n$, the curves for $n\in\{500,1000,2000\}$ are nearly indistinguishable (bottom row), indicating that in this regime the $m$-sensitivity is primarily governed by the contamination fraction.
Second, the underlying distribution substantially affects the slope of $m/n \mapsto \eta_{m/n}$ (top row): the Cauchy model yields earlier and sharper increases in $m$-sensitivity than the normal and uniform models, consistent with Corollary~\ref{cor:opt_attack_m_est_main} that shows the $m$-sensitivity is impacted by the order statistics.
Third, for a fixed asymptotic efficiency,  Huber's loss yields the smallest $m$-sensitivity across contamination levels, logcosh is intermediate, and the self-concordant loss is the most sensitive in these experiments.
Equivalently, by inverting $m/n \mapsto \eta_{m/n}(\hat\theta,x^{(n)})$, this ordering corresponds to larger threshold breakdown points $\BP_\eta(\hat\theta,x^{(n)})$ for Huber's loss relative to the two smooth counterparts at the same threshold $\eta$.
\subsection{Tightness of Breakdown Bounds: Wald-type Tests}
\label{sec:NI_wald_appendix}
Figure \ref{fig:bp_test_ratio} revealed that the bounds for the Wald test are visibly wider and the quality of the bounds is not easy to assess from that plot alone. We provide a complementary numerical analysis to show that our estimates improve as the sample size increases. We fix the effect size at $\theta=1$ and vary the sample size over
$
n \in \{50,250,450,\ldots,4050\}.
$
For each $n$, we generate data from $\mathcal N(\theta,1)$, $\mathrm{Unif}(0,1)+\theta$, and $\mathrm{Cauchy}(\theta,1)$ to reflect increasing tail-heaviness.
For each configuration, we repeat $100$ replications, condition on $\phi(x^{(n)})=1$, and compute the gap
$$
\mathrm{gap}(x^{(n)}) \;:=\; \BP_{\mathrm{reject}}^{\mathrm{up}}(\phi,x^{(n)})-\BP_{\mathrm{reject}}^{\mathrm{low}}(\phi,x^{(n)}),
$$
where $\BP_{\mathrm{reject}}^{\mathrm{up}}$ and $\BP_{\mathrm{reject}}^{\mathrm{low}}$ denote the upper and lower bounds in Theorem~\ref{thm:meta}, respectively.
Figure~\ref{fig:gap_new_scatter} displays $\mathrm{gap}(x^{(n)})$ versus $n$ with a fitted regression line, which shows that the bound gap decreases with $n$ across all configurations, supporting that the bounds become sharper as the sample size increases.
\begin{figure}[htbp]
  \centering
  \begin{subfigure}[b]{0.45\textwidth}
    \centering
    \includegraphics[width=\linewidth]{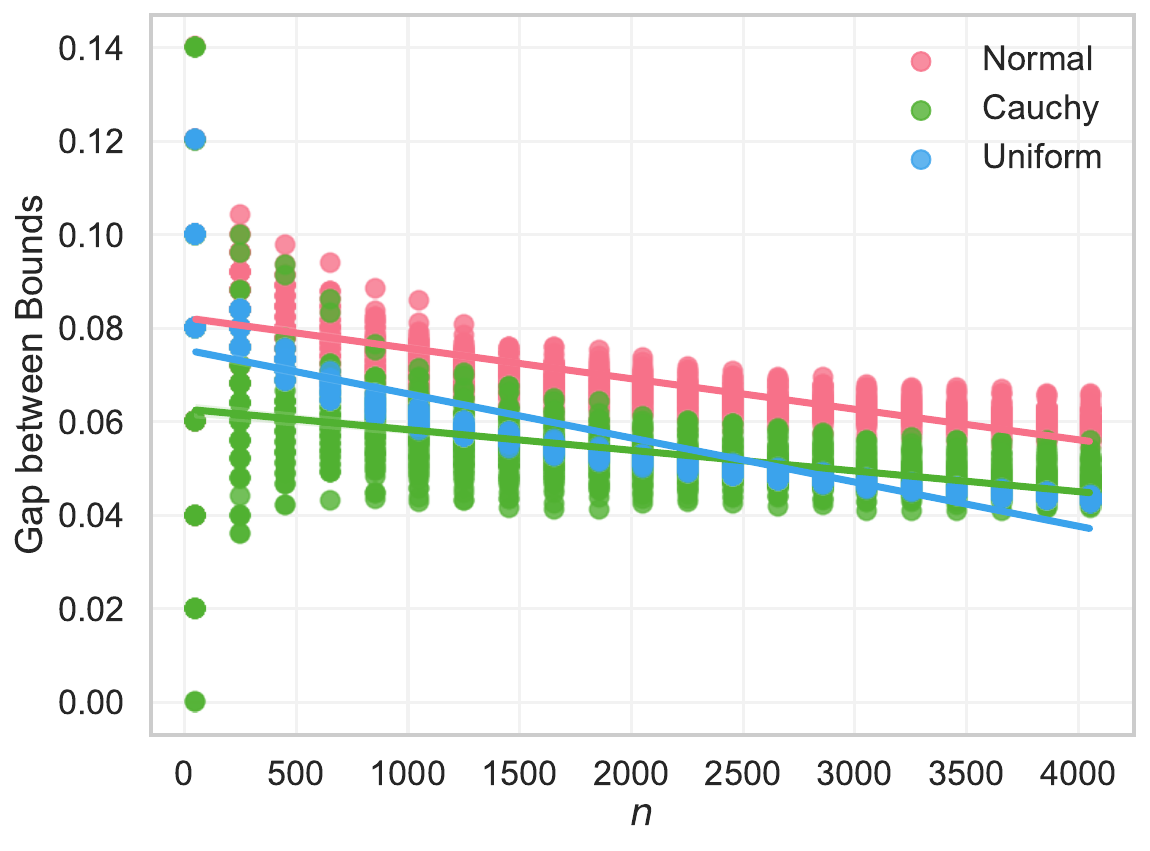}
    \caption{Huber's loss with regular Wald test.}
  \end{subfigure}
  \hfill
  \begin{subfigure}[b]{0.45\textwidth}
    \centering
    \includegraphics[width=\linewidth]{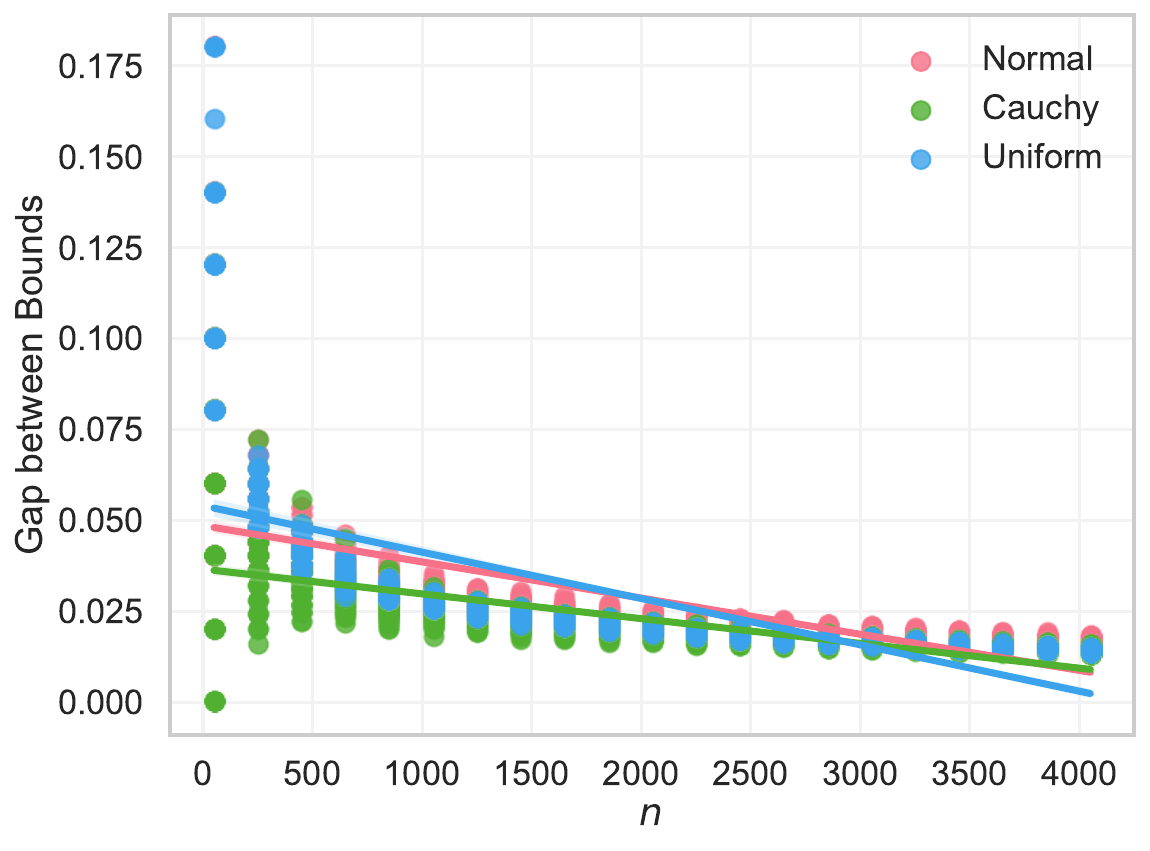}
    \caption{Huber's loss with restricted Wald test.}
  \end{subfigure}
  \vspace{1em}
  \begin{subfigure}[b]{0.45\textwidth}
    \centering
    \includegraphics[width=\linewidth]{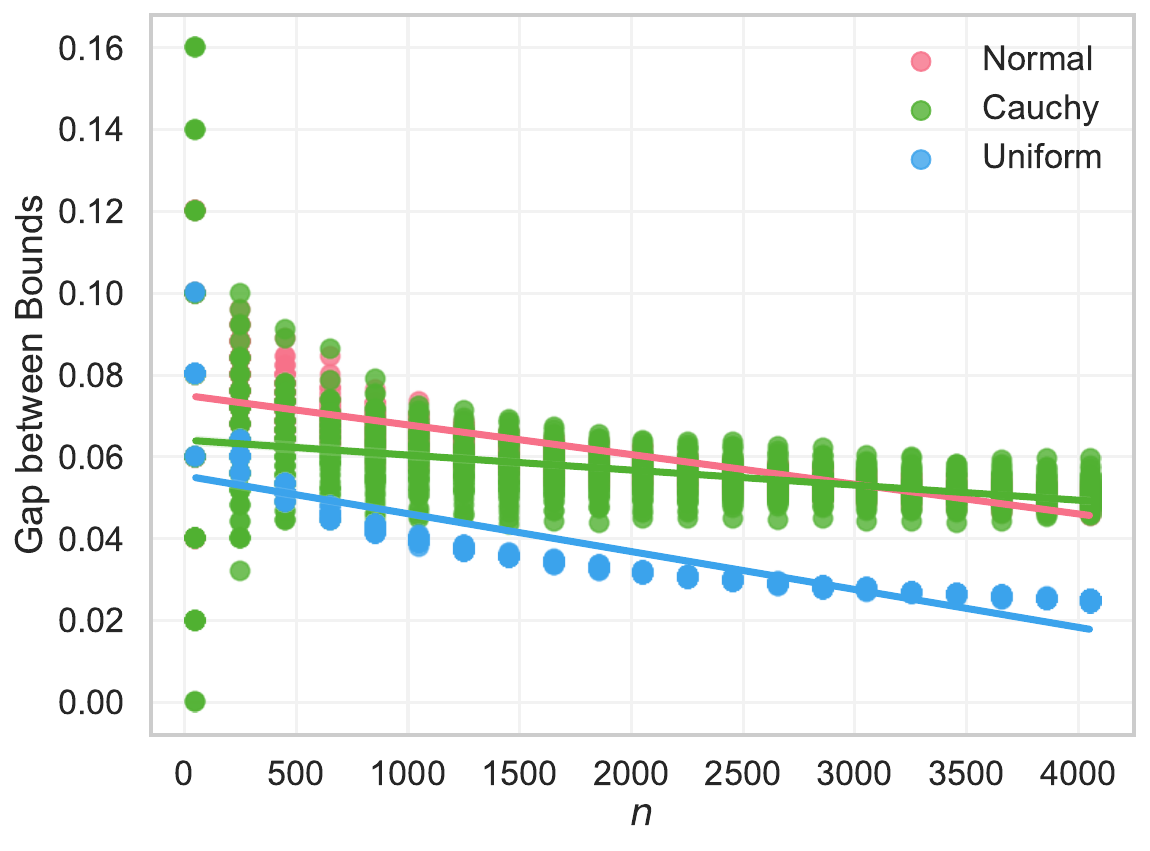}
    \caption{Logcosh loss with regular Wald test.}
  \end{subfigure}
  \hfill
  \begin{subfigure}[b]{0.45\textwidth}
    \centering
    \includegraphics[width=\linewidth]{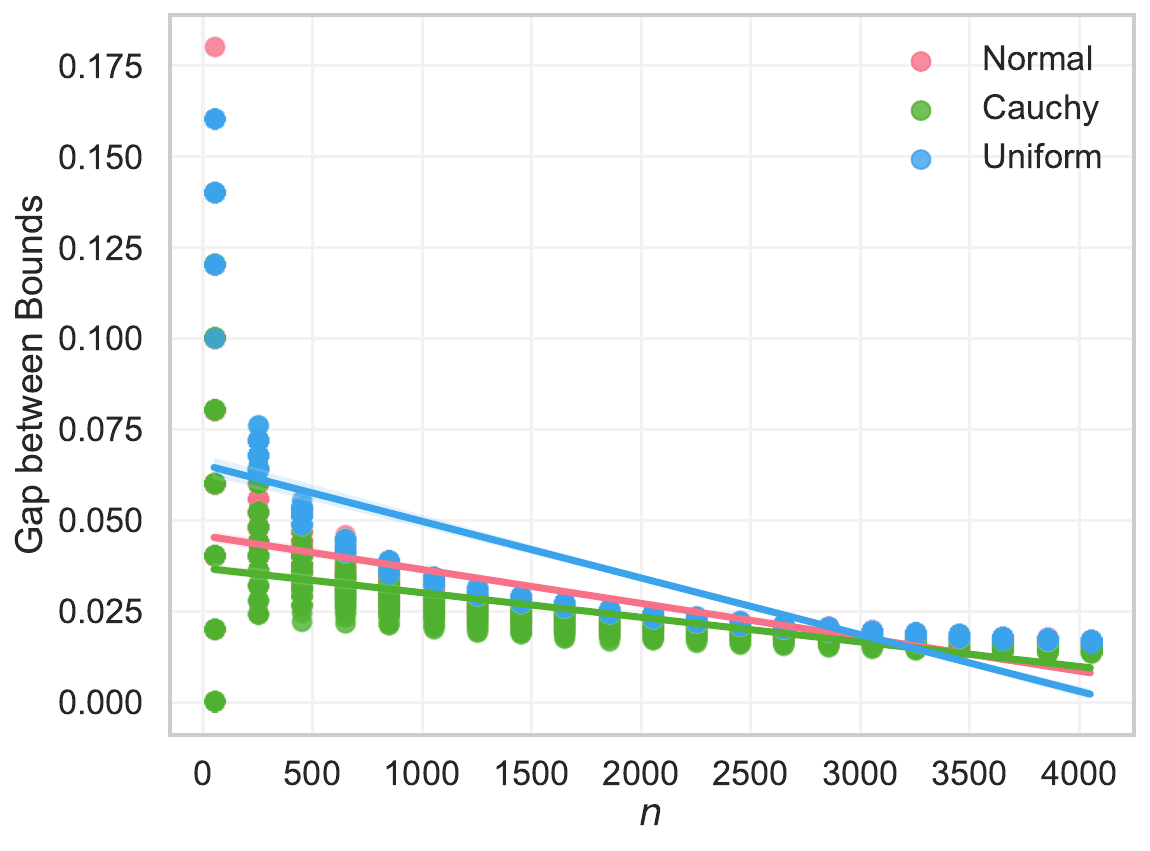}
    \caption{Logcosh loss with restricted Wald test.}
  \end{subfigure}
  \caption{ \small Gap between the upper and lower bounds on $\BP_{\mathrm{reject}}(\phi,x^{(n)})$ as a function of $n$, with $\theta=1$.
  Points are Monte Carlo replications; line and (almost invisible) band are a fitted regression and its $95\%$ confidence band.
  Line types distinguish the generating distribution.
  }
  \label{fig:gap_new_scatter}
\end{figure}
\subsection{Numerical Illustration: Two-sample Tests}\label{sec:NI_two_sample}
We follow the common simulation setup in Section \ref{sec:NI_location} for the choice of losses and tuning.
We consider the two-sample Wald-type test in \eqref{eq:test_two_sample} and evaluate the rejection breakdown point bounds described in Appendix \ref{sec:two_sample}, which is adjusted for two-sample tests.
We generate two independent samples with a location shift:
$$
x_1,\ldots,x_{n_x} \sim \mathcal N(0,1),
\qquad
y_1,\ldots,y_{n_y} \sim \mathcal N(\theta,1)
$$
for a grid of values of $\theta\in[-2,2]$ as in Section \ref{sec:NI_wald}.
We use balanced samples with $n_x=n_y\in\{50,100,200\}$.
For each replication, we repeatedly draw $(x^{(n_x)},y^{(n_y)})$ until the realized dataset satisfies $\phi(x^{(n_x)},y^{(n_y)})=1$ (rejects), and then compute the upper and lower bounds on $\BP_{\mathrm{reject}}(\phi,x^{(n_x)},y^{(n_y)})$ from Appendix~\ref{sec:two_sample}.
We report Monte Carlo averages over $100$ replications.
\begin{figure}[htbp]
  \centering
  \begin{subfigure}[b]{0.32\textwidth}
    \centering
    \includegraphics[width=\linewidth]{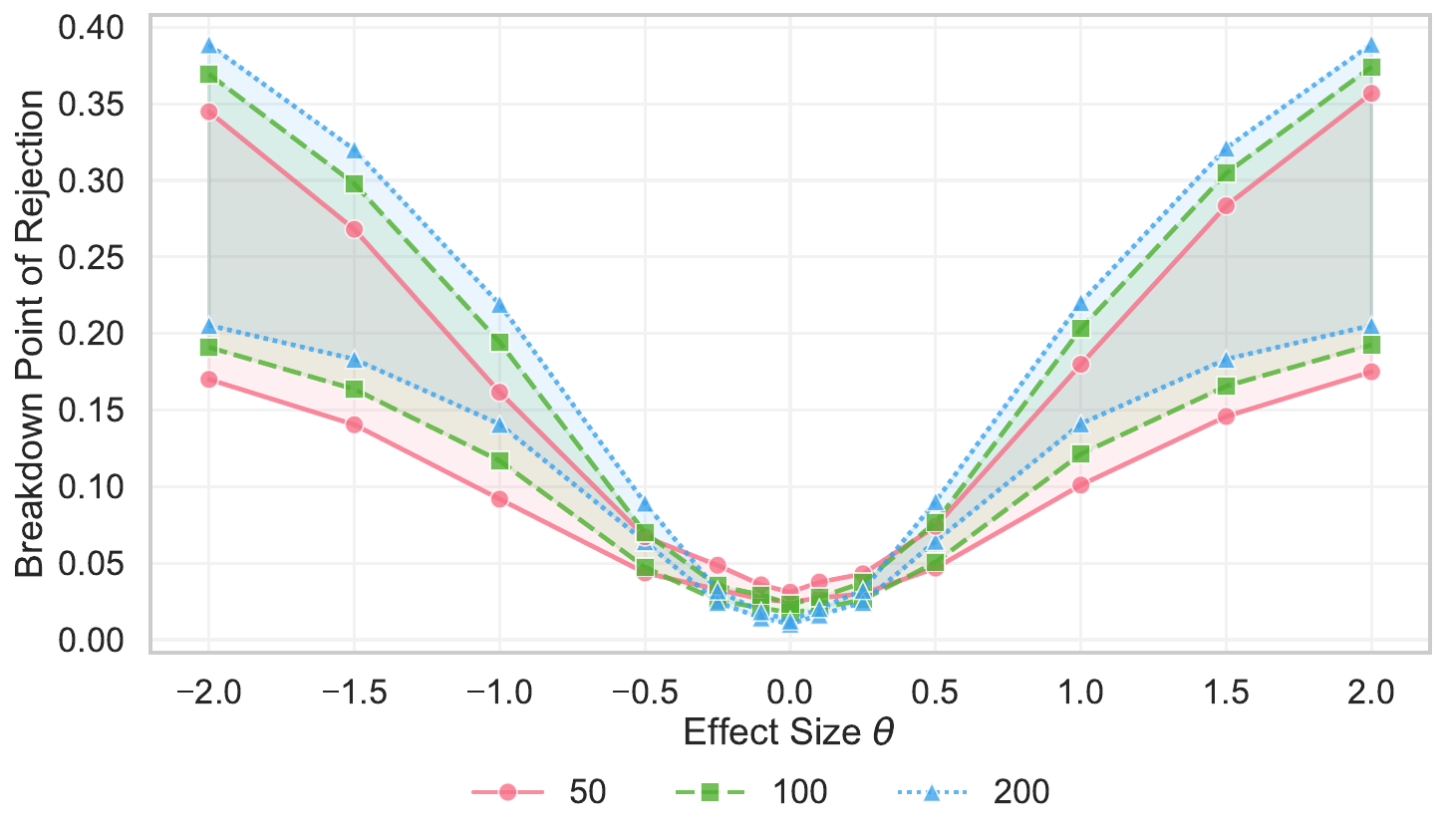}
    \caption{Huber's loss.}
    \label{fig:two_sample_huber_ratio}
  \end{subfigure}
  \hfill
  \begin{subfigure}[b]{0.32\textwidth}
    \centering
    \includegraphics[width=\linewidth]{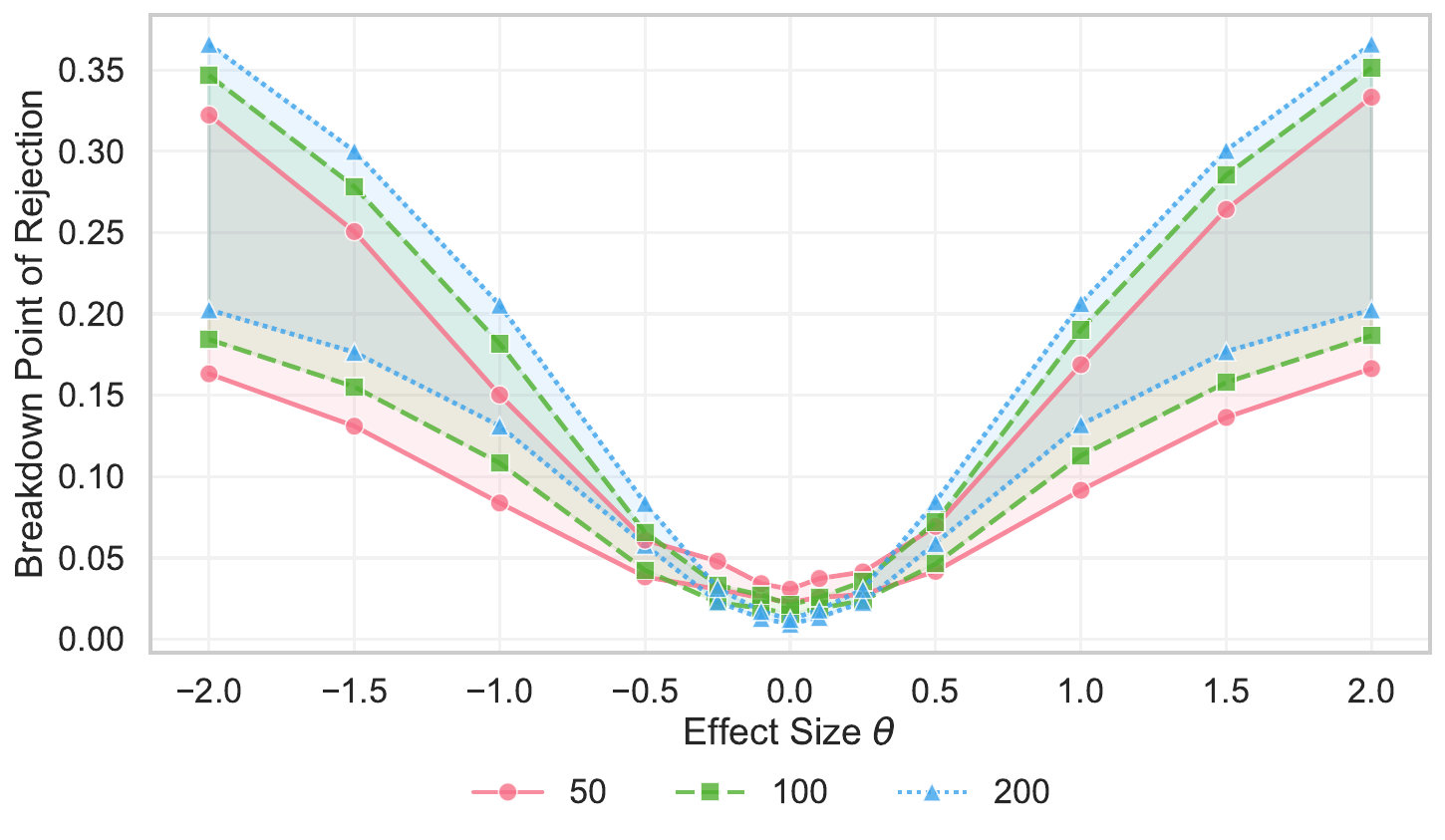}
    \caption{Logcosh loss.}
    \label{fig:two_sample_logcosh_ratio}
  \end{subfigure}
  \hfill
  \begin{subfigure}[b]{0.32\textwidth}
    \centering
    \includegraphics[width=\linewidth]{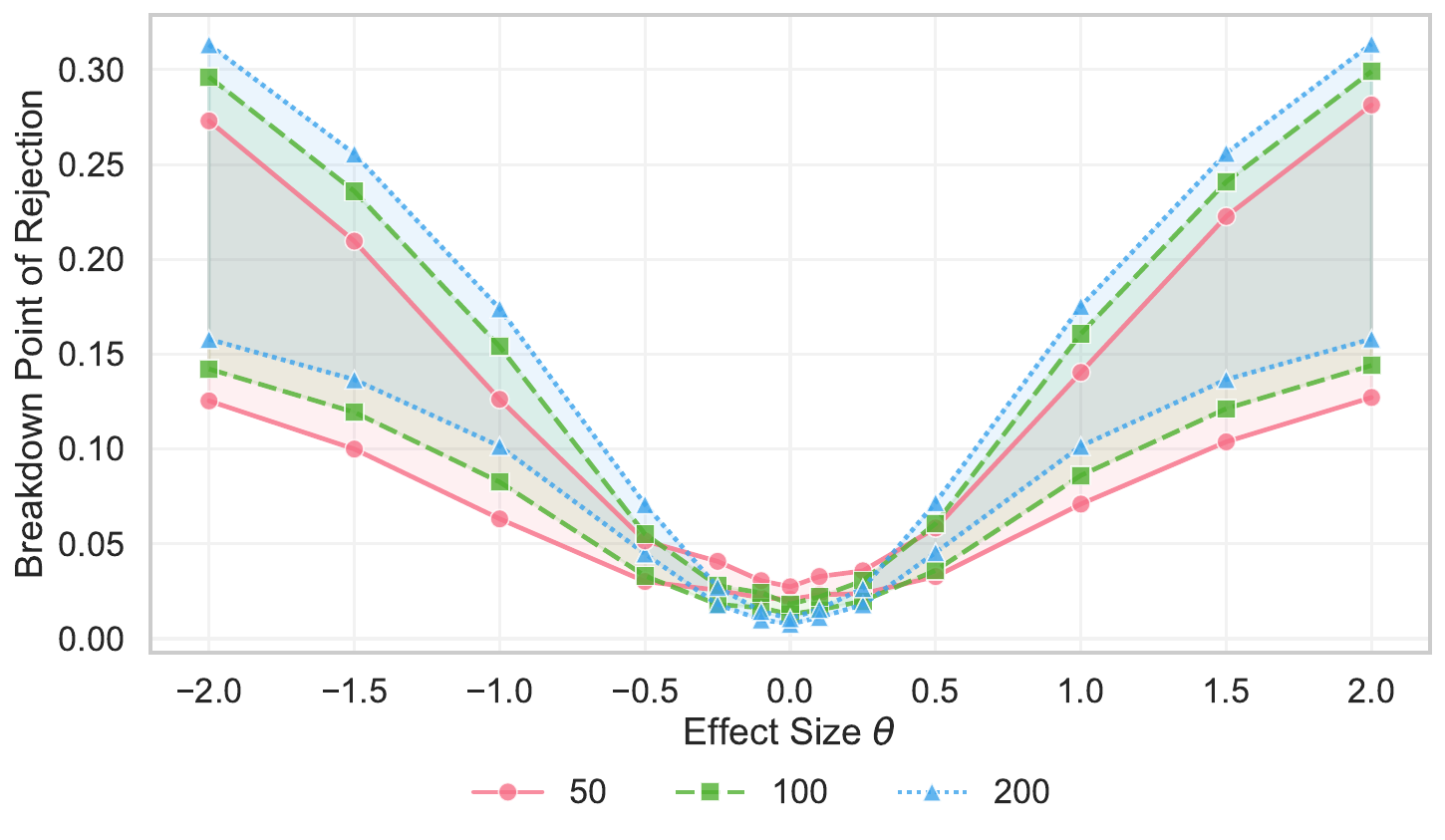}
    \caption{Self-concordant loss.}
    \label{fig:two_sample_con_ratio}
  \end{subfigure}
  \caption{\small Upper and lower bounds for the rejection breakdown point $\BP_{\mathrm{reject}}(\phi,x^{(n_x)},y^{(n_y)})$ as a function of the effect size $\theta$ in the two-sample problem.
  Line types distinguish $n_x=n_y\in\{50,100,200\}$; the shaded band indicates the interval between the lower and upper bounds.}
  \label{fig:bp_two_sample_ratio}
\end{figure}
Figure~\ref{fig:bp_two_sample_ratio} again exhibits a V-shape in $\theta$: the rejection breakdown point is minimized near $\theta=0$ and increases with $|\theta|$, reflecting that larger location separation is harder to overturn by contamination.
At matched Gaussian efficiency, the self-concordant loss yields smaller rejection breakdown points than Huber's and logcosh losses in Figure \ref{fig:bp_two_sample_ratio}, consistent with the robustness ordering observed for location estimation in Figure \ref{fig:bp_location}.
\subsection{Uniformity of the Bootstrap p-values}\label{sec:boot_simulation}
This section uses the randomized probability integral transform \citep{beran1988prepivoting, diebold1998evaluating, davidson2006power} to assess the quality of the multiplier bootstrap in approximating the sampling distribution of
$$
T=\sqrt{n}(\eta_{m/n\pm}(\hat \theta_b, \tilde F_n)-\eta_{\varepsilon \pm}).
$$
For Monte Carlo replication $j$, define
$$
T_j=\sqrt{n}(\eta_{m/n\pm}(\hat \theta^{(j)}, F_n^{(j)})-\eta_{\varepsilon \pm}), \qquad
T^{*(b)}_j=\sqrt{n}(\eta_{m/n\pm}(\hat \theta_b^{(j)}, \tilde F_n^{(b, j)})-\eta_{m/n\pm}(\hat \theta^{(j)}, F_{n}^{(j)})),\quad b=1,\dots,B,
$$
where $\eta_{m/n\pm}(\hat \theta^{(j)},  F_n^{(j)})$ is computed from i.i.d. standard normal sample $x^{(n)}$ under the Huber's loss ($\delta=1.345$), and $\eta_{m/n\pm}(\hat \theta_b^{(j)}, \tilde F_n^{(b, j)})$ is computed from the same $x^{(n)}$ using exponential multiplier weights $w_i\sim\mathrm{Exp}(1)$. To handle discreteness of the empirical bootstrap distribution, set
$$
R_j(t)=\sum_{b=1}^B \II\{T^{*(b)}_j\le t\},\qquad
U_j=\frac{R_j(T_j)+V_j}{B+1},\quad V_j\sim\mathrm{Unif}(0,1).
$$
If the bootstrap is valid for this statistic, then $\{U_j\}_{j=1}^M \overset{d}{\approx} \mathrm{Unif}(0,1)$. Accordingly, we display a PP-plot of the empirical quantiles of $\{U_j\}$ against the theoretical quantiles and a histogram of $\{U_j\}$ (with the $y=1$ reference line), and we report the Kolmogorov--Smirnov statistic and $p$-value for $H_0:U\sim\mathrm{Unif}(0,1)$. Figure~\ref{fig:pp_100} shows results for $n=100$, $M=1000$, $B=1000$, and Huber $\delta=1.345$ across $\varepsilon \in\{0.03,0.1,0.15\}$. The PP curves adhere closely to $y=x$ and the histograms concentrate around the unit baseline, indicating good agreement between the bootstrap conditional distribution and the sampling distribution, with only mild finite sample deviations in the tails. Figure~\ref{fig:pp_1000} shows results for $n=1000$, $M=3000$, $B=3000$ with less finite sample deviations.
\begin{figure}[htbp]
  \centering
  \begin{subfigure}[b]{0.32\textwidth}
    \centering
    \includegraphics[width=\linewidth]{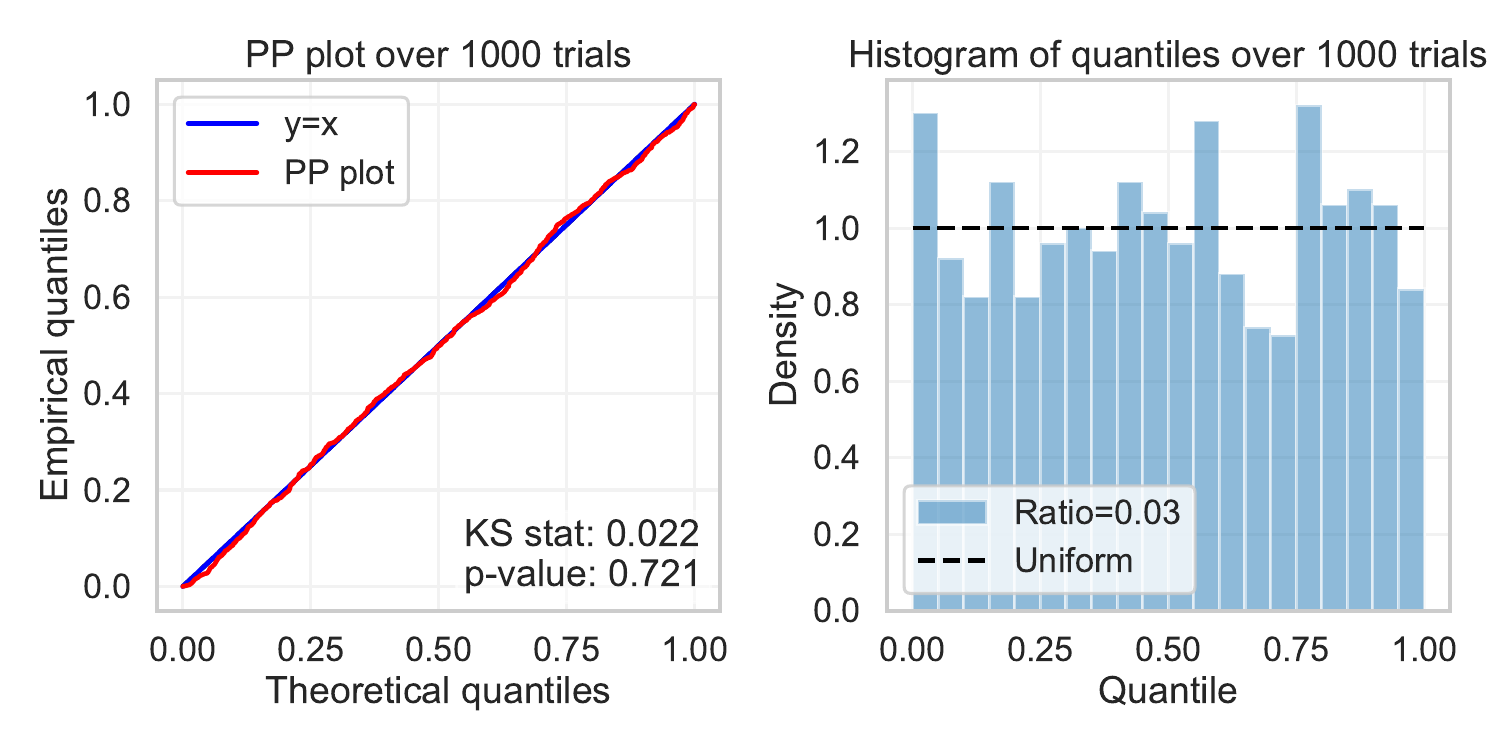}
    \label{fig:pp_3_100}
  \end{subfigure}
  \begin{subfigure}[b]{0.32\textwidth}
    \centering
    \includegraphics[width=\linewidth]{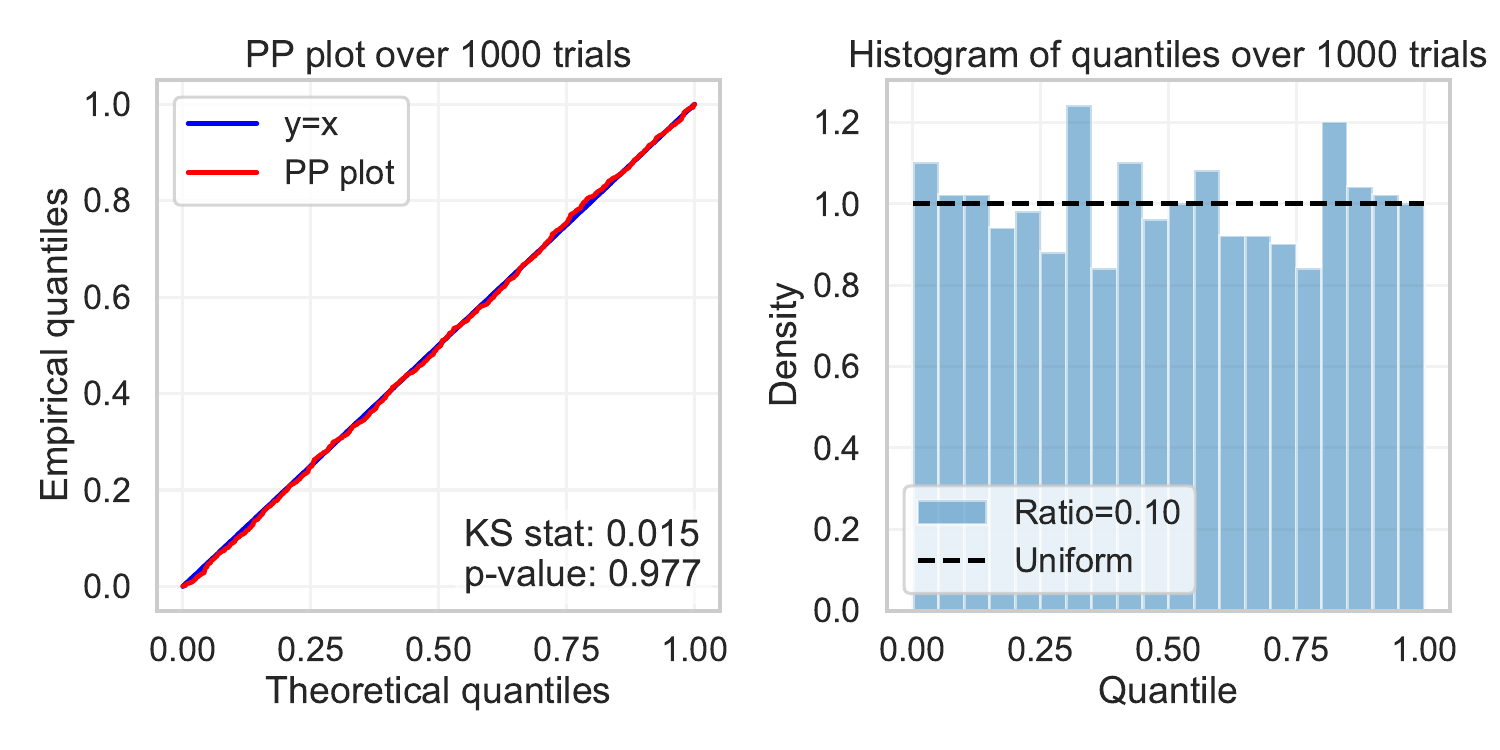}
    \label{fig:pp_10_100}
  \end{subfigure}
  \begin{subfigure}[b]{0.32\textwidth}
    \centering
    \includegraphics[width=\linewidth]{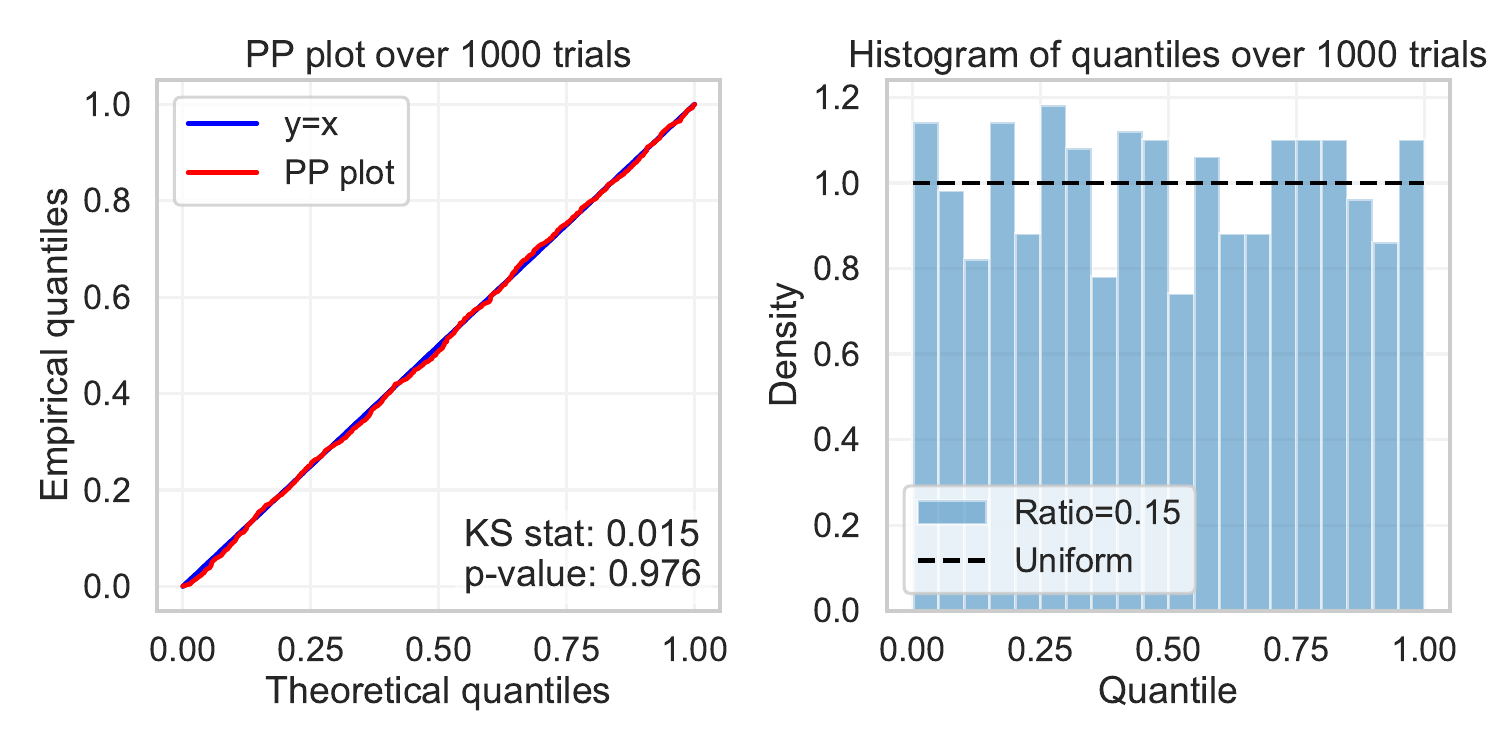}
    \label{fig:pp_15_100}
  \end{subfigure}
  \vspace{-15pt}
  \caption{ {\small Randomized PIT–based uniformity diagnostics for the multiplier bootstrap. Each panel corresponds to a different trimming proportion $\varepsilon \in \{0.03, 0.1, 0.15\}$, with $n=100$, $M=1000$, and $B=1000$. Under a valid bootstrap approximation, we expect the distribution to be uniform. Panels show PP curves near $y=x$ and histograms near the $y=1$ baseline, with in-panel Kolmogorov--Smirnov statistics and $p$-values indicating good fit.}}
  \label{fig:pp_100}
\end{figure}
\begin{figure}[htbp]
  \centering
  \begin{subfigure}[b]{0.32\textwidth}
    \centering
    \includegraphics[width=\linewidth]{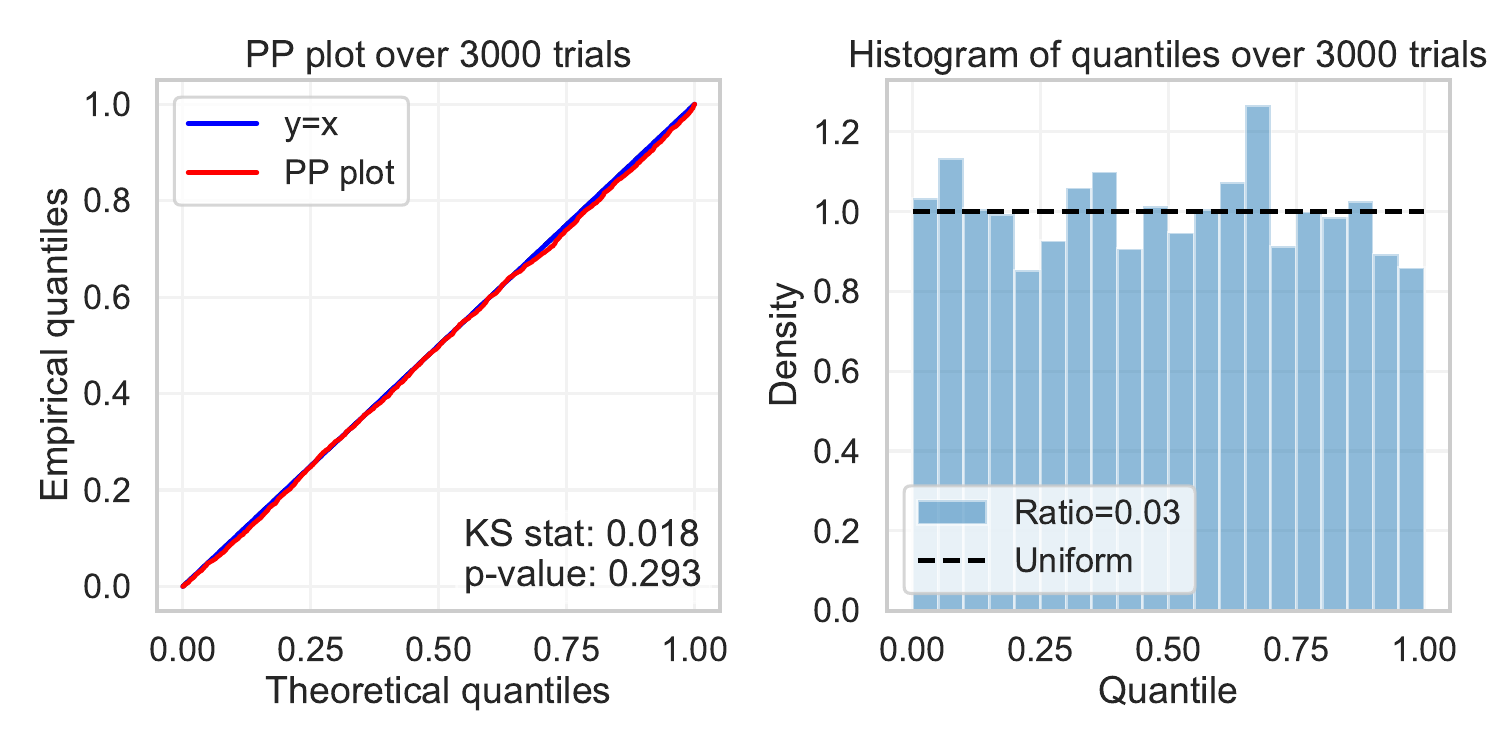}
    \label{fig:pp_3_1000}
  \end{subfigure}
  \begin{subfigure}[b]{0.32\textwidth}
    \centering
    \includegraphics[width=\linewidth]{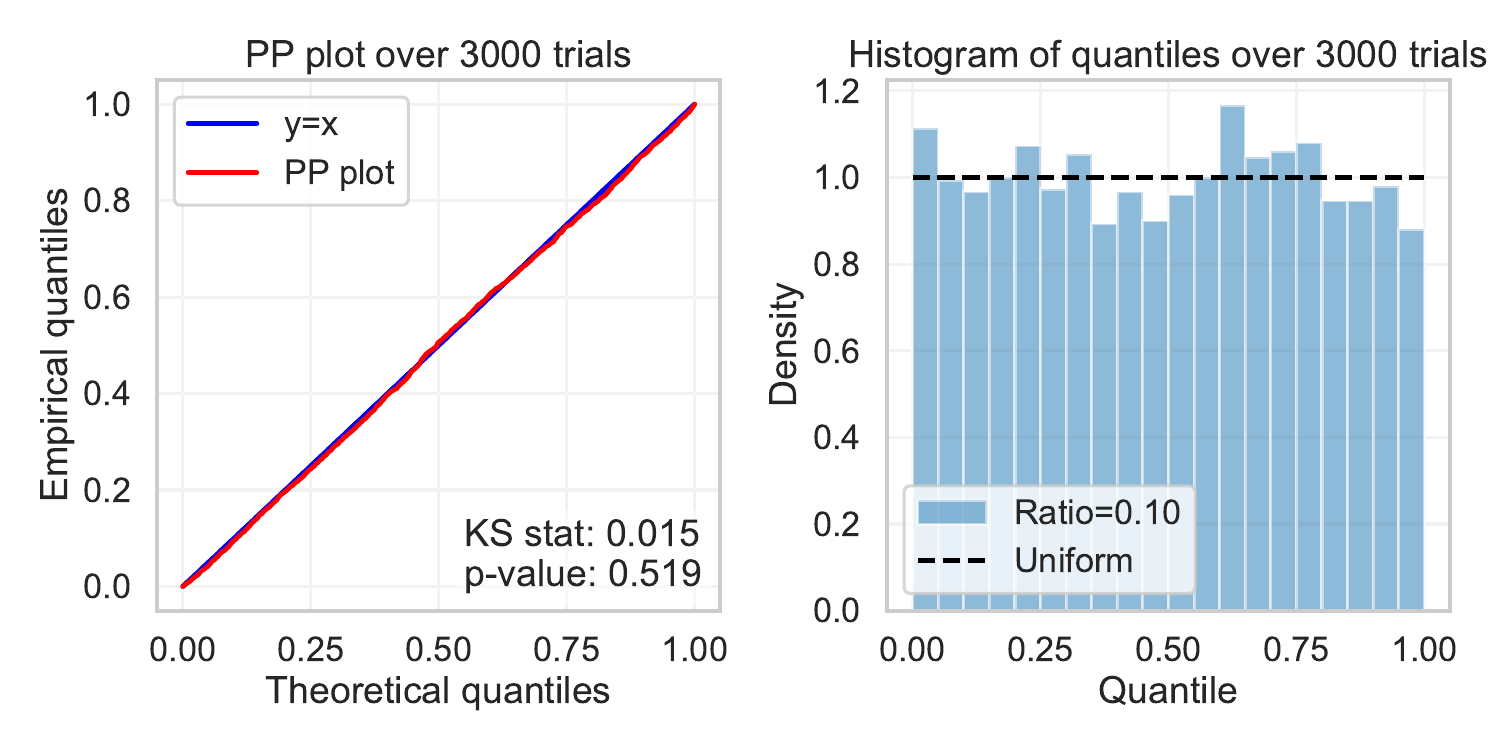}
    \label{fig:pp_10_1000}
  \end{subfigure}
  \begin{subfigure}[b]{0.32\textwidth}
    \centering
    \includegraphics[width=\linewidth]{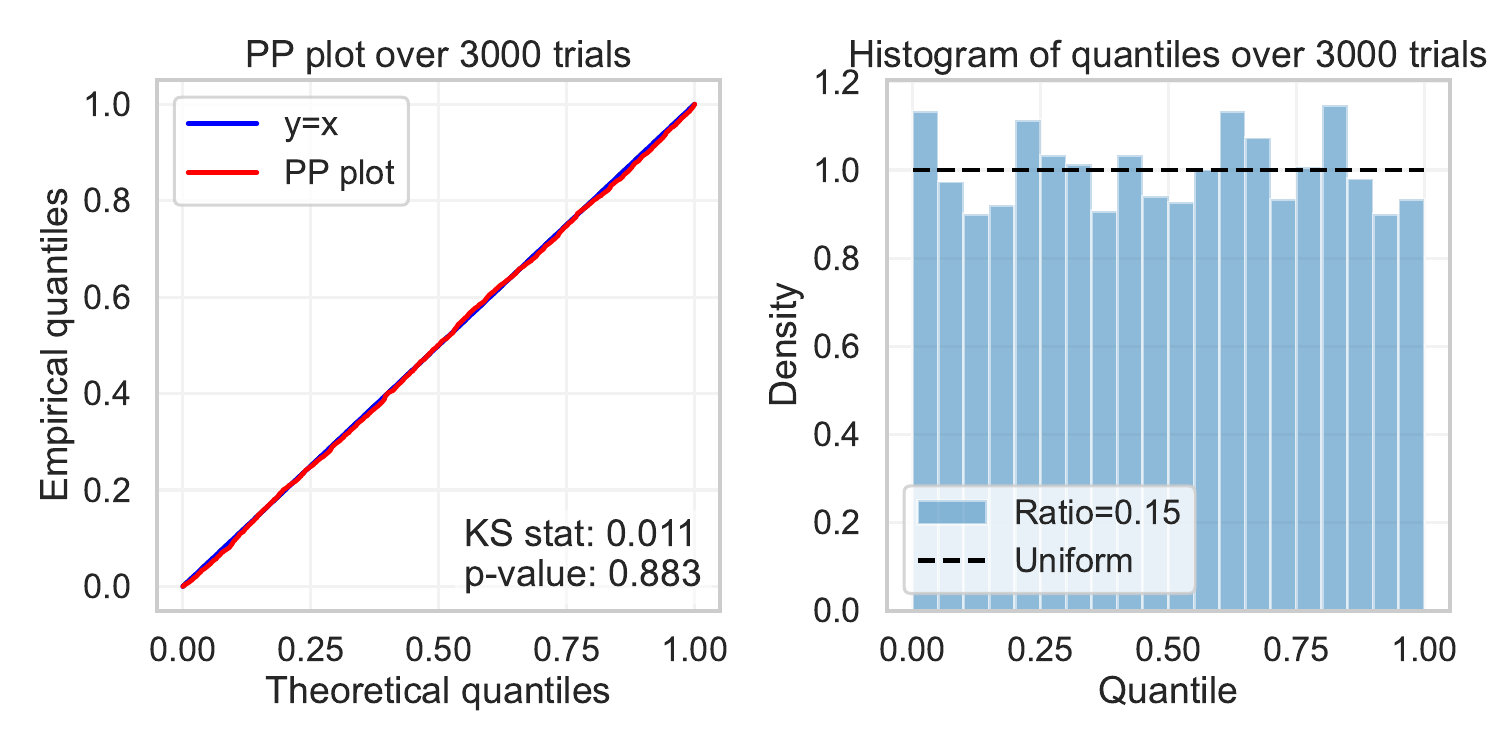}
    \label{fig:pp_15_1000}
  \end{subfigure}
  \vspace{-15pt}
  \caption{ {\small Randomized PIT–based uniformity diagnostics for the multiplier bootstrap. Each panel corresponds to a different trimming proportion $\varepsilon \in \{0.03, 0.1, 0.15\}$, with $n=1000$, $M=3000$, and $B=3000$. The finite sample deviation decreases compared to previous plots.}}
  \label{fig:pp_1000}
\end{figure}
\section{Proofs of Main Results}
\label{sec:proof}
We will restate the main results in a slightly different way and prove them below. We note that these alternative statement will be convenient for many of the breakdown point of tests results. Indeed, it will be quite useful to characterize the two directional breakdown points $\{ \BP_{\eta+},\BP_{\eta-}\}$ and $m$-sensitivities $\{\eta_{m/n+},\eta_{m/n-}\}$ for our different estimators.
\subsection{Proof for Location Estimation: Theorem \ref{thm:loc_main}}
Below we state Theorem \ref{thm:loc} which corresponds to Theorem \ref{thm:loc_main}. Note also that Corollary \ref{cor:opt_attack_m_est}   is a restatement of Corollary \ref{cor:opt_attack_m_est_main}.
\begin{theorem}
\label{thm:loc}
    Assume that $\psi(t-\theta)$ in \eqref{M-est} is non-increasing in $\theta$ that passes through 0. Then,
    \begin{align*}
        \BP_\eta (\hat\theta, x^{(n)}) = \min \left \{ \BP_{\eta+} (\hat\theta, x^{(n)}), \BP_{\eta-} (\hat\theta, x^{(n)})\right \},
    \end{align*}
    where
    \begin{align*}
        \BP_{\eta+} (\hat\theta, x^{(n)}) =  &~\frac{1}{n}\min \left\{m \in [n]: m \ge \frac{\sum_{i>m} \psi (x_{(i)} - (\hat{\theta} + \eta))}{-\psi(\infty)} \right \}, \\
    \BP_{\eta-} (\hat\theta, x^{(n)}) = &~  \frac{1}{n}\min \left\{m \in [n]: m \ge \frac{\sum_{i\le n-m} \psi(x_{(i)} - (\hat \theta - \eta))}{-\psi(-\infty)} \right \} .
    \end{align*}
\end{theorem}
\begin{proof}[Proof of Theorem \ref{thm:loc}]
If $\psi(\infty)$ is unbounded, then moving one point is enough to change the estimator to the right by any $\eta \in \RR_{>0}$, so $\BP_{\eta+}(\hat \theta, x^{(n)})=1/n$ in this case. Similarly, if $\psi(-\infty)$ is unbounded, $\BP_{\eta-}(\hat \theta, x^{(n)})=1/n$. In the following, we will assume $\psi(\pm \infty)$ is bounded.
Let's start with moving $\hat\theta(x^{(n)})$ to the right.
Consider any $m\in[n]$ and note that since $\psi(\cdot)$ is non-decreasing, for all $y^{(n)} \in B_\textsf{H}(x^{(n)},m)$,
\begin{align*}
\nonumber \frac{1}{n}\sum_{i=1}^n\psi(y_i-\theta)&\leq \frac{1}{n}\sum_{i>m}\psi(x_{(i)}-\theta)+\frac{m}{n}\psi(\infty).
\end{align*}
Our goal is to find the smallest $m$ such that
\begin{align*}
     \sup_{y^{(n)}\in B_\textsf{H}(x^{(n)},m) }\hat\theta(y^{(n)}) \ge \hat \theta(x^{(n)}) + \eta.
\end{align*}
Notice that any $\hat \theta(y^{(n)})$ must satisfy
\begin{align*}
    \frac{1}{n}\sum_{i=1}^n\psi(y_i-\hat \theta(y^{(n)})) = 0,
\end{align*}
which gives
\begin{align*}
    0 \le \frac{1}{n}\sum_{i>m}\psi(x_{(i)}-\sup_{y^{(n)}\in B_\textsf{H}(x^{(n)},m)} \hat \theta(y^{(n)}))+\frac{m}{n}\psi(\infty).
\end{align*}
Because $\psi(t-\theta)$ is non-increasing in $\theta$ and $\sup_{y^{(n)}\in B_\textsf{H}(x^{(n)},m) }\hat\theta(y^{(n)}) \ge \hat  \theta(x^{(n)}) + \eta$, we get
\begin{align}
    \label{eq:loc_sup}
    \frac{1}{n}\sum_{i>m}\psi(x_{(i)}-(\hat \theta(x^{(n)}) + \eta))+\frac{m}{n}\psi(\infty) \ge 0 \iff m\geq\frac{\sum_{i>m} \psi (x_{(i)} - (\hat{\theta}(x^{(n)}) + \eta))}{-\psi(\infty)}.
\end{align}
Let
$$
m^* := \min\bigg\{m: m\geq\frac{\sum_{i>m} \psi (x_{(i)} - (\hat{\theta}(x^{(n)}) + \eta))}{-\psi(\infty)} \bigg\}.
$$
and note that  for $m_0<m^*$ we must have that $\sup_{y^{(n)}\in B_\textsf{H}(x^{(n)},m_0)} \hat \theta(y^{(n)}) < \hat\theta(x^{(n)})+\eta$. Indeed, suppose this were not true and $\sup_{y^{(n)}\in B_\textsf{H}(x^{(n)},m_0)} \hat \theta(y^{(n)}) \ge \hat\theta(x^{(n)})+\eta$.  Then because $\psi(t-\theta)$ is non-increasing in $\theta$, we must have
$$
\frac{1}{n}\sum_{i>m_0}\psi(x_{(i)}-(\hat \theta(x^{(n)}) + \eta))+\frac{m_0}{n}\psi(\infty) \ge 0 \iff m_0\geq\frac{\sum_{i>m_0} \psi (x_{(i)} - (\hat{\theta}(x^{(n)}) + \eta))}{-\psi(\infty)},
$$
which contradicts the definition of $m^*$.
This gives one of the two terms inside the minimum in the desired result. The second term will follow from moving $\hat\theta(x^{(n)})$ to the left, which can be controlled with very similar arguments. In particular, consider now any $m\in[n]$ and note that since $\psi(\cdot)$ is non-decreasing, for all $y^{(n)} \in B_\textsf{H}(x^{(n)},m)$, {}
\begin{align*}
\nonumber \frac{1}{n}\sum_{i=1}^n\psi(y_i-\theta)\geq \frac{1}{n}\sum_{i\leq n-m}\psi(x_{(i)}-\theta)+\frac{m}{n}\psi(-\infty).
\end{align*}
And we want to find the smallest $m$ such that
\begin{align*}
    \inf_{y^{(n)}\in B_\textsf{H}(x^{(n)},m) }\hat\theta(y^{(n)}) \le \hat \theta - \eta.
\end{align*}
Notice that any $\hat \theta(y^{(n)})$ must satisfy
\begin{align*}
    \frac{1}{n}\sum_{i=1}^n\psi(y_i-\hat \theta(y^{(n)})) = 0,
\end{align*}
which gives
\begin{align*}
    {\frac{1}{n}\sum_{i\le n-m}\psi(x_{(i)}-\inf_{y^{(n)}\in B_\textsf{H}(x^{(n)},m)}\hat\theta(y^{(n)}))+\frac{m}{n}\psi(-\infty)\le0.}
\end{align*}
Because $\psi(t-\theta)$ is non-increasing in $\theta$ and $\inf_{y^{(n)}\in B_\textsf{H}(x^{(n)},m) }\hat\theta(y^{(n)}) \le \hat  \theta(x^{(n)}) - \eta$, we get
\begin{align}
\label{eq:loc_inf}
    \frac{1}{n}\sum_{i\le n-m}\psi(x_{(i)}-(\hat \theta(x^{(n)}) -\eta))+\frac{m}{n}\psi(-\infty) \leq 0 \iff m\geq\frac{\sum_{{i\le n-m}} \psi (x_{(i)} - (\hat{\theta}(x^{(n)}) - \eta))}{-\psi(-\infty)}.
\end{align}
This gives the second term inside the minimum of the claimed result and completes the proof.
\end{proof}
 The next corollary  follows directly from \eqref{eq:loc_sup} and \eqref{eq:loc_inf} in the proof of Theorem~\ref{thm:loc}.
\begin{corollary}
\label{cor:opt_attack_m_est}
     Assume that $\psi$  is non-decreasing and passes through $0$. Then,
    \begin{align*}
        \eta_{m/n}(\hat \theta,x^{(n)}) = \max\{\eta_{m/n+}(\hat \theta,x^{(n)}), \eta_{m/n-}(\hat \theta,x^{(n)})\},
    \end{align*}
    where
    \begin{align*}
        \eta_{m/n+}(\hat \theta,x^{(n)}) = & \max \left \{\eta: \sum_{i >m}\psi(x_{(i)}-(\hat \theta + \eta))+m \psi(\infty) \ge 0 \right \}, \\
        \eta_{m/n-}(\hat \theta,x^{(n)}) = &\max \left \{\eta:\sum_{i \le n-m}\psi(x_{(i)}-(\hat \theta-\eta))+m\psi(-\infty) \le 0\right \}.
    \end{align*}
\end{corollary}
\subsection{Proof of Scale and Two-stage M-estimators:  Propositions \ref{prop:two_stage_lb_main} and \ref{prop:two_stage}}
We first state a convenient variant of Corollary \ref{cor:scale_main}.
\begin{corollary}
\label{cor:scale}
Assume that $\chi(x)$ is non-decreasing in $|x|$, even, and passes through $0$.
For $\hat\sigma>0$,
\[
\BP_\eta(\hat\sigma,x^{(n)})
=\min\{\BP_{\eta+}(\hat\sigma,x^{(n)}),
        \BP_{\eta-}(\hat\sigma,x^{(n)})\}.
\]
{For $\eta>0$, the upward threshold is
\[
\BP_{\eta+}(\hat\sigma,x^{(n)})
=\frac1n\min\left\{m\in[n]:
\sum_{i>m}\chi\!\left(\frac{|x^{(n)}|_{(i)}}{\hat\sigma+\eta}\right)
       +m\chi(\infty)\ge0\right\}.
\]
For $0<\eta\le\hat\sigma$, the downward threshold is
\[
\BP_{\eta-}(\hat\sigma,x^{(n)})
=\frac1n\min\left\{m\in[n]:
\sum_{i\le n-m}\chi\!\left(\frac{|x^{(n)}|_{(i)}}{\hat\sigma-\eta}\right)
       +m\chi(0)\le0\right\}.
\]
At $\eta=\hat\sigma$, the score expression means its limit from below.
For $\eta>\hat\sigma$, $\BP_{\eta-}(\hat\sigma,x^{(n)})=+\infty$.
The one-sided sensitivities are
\begin{align*}
\eta_{m/n+}(\hat\sigma,x^{(n)})
&=\sup\left\{\eta\ge0:
\sum_{i>m}\chi\!\left(\frac{|x^{(n)}|_{(i)}}{\hat\sigma+\eta}\right)
       +m\chi(\infty)\ge0\right\},\\
\eta_{m/n-}(\hat\sigma,x^{(n)})
&=\sup\left\{0\le\eta\le\hat\sigma:
\sum_{i\le n-m}\chi\!\left(\frac{|x^{(n)}|_{(i)}}{\hat\sigma-\eta}\right)
       +m\chi(0)\le0\right\}.
\end{align*}
}
\end{corollary}
\begin{proof}[Proof of Corollary~\ref{cor:scale}]
    The corollary follows immediately from Theorem~\ref{thm:loc} by the log-scale reparametrization
    $$
    \tilde x_i^{(n)}=\log |x_i^{(n)}|,\qquad \tilde \sigma=\log \sigma,\qquad \tilde\chi(u)=\chi(e^u),
    $$
    with $\tilde x_i^{(n)}=-\infty$ when $x_i^{(n)}=0$. Then
    $
\chi({x_i^{(n)}}/{\sigma})=\tilde\chi(\tilde x_i^{(n)}-\tilde \sigma),
    $
    and $\tilde\chi$ satisfies the assumptions of Theorem~\ref{thm:loc}, so the formulas are exactly \eqref{eq:loc_sup} and \eqref{eq:loc_inf} written back on the original scale.
{For the downward direction, the target scale is nonnegative only when $\eta\le\hat\sigma$. The endpoint $\eta=\hat\sigma$ follows by taking the target scale down to zero.}
\end{proof}
\begin{proposition}
\label{prop:two_stage}
Assume that $\psi$  is non-decreasing, bounded, and passes through $0$. Assume also that $\hat\sigma(x^{(n)})$ is translationally invariant in $x$.
For $\tilde x^{(n)} \in \mathsf C^L_{m,\{\infty,\,x_{(n)}\}}(x^{(n)})$, define
$$
\mathsf T_\eta^+(\tilde x^{(n)})
:=
-\frac{\sum_{i\le n-m}\psi\big((\tilde x_{(i)}-(\hat\theta (x^{(n)})+\eta))/\hat\sigma(\tilde x^{(n)})\big)}
{\max\!\left\{\psi\big((\tilde x_{(n)}-(\hat\theta (x^{(n)})+\eta))/\hat\sigma(\tilde x^{(n)})\big),\,0\right\}}.
$$
For $\tilde x^{(n)} \in \mathsf C^R_{m,\{-\infty,\,x_{(1)}\}}(x^{(n)})$, define
$$
\mathsf T_\eta^-(\tilde x^{(n)})
:=
-\frac{\sum_{i>m}\psi\big((\tilde x_{(i)}-(\hat\theta (x^{(n)})-\eta))/\hat\sigma(\tilde x^{(n)})\big)}
{\min\!\left\{\psi\big((\tilde x_{(1)}-(\hat\theta (x^{(n)})-\eta))/\hat\sigma(\tilde x^{(n)})\big),\,0\right\}}.
$$
Then
\begin{align*}
&\BP_{\eta+}(\hat\theta,x^{(n)})
\le
\frac1n \min\Bigl\{m \in [n] : m \ge \min_{\tilde x^{(n)} \in \mathsf C^L_{m,\{\infty,\,x_{(n)}\}}} \mathsf T_\eta^+(\tilde x^{(n)})\Bigr\}, \\
&\BP_{\eta-}(\hat\theta,x^{(n)})
\le
\frac1n \min\Bigl\{m \in [n] : m \ge \min_{\tilde x^{(n)} \in \mathsf C^R_{m,\{-\infty,\,x_{(1)}\}}} \mathsf T_\eta^-(\tilde x^{(n)})\Bigr\}.
\end{align*}
Moreover, let $\tilde\theta(\tilde x^{(n)})$ denote the largest solution to \eqref{M-est-two} when $\tilde x^{(n)} \in \mathsf C^L_{m,\{\infty,\,x_{(n)}\}}(x^{(n)})$, and the smallest solution when $\tilde x^{(n)} \in \mathsf C^R_{m,\{-\infty,\,x_{(1)}\}}(x^{(n)})$. Then
\begin{align*}
&\eta_{m/n,+}(\hat\theta,x^{(n)})
\ge
\max_{\tilde x^{(n)} \in \mathsf C^L_{m,\{\infty,\,x_{(n)}\}}}
\bigl(\tilde\theta(\tilde x^{(n)})-\hat\theta (x^{(n)})\bigr),\\
&\eta_{m/n,-}(\hat\theta,x^{(n)})
\ge
\max_{\tilde x^{(n)} \in \mathsf C^R_{m,\{-\infty,\,x_{(1)}\}}}
\bigl(\hat\theta (x^{(n)})-\tilde\theta(\tilde x^{(n)})\bigr).
\end{align*}
\end{proposition}
\begin{proof}
Write $\hat\theta=\hat\theta(x^{(n)})$. For any sample $z^{(n)}$, define
$$
G_{z^{(n)}}(\theta)
:=
\sum_{i=1}^n
\psi\!\left(\frac{z_i-\theta}{\hat\sigma(z^{(n)})}\right).
$$
Since $\theta \mapsto \psi(t-\theta)$ is non-increasing for each fixed $t$, the map
$$
\theta \longmapsto G_{z^{(n)}}(\theta)
$$
is non-increasing for every fixed $z^{(n)}$.
We first prove the upper bound for $\BP_{\eta+}(\hat\theta,x^{(n)})$. Fix $m\in[n]$ and let $\tilde x^{(n)} \in \mathsf C^L_{m,\{\infty,\,x_{(n)}\}}(x^{(n)})$. By construction, the largest $m$ order statistics of $\tilde x^{(n)}$ all equal $\tilde x_{(n)}$, so
$$
G_{\tilde x^{(n)}}(\hat\theta+\eta)
=
\sum_{i\le n-m}
\psi\!\left(
\frac{\tilde x_{(i)}-(\hat\theta+\eta)}
{\hat\sigma(\tilde x^{(n)})}
\right)
+
m\,
\psi\!\left(
\frac{\tilde x_{(n)}-(\hat\theta+\eta)}
{\hat\sigma(\tilde x^{(n)})}
\right).
$$
Therefore, if
$$
m \ge \mathsf T_\eta^+(\tilde x^{(n)}),
$$
then
$$
G_{\tilde x^{(n)}}(\hat\theta+\eta)\ge 0.
$$
Since $G_{\tilde x^{(n)}}$ is non-increasing, the largest solution of the estimating equation for $\tilde x^{(n)}$ is then at least $\hat\theta+\eta$. Hence an $\eta$-breakdown to the right occurs with at most $m$ contaminated observations. Taking the smallest such $m$ over all admissible $\tilde x^{(n)}$ yields
$$
\BP_{\eta+}(\hat\theta,x^{(n)})
\le
\frac1n \min\Bigl\{m \in [n] : m \ge \min_{\tilde x^{(n)} \in \mathsf C^L_{m,\{\infty,\,x_{(n)}\}}} \mathsf T_\eta^+(\tilde x^{(n)})\Bigr\}.
$$
The proof for $\BP_{\eta-}(\hat\theta,x^{(n)})$ is analogous. Fix $m\in[n]$ and let $\tilde x^{(n)} \in \mathsf C^R_{m,\{-\infty,\,x_{(1)}\}}(x^{(n)})$. Then the smallest $m$ order statistics of $\tilde x^{(n)}$ all equal $\tilde x_{(1)}$, and
$$
G_{\tilde x^{(n)}}(\hat\theta-\eta)
=
m\,
\psi\!\left(
\frac{\tilde x_{(1)}-(\hat\theta-\eta)}
{\hat\sigma(\tilde x^{(n)})}
\right)
+
\sum_{i>m}
\psi\!\left(
\frac{\tilde x_{(i)}-(\hat\theta-\eta)}
{\hat\sigma(\tilde x^{(n)})}
\right).
$$
Thus, if
$$
m \ge \mathsf T_\eta^-(\tilde x^{(n)}),
$$
then
$$
G_{\tilde x^{(n)}}(\hat\theta-\eta)\le 0.
$$
Since $G_{\tilde x^{(n)}}$ is non-increasing, the smallest solution of the estimating equation for $\tilde x^{(n)}$ is then at most $\hat\theta-\eta$. Hence an $\eta$-breakdown to the left occurs with at most $m$ contaminated observations. Minimizing over $m$ and over admissible $\tilde x^{(n)}$ gives
$$
\BP_{\eta-}(\hat\theta,x^{(n)})
\le
\frac1n \min\Bigl\{m \in [n] : m \ge \min_{\tilde x^{(n)} \in \mathsf C^R_{m,\{-\infty,\,x_{(1)}\}}} \mathsf T_\eta^-(\tilde x^{(n)})\Bigr\}.
$$
For the sensitivity bounds, note that
$$
\mathsf C^L_{m,\{\infty,\,x_{(n)}\}}(x^{(n)})
\subseteq
B_{\mathrm H}(x^{(n)},m),
\qquad
\mathsf C^R_{m,\{-\infty,\,x_{(1)}\}}(x^{(n)})
\subseteq
B_{\mathrm H}(x^{(n)},m).
$$
Therefore,
$$
\eta_{m/n,+}(\hat\theta,x^{(n)})
\ge
\max_{\tilde x^{(n)} \in \mathsf C^L_{m,\{\infty,\,x_{(n)}\}}}
\bigl(\tilde\theta(\tilde x^{(n)})-\hat\theta\bigr)
$$
and
$$
\eta_{m/n,-}(\hat\theta,x^{(n)})
\ge
\max_{\tilde x^{(n)} \in \mathsf C^R_{m,\{-\infty,\,x_{(1)}\}}}
\bigl(\hat\theta-\tilde\theta(\tilde x^{(n)})\bigr).
$$
This completes the proof.
\end{proof}
\begin{proposition}
\label{prop:two_stage_lb}
Assume that  $\psi$ is non-decreasing, odd, and passes through $0$. Let $\hat\theta(x^{(n)})$ denote the two stage estimator defined in  \eqref{M-est-two}. For $m\in[n]$, set $\underline{\sigma}_{m/n} := \max\{\hat \sigma(x^{(n)}) - \eta_{m/n-}(\hat \sigma, x^{(n)}), 0 \}$ and $\overline{\sigma}_{m/n} := \hat \sigma(x^{(n)}) + \eta_{m/n+}(\hat \sigma, x^{(n)})$. Define the sign-split rescalings
$$
r_i
=
{\frac{x_i}{\underline{\sigma}_{m/n}}}
\II\{x_i \ge 0\}
+
\frac{x_i}{\overline{\sigma}_{m/n}}
\II\{x_i < 0\},
\quad
r'_i =
{\frac{x_i}{\underline{\sigma}_{m/n}}}
\II\{x_i < 0\}
+
\frac{x_i}{\overline{\sigma}_{m/n}}
\II\{x_i \ge 0\}
$$
Let $\hat\theta_{\mathrm{loc}}(x^{(n)})$ denote the solution to \eqref{M-est} (i.e., the vanilla location $M$-estimator evaluated at $x^{(n)}$ or the two-stage $M$-estimator with $\hat \sigma (x^{(n)}) = 1$).
Denote
\begin{align*}
    &\sigma_{m/n+} := \underline{\sigma}_{m/n}\II\{\hat\theta_{\rm loc}(r^{(n)}) + \eta_{m/n+}(\hat\theta_{\rm loc}, r^{(n)}) \le 0\} + \overline{\sigma}_{m/n} \II\{\hat\theta_{\rm loc}(r^{(n)}) + \eta_{m/n+}(\hat\theta_{\rm loc}, r^{(n)}) > 0\}, \\
    &\sigma_{m/n-} := \underline{\sigma}_{m/n}\II\{\hat\theta_{\rm loc}(r^{(n)'}) - \eta_{m/n-}(\hat\theta_{\rm loc}, r^{(n)'}) \ge 0\} + \overline{\sigma}_{m/n} \II\{\hat\theta_{\rm loc}(r^{(n)'}) - \eta_{m/n-}(\hat\theta_{\rm loc}, r^{(n)'}) < 0\}.
\end{align*}
Further define
\begin{align*}
    & \overline{\eta}_{m/n+}(\hat\theta,x^{(n)}) := \sigma_{m/n+} \cdot (\hat\theta_{\rm loc}(r^{(n)}) + \eta_{m/n+}(\hat\theta_{\rm loc}, r^{(n)}))-\hat\theta(x^{(n)}), \\
     & \overline \eta_{m/n-}(\hat\theta,x^{(n)}) :=  \hat\theta(x^{(n)}) - \sigma_{m/n-} \cdot (\hat\theta_{\rm loc}(r^{(n)'}) - \eta_{m/n-}(\hat\theta_{\rm loc}, r^{(n)'})).
\end{align*}
Then the one-sided $m$-sensitivity satisfies
\begin{align*}
  \eta_{m/n+}(\hat \theta,x^{(n)})
  \le
  \overline{\eta}_{m/n+}(\hat\theta,x^{(n)}),
  \quad
  \eta_{m/n-}(\hat \theta,x^{(n)})
  \le
  \overline{\eta}_{m/n-}(\hat\theta,x^{(n)}).
\end{align*}
Furthermore
\begin{align*}
        \BP_{\eta+} (\hat \theta, x^{(n)}) & \ge \frac{1}{n} \min \left \{m \in [n]: \overline{\eta}_{m/n+}(\hat\theta,x^{(n)}) \ge \eta \right \}, \\
        \BP_{\eta-} (\hat \theta, x^{(n)}) & \ge \frac{1}{n} \min \left \{m \in [n]: \overline{\eta}_{m/n-}(\hat\theta,x^{(n)}) \ge \eta \right \}.
\end{align*}
\end{proposition}
\begin{proof}[Proof of Proposition \ref{prop:two_stage_lb}]
Note that by scale invariance,
$$
\hat\theta(x^{(n)})
=
\hat\sigma(x^{(n)})\,
\hat\theta_{\mathrm{loc}}\!\bigl(x^{(n)}/\hat\sigma(x^{(n)})\bigr).
$$
We start with $\eta_{m/n+}$. Fix $m \in [n]$, and assume $\hat \sigma(x^{(n)}) - \eta_{m/n-}(\hat \sigma, x^{(n)}) \ge 0$ without loss of generality. By the definition of the one-sided scale sensitivities, any $m$-contaminated sample $\tilde x^{(n)}$ satisfies
$$
\hat\sigma(\tilde x^{(n)})\in
\Bigl[
\hat\sigma(x^{(n)})-\eta_{m/n-}(\hat\sigma,x^{(n)}),\,
\hat\sigma(x^{(n)})+\eta_{m/n+}(\hat\sigma,x^{(n)})
\Bigr].
$$
Consider the relaxed contamination  class $\widetilde{\mathsf C}_{m}(x^{(n)} / \hat \sigma (x^{(n)})$ obtained by allowing, after changing $m$ observations arbitrarily, each uncontaminated observation to be rescaled by an arbitrary factor in
$$
\left[
\frac{1}{\hat\sigma(x^{(n)})+\eta_{m/n+}(\hat\sigma,x^{(n)})},\,
\frac{1}{\hat\sigma(x^{(n)})-\eta_{m/n-}(\hat\sigma,x^{(n)})}
\right]
$$
instead of a fixed value for all observations.
This contains the original $m$-contamination class after standardization, hence
\begin{align*}
    \eta_{m/n+}(\hat\theta,x^{(n)})
    = &~
    \sup_{\tilde x^{(n)}\in B_{\sf H}(x^{(n)}, m)}
    \bigl(\hat\theta(\tilde x^{(n)})-\hat\theta(x^{(n)})\bigr) \\
    = & ~ \sup_{\tilde x^{(n)}\in B_{\sf H}(x^{(n)}, m)}
    \bigl(\hat \sigma(\tilde x^{(n)}) \hat\theta_{\rm loc}(\tilde x^{(n)}/ \hat \sigma(\tilde x^{(n)})-\hat\theta(x^{(n)})\bigr)
    \\
    \le & ~ \sup_{\tilde z^{(n)}\in \widetilde{\mathsf C}_{m}(x^{(n)}/\hat \sigma(x^{(n)}))} \sup_{\tilde x^{(n)}\in B_{\sf H}(x^{(n)}, m)} \bigl(\hat \sigma(\tilde x^{(n)}) \hat\theta_{\rm loc}(\tilde z^{(n)})-\hat\theta(x^{(n)})\bigr).
\end{align*}
By the monotonicity of $\psi$, the first supremum on the right is attained by replacing the $m$ smallest observations by arbitrarily large positive values and applying the sign-split rescaling, where the dataset is denoted as $\tilde z_+^{(n)} := \{r^{(n)}_{(i)}\}_{i=m+1}^n \cup \{+\infty\}_{j=1}^m$.
Hence,
$$
\hat\theta_{\rm loc}(\tilde z_+^{(n)}) = \hat\theta_{\rm loc}(r^{(n)}) + \eta_{m/n+}(\hat\theta_{\rm loc}, r^{(n)}).
$$
Therefore, we have
\begin{align*}
    \eta_{m/n+}(\hat\theta,x^{(n)}) \le \sup_{\tilde x^{(n)}\in B_{\sf H}(x^{(n)}, m)} \bigl(\hat \sigma(\tilde x^{(n)}) (\hat\theta_{\rm loc}(r^{(n)}) + \eta_{m/n+}(\hat\theta_{\rm loc}, r^{(n)}))-\hat\theta(x^{(n)})\bigr).
\end{align*}
If $\hat\theta_{\rm loc}(r^{(n)}) + \eta_{m/n+}(\hat\theta_{\rm loc}, r^{(n)}) < 0$, we should take the smallest possible $\underline{\sigma}_{m/n}$ for an upper bound. Conversely, if $\hat\theta_{\rm loc}(r^{(n)}) + \eta_{m/n+}(\hat\theta_{\rm loc}, r^{(n)}) > 0$, we should take the largest value $\overline{\sigma}_{m/n}$. Denote
$$
\sigma_{m/n+} := \underline{\sigma}_{m/n}\II\{\hat\theta_{\rm loc}(r^{(n)}) + \eta_{m/n+}(\hat\theta_{\rm loc}, r^{(n)}) \le 0\} + \overline{\sigma}_{m/n} \II\{\hat\theta_{\rm loc}(r^{(n)}) + \eta_{m/n+}(\hat\theta_{\rm loc}, r^{(n)}) > 0\}.
$$
This yields
\begin{align*}
    \eta_{m/n+}(\hat\theta,x^{(n)}) \le  \sigma_{m/n+} \cdot (\hat\theta_{\rm loc}(r^{(n)}) + \eta_{m/n+}(\hat\theta_{\rm loc}, r^{(n)}))-\hat\theta(x^{(n)}).
\end{align*}
The argument for $\eta_{m/n-}$ is symmetric. We have
\begin{align*}
    \eta_{m/n-}(\hat\theta,x^{(n)})
    = &~
    \sup_{\tilde x^{(n)}\in B_{\sf H}(x^{(n)}, m)}
    \bigl(\hat\theta(x^{(n)}) - \hat\theta(\tilde x^{(n)})\bigr) \\
    = & ~ \sup_{\tilde x^{(n)}\in B_{\sf H}(x^{(n)}, m)}
    \bigl(\hat\theta(x^{(n)})-\hat \sigma(\tilde x^{(n)}) \hat\theta_{\rm loc}(\tilde x^{(n)}/ \hat \sigma(\tilde x^{(n)})\bigr)
    \\
    \le & ~ \sup_{\tilde z^{(n)}\in \widetilde{\mathsf C}_{m}(x^{(n)}/\hat \sigma(x^{(n)}))} \sup_{\tilde x^{(n)}\in B_{\sf H}(x^{(n)}, m)} \bigl(\hat\theta(x^{(n)})-\hat \sigma(\tilde x^{(n)}) \hat\theta_{\rm loc}(\tilde z^{(n)})\bigr).
\end{align*}
By the monotonicity of $\psi$, the first supremum on the right is attained by replacing the $m$ largest observations by arbitrarily small values and applying the sign-split rescaling, where the dataset is denoted as $\tilde z_-^{(n)} := \{r^{(n)'}_{(i)}\}_{i=1}^{n-m} \cup \{-\infty\}_{j=1}^m$.
Similarly, we have
$$
\hat\theta_{\rm loc}(\tilde z_-^{(n)}) = \hat\theta_{\rm loc}(r^{(n)'}) - \eta_{m/n-}(\hat\theta_{\rm loc}, r^{(n)'}),
$$
which gives
$$
\eta_{m/n-}(\hat\theta,x^{(n)}) \le \sup_{\tilde x^{(n)}\in B_{\sf H}(x^{(n)}, m)} \bigl(\hat\theta(x^{(n)}) - \hat \sigma(\tilde x^{(n)}) (\hat\theta_{\rm loc}(r^{(n)'}) - \eta_{m/n-}(\hat\theta_{\rm loc}, r^{(n)'}))\bigr).
$$
If $\hat\theta_{\rm loc}(r^{(n)'}) - \eta_{m/n-}(\hat\theta_{\rm loc}, r^{(n)'}) < 0$, we should take the largest possible $\overline{\sigma}_{m/n}$ for an upper bound. Conversely, if $\hat\theta_{\rm loc}(r^{(n)'}) - {\eta_{m/n-}}(\hat\theta_{\rm loc}, r^{(n)'}) > 0$, we should take the smallest value $\underline{\sigma}_{m/n}$.  Denote
\[
\fitdisplay{\sigma_{m/n-} := \underline{\sigma}_{m/n}\II\{\hat\theta_{\rm loc}(r^{(n)'}) {- \eta_{m/n-}}(\hat\theta_{\rm loc}, r^{(n)'}) \ge 0\} + \overline{\sigma}_{m/n} \II\{\hat\theta_{\rm loc}(r^{(n)'}) {- \eta_{m/n-}}(\hat\theta_{\rm loc}, r^{(n)'}) < 0\}.}
\]
This yields
\begin{align*}
    \eta_{m/n-}(\hat\theta,x^{(n)}) \le  \hat\theta(x^{(n)}) - \sigma_{m/n-} \cdot (\hat\theta_{\rm loc}(r^{(n)'}) - \eta_{m/n-}(\hat\theta_{\rm loc}, r^{(n)'})).
\end{align*}
\end{proof}
\begin{proof}[Proof of Proposition \ref{prop:two_stage_lb_main}]
    This is a direct (and looser) result of Proposition \ref{prop:two_stage_lb}, where the interval formed by $\underline{\sigma}_{m/n}$ and $\overline{\sigma}_{m/n}$ is wider than the one in Proposition \ref{prop:two_stage_lb} since $\eta_{m/n\pm}(\hat \sigma, x^{(n)}) \le \eta_{m/n}(\hat \sigma, x^{(n)})$ by definition.
\end{proof}
\subsection{Proof of Variance Estimation Threshold BP}
\begin{lemma}
\label{lem:se_at_theta_0}
     Assume that $\psi$  is differentiable a.e., bounded,  non-decreasing and passes through $0$. Let the data be centered at $\theta_0$.
     Then, the one-sided threshold BP for the {restricted} plug-in estimate of the standard error evaluated at $\theta_0$ is
    \begin{equation*}
\fitdisplay{\BP_{\eta+} (\hatse (\theta_0), x^{(n)}) =  \frac{1}{n}\min \left\{m \in [n]: \max_{\substack{I \subseteq [n] \\ |I| = n - m}}
        \sqrt{\frac{\sum_{i\in I} \psi(x_i)^2 + m \psi^2_{\max}}{(\sum_{i \in I} \psi'(x_i))^2}} \mathrel{{\ge}}\hatse (\theta_0) + \eta
\right \},}
\end{equation*}
    where $\psi_{\max} = \max\{|\psi(\infty)|, |\psi(-\infty)|\}$.
    If we further assume that $\psi$ is odd, $\psi'(x)$ is non-increasing in $|x|$, we have
    \begin{equation}
\label{eq:bp+se}
\fitdisplay[\dimexpr\linewidth-3em\relax]{\BP_{\eta+} (\hatse (\theta_0), x^{(n)}) = \frac{1}{n}  \min \left\{m \in [n]: m \mathrel{{\ge}} \frac{(\hatse (\theta_0) + \eta)^2 \left ( \sum_{i>m} \psi'(x_{\pi_i}) \right )^2 - \sum_{i > m} \psi(x_{\pi_i})^2}{\psi^2_{\max}}
\right \},}
\end{equation}
    where $\pi$ is a permutation such that $|x_{\pi_1}| \le |x_{\pi_2}| \le \dots \le |x_{\pi_n}|$.
    The one-sided $m$-sensitivity for the plug-in estimate of the standard error satisfying \eqref{eq:bp+se} is
    \begin{gather*}
        \eta_{m/n+}(\hatse(\theta_0), x^{(n)}) = \sqrt{\frac{m \psi_{\max}^2 + \sum_{i>m} \psi(x_{\pi_i})^2}{(\sum_{i>m} \psi'(x_{\pi_i}))^2}} - \hatse(\theta_0).
    \end{gather*}
    Moreover, we have
    \begin{gather*}
\fitdisplay{{\BP_{\eta-}(\hatse(\theta_0),x^{(n)})
        =\frac1n\min\left\{m\in[n]:
        \sqrt{\frac{\sum_{i\le n-m}\psi(x_{\pi_i})^2}
        {(\sum_{i\le n-m}\psi'(x_{\pi_i})+m\psi'(0))^2}}
        \le\hatse(\theta_0)-\eta\right\},}}
\end{gather*}
    and that
    \begin{gather*}
        \eta_{m/n-}(\hatse(\theta_0), x^{(n)}) = \hatse(\theta_0) - \sqrt{\frac{\sum_{i\le n-m} \psi(x_{\pi_i})^2}{(\sum_{i\le n-m} \psi'(x_{\pi_i})+m\psi'(0))^2}} .
    \end{gather*}
\end{lemma}
\begin{proof}[Proof of Lemma \ref{lem:se_at_theta_0}]
After centering at $\theta_0$, the restricted plug-in estimate of the standard error is
\begin{equation*}
    \hatse(\theta_0)
    =
    \sqrt{\frac{\sum_{i=1}^n \psi(x_i)^2}{\left(\sum_{i = 1}^n \psi'(x_i)\right)^2}}.
\end{equation*}
\emph{Step 1: Worst contamination for a single point.}
Consider first the effect of changing a single observation $x_j$. We can write
\begin{equation*}
    \hatse(\theta_0)
    =
    \sqrt{\frac{\sum_{i\ne j} \psi(x_i)^2 + \psi(x_j)^2}{\left(\sum_{i \ne j} \psi'(x_i) + \psi'(x_j)\right)^2}}.
\end{equation*}
To increase this ratio, one wants to (i) make the numerator contribution $\psi(x_j)^2$ as large as possible, and (ii) make the denominator contribution $\psi'(x_j)$ as small as possible.
By monotonicity, $|\psi(x)|$ is maximized as $x \to \pm\infty$, and we denote the corresponding bound is $
  \psi_{\max} = \max\{|\psi(\infty)|, |\psi(-\infty)|\}.
$ Thus, by moving $x_j$ to $\pm\infty$, one can replace $\psi(x_j)^2$ by $\psi_{\max}^2$.
Moreover, since $\psi$ is non-decreasing and bounded, $\psi'(x)\ge 0$ for all $x$ and $\psi'(x)$ can be made arbitrarily small by moving $x$ sufficiently far into the tails. In the limiting contamination, we may effectively set $\psi'(x_j)$ to $0$. Hence, for a single-point contamination, the worst case for enlarging $\hatse(\theta_0)$ is to send that point to $\pm\infty$, replacing its contribution by
$$
\psi(x_j)^2 \to \psi_{\max}^2, \qquad \psi'(x_j) \to 0.
$$
\emph{Step 2: Worst $m$-point contamination.}
Now let $m \ge 1$ and consider an $m$-point contamination. Let $J \subset [n]$ be the set of $m$ contaminated indices and $I = [n]\setminus J$ the set of uncontaminated indices, with $|I| = n-m$. By repeating the one-point argument for each $j\in J$, one can (in the limit) achieve
\begin{equation*}
  \sum_{i=1}^n \psi(x_i)^2
  \to
  \sum_{i\in I} \psi(x_i)^2 + m \psi_{\max}^2,
  \qquad
  \sum_{i=1}^n \psi'(x_i)
  \to
  \sum_{i\in I} \psi'(x_i).
\end{equation*}
For a fixed set $I$, this yields the contaminated restricted standard error
\begin{equation*}
    \se_I^{(+)}
    =
    \sqrt{\frac{\sum_{i\in I} \psi(x_i)^2 + m \psi_{\max}^2}{\left(\sum_{i\in I} \psi'(x_i)\right)^2}}.
\end{equation*}
Maximizing over the choice of which $m$ points to contaminate is equivalent to maximizing over the choice of $I\subset[n]$ with $|I|=n-m$. This gives
\begin{equation*}
    \sup_{\tilde x \in B_{\mathrm{H}}(x^{(n)},m)} \hatse(\theta_0;\tilde x^{(n)})
    =
    \max_{\substack{I \subseteq [n] \\ |I| = n - m}}
    \sqrt{\frac{\sum_{i\in I} \psi(x_i)^2 + m \psi_{\max}^2}{\left(\sum_{i\in I} \psi'(x_i)\right)^2}}.
\end{equation*}
Consequently,
\begin{equation*}
    \BP_{\eta+}(\hatse(\theta_0),x^{(n)})
    =
    {\frac1n}\min\left\{
      {m\in[n]} : \max_{\substack{I \subseteq [n] \\ |I| = n - m}}
      \sqrt{\frac{\sum_{i\in I} \psi(x_i)^2 + m \psi_{\max}^2}{\left(\sum_{i\in I} \psi'(x_i)\right)^2}}
      \mathrel{{\ge}} \hatse(\theta_0) + \eta
    \right\},
\end{equation*}
which is the first claimed expression.

\emph{Step 3: Additional Symmetry.}
If, in addition, $\psi$ is odd and $\psi'(x)$ is non-increasing in $|x|$, we can also derive equality for the downward direction. In this case, one wants to \emph{decrease} the ratio. We consider the same partition $[n]=I\cup J$. To lower the numerator, it is optimal to drive the contributions of contaminated points to zero. Using that $\psi(0)=0$ for odd $\psi$, moving $x_j$ to $0$ eliminates its contribution to the numerator. At the same time, since $\psi'(0)\ge \psi'(x)$ for all $x$ under the monotonicity assumption on $|x|$, the denominator contribution of each moved point can be increased up to $\psi'(0)$ by moving it to $0$. Hence, for a fixed $I$ and $J=[n]\setminus I$, there exists a contaminated sample $\tilde{x}^{(n)}$
 such that
 \begin{equation*}
  \sum_{i=1}^n \psi(\tilde x_i)^2
  =
  \sum_{i\in I} \psi(x_i)^2 \quad
  \mbox{ and } \quad
  \sum_{i=1}^n \psi'(\tilde x_i)
  =
  \sum_{i\in I} \psi'(x_i) + m \psi'_{\max},
\end{equation*}
where $\psi'_{\max} = \max_{x\in\RR}\psi'(x)$ and in particular $\psi'(0) = \psi'_{\max}$. This yields
\begin{equation*}
    \se_I^{(-)}
    =
    \sqrt{\frac{\sum_{i\in I} \psi(x_i)^2}{\left(\sum_{i\in I} \psi'(x_i) + m \psi'_{\max} \right)^2}},
\end{equation*}
and by minimizing over $I$ we obtain
\begin{equation*}
    \BP_{\eta-}(\hatse(\theta_0),x^{(n)})
    =
    {\frac1n}\min\left\{
      {m\in[n]} : \min_{\substack{I \subseteq [n] \\ |I| = n - m}}
      \sqrt{\frac{\sum_{i\in I} \psi(x_i)^2}{\left(\sum_{i\in I} \psi'(x_i) + m \psi'_{\max} \right)^2}}
      \mathrel{{\le}} \hatse(\theta_0) - \eta
    \right\},
\end{equation*}
However, notice that under the additional assumptions, $|\psi(x)|$ and $\psi'(x)$ can be ordered according to $|x|$. In particular, to enlarge the ratio, one should replace the $m$ points closest to $0$ by $\pm\infty$ (these contribute least to the numerator and most to the denominator), and to decrease the ratio, one should move the $m$ points farthest from $0$ to the origin. This yields the ordered expressions in terms of $(x_{\pi_1},\dots,x_{\pi_n})$ in the statement of the lemma, and the corresponding one-sided sensitivities follow by inverting the inequalities in $m$.
This completes the proof.
\end{proof}
\begin{proof}[Proof of Lemma \ref{lem:se_at_theta_0_main}]
    This is a restatement of Lemma \ref{lem:se_at_theta_0}.
\end{proof}
\begin{proof}[Proof of Lemma \ref{lem:opt_attack_m_est_se}]
We start with $\eta_{m/n+}(\hatse(\hat \theta), x^{(n)})$. The idea is to bound the numerator and denominator of $\hatse(\tilde\theta)$ separately, where $\tilde x \in B_\textsf{H}(x^{(n)},m)$ is an contaminated sample and $\tilde \theta =\hat\theta(\tilde x^{(n)})$ is the corresponding M-estimate of location.
Write
\begin{equation*}
\fitdisplay{\hatse(\tilde{\theta}) = \sqrt{ \frac{ \sum_{i=1}^n \psi(\tilde x_i - \hat{\theta})^2 + \Bigl[\sum_{i=1}^n \psi(\tilde x_i - \tilde{\theta})^2 - \sum_{i=1}^n \psi(\tilde x_i - \hat{\theta})^2\Bigr] }{ \Bigl( {\sum_{i=1}^n\psi'(x_i-\hat\theta) +\sum_{x_i\ne\tilde x_i} [\psi'(\tilde x_i-\tilde\theta)-\psi'(x_i-\hat\theta)] \quad+\sum_{x_i=\tilde x_i} [\psi'(x_i-\tilde\theta)-\psi'(x_i-\hat\theta)] } \Bigr)^2 } }.}
\end{equation*}
\emph{Numerator: upper bound for $\eta_{m/n+}$.}
Let
$$
a_i := \psi(\tilde x_i - \tilde{\theta}),
\qquad
b_i := \psi(\tilde x_i - \hat{\theta}),
\qquad i=1,\dots,n.
$$
Then
\begin{align*}
\sum_{i=1}^n \psi(\tilde x_i - \tilde{\theta})^2
-
\sum_{i=1}^n \psi(\tilde x_i - \hat{\theta})^2
&=
\sum_{i=1}^n (a_i+b_i)(a_i-b_i).
\end{align*}
Since $|\psi(u)| \le \psi_{\max}$ for all $u$,
$$
|a_i+b_i| \le 2\psi_{\max}.
$$
{Since $\psi$ is nondecreasing, the differences $a_i-b_i$ all have the same sign: they are nonpositive when $\tilde\theta\ge\hat\theta$ and nonnegative when $\tilde\theta\le\hat\theta$. Thus the following bound applies to either direction of the location change.}
Therefore,
\begin{align*}
\sum_{i=1}^n \psi(\tilde x_i - \tilde{\theta})^2
-
\sum_{i=1}^n \psi(\tilde x_i - \hat{\theta})^2
&=
\sum_{i=1}^n (a_i+b_i)(a_i-b_i) \\
&\le
\sum_{i=1}^n |a_i+b_i|\,|a_i-b_i| \\
&\le
2\psi_{\max}\sum_{i=1}^n |a_i-b_i|.
\end{align*}
Because all $a_i-b_i$ have the same sign,
\begin{align*}
\sum_{i=1}^n |a_i-b_i|
&=
\left|
\sum_{i=1}^n (a_i-b_i)
\right| \\
&=
\left|
\sum_{i=1}^n \psi(\tilde x_i-\tilde\theta)
-
\sum_{i=1}^n \psi(\tilde x_i-\hat\theta)
\right|.
\end{align*}
Now $\tilde\theta$ is the $M$-estimator based on the contaminated sample $\tilde x^{(n)}$, so
$$
\sum_{i=1}^n \psi(\tilde x_i-\tilde\theta)=0.
$$
Hence
\begin{align*}
\sum_{i=1}^n |a_i-b_i|
&=
\left|
\sum_{i=1}^n \psi(\tilde x_i-\hat\theta)
\right|.
\end{align*}
On the other hand, $\hat\theta$ is the $M$-estimator based on the original sample $x^{(n)}$, so
$$
\sum_{i=1}^n \psi(x_i-\hat\theta)=0.
$$
Let $\mathcal I:=\{i:x_i\neq \tilde x_i\}$ denote the set of contaminated indices. Since $|\mathcal I|\le m$,
\begin{align*}
\left|
\sum_{i=1}^n \psi(\tilde x_i-\hat\theta)
\right|
&=
\left|
\sum_{i=1}^n \bigl[\psi(\tilde x_i-\hat\theta)-\psi(x_i-\hat\theta)\bigr]
\right| \\
&=
\left|
\sum_{i\in\mathcal I} \bigl[\psi(\tilde x_i-\hat\theta)-\psi(x_i-\hat\theta)\bigr]
\right| \\
&\le
\sum_{i\in\mathcal I}
\left|
\psi(\tilde x_i-\hat\theta)-\psi(x_i-\hat\theta)
\right| \\
&\le
2m\psi_{\max}.
\end{align*}
Combining the previous displays yields
\begin{align*}
\sum_{i=1}^n \psi(\tilde x_i - \tilde{\theta})^2
-
\sum_{i=1}^n \psi(\tilde x_i - \hat{\theta})^2
\le
4m\psi_{\max}^2.
\end{align*}
Combining this with Lemma~\ref{lem:se_at_theta_0}, applied at $\theta_0=\hat\theta$, gives
\begin{align*}
\sum_{i=1}^n \psi(\tilde x_i-\tilde\theta)^2
\le
m\psi_{\max}^2
+
\sum_{i>m}\psi(x_{\pi_i}-\hat\theta)^2
+
4m\psi_{\max}^2,
\end{align*}
that is,
\begin{align*}
\sum_{i=1}^n \psi(\tilde x_i-\tilde\theta)^2
\le
5m\psi_{\max}^2
+
\sum_{i>m}\psi(x_{\pi_i}-\hat\theta)^2,
\end{align*}
where $\pi$ is a permutation such that
$$
|x_{\pi_1}-\hat\theta|
\le
|x_{\pi_2}-\hat\theta|
\le
\cdots
\le
|x_{\pi_n}-\hat\theta|.
$$
\emph{Denominator: lower bound for $\eta_{m/n+}$.}
For the denominator, note first that $\sum_{x_i \ne \tilde x_i} \psi'(\tilde x_i - \tilde{\theta}) \ge 0$. We also control $\sum_{x_i = \tilde x_i} \psi'(x_i - \hat{\theta})$ by picking the $n-m$ smallest $\psi'$ using the same idea in Lemma~\ref{lem:se_at_theta_0}. Thus we only need a lower bound on
\begin{equation*}
  \sum_{x_i = \tilde x_i} \psi'(x_i - \tilde{\theta}) -
  \sum_{x_i = \tilde x_i} \psi'(x_i - \hat{\theta}).
\end{equation*}
We only care about indices where the derivative decreases, i.e., those with
\begin{equation*}
  \psi'(x_i - \tilde{\theta}) < \psi'(x_i - \hat{\theta})
  \quad\mathrel{{\Longrightarrow}}\quad
  |x_i - \hat \theta| < |x_i - \tilde \theta|.
\end{equation*}
For such $i$, since $|\tilde\theta - \hat\theta| \le \eta_{m/n}(\hat \theta, x^{(n)})$ and by definition of $\Delta(\cdot)$,
\begin{equation*}
  \psi'(x_i - \tilde{\theta}) - \psi'(x_i - \hat{\theta})
  \ge - \Delta(\eta_{m/n}(\hat \theta, x^{(n)})).
\end{equation*}
By construction there are at most $S_{m+}$ such indices and therefore
\begin{equation*}
  \sum_{x_i = \tilde x_i} \psi'(x_i - \tilde{\theta}) -
  \sum_{x_i = \tilde x_i} \psi'(x_i - \hat{\theta})
  \ge - S_{m+} \Delta(\eta_{m/n}(\hat \theta, x^{(n)})).
\end{equation*}
Indeed, by definition of $S_{m+}$, we are taking the worst-case (largest possible) number of such indices over all $m$-point contaminations and all $\tilde\theta$ in the interval
$$
[\hat \theta - \eta_{m/n-}(\hat \theta, x^{(n)}){,} \hat \theta + \eta_{m/n+}(\hat \theta, x^{(n)})].
$$
{Put $u_i=x_i-\hat\theta$ and $h=\tilde\theta-\hat\theta$.
For $h>0$,
\[
|u_i|<|u_i-h|\iff u_i<h/2,
\qquad
|u_i|>|u_i-h|\iff u_i>h/2.
\]
For $h<0$, these conditions become $u_i>h/2$ and $u_i<h/2$,
respectively. The largest outward count in each direction is obtained
at the corresponding nonzero endpoint
$h=\eta_{m/n+}(\hat\theta,x^{(n)})$ or
$h=-\eta_{m/n-}(\hat\theta,x^{(n)})$.
For the inward count, a sufficiently small positive $h$ includes every
$u_i>0$, and a sufficiently small negative $h$ includes every $u_i<0$.
We can retain at most $n-m$ of these observations. A direction with
$\eta_{m/n\pm}(\hat\theta,x^{(n)})=0$ contributes zero.
These observations give both formulas for $S_{m\pm}$, including ties and
zero sensitivity.
}
Combining the numerator and denominator bounds then gives the stated upper bound for $\eta_{m/n+}(\hatse(\hat\theta),x^{(n)})$.
We omit the argument for the lower-side $m$-sensitivity as it is analogous, but with the inequalities reversed.
For Huber’s loss, we defined  ${\psi'(x)=\II\{|x|\leq\delta\}}$ and can sharpen the denominator as follows
\begin{align*}
    \sum_{x_i = \tilde x_i} \psi'(x_i - \tilde{\theta}) -
  \sum_{x_i = \tilde x_i} \psi'(x_i - \hat{\theta}) = \biggl(\sum_{i = 1}^n - \sum_{x_i \ne \tilde x_i} \biggr) (\psi'(x_i - \tilde{\theta}) - \psi'(x_i - \hat{\theta})),
\end{align*}
where the first summation is bounded by {$q_{m\pm}$} by definition, and the second summation is bounded by {$\pm m$}. Hence, we have
\begin{equation*}
  {(q_{m-}+m)}\ge \sum_{x_i = \tilde x_i} \psi'(x_i - \tilde{\theta}) -
  \sum_{x_i = \tilde x_i} \psi'(x_i - \hat{\theta}) \ge {(q_{m+}-m)}.
\end{equation*}
\end{proof}
\begin{proof}[Proof of Corrolary \ref{lem:opt_attack_m_est_se_envelope}]
We only need to give an upper and lower bound for
$$
\sum_{i=1}^n \psi(\tilde x_i - \tilde \theta)^2.
$$
Write
$$
t:=\tilde\theta-\hat\theta
\in
\Bigl[
-\eta_{m/n-}(\hat\theta,x^{(n)}),\,
\eta_{m/n+}(\hat\theta,x^{(n)})
\Bigr],
$$
and define
$$
h_i(t):=\psi(x_i-\hat\theta-t)^2,\qquad i=1,\dots,n.
$$
Let
$
h_{(1)}(t)\le \cdots \le h_{(n)}(t)
$
be the order statistics of $\{h_i(t)\}_{i=1}^n$. By the same order-statistics argument as in the proof of Theorem~\ref{thm:loc}, if at most $m$ observations are contaminated, then for fixed $t$ the unchanged $n-m$ observations contribute some subset of size $n-m$ from $\{h_i(t)\}_{i=1}^n$, while the $m$ contaminated observations can contribute anywhere between $0$ and $\psi_{\max}^2$, since $\psi(0)=0$ and $|\psi|\le \psi_{\max}$. Therefore, for each admissible $t$,
$$
\sum_{i=1}^{n-m} h_{(i)}(t)
\le
\sum_{i=1}^n \psi(\tilde x_i-\tilde\theta)^2
\le
m\psi_{\max}^2+\sum_{i=m+1}^n h_{(i)}(t).
$$
Taking the infimum and supremum over all admissible shifts $t$ yields
$$
\underline N_m
:=
\inf_t \sum_{i=1}^{n-m} h_{(i)}(t)
\le
\sum_{i=1}^n \psi(\tilde x_i-\tilde\theta)^2
\le
\sup_t\left\{
m\psi_{\max}^2+\sum_{i=m+1}^n h_{(i)}(t)
\right\}
=: \overline N_m.
$$
Substituting these bounds into the same denominator bounds as in Lemma~\ref{lem:opt_attack_m_est_se} gives the claimed inequalities.
\end{proof}

\subsection{Population Breakdown and Proof of Theorem~\ref{thm:conv}}
\label{sec:def_pbp}
\label{app:replacement_proof}

The neighborhood \eqref{eq:replacement_neighborhood} allows both the removal
of selected mass from \(F\) and the addition of arbitrary replacement mass.
The population decision targets
\eqref{eq:population_rejection}--\eqref{eq:population_acceptance} use the
intervals in \eqref{eq:population_interval}. For rejection we take
\(F=F_\theta\), with \(\theta\) an alternative. For acceptance we take
\(F=F_{\theta_0}\). These are two applications of Theorem~\ref{thm:conv}
under different sampling laws. The sample quantities follow the
decision-based viewpoint of \citet{zhang1996}. The population targets use
the same decision criterion. They are related to, but need not equal,
the statistic-based power and level functions of
\citet{he:simpson:portnoy1990}. In particular,
\eqref{eq:population_acceptance} asks for any rejected population decision;
it does not require reaching a specified alternative statistic value.

For a deterministic sample \(x^{(n)}\), we write 
\(\widehat F_x=n^{-1}\sum_{i=1}^n\delta_{x_i}\) and define
\begin{equation*}
 \mathcal E_{n,\varepsilon}(x^{(n)})
 =\{\widehat F_y:y^{(n)}\in
 B_{\mathsf H}(x^{(n)},\lfloor n\varepsilon\rfloor)\}.
\end{equation*}
Also remind that
$$
\mathcal R_\varepsilon(F)=\{H\in\mathcal P:d_{\mathrm{TV}}(F,H)\le\varepsilon\}.
$$
\begin{samepage}
The following result gives an approximation of replacements: for all empirical $\hat P$ in a replacement neighborhood of $\hat F_x$, there is a population distribution $H$ in a TV neighborhood of $F$ such that $\hat P $ and $H$ are (almost) as close as $\hat F_x$ and $F$ in Kolmogorov distance, and vice versa.

\begin{lemma}
\label{lem:replacement_approximation}
Let \(F\) be any probability measure on \(\RR\). For every
\(0\le\varepsilon\le1\):
\begin{enumerate}
 \item Every \(P\in\mathcal E_{n,\varepsilon}(x^{(n)})\) has an
 \(H\in\mathcal R_\varepsilon(F)\) with
 \(\distK(P,H)\le\distK(\widehat F_x,F)\).
 \item Every \(H\in\mathcal R_\varepsilon(F)\) has a
 \(P\in\mathcal E_{n,\varepsilon}(x^{(n)})\) with
 \(\distK(P,H)\le\distK(\widehat F_x,F)+2/n\).
\end{enumerate}
\end{lemma}
\end{samepage}

\begin{proof}

We will prove a more general result first. Given two clean distributions $P$ and $Q$ and one contaminated distribution $H \in \mathcal R_\varepsilon(P)$, we construct a suitable contaminated distribution $\tilde H$ for $Q$ that is close to $H$ in Kolmogorov distance, i.e., \(\distK(H,\widetilde H)\le\distK(P,Q)\). The main idea is to use the quantiles of $P$ and $Q$ to transfer the contamination part in $H$ to $\tilde H$. 

For any probability measure \(P\),
\begin{equation*}
 H\in\mathcal R_\varepsilon(P)
 \quad\Longleftrightarrow\quad
 H=P-R+S,\quad 0\le R\le P,\quad R(\RR)=S(\RR)\le\varepsilon,
\end{equation*}
where \(R,S\) are nonnegative finite measures. One may take \(R,S\) to be
the negative and positive parts of \(H-P\); their common mass is
\(d_{\mathrm{TV}}(P,H)\). Conversely, any such decomposition has
\(d_{\mathrm{TV}}(P,H)\le R(\RR)\).

Now we match the two laws $P$ and $Q$ by
quantiles to transfer the replacement $R$ in \(H\) to \(\tilde H\). We write $P(t)=P( -\infty,t]$ and let \(P^{-1}(u)=\inf\{t:P(t)\ge u\}\) be the quantile
function. For \(0<u<1\), the fraction of mass removed at percentile \(u\)
of \(P\) is
\[
 w(u)=\frac{dR}{dP}(P^{-1}(u)),\qquad 0\le w(u)\le1.
\]
We remove the same fraction \(w(u)\) at percentile \(u\) of \(Q\) and
add back \(S\) at the same locations as in the original replacement.
The resulting law are
\begin{equation}\label{eq:transferred_removal}
 R_Q(A)=\int_0^1w(u)\II\{Q^{-1}(u)\in A\}\,du,
 \qquad \widetilde H=Q-R_Q+S.
\end{equation}
Since \(0\le w(u)\le1\), we have \(R_Q\le Q\).  By construction,
\(\widetilde H\in\mathcal R_\varepsilon(Q)\).
Note that 
\begin{align}
 H(t)-\widetilde H(t)
 &=P(t)-Q(t)- (R(t)-R_Q(t) )=\int_{Q(t)}^{P(t)}(1-w(u))\,du,
       \notag %
\end{align}
and since \(0\le 1-w(u)\le 1\),
\begin{equation}
  |H(t)-\widetilde H(t)|\leq  \left|\int_{Q(t)}^{P(t)}(1-w(u)) \,du \right| \le |P(t)-Q(t)|
\label{eq:transfer_cdf_bound} 
\end{equation}
Consequently \(\distK(H,\widetilde H)\le\distK(P,Q)\). This proves the general result that we will repeatedly use below.

For point 1, we apply \eqref{eq:transfer_cdf_bound} with \(P=\widehat F_x\) and \(Q=F\). A sample with at most \(\lfloor n\varepsilon\rfloor\) replacements has
empirical law in \(\mathcal R_\varepsilon(\widehat F_x)\). Then by our construction, we establish a contaminated distribution $H\in\mathcal R_\varepsilon(F)$.

For point 2, we apply \eqref{eq:transfer_cdf_bound} again but with \(Q=\widehat F_x\) and \(P=F\) from \(H\in\mathcal R_\varepsilon(F)\). This will give us a contaminated distribution $\overline H_n \in \mathcal R_\varepsilon(\widehat F_x )$ that has the following properties:
\[
 \overline H_n=\widehat F_x-R_n+S,\quad
 R_n\le\widehat F_x,  ,\quad R_n(\RR)=S(\RR)=s\le\varepsilon,
 \quad\distK(H,\overline H_n)\le\distK(F,\widehat F_x).
\]
We want to conclude that $\overline H_n \in \mathcal E_{n,\varepsilon}(x^{(n)})$, which is not generally true. But we can apply a simple rounding scheme to find a dataset $y^{(n)}$ such that its empirical CDF $\hat F_y$ satisfies
$$
\distK(\widehat F_y,H)
 \le 2/n+\distK(\widehat F_x,F).
$$
The idea is to round the weights at each point to either 0 or 1, such that the dataset is close to the original one. Split the removed mass among the labeled, sorted observations:
\begin{equation*}
 R_n=\frac1n\sum_{i=1}^n a_i\delta_{x_{(i)}},\qquad
 0\le a_i\le1,\qquad\sum_{i=1}^na_i=ns.
\end{equation*}
Set
\begin{equation*}
 b_i=\left\lfloor\sum_{j\le i}a_j\right\rfloor
      -\left\lfloor\sum_{j<i}a_j\right\rfloor,\qquad
 R_n^\circ=\frac1n\sum_{i=1}^n b_i\delta_{x_{(i)}}.
\end{equation*}
Each \(b_i\) is zero or one, and the number of removed observations is
\(k=\sum_i b_i=\lfloor ns\rfloor\le\lfloor n\varepsilon\rfloor\). Let
\(
m(t) = \sum_{i=1}^n \mathbb I\{x_{(i)}\le t\}, A_i=\sum_{j=1}^i a_j
\) with $A_0 = 0$.
By construction,
\[
nR_n(t)=A_{m(t)},\qquad
nR_n^\circ(t)
=\sum_{i=1}^{m(t)}
  \bigl(\lfloor A_i\rfloor-\lfloor A_{i-1}\rfloor\bigr)
=\lfloor A_{m(t)}\rfloor.
\]
So for all $t$,
\begin{equation}\label{eq:deletion_rounding_error}
R_n(t)-R_n^\circ(t)=(A_{m(t)}-\lfloor A_{m(t)}\rfloor)/n<1/n.
\end{equation}
This give us the removed data points.

Now, we construct the added data points from population quantiles. Write \(G=S/s\), choose the $k$ replacement points in the center of each step of size $1/ns$:
\begin{equation*}
 z_j=G^{-1}\left(\frac{j-1/2}{ns}\right),\quad j=1,\ldots,k,
 \qquad S_n^\circ=\frac1n\sum_{j=1}^k\delta_{z_j}.
\end{equation*}
The generalized inverse gives
\[
z_j\le t
\quad\Longleftrightarrow\quad
\frac{j-\tfrac12}{ns}\le G(t)
\quad\Longleftrightarrow\quad
j\le nS(t)+\tfrac12.
\]
Therefore, the number of added points at or below \(t\) is
\[
nS_n^\circ(t)
=\sum_{j=1}^{k}\mathbb I\{j\le nS(t)+\tfrac12\}
=\min\{k,\lfloor nS(t)+\tfrac12\rfloor\}.
\]
When \(nS(t) \le k\), this rounds \(nS(t)\) to the nearest integer, with error at most \(1/2\). When \(nS(t)>k\), the count is capped at \(k\), so the error is \(nS(t)-k\le ns-k<1\).
Hence, for all $t$,
\begin{equation}\label{eq:insertion_rounding_error}
 |S_n^\circ(t)-S(t)|\le1/n.
\end{equation}

Now we get the measure
\(\widehat F_y=\widehat F_x-R_n^\circ+S_n^\circ\) combining both parts. This is the empirical law of
a sample obtained by at most \(k\) replacements. By
\eqref{eq:deletion_rounding_error}--\eqref{eq:insertion_rounding_error},
\[
 \distK(\widehat F_y,H)
 \le \distK(\widehat F_y,\overline H_n)+\distK(\overline H_n,H)
 \le 2/n+\distK(\widehat F_x,F).
\]
This proves the second part of Lemma~\ref{lem:replacement_approximation}.
\end{proof}
\par

Now we are ready to prove Theorem \ref{thm:conv}.

\begin{proof}[Proof of Theorem~\ref{thm:conv}]
The test decision is determined by whether the interval contains \(\theta_0\). For each replacement budget, we therefore compare the smallest and largest values of the sample interval endpoints with their population counterparts. Lemma~\ref{lem:replacement_approximation} lets us compare  sample and population endpoints. Conditions (C1) and (C2) further control how distant the endpoints can be in a replacement/TV  neighborhood.

First note that by the Glivenko--Cantelli theorem, \(\distK(\widehat F_X,F)\to0\)
almost surely. Work on a sample path where this convergence and
\eqref{eq:sample_approximation} both hold. Abbreviate
\(\mathcal B_n(\varepsilon)
=B_{\textsf{H}}(X^{(n)},\lfloor n\varepsilon\rfloor)\).
Every \(y^{(n)}\in\mathcal B_n(\bar\varepsilon)\) satisfies
\begin{equation*}
 \distK(\widehat F_y,F)
 \le\distK(\widehat F_y,\widehat F_X)+\distK(\widehat F_X,F)
 \le\bar\varepsilon+\distK(\widehat F_X,F)<\rho
\end{equation*}
for all sufficiently large \(n\). Furthermore,
\(\distK(H,F)\le d_{\mathrm{TV}}(H,F)\le\bar\varepsilon<\rho\) for every
\(H\in\mathcal R_{\bar\varepsilon}(F)\). Thus the continuity bound
\eqref{eq:population_regularity} of Condition (C1) applies to all comparisons below.

Let \((V,V_n)\) denote either \((L,L_n)\) or \((U,U_n)\). Define
the uniform endpoint approximation error in (C2) by
\[
 \Delta_n=\sup_{y^{(n)}\in\mathcal B_n(\bar\varepsilon)}
 \max\bigl\{|L_n(y^{(n)})-L(\widehat F_y)|,\,
             |U_n(y^{(n)})-U(\widehat F_y)|\bigr\},
\]
and set
\[
 \gamma_n=\Delta_n+\omega\bigl(\distK(\widehat F_X,F)+2/n\bigr).
\]
By (C2), \(\Delta_n\to0\) almost surely. By the Glivenko--Cantelli theorem,
\(\distK(\widehat F_X,F)+2/n\to0\) almost surely. Since \(\omega(t)\to0\)
as \(t\downarrow0\) in (C1), it follows that \(\gamma_n\to0\) almost surely.

Fix \(0\le\varepsilon\le\bar\varepsilon\).
Lemma~\ref{lem:replacement_approximation} matches each sample replacement
\(y^{(n)}\in\mathcal B_n(\varepsilon)\) with some population replacement
\(H\in\mathcal R_\varepsilon(F)\), and conversely, with
\(\distK(\widehat F_y,H)\le\distK(\widehat F_X,F)+2/n\).
For every such pair, (C1) and (C2) give
\begin{equation}\label{eq:paired_endpoint_bound}
\begin{aligned}
 |V_n(y^{(n)})-V(H)|
 &\le |V_n(y^{(n)})-V(\widehat F_y)|
       +|V(\widehat F_y)-V(H)|\\
 &\le \Delta_n+\omega\bigl(\distK(\widehat F_X,F)+2/n\bigr)=\gamma_n.
\end{aligned}
\end{equation}
From this inequality, taking a suprema we see that 
\[
\overline V_{n,\varepsilon}
 :=\sup_{y^{(n)}\in\mathcal B_n(\varepsilon)} V_n(y^{(n)})\le V(H)+\gamma_n
 \le \overline V_\varepsilon(F)+\gamma_n,
\]
\[
 \overline V_\varepsilon(F)=\sup_{H\in\mathcal R_\varepsilon(F)}V(H) \le V_n(y^{(n)})+\gamma_n
 \le \overline V_{n,\varepsilon}+\gamma_n,
\]
and therefore
\[\bigl|\overline V_{n,\varepsilon}-\overline V_\varepsilon(F)\bigr|\le\gamma_n.\]
Applying the same argument to \(-V_n\) and \(-V\) gives the corresponding
bound for
\(\underline V_{n,\varepsilon}
:=\inf_{y^{(n)}\in\mathcal B_n(\varepsilon)}V_n(y^{(n)})\).
We conclude that almost surely
\begin{align}
 \sup_{0\le\varepsilon\le\bar\varepsilon}
 \max_{V\in\{L,U\}}
 |\underline V_{n,\varepsilon}-\underline V_\varepsilon(F)|
 &\le\gamma_n\longrightarrow0,
 \label{eq:lower_endpoint_convergence}\\
 \sup_{0\le\varepsilon\le\bar\varepsilon}
 \max_{V\in\{L,U\}}
 |\overline V_{n,\varepsilon}-\overline V_\varepsilon(F)|
 &\le\gamma_n\longrightarrow0.
 \label{eq:upper_endpoint_convergence}
\end{align}
Now we are ready to show the main theorem by applying these convergence results.

\emph{Rejection breakdown.} We give the argument when \(L(F)>\theta_0\). The case \(U(F)<\theta_0\) can be handled with similar arguments by reversing inequalities in obvious ways.
It follows from \eqref{eq:lower_endpoint_convergence}-\eqref{eq:upper_endpoint_convergence} that \(L_n(X^{(n)})>\theta_0\) with probability tending to one, so the clean sample
rejects.
Fix \(0<\varepsilon<r_{\mathrm{reject}}\), i.e., below the population threshold. Equations
\eqref{eq:rejection_crossing_right} and
\eqref{eq:lower_endpoint_convergence} imply
\(\underline L_{n,\varepsilon}>\theta_0\) eventually. Hence
\(L_n(y^{(n)})>\theta_0\) for every
\(y^{(n)}\in\mathcal B_n(\varepsilon)\) and $n$ large enough,
\begin{equation}\label{eq:rejection_lower_bound}
 \BP_{\mathrm{reject}}(\phi_n,X^{(n)})
 >\frac{\lfloor n\varepsilon\rfloor}{n}.
\end{equation}

Fix \(r_{\mathrm{reject}}<\varepsilon<\bar\varepsilon\), i.e., above the population threshold. Eventually
\(\underline L_{n,\varepsilon}<\theta_0\), so some
\(y^{(n)}\in\mathcal B_n(\varepsilon)\) has
\(L_n(y^{(n)})<\theta_0\). If \(U_n(y^{(n)})\ge\theta_0\), this sample
already accepts. Otherwise its entire interval is below the null.  A continuity argument can however be used to show that even in this case the rejection breakdown must have been attained.

Indeed, if  \(U_n(y^{(n)})<\theta_0\), fix this \(n\) and write
\(x^{(n)}=(x_1,\ldots,x_n)\) for the realized values of
\(X^{(n)}\). Let
\[
 I=\{i\in\{1,\ldots,n\}:y_i\ne x_i\},
 \qquad |I|\le\lfloor n\varepsilon\rfloor.
\]
Move each of these observations continuously from \(x_i\) to \(y_i\),
keeping all other observations fixed. Specifically, for \(0\le t\le1\),
define the sample \(z^{(n)}(t)=(z_1(t),\ldots,z_n(t))\) by
\begin{equation*}
 z_i(t)=
 \begin{cases}
 x_i+t(y_i-x_i),&i\in I,\\
 x_i,&i\notin I.
 \end{cases}
\end{equation*}
For every \(t\),
\(
 \sum_{i=1}^n\II\{z_i(t)\ne x_i\}
 \le |I|\le\lfloor n\varepsilon\rfloor,
\)
so \(z^{(n)}(t)\in \mathcal B_n(\varepsilon)\). Since
\(L_n(X^n) \le \widehat\theta_n\),
\[
 \widehat\theta_n(x^{(n)})>\theta_0 \quad\mbox{ and }
 \quad \widehat\theta_n(y^{(n)})<\theta_0.
\]
The function \(t\mapsto\widehat\theta_n(z^{(n)}(t))\) is continuous,
so some \(t_*\in(0,1)\) satisfies
\(\widehat\theta_n(z^{(n)}(t_*))=\theta_0\). By
\eqref{eq:continuous_center}, this sample accepts. Thus in either case,
\begin{equation}\label{eq:rejection_upper_bound}
 \BP_{\mathrm{reject}}(\phi_n,X^{(n)})
 \le\frac{\lfloor n\varepsilon\rfloor}{n}.
\end{equation}
Letting \(\varepsilon\) approach \(r_{\mathrm{reject}}\) in \eqref{eq:rejection_lower_bound}--\eqref{eq:rejection_upper_bound}
proves that $
\BP_{\mathrm{reject}}(\phi_n,X^{(n)})
 \xrightarrow{\mathrm{a.s.}}
 r_{\mathrm{reject}}$.

For \(\varepsilon<r_{\mathrm{reject}}\), all population lower endpoints
exceed \(\theta_0\), so acceptance is impossible. For
\(r_{\mathrm{reject}}<\varepsilon<\bar\varepsilon\), choose
\(H\in\mathcal R_\varepsilon(F)\) with \(L(H)<\theta_0\), and define
\begin{equation*}
 H_t=(1-t)F+tH,\qquad 0\le t\le1.
\end{equation*}
Since \(d_{\mathrm{TV}}(F,H_t)=td_{\mathrm{TV}}(F,H)\le\varepsilon\), the path stays in the
same TV neighborhood. Moreover,
\(\distK(H_s,H_t)=|s-t|\distK(F,H)\), so
\eqref{eq:population_regularity} makes \(L(H_t)\) continuous in \(t\).
It crosses \(\theta_0\), and at a crossing \(L(H_{t_*})=\theta_0\), the
closed interval \(C(H_{t_*})\) accepts. This proves
\(\varepsilon_{\mathrm{reject}}(F)=r_{\mathrm{reject}}\).

\emph{Acceptance breakdown.}
Since \(L(F)<\theta_0<U(F)\), it follows from \eqref{eq:lower_endpoint_convergence}-\eqref{eq:upper_endpoint_convergence} that
\(L_n(X^{(n)})<\theta_0<U_n(X^{(n)})\) eventually, so the clean sample
accepts. Furthermore, \eqref{eq:acceptance_crossing} and
\eqref{eq:lower_endpoint_convergence}--\eqref{eq:upper_endpoint_convergence} imply that for a fix \(0<\varepsilon<r_{\mathrm{accept}}\) we
eventually have that
\[
 \overline L_{n,\varepsilon}<\theta_0<\underline U_{n,\varepsilon}.
\]
Hence every permitted sample interval therefore contains the null and
\begin{equation*}
 \BP_{\mathrm{accept}}(\phi_n,X^{(n)})
 >\frac{\lfloor n\varepsilon\rfloor}{n}.
\end{equation*}
For \(r_{\mathrm{accept}}<\varepsilon<\bar\varepsilon\), eventually
\(
 \overline L_{n,\varepsilon}>\theta_0
 \quad\text{or}\quad \underline U_{n,\varepsilon}<\theta_0.
\)
In the first case some permitted sample has \(L_n(y^{(n)})>\theta_0\);
in the second case some permitted sample has \(U_n(y^{(n)})<\theta_0\).
Either sample rejects. Hence
\begin{equation*}
 \BP_{\mathrm{accept}}(\phi_n,X^{(n)})
 \le\frac{\lfloor n\varepsilon\rfloor}{n}.
\end{equation*}
These bounds prove convergence to \(r_{\mathrm{accept}}\). The proof is completed by noting that  the
inequalities in \eqref{eq:acceptance_crossing} imply that
\(\varepsilon_{\mathrm{accept}}(F)=r_{\mathrm{accept}}\).
\end{proof}
\par
\subsection{Proof of Population Sensitivity and the Maxbias Curve}
We will first give a useful lemma illustrating the nice properties of the solution map $\varepsilon \mapsto \eta_{\varepsilon\pm}$ for \eqref{eq:population_eta+} and \eqref{eq:population_eta-}.
\begin{lemma}
    \label{lem:sens-bias_property}
    Suppose (A2)-(A4) is true for some $\varepsilon^* \in (0, 0.5)$.
    Further assume that
    $$
    F\!\left(\left\{x\in\mathbb R:\psi' \text{ is discontinuous at } x-(\theta_0 \pm \eta_{\varepsilon^*\pm})\right\}\right)=0.
    $$
    Then,
    \begin{gather}
    \label{eq:derivative}
\fitdisplay[\dimexpr\linewidth-3em\relax]{\frac{\dd\eta_{\varepsilon+}}{\dd\varepsilon} \Big |_{\varepsilon = \varepsilon^*} = \frac{ -\psi(q_{\varepsilon^*} - (\theta_0 + \eta_{\varepsilon^*+})) + \|\psi\|_\infty}{\int_{{q_{\varepsilon^*}}}^{\infty} \psi'(x - (\theta_0 + \eta_{\varepsilon^*+})) \dd F(x)} > 0,}\\ \nonumber
\fitdisplay[\dimexpr\linewidth-3em\relax]{\frac{\dd\eta_{\varepsilon-}}{\dd\varepsilon} \Big |_{\varepsilon = \varepsilon^*} =\frac{ -\psi(q_{1-\varepsilon^*} - (\theta_0 - \eta_{\varepsilon^*-})) - \|\psi\|_\infty}{-\int_{-\infty}^{q_{1-\varepsilon^*}} \psi'\big(x-(\theta_0-\eta_{\varepsilon^*-})\big) \dd F(x)} > 0.}
\end{gather}
{If these assumptions hold at every $\varepsilon\in(0,1/2)$, the solution map is continuously differentiable and strictly increasing on $(0,1/2)$. It is therefore invertible on its image.}
\end{lemma}
\begin{proof}
    We only show for the ''+`` direction, as the other direction is similar.
    Recall the implicit function $\eta_{\varepsilon+}$ is defined by:
    \begin{align*}   0= H_{+}(\eta, \varepsilon) &= (1-\varepsilon) \cdot \mathbb{E}\left[ \psi\left(X - (\theta_0 + \eta)\right) \mid X > q_\varepsilon \right] + \varepsilon \|\psi\|_\infty
    \\
    &= \int_{q_\varepsilon}^{\infty} \psi(x - (\theta_0 + \eta)) \, \dd F(x) + \varepsilon \|\psi\|_\infty.
    \end{align*}
    The partial derivative exists near $\eta_{\varepsilon^*+}$ by (A4) and the fact that $\psi$ is $L$-Lipschitz. Now fix any $\eta$, for small $h>0$,
    \begin{align}
     \nonumber  &~ H_{+}(\eta, \varepsilon^* + h) - H_{+}(\eta, \varepsilon^*) \\
    \nonumber    = &~-\int^{q_{\varepsilon^*+h}}_{q_{\varepsilon^*}} \psi(x - (\theta_0+\eta)) \dd F(x) + h\|\psi\|_\infty\\
        = &~ - \psi(q_{\varepsilon^*} - (\theta_0 + \eta)) (F(q_{\varepsilon^*+h}) - F(q_{\varepsilon^*})) + o(h) + h\|\psi\|_\infty,
        \label{eq:diff_H+_wrt_epsilon}
    \end{align}
    where the last equality is due to the continuity of $\psi$. When $F$ is continuous around $q_{\varepsilon^*}$, $F(q_{\varepsilon^*+h}) - F(q_{\varepsilon^*}) = h$. Now, we can safely conclude the second partial derivative exists near $\varepsilon^*$ by (A3).
    Therefore, by dominated convergence, we can derive
    \begin{align*}
        \frac{\partial H_{+}}{\partial \eta} = &~ -\int_{q_\varepsilon}^{\infty} \psi'(x - (\theta_0 + \eta)) \dd F(x), \\
        \frac{\partial H_{+}}{\partial \varepsilon} = &~ -\psi(q_\varepsilon - (\theta_0 + \eta)) + \|\psi\|_\infty.
    \end{align*}
    {
        We will show that both partial derivatives are continuous at $(\eta_{\varepsilon^*\pm},\varepsilon^*)$. We will only consider the $+$ side as the other side is analogous.
        Take any sequence $(\eta_n,\varepsilon_n)\to(\eta^*,\varepsilon^*)$.
        By (A3), $F$ has a positive density in a neighborhood of $q_{\varepsilon^*}$, so in particular $F$ is continuous at $q_{\varepsilon^*}$. Hence the quantile map is continuous at $\varepsilon^*$, and therefore
        $
        q_{\varepsilon_n}\to q_{\varepsilon^*}.
        $
        We first prove continuity of ${\partial H_{+}}/{\partial \eta}$. Define
        $$
        g_n(x):=\psi'(x-(\theta_0+\eta_n))\II\{x>q_{\varepsilon_n}\},
        \qquad
        g(x):=\psi'(x-(\theta_0+\eta^*))\II\{x>q_{\varepsilon^*}\}.
        $$
        Then
        $$
        \frac{\partial H_{+}}{\partial \eta}(\eta_n,\varepsilon_n)=-\int g_n(x)\dd F(x),
        \qquad
        \frac{\partial H_{+}}{\partial \eta}(\eta^*,\varepsilon^*)=-\int g(x)\dd F(x).
        $$
        Fix $x$ such that $\psi'$ is continuous at $x-(\theta_0+\eta^*)$ and $x\neq q_{\varepsilon^*}$. Since $\eta_n\to\eta^*$,
        $$
        x-(\theta_0+\eta_n)\to x-(\theta_0+\eta^*),
        $$
        and therefore
        $$
        \psi'(x-(\theta_0+\eta_n))
        \to
        \psi'(x-(\theta_0+\eta^*)).
        $$
        Also, if $x\neq q_{\varepsilon^*}$, since $q_{\varepsilon_n}\to q_{\varepsilon^*}$,  we have that
        $
        \II\{x>q_{\varepsilon_n}\}\to \II\{x>q_{\varepsilon^*}\}.
        $
        Hence
        $
        g_n(x)\to g(x).
        $
        By the extra assumption,
        $
        F\left(\left\{x:\psi' \text{ is discontinuous at } x-(\theta_0+\eta_{\varepsilon^*+})\right\}\right)=0,
        $
        and by (A3),
        $
        F(\{q_{\varepsilon^*}\})=0.
        $
        Therefore
        $$
        g_n(x)\to g(x)
        \qquad
        \text{for }F\text{-a.e. }x.
        $$
        Next, by (A2), $\psi$ is $L$-Lipschitz. Hence, $\psi'$ satisfies
        $
        |\psi'|\le L~
        \text{a.e.}
        $
        By dominated convergence,
        $$
        \int g_n(x)\dd F(x) \to \int g(x)\dd F(x).
        $$
        Equivalently,
        $$
         \frac{\partial H_{+}}{\partial \eta}(\eta_n,\varepsilon_n)\to  \frac{\partial H_{+}}{\partial \eta}(\eta_{\varepsilon^*+},\varepsilon^*).
        $$
        So
        $$
        \frac{\partial H_+}{\partial \eta}
        =
        -\int_{q_\varepsilon}^{\infty}\psi'(x-(\theta_0+\eta))\,\dd F(x)
        $$
        is continuous at $(\eta^*,\varepsilon^*)$.
        Now consider $ {\partial H_{+}}/{\partial \varepsilon}$. Since $q_{\varepsilon_n}\to q_{\varepsilon^*}$ and $\eta_n\to\eta_{\varepsilon^*+}$,
        $$
        q_{\varepsilon_n}-(\theta_0+\eta_n)\to q_{\varepsilon^*}-(\theta_0+\eta_{\varepsilon^*+}).
        $$
        By (A2), $\psi$ is $L$-Lipschitz, hence continuous. Therefore
        $$
        \psi\bigl(q_{\varepsilon_n}-(\theta_0+\eta_n)\bigr)
        \to
        \psi\bigl(q_{\varepsilon^*}-(\theta_0+\eta^*)\bigr).
        $$
        It follows that
        $$
        \frac{\partial H_{+}}{\partial \varepsilon}(\eta_n,\varepsilon_n)
        =
        -\psi\bigl(q_{\varepsilon_n}-(\theta_0+\eta_n)\bigr)+\|\psi\|_\infty
        \to
        -\psi\bigl(q_{\varepsilon^*}-(\theta_0+\eta^*)\bigr)+\|\psi\|_\infty
        =
        \frac{\partial H_{+}}{\partial \varepsilon}(\eta^*,\varepsilon^*).
        $$
        Therefore,
        $$
        \frac{\partial H_+}{\partial \varepsilon}
        =
        -\psi\bigl(q_\varepsilon-(\theta_0+\eta)\bigr)+\|\psi\|_\infty
        $$
        is continuous at $(\eta^*,\varepsilon^*)$.
    }
    Hence, $H_+(\eta, \varepsilon)$ is continuously differentiable in $(\eta, \varepsilon)$ near $(\eta_{\varepsilon^*+}, \varepsilon^*)$. Furthermore, we have $\frac{\partial H_{+}}{\partial \eta} \big |_{\eta = \eta_{\varepsilon^*+}, \varepsilon =\varepsilon^*} \ne 0$ by (A4). Thus, we can apply implicit function theorem and derive
    $$
    \frac{\dd\eta}{\dd\varepsilon} = -\frac{\partial H_{+}/\partial \varepsilon}{\partial H_{+}/\partial \eta} = \frac{ -\psi(q_\varepsilon - (\theta_0 + \eta)) + \|\psi\|_\infty}{\int_{q_\varepsilon}^{\infty} \psi'(x - (\theta_0 + \eta)) \dd F(x)}.
    $$
    We need to show that $-\psi(q_{\varepsilon^*} - (\theta_0 + \eta_{\varepsilon^*+})) + \|\psi\|_\infty > 0$. To do so we will argue by contradiction. Notice that because $H_{+}(\eta_{\varepsilon^*+}, \varepsilon^*) = 0$, we have
    $$
    \int_{q_{\varepsilon^*}}^{\infty} \psi(x - (\theta_0 + \eta_{\varepsilon^*+})) \, \dd F(x) = - \varepsilon^* \|\psi\|_\infty < 0.
    $$
    Suppose $\psi(q_{\varepsilon^*} - (\theta_0 + \eta_{\varepsilon^*+})) = \|\psi\|_\infty$, then by monotonicity of $\psi$ we must have
    $$
\int_{q_{\varepsilon^*}}^{\infty} \psi(x - (\theta_0 + \eta_{\varepsilon^*+})) \, \dd F(x) = \|\psi\|_\infty (1 - \varepsilon^*) > 0,
    $$
    which is a contradiction. This shows that the map is $C^1((0, 0.5), \mathbb{R}_{\geq 0})$ and strictly increasing.
{Applying the local result at every $\varepsilon\in(0,1/2)$ gives the stated differentiability and strict monotonicity, and hence invertibility on the image.}
\end{proof}
{
\begin{proof}[Proof of Proposition \ref{prop:sensitivity_map}]
    This is an immediate result of Lemma \ref{lem:sens-bias_property}. However, since we now have that (A3) is true for $\varepsilon \in (0, 0.5)$, we can derive continuity of ${\partial H_{+}}/{\partial \eta}$ without assuming $
    F\!\left(\left\{x\in\mathbb R:\psi' \text{ is discontinuous at } x-(\theta_0 \pm \eta_{\varepsilon^*\pm})\right\}\right)=0.
    $
    Let $(\eta_n,\varepsilon_n)\to(\eta_{\varepsilon^*+},\varepsilon^*)$. We show
    $$
    \int_{q_{\varepsilon_n}}^{\infty}\psi'\bigl(x-(\theta_0+\eta_n)\bigr)\,dF(x)
    \to
    \int_{q_{\varepsilon^*}}^{\infty}\psi'\bigl(x-(\theta_0+\eta_{\varepsilon^*+})\bigr)\,dF(x).
    $$
    Choose $\delta>0$ such that
    $$
    0<\varepsilon^*-\delta<\varepsilon^*+\delta<0.5.
    $$
    For all large $n$, $\varepsilon_n\in(\varepsilon^*-\delta,\varepsilon^*+\delta)$, hence
    $$
    q_{\varepsilon_n}\in [q_{\varepsilon^*-\delta},\,q_{\varepsilon^*+\delta}].
    $$
    Fix $c>q_{\varepsilon^*+\delta}$. Since (A3) holds for every $\varepsilon\in(0,0.5)$, so on
    $$
    [q_{\varepsilon^*-\delta},c]
    $$
    we may write
    $$
    \dd F(x)=f_c(x)\,\dd x
    $$
    for some $f_c\in L^1([q_{\varepsilon^*-\delta},c])$.
    Now decompose
    \begin{align*}
    &\int_{q_{\varepsilon_n}}^{\infty}\psi'\bigl(x-(\theta_0+\eta_n)\bigr)\,dF(x)
    -
    \int_{q_{\varepsilon^*}}^{\infty}\psi'\bigl(x-(\theta_0+\eta_{\varepsilon^*+})\bigr)\,dF(x) \\
    &=
    \Bigg[
    \int_{q_{\varepsilon_n}}^{c}\psi'\bigl(x-(\theta_0+\eta_n)\bigr)f_c(x)\,dx
    -
    \int_{q_{\varepsilon^*}}^{c}\psi'\bigl(x-(\theta_0+\eta_{\varepsilon^*+})\bigr)f_c(x)\,dx
    \Bigg] \\
    &\qquad+
    \Bigg[
    \int_c^\infty \psi'\bigl(x-(\theta_0+\eta_n)\bigr)\,dF(x)
    -
    \int_c^\infty \psi'\bigl(x-(\theta_0+\eta_{\varepsilon^*+})\bigr)\,dF(x)
    \Bigg] \\
    &=: A_{n,c}+R_{n,c}.
    \end{align*}
    Since $\psi$ is $L$-Lipschitz, $|\psi'|\le L$ a.e., so
    $$
    |R_{n,c}|\le 2L\,F((c,\infty)).
    $$
    For $A_{n,c}$, write
    \begin{align*}
    A_{n,c}
    &=
    \int_{q_{\varepsilon^*}}^{c}
    \Big[
    \psi'\bigl(x-(\theta_0+\eta_n)\bigr)
    -
    \psi'\bigl(x-(\theta_0+\eta_{\varepsilon^*+})\bigr)
    \Big]f_c(x)\,dx \\
    &\qquad+
    \Bigg[
    \int_{q_{\varepsilon_n}}^{c}\psi'\bigl(x-(\theta_0+\eta_n)\bigr)f_c(x)\,dx
    -
    \int_{q_{\varepsilon^*}}^{c}\psi'\bigl(x-(\theta_0+\eta_n)\bigr)f_c(x)\,dx
    \Bigg] \\
    &=: B_{n,c}+C_{n,c}.
    \end{align*}
    For the second term,
    $$
    |C_{n,c}|
    \le
    L\int_{q_{\varepsilon_n}\wedge q_{\varepsilon^*}}^{q_{\varepsilon_n}\vee q_{\varepsilon^*}} f_c(x)\,dx
    \to 0,
    $$
    because $q_{\varepsilon_n}\to q_{\varepsilon^*}$ by continuity of the quantile map at $\varepsilon^*$.
    For the first term, let
    $$
    h_c(x):=f_c(x)\mathbf 1_{[q_{\varepsilon^*},c]}(x)\in L^1(\mathbb R).
    $$
    Changing variables in the two integrals gives
    $$
    B_{n,c}
    =
    \int_{\mathbb R}\psi'(u)\Big[h_c(u+\theta_0+\eta_n)-h_c(u+\theta_0+\eta_{\varepsilon^*+})\Big]\,du,
    $$
    hence
    $$
    |B_{n,c}|
    \le
    L\int_{\mathbb R}\Big|h_c(u+\theta_0+\eta_n)-h_c(u+\theta_0+\eta_{\varepsilon^*+})\Big|\,du
    \to 0,
    $$
    since translations are continuous in $L^1$ and $\eta_n\to\eta_{\varepsilon^*+}$.
    Thus $A_{n,c}\to 0$ for each fixed $c$, and therefore
    $$
    \limsup_{n\to\infty}
    \left|
    \int_{q_{\varepsilon_n}}^{\infty}\psi'\bigl(x-(\theta_0+\eta_n)\bigr)\,dF(x)
    -
    \int_{q_{\varepsilon^*}}^{\infty}\psi'\bigl(x-(\theta_0+\eta_{\varepsilon^*+})\bigr)\,dF(x)
    \right|
    \le 2L\,F((c,\infty)).
    $$
    Letting $c\to\infty$ yields the claim.
\end{proof}
}
\subsection{Proof of the Coupled M-Estimation Equivalence}
\label{sec:identification_proof}
\begin{proof}[Proof of Proposition~\ref{prop:finite_identification_sensitivity}]
We prove only the $+$ case, since the $-$ case is entirely analogous after reversing the order of the sample. Fix $m \in [n]$, by construction,
$$
\hat q_{m+} \in [x_{(m)},x_{(m+1)}),
$$
so
$$
\EE_{F_n}\bigl[\II\{X\le \hat q_{m+}\}\bigr]
=
\frac{1}{n}\sum_{i=1}^n \II\{x_i\le \hat q_{m+}\}
=
\frac{m}{n}.
$$
Hence the third coordinate of
$\EE_{F_n}[\Psi_+(X;\hat{\vartheta}_+^{\mathrm{sen}})]$
is equal to $0$.
Next, by definition, $\hat\theta$ is the unique selection of solution to
$$
\EE_{F_n}[\psi(X-\theta)] = 0,
$$
so the first coordinate is also equal to $0$.
It remains to verify the second coordinate. For the fixed pair $(\hat\theta,\hat q_{m+})$, define
$$
H_{n+}(\eta)
:=
\EE_{F_n}\bigl[\psi(X-\hat\theta-\eta)\II\{X>\hat q_{m+}\} + (m/n)B\bigr].
$$
Since $\hat q_{m+}\in [x_{(m)},x_{(m+1)})$, exactly the largest $n-m$ observations exceed $\hat q_{m+}$, and therefore
$$
H_{n+}(\eta)
=
\frac{1}{n}\sum_{i>m}\psi(x_{(i)}-\hat\theta-\eta) + \frac{m}{n}B.
$$
By continuity and monotonicity of $\psi$, the map $\eta \mapsto H_{n+}(\eta)$ is continuous and nonincreasing on $\RR_{>0}$. Moreover,
\begin{align*}
H_{n+}(0)
&=
\frac{1}{n}\sum_{i>m}\psi(x_{(i)}-\hat\theta)+\frac{m}{n}B \\
&>
\frac{1}{n}\sum_{i>m}\psi(x_{(i)}-\hat\theta)
+
\frac{1}{n}\sum_{i\le m}\psi(x_{(i)}-\hat\theta) \\
&=
\EE_{F_n}[\psi(X-\hat\theta)]
=
0,
\end{align*}
where we used $\psi\le B$ and for a solution $\hat \theta$ to exist, one must have at least $\psi(x_{(1)} - \hat \theta) \le 0$.
On the other hand,
$$
\lim_{\eta\to\infty} H_{n+}(\eta)
=
-\frac{n-m}{n}B+\frac{m}{n}B
=
\frac{2m-n}{n}B.
$$
For the $+$ direction, the admissible range is $m/n\le 1/2$, so the right-hand side is nonpositive. Hence, by the intermediate value theorem, there exists at least one $\eta\in\RR_{>0}$ such that
$$
H_{n+}(\eta)=0.
$$
Since there exist at least one solution, and $\eta_{m/n+}(\hat\theta,x^{(n)})$ is the largest solution to $H_{n+}(\eta) = 0$ by Corollary \ref{cor:opt_attack_m_est}, we have
$$
\EE_{F_n}\bigl[\psi(X-\hat\theta-\eta_{m/n+}(\hat\theta,X^{(n)}))\II\{X>\hat q_{m+}\}+(m/n)B\bigr]=0,
$$
so the second coordinate is also equal to $0$.
Combining the three coordinates yields
$$
\EE_{F_n}\bigl[\Psi_+(X;\hat{\vartheta}_+^{\mathrm{sen}})\bigr]=0.
$$
This completes the proof.
\end{proof}
\begin{proof}[Proof of Proposition~\ref{prop:finite_identification_BP}]
Again, we prove only the $+$ case. By construction,
$$
\hat q_{k_{\eta+}+}\in [x_{(k_{\eta+})},x_{(k_{\eta+}+1)}),
$$
and therefore
$$
\EE_{F_n}\bigl[\II\{X\le \hat q_{k_{\eta+}+}\}\bigr]
=
\frac{k_{\eta+}}{n}
=
\BP_{\eta+}(\hat\theta,X^{(n)}).
$$
Hence the third coordinate of
$\EE_{F_n}[\Psi_+(X;\hat{\vartheta}_+^{\mathrm{BP}})]$
is equal to $0$.
The first coordinate is also equal to $0$, since $\hat\theta$ uniquely solves
$$
\EE_{F_n}[\psi(X-\theta)] = 0.
$$
It remains to study the second coordinate. Define, for $k\in [n]$,
$$
G_{n+}(k;\eta)
:=
\frac{1}{n}\sum_{i>k}\psi(x_{(i)}-\hat\theta-\eta)+\frac{k}{n}B.
$$
Since $\hat q_{k_{\eta+}+}\in [x_{(k_{\eta+})},x_{(k_{\eta+}+1)})$, exactly the largest $n-k_{\eta+}$ observations exceed $\hat q_{k_{\eta+}+}$, and thus
$$
\EE_{F_n}\bigl[\psi(X-\hat\theta-\eta)\II\{X>\hat q_{k_{\eta+}+}\}+\BP_{\eta+}(\hat\theta,X^{(n)})B\bigr]
=
G_{n+}(k_{\eta+};\eta).
$$
Now, by the characterization of $\BP_{\eta+}(\hat\theta,x^{(n)})$ in Theorem~\ref{thm:loc}, $k_{\eta+}$ is the smallest index such that
$$
G_{n+}(k_{\eta+};\eta)\ge 0.
$$
Hence
$$
G_{n+}(k_{\eta+}-1;\eta)<0\le G_{n+}(k_{\eta+};\eta).
$$
Then, $|\psi|\le B$ gives
\begin{align*}
G_{n+}(k_{\eta+};\eta)-G_{n+}(k_{\eta+}-1;\eta)
&=
\frac{1}{n}\Bigl(B-\psi(X_{(k_{\eta+})}-\hat\theta-\eta)\Bigr)\leq \frac{2B}{n}.
\end{align*}
Since
$$
G_{n+}(k_{\eta+}-1;\eta)<0\le G_{n+}(k_{\eta+};\eta),
$$
it follows that
\begin{align*}
0
\le
G_{n+}(k_{\eta+};\eta)
&<
G_{n+}(k_{\eta+};\eta)-G_{n+}(k_{\eta+}-1;\eta) \le
\frac{2B}{n},
\end{align*}
 Hence, we conclude that there exists
$$
r_{n+}:=G_{n+}(k_{\eta+};\eta)\in[0,2B/n]
$$
such that
$$
\EE_{F_n}\bigl[\psi(X-\hat\theta-\eta)\II\{X>\hat q_{k_{\eta+}+}\}+\BP_{\eta+}(\hat\theta,X^{(n)})B\bigr]
=
r_{n+}.
$$
Therefore,
$$
\EE_{F_n}\bigl[\Psi_+(X;\hat{\vartheta}_+^{\mathrm{BP}})\bigr]
=
(0,r_{n+},0).
$$
This proves the $+$ case, and the $-$ case follows analogously.
\end{proof}
Next, we will give a lemma characterizing that the bootstrap threshold BP and $m$-sensitivity derived from Theorem \ref{thm:loc_bootstrap} are also approximate solutions to the Z-system.
\begin{lemma}
\label{lem:finite_identification_bootstrap}

    {Assume (A1)--(A4) and use the multiplier weights from \eqref{eq:boostrap_condition}. Let $m/n\to\varepsilon\in(0,1/2)$ as in (A1).} Let $\eta_{m/n\pm}(\hat \theta_b, \tilde F_n)$
    be defined as in Definition~\ref{def:eta_extended}, where $\hat \theta_ b$ solves \eqref{eq:bootstrap}. Let $X_i$ be sorted in a non-decreasing order. Define $\tilde w_i:=W_i/\sum_{j=1}^n W_j$. Let $i_m^+$ and $i_m^-$ be the random indices with $$\sum_{i=1}^{i^+_m-1} \tilde w_i<m/n\le \sum_{i=1}^{i^+_m} \tilde w_i \quad \text{and} \quad \sum_{i=1}^{i^-_m-1} \tilde w_i \le 1- m/n < \sum_{i=1}^{i^-_m} \tilde w_i.$$
    Then,
    $\tilde {\vartheta}_{\pm}^{\mathrm{sen}} = (\hat \theta_b, \eta_{m/n\pm}(\hat \theta_b, \tilde F_n), X_{i_m^\pm}, m/n)$ satisfies
 $\EE_{\tilde F_n} [\Psi_{\pm}(X; \tilde {\vartheta}_{\pm}^{\mathrm{sen}})] = o_p(n^{-1/2})$.
    Similarly, fix some $\eta \in \RR_{>0}$ {whose two population breakdown levels lie in $(0,1/2)$, and assume (A2)--(A4) at each level}. Let $\BP_{\eta \pm}(\hat \theta_b, \tilde F_n)$ be defined as in Definition~\ref{def:BP_extended}, where $\hat \theta_b$ solves \eqref{eq:bootstrap}. Let $X_i$ be sorted in a non-decreasing order. Denote
\[
\fitdisplay{\sum_{i=1}^{\tilde k_{\eta+}-1}\tilde w_i < {\BP_{\eta+}}(\hat\theta_b,\tilde F_n) \le \sum_{i=1}^{\tilde k_{\eta+}}\tilde w_i \quad\text{and}\quad \sum_{i=1}^{\tilde k_{\eta-}-1}\tilde w_i \le 1-{\BP_{\eta-}}(\hat\theta_b,\tilde F_n) < \sum_{i=1}^{\tilde k_{\eta-}}\tilde w_i.}
\]
    Then, $\tilde  {\vartheta}_{\pm}^{\mathrm{BP}} = (\hat \theta_b, \eta, X_{\tilde k_{\eta\pm}}, \BP_{\eta \pm}(\hat \theta_b, \tilde F_n))$ satisfies
     $\EE_{\tilde F_n} [\Psi_{\pm}(X; \tilde {\vartheta}_{\pm}^{\mathrm{BP}})] = o_p(n^{-1/2})$.
\end{lemma}
\begin{proof}[Proof of Lemma~\ref{lem:finite_identification_bootstrap}]
By Assumption (A3), the sample has no ties almost surely.
{Work on $\sum_iW_i>0$ and on the finite roots at the interior levels in the lemma. The population sign changes and uniform convergence of the weighted scores give these events probability tending to one.}
Recall that $\hat\theta_b$ solves \eqref{eq:bootstrap}, {so the first component of the estimating equations is zero.} We show the $+$ direction first.
{Set $\tilde q=X_{i_m^+}$. By the definition of $i_m^+$,
\[
\fitdisplay{0\le\EE_{\tilde F_n}[(\Psi_+(X;(\hat\theta_b,\eta,\tilde q,m/n)))_3]=\tilde F_n(\tilde q)-m/n\le\tilde w_{i_m^+}\le\max_i\tilde w_i.}
\]
Thus the quantile equation has a residual bounded by the largest normalized weight.}
Define
\[
\fitdisplay{\tilde H_{n+}(\eta):=\sum_{i>i_m^+}\tilde w_i\psi(X_i-(\hat\theta_b+\eta))+\Big(\sum_{i=1}^{i_m^+}\tilde w_i-\frac mn\Big)\psi(X_{i_m^+}-(\hat\theta_b+\eta))+\frac mn\psi(\infty).}
\]
By Theorem~\ref{thm:loc_bootstrap},
\[
\fitdisplay{\eta_{m/n+}(\hat\theta_b,\tilde F_n)=\max\{\eta:\tilde H_{n+}(\eta)=0\},}
\]
because $\tilde H_{n+}$ is nonincreasing. Using $\tilde q=X_{i_m^+}$, define
\[
\fitdisplay{H_{n+}(\eta):=\EE_{\tilde F_n}[(\Psi_+(X;(\hat\theta_b,\eta,\tilde q,m/n)))_2]=\sum_{i>i_m^+}\tilde w_i\psi(X_i-(\hat\theta_b+\eta))+\frac mn B.}
\]
Since $\|\psi\|_\infty=B$,
\[
\fitdisplay{\sup_\eta|\tilde H_{n+}(\eta)-H_{n+}(\eta)|=\sup_\eta\Big|\Big(\sum_{i=1}^{i_m^+}\tilde w_i-\frac mn\Big)\psi(X_{i_m^+}-({\hat\theta_b}+\eta))\Big|\le B\max_i\tilde w_i.}
\]
The moment condition \eqref{eq:boostrap_condition} gives $\max_iW_i=O_p(n^{1/(2+\kappa)})$ and $\sum_iW_i=n+O_p(n^{1/2})$. Thus $\max_i\tilde w_i=o_p(n^{-1/2})$ and
\begin{equation}\label{eq:proof_normality_step1}
\fitdisplay[\dimexpr\linewidth-3em\relax]{\sup_\eta|\tilde H_{n+}(\eta)-H_{n+}(\eta)|=o_p(n^{-1/2}).}
\end{equation}
Using \eqref{eq:proof_normality_step1} and $\tilde H_{n+}(\eta_{m/n+}(\hat\theta_b,\tilde F_n))=0$, we obtain
\[
\fitdisplay{H_{n+}(\eta_{m/n+}(\hat\theta_b,\tilde F_n))=o_p(n^{-1/2}).}
\]
{Together with the first component and the quantile residual bound, this proves the sensitivity conclusion.}

Now consider $\tilde\vartheta_+^{\mathrm{BP}}$. {Write $m=n\BP_{\eta+}(\hat\theta_b,\tilde F_n)$ and set $\tilde q=X_{i_m^+}$. The same bound gives $0\le\tilde F_n(\tilde q)-m/n\le\max_i\tilde w_i$. Using normalized weights,}
\[
\fitdisplay{{\EE_{\tilde F_n}[(\Psi_+(X;(\hat\theta_b,\eta,\tilde q,m/n)))_2]=\sum_{i=1}^n\tilde w_i\psi(X_i-\hat\theta_b-\eta)\II\{X_i>\tilde q\}+\frac mn B.}}
\]
It remains to show that this expression is $o_p(n^{-1/2})$.
By Theorem~\ref{thm:loc_bootstrap}, $m$ is the smallest integer for which
\[
\fitdisplay{G_{n+}(m;\eta):=\sum_{i>i_m^+}\tilde w_i\psi(X_i-(\hat\theta_b+\eta))+\Big(\sum_{i=1}^{i_m^+}\tilde w_i-\frac mn\Big)\psi(X_{i_m^+}-(\hat\theta_b+\eta))+\frac mn\psi(\infty)\ge0.}
\]
{At zero replacements, set $G_{n+}(0;\eta)=\sum_i\tilde w_i\psi(X_i-\hat\theta_b-\eta)$. By minimality, $G_{n+}(m-1;\eta)<0$. Increasing the count from $m-1$ to $m$ removes exactly another $1/n$ of mass, whose score lies between $-B$ and $B$, and adds $B/n$. Hence
\[
\fitdisplay{0\le G_{n+}(m;\eta)-G_{n+}(m-1;\eta)\le\frac{2B}{n}.}
\]
This bound holds regardless of how many observations lie between the two weighted cutoffs. Therefore,
\[
\fitdisplay{0\le G_{n+}(m;\eta)\le G_{n+}(m;\eta)-G_{n+}(m-1;\eta)\le\frac{2B}{n}=o_p(n^{-1/2}).}
\]
}
The difference between $G_{n+}(m;\eta)$ and the second estimating equation is the retained boundary term, and
\[
\fitdisplay{\Big|\Big(\sum_{i=1}^{i_m^+}\tilde w_i-\frac mn\Big)\psi(X_{i_m^+}-(\hat\theta_b+\eta))\Big|\le B\max_i\tilde w_i=o_p(n^{-1/2}).}
\]
{Combining the three components proves the BP conclusion.
For the ``$-$'' direction, use the lower cutoffs in the lemma. The CDF residual and the retained boundary mass are again bounded by $\max_i\tilde w_i$. Each additional $1/n$ of replaced mass decreases the trimmed score by a value in $[0,2B/n]$, giving the same absolute residual bound.}
\end{proof}
\begin{proof}[Proof of Proposition \ref{prop:sensitivity_map_finite_main}]
    This is a reduced version of Proposition \ref{prop:sensitivity_map_finite}.
\end{proof}
\subsection{Proof of Consistency}
\label{sec:consistency_proof}
\begin{proof}[Proof of Theorem \ref{thm:eta_consistency_inP}]
    By Assumption (A3), the sample has no ties almost surely.
    We will only give a proof for $\eta_{m/n+}(\hat \theta, X^{(n)})$ as the one for  $\eta_{m/n-}(\hat \theta, X^{(n)})$ follows from a similar argument.
    Bounded and Lipschitz $\psi$ implies $\EE_F|\psi(X-\theta)|<\infty$ and continuity of $\EE_F[\psi(X-\theta)]$. By the weak law of large numbers, $n^{-1}\sum_{i=1}^n \psi(X_i-\theta) \to \EE_F[\psi(X-\theta)]$ in probability for each fixed $\theta$. Since $\EE_F[\psi(X-\theta)]$ is strictly decreasing at its unique zero $\theta_0$, we have $\hat\theta \overset{p}{\to} \theta_0$ by a standard consistency result for M-estimators (see for example, Lemma 5.10 in \citet{van2000asymptotic}).
    Next, remind that $\eta_{m/n+}(\hat \theta, X^{(n)}) = \max  \{\eta: \sum_{i >m}\psi(X_{(i)}-(\hat \theta + \eta))+m \psi(\infty) = 0 \}$ in Corollary \ref{cor:opt_attack_m_est}. To show the convergence of our $m$-sensitivity, we will show a uniform convergence result for the quantity
    $$
    T_{n+}(\eta):=\frac1n\sum_{i>m}\psi \big(X_{(i)}-(\hat\theta+\eta)\big)
    =\frac1n\sum_{i=1}^n \psi \big(X_i-(\hat\theta+\eta)\big)\II \{X_i>X_{(m)}\}.
    $$
    Specifically, define the population version as
    $$
    T_{+}(\eta):=(1-\varepsilon)\EE_F[\psi\big(X-(\theta_0+\eta)\big) \big|X>q_\varepsilon]
    =\EE_F\big[\psi\big(X-(\theta_0+\eta)\big)\II\{X>q_\varepsilon\}\big].
    $$
    We will show, for any compact $K \subset \RR$
    $$
    \sup_{\eta\in K}|T_{n+}(\eta)-T_+(\eta)| \overset{p}{\to} 0.
    $$
    Notice that for all $\eta \in K$, we have
    \begin{align*}
        |T_{n+}(\eta)-T_{+}(\eta)| \le A_n(\eta) + B_n(\eta) +C_n(\eta),
    \end{align*}
    where
    \begin{align*}
    A_n(\eta)&:=\Big|\frac1n\sum_{i=1}^n\big[\psi(X_i-(\hat\theta+\eta))-\psi(X_i-(\theta_0+\eta))\big]\II\{X_i>X_{(m)}\}\Big|,\\
    B_n(\eta)&:=\Big|\frac1n\sum_{i=1}^n\psi(X_i-(\theta_0+\eta))\big(\II\{X_i>X_{(m)}\}-\II\{X_i>q_\varepsilon\}\big)\Big|,\\
    C_n(\eta)&:=\Big|\frac1n\sum_{i=1}^n\psi(X_i-(\theta_0+\eta))\II\{X_i>q_\varepsilon\}
    - \mathbb E[\psi(X-(\theta_0+\eta))\II\{X>q_\varepsilon\}]\Big|.
    \end{align*}
    $L$-Lipschitzness of $\psi$ and consistency of $\hat\theta$ gives
    $$
    A_n(\eta) \le L \cdot |\hat \theta - \theta_0| \overset{p}{\to} 0
    $$
    uniformly over all $\eta$.
    The boundedness of $\psi$ {and $F_n(X_{(m)})=m/n$} give
    \[
    \fitdisplay{B_n(\eta)\le B\,|F_n(q_\varepsilon)-F_n(X_{(m)})|\le {B\left(|F_n(q_\varepsilon)-\varepsilon|+|m/n-\varepsilon|\right)}\xrightarrow{p}0.}
    \]
    {The first term converges by the law of large numbers and $F(q_\varepsilon)=\varepsilon$; the second converges by (A1). The bound is uniform in $\eta$.}
    For $C_n$, Lemma \ref{lem:donsker} implies that the function class $\{\psi(\cdot-(\theta_0+\eta))\II\{\cdot>q_\varepsilon\}:\eta\in K\}$ is a Glivenko-Cantelli class, so $\sup_{\eta\in K}C_n(\eta) \overset{p}{\to} 0$.
    Combining the above, we showed the uniform convergence for $T_n(\eta)$. Now, for any compact $K \subset \RR$, we also have
    \begin{align}
    \label{eq:consistency_of_H}
        \sup_{\eta\in K}|H_{n+}(\eta)-H_{+}(\eta)| \overset{p}{\to} 0
    \end{align}
    for $H_{n+} (\eta) := T_{n+}(\eta) + m/n \cdot B$ and $H_{+} = T_{+} + \varepsilon \cdot B$.
    By continuity and strict decrease of $H_+$ at its unique zero $\eta_{\varepsilon+}$, for every $\delta>0$ there exists $c_\delta>0$ with
    $$
    \inf_{|\eta-\eta_{\varepsilon+}|\ge \delta}|H_+(\eta)|\ \ge\ c_\delta.
    $$
    Pick a compact $K$ containing $\eta_{\varepsilon+}\pm\delta$. Then
    $$
    \PP \left(\sup_{\eta\in K}|H_{n+}(\eta)-H_+(\eta)|<\frac{c_\delta}{2}\right)\ \to\ 1,
    $$
    and on the event $\{\sup_{\eta\in K}|H_{n+}(\eta)-H_+(\eta)|<{c_\delta}/{2}\}$,
    $$
    H_{n+}(\eta_{\varepsilon+}-\delta)> \frac{c_\delta}{2}>0,
    \qquad
    H_{n+}(\eta_{\varepsilon+}+\delta)< -\frac{c_\delta}{2}<0.
    $$
    Since $H_{n+}$ is nonincreasing, the only possible zeroes are inside $(\eta_{\varepsilon+}-\delta,\eta_{\varepsilon+}+\delta)$. Notice that the (maximal) zero of $H_{n+}(\eta)$ is $\eta_{m/n+}(\hat\theta,X^{(n)})$ by definition, so
    \begin{align*}
        \biggl \{\sup_{\eta \in K}|H_{n+}(\eta)-H_+(\eta)| < \frac{c_\delta}{2} \biggr\} \subseteq
        \{|\eta_{m/n+}(\hat\theta,X^{(n)})-\eta_{\varepsilon+}|\le\delta\}.
    \end{align*}
    Therefore,
    $$
    \PP\big(|\eta_{m/n+}(\hat\theta,X^{(n)})-\eta_{\varepsilon+}|>\delta\big)
    \le
    \PP \left(\sup_{\eta\in K}|H_n(\eta)-H(\eta)|\ge\frac{c_\delta}{2}\right)\ \to 0.
    $$
    This concludes the proof.
\end{proof}
\begin{proof}[Proof of Corollary \ref{cor:BP_consistency}]
    We will give a proof for $\BP_{\eta+}(\hat \theta, X^{(n)})$, where the proof for $\BP_{\eta-}(\hat \theta, X^{(n)})$ is very similar.
    Let $\rho > 0$ be arbitrary. We will show that for large $n$, with high probability:
    $$
    \varepsilon^* - \rho < \BP_{\eta+}(\hat \theta, X^{(n)}) < \varepsilon^* + \rho.
    $$
    Lemma \ref{lem:sens-bias_property} implies that the function $\varepsilon \mapsto \eta_{\varepsilon+}$ is strictly increasing and continuous at $\varepsilon^*$. Therefore, there exists $\delta > 0$ such that $\varepsilon \leq \varepsilon^* - \rho$ gives $\eta_{\varepsilon +} \leq \eta - \delta$ and that $\varepsilon \geq \varepsilon^* + \rho$ gives $\eta_{\varepsilon +} \geq \eta + \delta$. Now define two sequences of indices:
    $$
    m_1(n) = \lfloor n(\varepsilon^* - \rho) \rfloor, \quad m_2(n) = \lceil n(\varepsilon^* + \rho) \rceil.
    $$
    Note that
    $
    \frac{m_1(n)}{n} \to \varepsilon^* - \rho \text{ and } \frac{m_2(n)}{n} \to \varepsilon^* + \rho.
    $
    By Theorem \ref{thm:eta_consistency_inP}, we have pointwise convergence at these specific sequences:
    $$
    \eta_{m_1(n)/n+} \overset{p}{\to} \eta_{(\varepsilon^* - \rho)+} \leq \eta - \delta, \quad \eta_{m_2(n)/n+} \overset{p}{\to} \eta_{(\varepsilon^* + \rho)+} \geq \eta + \delta
    $$
    Therefore, with probability converging to 1, $\eta_{m_1(n)/n+} < \eta$ and $\eta_{m_2(n)/n+} > \eta$. By the definition of $\BP_{\eta+}$, if $\eta_{m_1(n)/n+} < \eta$ and $\eta_{m_2(n)/n+} > \eta$, then necessarily:
    $$
    \varepsilon^* - \rho < \BP_{\eta+}(\hat \theta, X^{(n)}) < \varepsilon^* + \rho.
    $$
    Hence, we get weak consistency for $\BP_{\eta+}(\hat \theta, X^{(n)})$.
\end{proof}
\subsection{Proof of Asymptotic Normality and Bootstrap Validity}
\label{sec:normality_proof}
    \begin{theorem}
    \label{thm:bootstrap_normality}
    Suppose (A1)-(A4) are true.
     Let $\eta_{m/n\pm}(\hat \theta, F_n)$ and $\eta_{m/n\pm}(\hat \theta_b, \tilde F_n)$ be defined as in Definition \ref{def:eta_extended}, where $\hat \theta$ solves \eqref{M-est} and $\hat \theta_b$ solves \eqref{eq:bootstrap}. Then,
    $$
    \sqrt{n}\bigl(\eta_{m/n\pm}(\hat \theta, F_n) - \eta_{\varepsilon \pm}) \bigr)
    {\ \rightsquigarrow \ }
    Z \sim
    \mathcal{N}\left(0,
    V_\pm
    \right)
    \text{ and }
    \sqrt{n}\bigl(\eta_{m/n\pm}(\hat \theta_b, \tilde F_n) - \eta_{m/n\pm}(\hat \theta,  F_n)\bigr)
    \overunderset{p}{W}{\rightsquigarrow}
    Z,
    $$
    where $V_\pm$ is given in \eqref{eq:var_eta+}-\eqref{eq:var_eta-}, and ${\rightsquigarrow}$ and $\overunderset{p}{W}{\rightsquigarrow}$ denote weak and  conditional weak convergence.
\end{theorem}
We also obtain the asymptotic normality and bootstrap validity for the threshold breakdown point itself.
\begin{theorem}
\label{thm:BP_normality}
    Let $X_1,\dots,X_n \overset{i.i.d.}{\sim} F$ for some distribution $F$. Suppose $\eta$ is a solution to \eqref{eq:population_eta+} or \eqref{eq:population_eta-} for some $\varepsilon^* \in (0, 0.5)$, which we denote as $H_{\pm}(\eta) := H_{\pm}(\eta_{\varepsilon^*\pm}) = 0$. Suppose (A2)-(A4) are true for $\varepsilon^*$. Let $\BP_{\eta \pm}(\hat \theta_b, \tilde F_n)$ and $\BP_{\eta \pm}(\hat \theta, F_n)$ be defined as in Definition \ref{def:BP_extended}, where $\hat \theta$ solves \eqref{M-est} and $\hat \theta_b$ solves \eqref{eq:bootstrap}. Then,
    $$
    \sqrt{n}\bigl(\BP_{\eta \pm}(\hat \theta, F_n) - \varepsilon^* \bigr)
    {\ \rightsquigarrow \ }
    Z \sim
    \mathcal{N}\left(0,
    \sigma^2_{\pm}
    \right)
    \text{ and }
    \sqrt{n}\bigl(\BP_{\eta \pm}(\hat \theta_b, \tilde F_n) - \BP_{\eta \pm}(\hat \theta, F_n)\bigr)
    \overunderset{p}{W}{\rightsquigarrow}
    Z,
    $$
    where $\theta_0$ is the unique zero to $\EE[\psi(X - \theta)] = 0$ and
    \begin{align*}
        \sigma^2_+ := &~ \left (\frac{\int_{q_{\varepsilon^*}}^{\infty} \psi'(x - (\theta_0 + \eta_{\varepsilon^*+})) \dd F(x)}{ -\psi(q_{\varepsilon^*} - (\theta_0 + \eta_{\varepsilon^*+})) + \|\psi\|_\infty} \right)^2 V_+, \\
        \sigma^2_- := &~ \left (\frac{-\int_{-\infty}^{q_{1-\varepsilon^*}} \psi'\big(x-(\theta_0-\eta_{\varepsilon^*-})\big) \dd F(x)}{ -\psi(q_{1-\varepsilon^*} - (\theta_0 - \eta_{\varepsilon^*-})) - \|\psi\|_\infty} \right)^2 V_-
    \end{align*}
    where $V_\pm$ is given in \eqref{eq:var_eta+}-\eqref{eq:var_eta-}, and ${\rightsquigarrow}$ and $\overunderset{p}{W}{\rightsquigarrow}$ denote weak and  conditional weak convergence.
\end{theorem}
\begin{proof}[Proof of Theorem \ref{thm:bootstrap_normality}]
    By Proposition \ref{prop:finite_identification_sensitivity}, $\eta_{m/n\pm}(\hat \theta, F_n)$ is a solution to $\EE_{F}[\Psi_{\pm}(X; \vartheta^{\mathrm{sen}}_\pm)] = 0$.
    Note also that  Assumption (A3) guarantees there are no ties almost surely.
    By Lemma \ref{lem:finite_identification_bootstrap}, $\eta_{m/n\pm}(\hat \theta_b, \tilde F_n)$ is a root-n solution to the estimation equations \eqref{eq:Z-system}.
    Hence, we can derive the asymptotic normality results using M-estimation theory. The idea is to apply Theorem 10.16 in \citet{kosorok2008introduction}. Thus, we only need to check the conditions. For convenience, we restate this result as Theorem \ref{prop:master} in Appendix \ref{sec:technical}.
    Recall that assumptions (A2) and (A4) ensure that for each $\varepsilon\in(0,0.5]$ there exist unique solutions to $\EE[\Psi_{\pm}(X; \vartheta)]=0$ given by
    $$
    \vartheta_{0+} = (\theta_0,\eta_{\varepsilon+},q_\varepsilon,\varepsilon),
    \qquad
    \vartheta_{0-} = (\theta_0,\eta_{\varepsilon-},q_{1-\varepsilon},\varepsilon),
    $$
    where $\theta_0$ and $\eta_{\varepsilon\pm}$ are the same as in Section~\ref{sec:assumptions}.
    Following the argument in Theorem \ref{thm:eta_consistency_inP}, $\theta_0$ and $\eta_{\varepsilon\pm}$ are unique. Because $F$ has a positive density at $q_\varepsilon$ and $q_{1-\varepsilon}$, the population quantiles are also unique. Hence, our setting satisfies the identifiability condition (A).
    Let $\Psi_\pm$ be defined as in  \eqref{eq:Z-system}, and let
    \begin{align}
    \label{eq:B}
        \mathcal F=\{\Psi_\pm(X;\vartheta):\vartheta\in\Theta\},
    \qquad
    \mathcal F_j=\{(\Psi_\pm(X;\vartheta))_j:\vartheta\in\Theta\},
    \quad j=1,2,3.
    \end{align}
    Lemma \ref{lem:donsker} implies that $\{\psi\{\cdot - \theta\} \II \{\cdot \le q\}: \theta, q \in \RR\} $ is a Donsker class, so each $\mathcal F_j$ is Donsker. Theorem 7.17 in \citet{kosorok2008introduction} implies that $\mathcal F$ is Donsker, so (B) is satisfied as Donsker classes are strong Glivenko-Cantelli classes by Lemma 8.17 in \citet{kosorok2008introduction}.
    For (C), we already have the Donsker result, we only need to check $L_2$ convergence. Since $F$ is continuous, it is enough to control the three coordinates separately. Write, for either $\Psi=\Psi_+$ or $\Psi=\Psi_-$,
    $$
    \Psi(X;\vartheta)-\Psi(X;\vartheta_0)
    =
    \begin{pmatrix}
    \Delta_1(X)\\
    \Delta_2(X)\\
    \Delta_3(X)
    \end{pmatrix},
    \qquad
    \vartheta=(\theta,\eta,q,\varepsilon),\quad
    \vartheta_0=(\theta_0,\eta_0,q_0,\varepsilon_0).
    $$
    Then
    $$
    \EE_F\bigl[\|\Psi(X;\vartheta)-\Psi(X;\vartheta_0)\|^2\bigr]
    =
    \sum_{j=1}^3 \EE_F[\Delta_j(X)^2].
    $$
    For the first coordinate, since $\psi$ is $L$-Lipschitz,
    $$
    |\Delta_1(X)|
    =
    |\psi(X-\theta)-\psi(X-\theta_0)|
    \le L|\theta-\theta_0|,
    $$
    so when $\|\vartheta-\vartheta_0\|\to 0$ we get
    $$
    \EE_F[\Delta_1(X)^2]\le L^2|\theta-\theta_0|^2 \to 0.
    $$
    For the second coordinate, using that $\|\psi\|_\infty\le B$, in either case we have
    $$
    |\Delta_2(X)|
    \le
    L\bigl(|\theta-\theta_0|+|\eta-\eta_0|\bigr)
    +
    B\,\bigl|\II\{X>q\}-\II\{X>q_0\}\bigr|
    +
    B|\varepsilon-\varepsilon_0|,
    $$
    or with $\II\{X<q\}$ in place of $\II\{X>q\}$ for $\Psi_-$. Hence
    $$
    \EE_F[\Delta_2(X)^2]
    \lesssim
    \bigl(|\theta-\theta_0|+|\eta-\eta_0|\bigr)^2
    +
    \EE_F\bigl|\II\{X>q\}-\II\{X>q_0\}\bigr|
    +
    |\varepsilon-\varepsilon_0|^2,
    $$
    or the analogous bound with $\II\{X<q\}$. Since
    $$
    \bigl|\II\{X>q\}-\II\{X>q_0\}\bigr|
    \le
    \II\{q\wedge q_0 < X \le q\vee q_0\},
    $$
    and likewise for $\II\{X<q\}$, continuity of $F$ at $q_0$ gives
    $$
    \EE_F[\Delta_2(X)^2]\to 0.
    $$
    For the third coordinate,
    $$
    \Delta_3(X)=\II\{X\le q\}-\II\{X\le q_0\}-(\varepsilon-\varepsilon_0),
    $$
    up to replacing $\varepsilon$ by $1-\varepsilon$ in the $\Psi_-$ case, which makes no difference. Thus
    $$
    |\Delta_3(X)|
    \le
    |\II\{X\le q\}-\II\{X\le q_0\}|+|\varepsilon-\varepsilon_0|,
    $$
    so
    $$
    \EE_F[\Delta_3(X)^2]
    \lesssim
    \EE_F|\II\{X\le q\}-\II\{X\le q_0\}|+|\varepsilon-\varepsilon_0|^2.
    $$
    Again,
    $$
    |\II\{X\le q\}-\II\{X\le q_0\}|
    =
    \II\{q\wedge q_0 < X \le q\vee q_0\},
    $$
    hence continuity of $F$ implies
    $$
    \EE_F[\Delta_3(X)^2]\to 0.
    $$
    Therefore,
    \begin{align}
    \label{eq:C}
        \EE_F[\|\Psi(X;\vartheta)-\Psi(X;\vartheta_0)\|^2]\to 0
    \qquad\text{as}\qquad
    \|\vartheta-\vartheta_0\|\to 0.
    \end{align}
    For (D), we check that the derivative matrices $V_{\vartheta_{0\pm}\pm}$ in \eqref{eq:jac} are non-singular.
     In addition, the second moment of $\Psi$ is finite since $\psi$ is bounded. Hence, (D) is also satisfied.
    In Proposition \ref{prop:finite_identification_sensitivity}, we have shown that $(\hat \theta, \eta_{m/n\pm}(\hat \theta, F_n), \hat q, m/n)$ is a $o_p(n^{-1/2})$ solution to the estimating equations defined by \eqref{eq:Z-system}. Notice that because $m/n = \varepsilon + O(1/n)$, we also have that
    $(\hat \theta, \eta_{m/n\pm}(\hat \theta, F_n), \hat q, \varepsilon)$ is a $o_p(n^{-1/2})$ solution to the same estimating equations. This is also true for the bootstrap, for which we have shown it is also a $o_p(n^{-1/2})$ solution in Lemma \ref{lem:finite_identification_bootstrap}. Thus, we verified (E).
    Hence, we can safely apply Theorem \ref{prop:master}. Notice that we still need to compute the covariance matrix $V_{\vartheta_{0\pm}\pm}^{-1} \EE_F\big[\Psi_\pm(X;\vartheta_{0\pm}) \Psi_\pm(X;\vartheta_{0\pm})^\top\big] V_{\vartheta_{0\pm}\pm}^{-\top}$. The derivative matrix $V_{\vartheta_{0\pm}\pm}$ of $\EE_F[\Psi_\pm(X;\vartheta)]$ at $\vartheta_{0\pm}$ can be computed as follows. By the chain rule, we have $\frac{\partial}{\partial q}\EE_F\big[g(X)\II\{X>q\}\big] = -g(q)f(q)$, $\frac{\partial}{\partial q}\EE_F\big[g(X)\II\{X<q\}\big] = g(q)f(q)$. Denote
    \begin{align*}
        & a = \EE_F[\psi'(X-\theta_0)], \\
        & b_+ = \EE_F\big[\psi'(X-(\theta_0+\eta_{\varepsilon+}))\II\{X>q_\varepsilon\}\big], \quad b_- = \EE_F\big[\psi'(X-(\theta_0-\eta_{\varepsilon-}))\II\{X<q_{1-\varepsilon}\}\big], \\
        & c_+ =  \psi\big(q_\varepsilon-(\theta_0+\eta_{\varepsilon+})\big)f(q_\varepsilon), \quad c_- = \psi\big(q_{1-\varepsilon}-(\theta_0 - \eta_{\varepsilon-})\big)f(q_{1-\varepsilon}), \\
        & d_+ = f(q_\varepsilon), \quad d_- = f(q_{1-\varepsilon}),
    \end{align*}
    we have
    \begin{align}
    \label{eq:jac}
        V_{\vartheta_{0+}+}
    =
    \begin{pmatrix}
        -a & 0 & 0\\
        -b_+ & -b_+ & -c_+ \\
        0 & 0 & d_+
    \end{pmatrix}
    \text{ and }
        V_{\vartheta_{0-}-}
    =
    \begin{pmatrix}
        -a & 0 & 0\\
        -b_- & b_- & c_- \\
        0 & 0 & d_-
    \end{pmatrix}.
    \end{align}
    By (A3) and (A4), $a>0$, $b_+>0$, $b_->0$, $d_+ > 0$, and $d_- > 0$. Then,
    $$
    \det(V_{\vartheta_{0+}+}) = a b_+ d_+ \neq 0,
    \qquad
    \det(V_{\vartheta_{0-}-}) = - a b_- d_- \neq 0.
    $$
    Hence both derivative matrices $V_{\vartheta_{0+}+}$ and $V_{\vartheta_{0-}-}$ are nonsingular.
    The inverse of the matrices is derived as
    \begin{align}
    \label{eq:jac_inv}
        V_{\vartheta_{0+}+}^{-1}
    =
    \begin{pmatrix}
        -1/a & 0 & 0\\
        1/a & -1/b_+ & {-c_+/(b_+d_+)} \\
        0 & 0 & 1/d_+
    \end{pmatrix}
    \text{ and }
        V_{\vartheta_{0-}-}^{-1}
    =
    \begin{pmatrix}
        -1/a & 0 & 0\\
        -1/a & 1/b_- & -c_-/(b_-d_-) \\
        0 & 0 & 1/d_-
    \end{pmatrix}.
    \end{align}
    Now we derive the second moment matrix $\EE_F\big[\Psi_\pm(X;\vartheta_{0\pm}) \Psi_\pm(X;\vartheta_{0\pm})^\top\big]$. Denote $\Psi_\pm(X;\vartheta_{0\pm}) := (\psi_1(X;\vartheta_{0\pm}), \psi_2(X;\vartheta_{0\pm}), \psi_3(X;\vartheta_{0\pm}))^\top$, then we have
    \begin{align*}
        &~\EE_F\big[\Psi_\pm(X;\vartheta_{0\pm}) \Psi_\pm(X;\vartheta_{0\pm})^\top\big] \\
        = &~
        \begin{pmatrix}
        \EE_F[\psi_1(X;\vartheta_{0\pm})^2] & \EE_F[\psi_1(X;\vartheta_{0\pm})\psi_2(X;\vartheta_{0\pm})] & \EE_F[\psi_1(X;\vartheta_{0\pm})\psi_3(X;\vartheta_{0\pm})]\\
        \EE_F[\psi_1(X;\vartheta_{0\pm})\psi_2(X;\vartheta_{0\pm})] & \EE_F[\psi_2(X;\vartheta_{0\pm})^2] & \EE_F[\psi_2(X;\vartheta_{0\pm})\psi_3(X;\vartheta_{0\pm})]\\
        \EE_F[\psi_1(X;\vartheta_{0\pm})\psi_3(X;\vartheta_{0\pm})] & \EE_F[\psi_2(X;\vartheta_{0\pm})\psi_3(X;\vartheta_{0\pm})] & \EE_F[\psi_3(X;\vartheta_{0\pm})^2]
        \end{pmatrix},
    \end{align*}
    where
    \begin{align*}
    & \EE_F[\psi_1(X;\vartheta_{0\pm})^2] = \EE_F[\psi(X-\theta_0)^2], \\
    &\EE_F[\psi_1(X;\vartheta_{0+})\psi_2(X;\vartheta_{0+})]
    = \EE_F\big[\psi(X-\theta_0)\psi\big(X-(\theta_0+\eta_{\varepsilon+})\big)\II\{X>q_\varepsilon\}\big],\\
    & \EE_F[\psi_1(X;\vartheta_{0-})\psi_2(X;\vartheta_{0-})]
    = \EE_F \big[\psi(X-\theta_0)\psi\big(X-(\theta_0-\eta_{\varepsilon-})\big)\II\{X<q_{1-\varepsilon}\}\big],\\
    &\EE_F[\psi_1(X;\vartheta_{0+})\psi_3(X;\vartheta_{0+})] = \EE_F\big[\psi(X-\theta_0)(\II\{X\le q_\varepsilon\}-\varepsilon)\big], \\
    &\EE_F[\psi_1(X;\vartheta_{0-})\psi_3(X;\vartheta_{0-})] = \EE_F\big[\psi(X-\theta_0)(\II\{X\le q_{1-\varepsilon}\}-(1-\varepsilon))\big], \\
    &\EE_F[\psi_2(X;\vartheta_{0+})^2] = \EE_F\big[\psi\big(X-(\theta_0+\eta_{\varepsilon+})\big)^2\II\{X>q_\varepsilon\}\big] - \varepsilon^2B^2, \\
    & \EE_F[\psi_2(X;\vartheta_{0-})^2] = \EE_F\big[\psi\big(X-(\theta_0-\eta_{\varepsilon-})\big)^2\II\{X<q_{1-\varepsilon}\}\big] - \varepsilon^2B^2, \\
    &\EE_F[\psi_2(X;\vartheta_{0+})\psi_3(X;\vartheta_{0+})] =\EE_F[\psi_2(X;\vartheta_{0-})\psi_3(X;\vartheta_{0-})] = \varepsilon^2B. \\
    &\EE_F[\psi_3(X;\vartheta_{0+})^2] = \EE_F[\psi_3(X;\vartheta_{0-})^2] = \varepsilon(1-\varepsilon).
    \end{align*}
    We only need the variance for $\eta_{m/n\pm}$, which is
    \begin{equation}
\label{eq:var_eta+}
\fitdisplay[\dimexpr\linewidth-3em\relax]{V_+=\EE_F\!\left[\left(\frac{\psi_1(X;\vartheta_{0+})}{a}-\frac{\psi_2(X;\vartheta_{0+})}{b_+}{-}\frac{c_+\psi_3(X;\vartheta_{0+})}{b_+d_+}\right)^2\right]}
\end{equation}
    and
    \begin{align} \nonumber
    V_-:= &~\Big( V_{\vartheta_{0-}-}^{-1}  \EE_F\big[\Psi_-(X;\vartheta_{0-})\Psi_-(X;\vartheta_{0-})^\top \big]  V_{\vartheta_{0-}-}^{-\top} \Big)_{2,2} \\ \nonumber
    =&~ \frac{\EE_F\big[\psi(X-\theta_0)^2\big]}{\EE_F\big[\psi'(X-\theta_0)\big]^2}
    + \frac{\EE_F\big[\psi\big(X-(\theta_0-\eta_{\varepsilon-})\big)^2\II\{X<q_{1-\varepsilon}\}\big]-\varepsilon^2B^2}{\EE_F\big[\psi'(X-(\theta_0-\eta_{\varepsilon-}))\II\{X<q_{1-\varepsilon}\}\big]^2} \\ \nonumber
    &\quad + \frac{\varepsilon(1-\varepsilon)\psi\big(q_{1-\varepsilon}-(\theta_0-\eta_{\varepsilon-})\big)^2}{\EE_F\big[\psi'(X-(\theta_0-\eta_{\varepsilon-}))\II\{X<q_{1-\varepsilon}\}\big]^2} \\ \nonumber
    &\quad - \frac{2\EE_F\Big[\psi(X-\theta_0)\psi\big(X-(\theta_0-\eta_{\varepsilon-})\big)\II\{X<q_{1-\varepsilon}\}\Big]}
    {\EE_F\big[\psi'(X-\theta_0)\big]\;\EE_F\big[\psi'(X-(\theta_0-\eta_{\varepsilon-}))\II\{X<q_{1-\varepsilon}\}\big]} \\ \nonumber
    &\quad + \frac{2\psi\big(q_{1-\varepsilon}-(\theta_0-\eta_{\varepsilon-})\big)\EE_F\Big[\psi(X-\theta_0)(\II\{X\le q_{1-\varepsilon}\}-(1-\varepsilon))\Big]}
    {\EE_F\big[\psi'(X-\theta_0)\big]\;\EE_F\big[\psi'(X-(\theta_0-\eta_{\varepsilon-}))\II\{X<q_{1-\varepsilon}\}\big]} \\ \label{eq:var_eta-}
    &\quad - \frac{2\psi\big(q_{1-\varepsilon}-(\theta_0-\eta_{\varepsilon-})\big)\varepsilon^2B}{\EE_F\big[\psi'(X-(\theta_0-\eta_{\varepsilon-}))\II\{X<q_{1-\varepsilon}\}\big]^2}.
    \end{align}
\end{proof}
\begin{proof}[Proof of Theorem \ref{thm:bootstrap_normality_main}]
    This is an immediate consequence of Theorem \ref{thm:bootstrap_normality}.
\end{proof}
{
\begin{proof}[Proof of Theorem \ref{thm:BP_normality}]
We only prove the result for $\BP_{\eta+}(\hat \theta, X^{(n)})$, where the proof for $\BP_{\eta-}(\hat \theta, X^{(n)})$ is very similar. Nonetheless, we will derive the derivative matrix for the ``-'' case for completeness.
Recall that
$$
\vartheta_{0+} = (\theta_0,\eta_{\varepsilon+},q_\varepsilon,\varepsilon),
\qquad
\vartheta_{0-} = (\theta_0,\eta_{\varepsilon-},q_{1-\varepsilon},\varepsilon),
$$
and throughout this proof all quantities are evaluated at $\varepsilon=\varepsilon^*$.
By Proposition~\ref{prop:finite_identification_BP}, conditional on $X^{(n)}$,  $\BP_{\eta+}(\hat \theta, X^{(n)})$ is a solution to $\EE_{F_n}\bigl[\Psi_{\pm}(X; \hat {\vartheta}_{\pm}^{\mathrm{BP}})\bigr]
=
(0, r_{n\pm}, 0)$, where $r_{n\pm} = O(1/n)$. Note also that Assumption (A3) guarantees there are no ties almost surely. By Lemma \ref{lem:finite_identification_bootstrap}, $\BP_{\eta +}(\hat \theta_b, \tilde F_n)$ is a root-$n$ solution to the estimation equations \eqref{eq:Z-system}. Hence, we can derive the asymptotic normality results using M-estimation theory by applying. In particular, we only need to invoke the conditions (A)--(E) required by Theorem \ref{prop:master}.
The identifiability condition (A) follows from the same argument as in Theorem~\ref{thm:bootstrap_normality}. Indeed, $\theta_0$ and $\eta_{\varepsilon^*+}$ are unique by (A2) and (A4), and because $F$ has a positive density at $q_{\varepsilon^*}$, the population quantile is also unique.
Set
$$
G_+(\varepsilon)
:=
\int_{q_\varepsilon}^{\infty}
\psi\bigl(x-(\theta_0+\eta_{\varepsilon^*+})\bigr)\,\dd F(x)
+
\varepsilon \|\psi\|_\infty.
$$
Since $H_+$ defined in \eqref{eq:population_eta+} is such that $H_+(\eta_{\varepsilon^*+})=0$, we have $G_+(\varepsilon^*)=0$. We claim that $\varepsilon^*$ is the unique zero of $G_+$. Indeed, for $\varepsilon_2>\varepsilon_1$,
\begin{align*}
G_+(\varepsilon_2)-G_+(\varepsilon_1)
&=
\int_{(q_{\varepsilon_1},q_{\varepsilon_2}]}
\Bigl(
\|\psi\|_\infty-\psi\bigl(x-(\theta_0+\eta_{\varepsilon^*+})\bigr)
\Bigr)\,\dd F(x)
\ge 0,
\end{align*}
so $G_+$ is nondecreasing. If there were two zeros $\varepsilon_1<\varepsilon_2$
then the display above would be zero, which implies
$$
\psi\bigl(x-(\theta_0+\eta_{\varepsilon^*+})\bigr)=\|\psi\|_\infty
\qquad
\text{for }F\text{-a.e. }x\in(q_{\varepsilon_1},q_{\varepsilon_2}].
$$
By monotonicity of $\psi$, this yields
$$
\psi\bigl(x-(\theta_0+\eta_{\varepsilon^*+})\bigr)=\|\psi\|_\infty
\qquad
\text{for all }x>q_{\varepsilon_1},
$$
and hence
$$
G_+(\varepsilon_1)=\|\psi\|_\infty>0,
$$
a contradiction. Therefore, the population parameter is identified, satisfying condition (A).
The verification of (B) and (C) is exactly the same as in the proof of Theorem \ref{thm:bootstrap_normality}, as we also need \eqref{eq:B} to be strong Glivenko-Cantelli, and \eqref{eq:C} to hold.
It remains to compute the asymptotic variance. Treating $\eta=\eta_{\varepsilon^*+}$ as fixed, we view $\Psi_+$ as a function of the free coordinates $(\theta,\varepsilon,q)$.
Let us introduce some additional notation
\begin{align*}
a
&=
\EE_F[\psi'(X-\theta_0)], \\
b_+
&=
\EE_F\big[\psi'(X-(\theta_0+\eta_{\varepsilon^*+}))\II\{X>q_{\varepsilon^*}\}\big]
=
\int_{q_{\varepsilon^*}}^{\infty}
\psi'(x-(\theta_0+\eta_{\varepsilon^*+}))\,\dd F(x), \\
c_+
&=
\psi\bigl(q_{\varepsilon^*}-(\theta_0+\eta_{\varepsilon^*+})\bigr) f(q_{\varepsilon^*}), \\
d_+
&=
f(q_{\varepsilon^*}), \\
e_+
&=
\|\psi\|_\infty - \psi\bigl(q_{\varepsilon^*}-(\theta_0+\eta_{\varepsilon^*+})\bigr).
\end{align*}
Note that the derivative matrix of $\EE_F[\Psi_+(X;\vartheta)]$ with respect to $(\theta,\varepsilon,q)$ at $\vartheta_{0+}=
(\theta_0,\eta_{\varepsilon^*+},q_{\varepsilon^*},\varepsilon^*)$ is
\begin{align}
\label{eq:jac_BP}
    V_{\vartheta_{0+}+}^{\mathrm{BP}}
=
\begin{pmatrix}
-a & 0 & 0\\
-b_+ & \|\psi\|_\infty & -c_+ \\
0 & -1 & d_+
\end{pmatrix}
\end{align}
and its determinant is
\begin{align*}
\det\bigl(V_{\vartheta_{0+}+}^{\mathrm{BP}}\bigr)
&=
-a
\begin{vmatrix}
\|\psi\|_\infty & -c_+ \\
-1 & d_+
\end{vmatrix} \\
&=
-a\bigl(\|\psi\|_\infty d_+ - c_+\bigr) \\
&=
-a d_+
\Bigl(
\|\psi\|_\infty - \psi\bigl(q_{\varepsilon^*}-(\theta_0+\eta_{\varepsilon^*+})\bigr)
\Bigr) \\
&=
-a d_+ e_+ \neq 0,
\end{align*}
where  $a, d_+ > 0$ by (A3) and (A4), and $e_+ > 0$ as we will argue below.
Hence $V_{\vartheta_{0+}+}^{\mathrm{BP}}$ is nonsingular.
Its inverse is
\begin{align}
    \label{eq:jac_BP_inv}
    \bigl(V_{\vartheta_{0+}+}^{\mathrm{BP}}\bigr)^{-1}
=
\begin{pmatrix}
-1/a & 0 & 0 \\
-b_+/(a e_+) & 1/e_+ & c_+/(d_+e_+) \\
-b_+/(a d_+ e_+) & 1/(d_+ e_+) & \|\psi\|_\infty/(d_+ e_+)
\end{pmatrix}.
\end{align}
It remains to show that $e_+>0$. Suppose for the sake of establishing a contradiction that
$$
\psi\bigl(q_{\varepsilon^*}-(\theta_0+\eta_{\varepsilon^*+})\bigr)=\|\psi\|_\infty.
$$
Because $\psi$ is nondecreasing, for every $x>q_{\varepsilon^*}$ we must then have
$$
\psi\bigl(x-(\theta_0+\eta_{\varepsilon^*+})\bigr)=\|\psi\|_\infty.
$$
Hence
$$
\int_{q_{\varepsilon^*}}^{\infty}
\psi\bigl(x-(\theta_0+\eta_{\varepsilon^*+})\bigr)\,\dd F(x)
=
\|\psi\|_\infty (1-\varepsilon^*)
>
0.
$$
On the other hand, since $H_+(\eta_{\varepsilon^*+})=0$, we have
$$
\int_{q_{\varepsilon^*}}^{\infty}
\psi\bigl(x-(\theta_0+\eta_{\varepsilon^*+})\bigr)\,\dd F(x)
=
-\varepsilon^* \|\psi\|_\infty
<
0,
$$
which is a contradiction. Therefore $e_+>0$.
Now let
$$
\Sigma_+
:=
\EE_F\big[\Psi_+(X;\vartheta_{0+})\Psi_+(X;\vartheta_{0+})^\top\big].
$$
Then by Theorem~\ref{prop:master},
$$
\sqrt{n}\bigl(\BP_{\eta +}(\hat \theta, X^{(n)}) - \varepsilon^* \bigr)
{\ \rightsquigarrow \ }
\mathcal{N}\left(
0,
\Bigl(
\bigl(V_{\vartheta_{0+}+}^{\mathrm{BP}}\bigr)^{-1}
\Sigma_+
\bigl(V_{\vartheta_{0+}+}^{\mathrm{BP}}\bigr)^{-\top}
\Bigr)_{2,2}
\right),
$$
and similarly for the conditional bootstrap.
We now compare this variance with $V_+$. In the proof of Theorem~\ref{thm:bootstrap_normality}, the derivative matrices given by \eqref{eq:jac} and \eqref{eq:jac_inv} for $\eta_{m/n+}$ are
$$
V_{\vartheta_{0+}+}
=
\begin{pmatrix}
-a & 0 & 0\\
-b_+ & -b_+ & -c_+ \\
0 & 0 & d_+
\end{pmatrix},
\qquad
V_{\vartheta_{0+}+}^{-1}
    =
    \begin{pmatrix}
        -1/a & 0 & 0\\
        1/a & -1/b_+ & {-c_+/(b_+d_+)} \\
        0 & 0 & 1/d_+
    \end{pmatrix}.
$$
so that
$$
V_+
=
\Bigl(
V_{\vartheta_{0+}+}^{-1}
\Sigma_+
V_{\vartheta_{0+}+}^{-\top}
\Bigr)_{2,2}.
$$
Hence, the second row of $V_{\vartheta_{0+}+}^{-1}$ is
$$
\left(
\frac{1}{a},
-\frac{1}{b_+},
-\frac{\psi(q_{\varepsilon^*}-(\theta_0+\eta_{\varepsilon^*+}))}{b_+}
\right).
$$
On the other hand, from \eqref{eq:jac_BP_inv} we see that the second row of $\bigl(V_{\vartheta_{0+}+}^{\mathrm{BP}}\bigr)^{-1}$  is
$$
\left(
-\frac{b_+}{a e_+},
\frac{1}{e_+},
\frac{\psi(q_{\varepsilon^*}-(\theta_0+\eta_{\varepsilon^*+}))}{e_+}
\right).
$$
Therefore,
$$
\left(
-\frac{b_+}{a e_+},
\frac{1}{e_+},
\frac{\psi(q_{\varepsilon^*}-(\theta_0+\eta_{\varepsilon^*+}))}{e_+}
\right)
=
-\frac{b_+}{e_+}
\left(
\frac{1}{a},
-\frac{1}{b_+},
-\frac{\psi(q_{\varepsilon^*}-(\theta_0+\eta_{\varepsilon^*+}))}{b_+}
\right).
$$
Hence the second row of $\bigl(V_{\vartheta_{0+}+}^{\mathrm{BP}}\bigr)^{-1}$ is exactly
$
-\frac{b_+}{e_+}
$
times the second row of $V_{\vartheta_{0+}+}^{-1}$. Since both variances are obtained by sandwiching the same second-moment matrix $\Sigma_+$, we obtain
\begin{align*}
\Bigl(
\bigl(V_{\vartheta_{0+}+}^{\mathrm{BP}}\bigr)^{-1}
\Sigma_+
\bigl(V_{\vartheta_{0+}+}^{\mathrm{BP}}\bigr)^{-\top}
\Bigr)_{2,2}
&=
\left(\frac{b_+}{e_+}\right)^2
\Bigl(
V_{\vartheta_{0+}+}^{-1}
\Sigma_+
V_{\vartheta_{0+}+}^{-\top}
\Bigr)_{2,2} \\
&=
\left(\frac{b_+}{e_+}\right)^2 V_+.
\end{align*}
Substituting the definitions of $b_+$ and $e_+$ gives
$$
\sigma^2_+
=
\left (
\frac{\int_{q_{\varepsilon^*}}^{\infty} \psi'(x - (\theta_0 + \eta_{\varepsilon^*+})) \dd F(x)}
{ -\psi(q_{\varepsilon^*} - (\theta_0 + \eta_{\varepsilon^*+})) + \|\psi\|_\infty}
\right)^2 V_+.
$$
For the ``-'' case, define
\begin{align*}
b_-
&=
\EE_F\big[\psi'(X-(\theta_0-\eta_{\varepsilon^*-}))\II\{X<q_{1-\varepsilon^*}\}\big]
=
\int_{-\infty}^{q_{1-\varepsilon^*}}
\psi'\bigl(x-(\theta_0-\eta_{\varepsilon^*-})\bigr)\,\dd F(x), \\
c_-
&=
\psi\bigl(q_{1-\varepsilon^*}-(\theta_0-\eta_{\varepsilon^*-})\bigr) f(q_{1-\varepsilon^*}), \\
d_-
&=
f(q_{1-\varepsilon^*}), \\
e_-
&=
\|\psi\|_\infty + \psi\bigl(q_{1-\varepsilon^*}-(\theta_0-\eta_{\varepsilon^*-})\bigr).
\end{align*}
Then one checks exactly as above that $e_->0$, that
$$
V_{\vartheta_{0-}-}^{\mathrm{BP}}
=
\begin{pmatrix}
-a & 0 & 0\\
-b_- & -\|\psi\|_\infty & c_- \\
0 & 1 & d_-
\end{pmatrix}
$$
is nonsingular, and that the second row of $\bigl(V_{\vartheta_{0-}-}^{\mathrm{BP}}\bigr)^{-1}$ is
$
-\frac{b_-}{e_-}
$
times the second row of $V_{\vartheta_{0-}-}^{-1}$. Therefore,
\begin{align*}
\sigma^2_-
&=
\Bigl(
\bigl( V_{\vartheta_{0-}-}^{\mathrm{BP}}\bigr)^{-1}
\Sigma_-
\bigl( V_{\vartheta_{0-}-}^{\mathrm{BP}}\bigr)^{-\top}
\Bigr)_{2,2} \\
&=
\left(\frac{b_-}{e_-}\right)^2 V_- \\
&=
\left (
\frac{-\int_{-\infty}^{q_{1-\varepsilon^*}} \psi'\bigl(x-(\theta_0-\eta_{\varepsilon^*-})\bigr) \dd F(x)}
{-\psi(q_{1-\varepsilon^*} - (\theta_0 - \eta_{\varepsilon^*-})) - \|\psi\|_\infty}
\right)^2 V_-.
\end{align*}
This concludes the proof.
\end{proof}
}
\begin{proof}[Proof of Theorem  \ref{thm:BP_normality_main}]
    This is an immediate consequence of Theorem \ref{thm:BP_normality}.
\end{proof}
{
\begin{proof}[Proof of Proposition \ref{prop:sensitivity_map_finite}]
By Assumption (A3), the sample has no ties almost surely.
We prove only the $+$ case; the $-$ case is analogous. Write
$$
\hat\eta:=\eta_{m/n+}(\hat\theta,F_n),
\qquad
\hat\varepsilon:=\BP_{\eta^*+}(\hat\theta,F_n),
$$
and
$$
\hat\eta_b:=\eta_{m/n+}(\hat\theta_b,\tilde F_n),
\qquad
\hat\varepsilon_b:=\BP_{\eta^*+}(\hat\theta_b,\tilde F_n).
$$
Also let
$$
\vartheta_{0+}:=(\theta_0,\eta^*,q_{\varepsilon^*},\varepsilon^*),
\qquad
\Psi_+^*(x):=\Psi_+(x;\vartheta_{0+}).
$$
By Proposition \ref{prop:finite_identification_sensitivity},
$$
\EE_{F_n}\bigl[\Psi_+(X;\hat{\vartheta}_{+}^{\mathrm{sen}})\bigr]=0.
$$
By Proposition \ref{prop:finite_identification_BP}, there exists $r_{n+}\in[-2B/n,2B/n]$ such that
$$
\EE_{F_n}\bigl[\Psi_+(X;\hat{\vartheta}_{+}^{\mathrm{BP}})\bigr]
=
(0,r_{n+},0),
$$
hence in particular $r_{n+}=o(n^{-1/2})$.
By Lemma \ref{lem:finite_identification_bootstrap},
$$
\EE_{\tilde F_n}\bigl[\Psi_+(X;\tilde{\vartheta}_{+}^{\mathrm{sen}})\bigr]
=
o_p(n^{-1/2}),
\qquad
\EE_{\tilde F_n}\bigl[\Psi_+(X;\tilde{\vartheta}_{+}^{\mathrm{BP}})\bigr]
=
o_p(n^{-1/2}).
$$
We remind that
$$
\hat{\vartheta}_{+}^{\mathrm{sen}}
:=
(\hat\theta,\hat\eta,\hat q_{m+},m/n),
\qquad
\hat{\vartheta}_{+}^{\mathrm{BP}}
:=
(\hat\theta,\eta^*,\hat q_{k_{\eta^*+}+},\hat\varepsilon),
$$
and
$$
\tilde{\vartheta}_{+}^{\mathrm{sen}}
:=
(\hat\theta_b,\hat\eta_b,X_{(i_m^+)},m/n),
\qquad
\tilde{\vartheta}_{+}^{\mathrm{BP}}
:=
(\hat\theta_b,\eta^*,X_{(\tilde k_{\eta^*+})},\hat\varepsilon_b)
$$
as in Proposition \ref{prop:finite_identification_sensitivity}, Proposition \ref{prop:finite_identification_BP}, and Lemma \ref{lem:finite_identification_bootstrap}.
Now consider the two square subsystems induced by $\Psi_+$. First, fix $\varepsilon=\varepsilon^*$ and regard
$$
(\theta,\eta,q)\longmapsto \Psi_+(x;\theta,\eta,q,\varepsilon^*).
$$
Its derivative matrix at $(\theta_0,\eta^*,q_{\varepsilon^*})$ is exactly
$V_{\vartheta_{0+}+}$ in \eqref{eq:jac}. Therefore, Theorem \ref{prop:master}
and its bootstrap analogue with the delta method yield
$$
\sqrt n
\begin{pmatrix}
\hat\theta-\theta_0\\
\hat\eta-\eta^*\\
\hat q_{m+}-q_{\varepsilon^*}
\end{pmatrix}
=
-
V_{\vartheta_{0+}+}^{-1}\,
\cdot
\frac{1}{\sqrt{n}}
    \sum_{i=1}^n
    \bigl(
    \Psi_+^*(X_i)-\EE_F[\Psi_+^*(X_i)]
    \bigr)
+
o_p(1),
$$
and
$$
\sqrt n
\begin{pmatrix}
\hat\theta_b-\hat\theta\\
\hat\eta_b-\hat\eta\\
X_{(i_m^+)}-\hat q_{m+}
\end{pmatrix}
=
-
V_{\vartheta_{0+}+}^{-1}\,
\cdot
\frac{1}{\sqrt{n}}
    \sum_{i=1}^n
    \left(
    \frac{W_i}{\bar W}-1
    \right)
    \Psi_+^*(X_i)
+
o_p^W(1).
$$
Second, fix $\eta=\eta^*$ and regard
$$
(\theta,\varepsilon,q)\longmapsto \Psi_+(x;\theta,\eta^*,q,\varepsilon).
$$
Its derivative matrix at $(\theta_0,\varepsilon^*,q_{\varepsilon^*})$ is exactly
$V_{\vartheta_{0+}+}^{\mathrm{BP}}$ in \eqref{eq:jac_BP}. Hence again by
Theorem \ref{prop:master} and its bootstrap analogue with delta method,
$$
\sqrt n
\begin{pmatrix}
\hat\theta-\theta_0\\
\hat\varepsilon-\varepsilon^*\\
\hat q_{k_{\eta^*+}+}-q_{\varepsilon^*}
\end{pmatrix}
=
-
\bigl(V_{\vartheta_{0+}+}^{\mathrm{BP}}\bigr)^{-1}\,
\cdot
\frac{1}{\sqrt{n}}
    \sum_{i=1}^n
    \bigl(
    \Psi_+^*(X_i)-\EE_F[\Psi_+^*(X_i)]
    \bigr)
+
o_p(1),
$$
and
$$
\sqrt n
\begin{pmatrix}
\hat\theta_b-\hat\theta\\
\hat\varepsilon_b-\hat\varepsilon\\
X_{(\tilde k_{\eta^*+})}-\hat q_{k_{\eta^*+}+}
\end{pmatrix}
=
-
\bigl(V_{\vartheta_{0+}+}^{\mathrm{BP}}\bigr)^{-1}\,
\cdot
\frac{1}{\sqrt{n}}
    \sum_{i=1}^n
    \left(
    \frac{W_i}{\bar W}-1
    \right)
    \Psi_+^*(X_i)
+
o_p^W(1).
$$
Write
\begin{align*}
    \frac{1}{\sqrt{n}}
    \sum_{i=1}^n
    \bigl(
    \Psi_+^*(X_i)-\EE_F[\Psi_+^*(X_i)]
    \bigr)=(G_{1n},G_{2n},G_{3n})^\top,
\\
\frac{1}{\sqrt{n}}
    \sum_{i=1}^n
    \left(
    \frac{W_i}{\bar W}-1
    \right)
    \Psi_+^*(X_i)=(\tilde G_{1n},\tilde G_{2n},\tilde G_{3n})^\top.
\end{align*}
Also recall
$$
e_+
:=
\partial_\varepsilon H_+(\eta^*,\varepsilon^*)
=
B-\frac{c_+}{d_+}.
$$
By \eqref{eq:jac_inv}, the second row of $V_{\vartheta_{0+}+}^{-1}$ is
$$
\left(
\frac{1}{a},
-\frac{1}{b_+},
-\frac{c_+}{b_+d_+}
\right).
$$
By \eqref{eq:jac_BP_inv}, the second row of
$\bigl(V_{\vartheta_{0+}+}^{\mathrm{BP}}\bigr)^{-1}$ is
$$
\left(
-\frac{b_+}{ae_+},
\frac{1}{e_+},
\frac{c_+}{d_+e_+}
\right).
$$
Therefore, if we define
$$
\Xi_n
:=
G_{2n}
-
\frac{b_+}{a}G_{1n}
+
\frac{c_+}{d_+}G_{3n},
\qquad
\tilde\Xi_n
:=
\tilde G_{2n}
-
\frac{b_+}{a}\tilde G_{1n}
+
\frac{c_+}{d_+}\tilde G_{3n},
$$
then the second coordinates of the above asymptotic linearizations give
$$
\sqrt n(\hat\eta-\eta^*)
=
\frac{\Xi_n}{b_+}
+
o_p(1),
\qquad
\sqrt n(\hat\eta_b-\hat\eta)
=
\frac{\tilde\Xi_n}{b_+}
+
o_p^W(1),
$$
and
$$
\sqrt n(\hat\varepsilon-\varepsilon^*)
=
-\frac{\Xi_n}{e_+}
+
o_p(1),
\qquad
\sqrt n(\hat\varepsilon_b-\hat\varepsilon)
=
-\frac{\tilde\Xi_n}{e_+}
+
o_p^W(1).
$$
Eliminating the common fluctuations $\Xi_n$ and $\tilde\Xi_n$ yields
$$
\sqrt n(\hat\varepsilon-\varepsilon^*)
=
-\frac{b_+}{e_+}\sqrt n(\hat\eta-\eta^*)
+
o_p(1),
$$
and
$$
\sqrt n(\hat\varepsilon_b-\hat\varepsilon)
=
-\frac{b_+}{e_+}\sqrt n(\hat\eta_b-\hat\eta)
+
o_p^W(1).
$$
Finally, along the population sensitivity--breakdown curve,
$$
H_+(\eta,\varepsilon)
=
\int_{q_\varepsilon}^{\infty}
\psi\bigl(x-(\theta_0+\eta)\bigr)\,\dd F(x)
+
\varepsilon B
=
0.
$$
Since
$$
\partial_\eta H_+(\eta^*,\varepsilon^*)=-b_+,
\qquad
\partial_\varepsilon H_+(\eta^*,\varepsilon^*)=e_+,
$$
the implicit function theorem gives
$$
\frac{\dd\varepsilon}{\dd\eta}\Big|_{\eta=\eta^*}
=
-
\frac{\partial_\eta H_+(\eta^*,\varepsilon^*)}
{\partial_\varepsilon H_+(\eta^*,\varepsilon^*)}
=
\frac{b_+}{e_+}.
$$
Substituting this into the previous two displays, we obtain
$$
\BP_{\eta^*+}(\hat\theta,F_n)-\varepsilon^*
=
-\frac{\dd\varepsilon}{\dd\eta}\Big|_{\eta=\eta^*}
\bigl(\eta_{m/n+}(\hat\theta,F_n)-\eta^*\bigr)
+
o_p\!\left(\frac{1}{\sqrt n}\right),
$$
and
$$
\BP_{\eta^*+}(\hat\theta_b,\tilde F_n)
-
\BP_{\eta^*+}(\hat\theta,F_n)
=
-\frac{\dd\varepsilon}{\dd\eta}\Big|_{\eta=\eta^*}
\bigl(
\eta_{m/n+}(\hat\theta_b,\tilde F_n)
-
\eta_{m/n+}(\hat\theta,F_n)
\bigr)
+
o_p^W\!\left(\frac{1}{\sqrt n}\right).
$$
This proves the proposition.
\end{proof}
}
\subsection{Proof of Multiplier Bootstrap}
\begin{proof}[Proof of Theorem~\ref{thm:loc_bootstrap}]
{Let $x_1\le\cdots\le x_n$, write $W_k=\sum_{i=1}^kw_i$ with $W_0=0$, and fix $m\in[n]$. Choose $i_m$ and $i'_m$ as in the theorem. A law $H\in B_{\mathrm{TV}}(F_{x^{(n)}}^w,m/n)$ is obtained by removing at most $m/n$ mass and adding the same mass elsewhere.
For each finite $\theta$, monotonicity of $\psi$ gives
\begin{gather*}
\fitdisplay{\int\psi(t-\theta)\,\dd H(t)\le\sum_{i>i_m}w_i\psi(x_i-\theta)+(W_{i_m}-m/n)\psi(x_{i_m}-\theta)+(m/n)\psi(\infty),}\\
\fitdisplay{\int\psi(t-\theta)\,\dd H(t)\ge\sum_{i<i'_m}w_i\psi(x_i-\theta)+(1-m/n-W_{i'_m-1})\psi(x_{i'_m}-\theta)+(m/n)\psi(-\infty).}
\end{gather*}
Indeed, the upper bound removes mass from the smallest observations and places it at $+\infty$; the lower bound removes mass from the largest observations and places it at $-\infty$. Both bounds are attained by the laws
\begin{gather*}
\fitdisplay{H_{m,+}=\sum_{i>i_m}w_i\delta_{x_i}+(W_{i_m}-m/n)\delta_{x_{i_m}}+(m/n)\delta_{+\infty},}\\
\fitdisplay{H_{m,-}=\sum_{i<i'_m}w_i\delta_{x_i}+(1-m/n-W_{i'_m-1})\delta_{x_{i'_m}}+(m/n)\delta_{-\infty}.}
\end{gather*}
These are probability laws in $B_{\mathrm{TV}}(F_{x^{(n)}}^w,m/n)$, since each moves at most $m/n$ mass.
Write $\hat\theta=\hat\theta(F_{x^{(n)}}^w)$. The score is nonincreasing in its location argument. Evaluating the upper extremal score at $\hat\theta+\eta$ therefore gives the upward threshold condition
\begin{equation}\label{eq:loc_sup_bootstrap}
\fitdisplay[\dimexpr\linewidth-3em\relax]{\sum_{i>i_m}w_i\psi(x_i-(\hat\theta+\eta))+(W_{i_m}-m/n)\psi(x_{i_m}-(\hat\theta+\eta))+(m/n)\psi(\infty)\ge0.}
\end{equation}
Similarly, evaluating the lower extremal score at $\hat\theta-\eta$ gives
\begin{equation}\label{eq:loc_inf_bootstrap}
\fitdisplay[\dimexpr\linewidth-3em\relax]{\sum_{i<i'_m}w_i\psi(x_i-(\hat\theta-\eta))+(1-m/n-W_{i'_m-1})\psi(x_{i'_m}-(\hat\theta-\eta))+(m/n)\psi(-\infty)\le0.}
\end{equation}
Taking the smallest $m$ satisfying the corresponding condition gives the one-sided threshold breakdown point. For a fixed $m$, taking the largest permitted $\eta$ gives the one-sided sensitivity. This proves the stated formulas.
}
\end{proof}
\subsection{Technical Auxiliary Results}
\label{sec:technical}
This section collects some technical results used throughout the proof. Recall for a measurable function class $\mathcal{F}$, define the empirical process by
    $$
    \mathbb{G}_n (f)
    :=
    \frac{1}{\sqrt{n}}
    \sum_{i=1}^n
    \bigl(
    f(X_i)-\EE_F[f(X_i)]
    \bigr),
    \qquad f \in \mathcal{F}.
    $$
    For the multiplier bootstrap, let $W_1,\dots,W_n$ be bootstrap weights and $\bar W := n^{-1}\sum_{i=1}^n W_i$. Define the bootstrap empirical process by
    $$
    \tilde{\mathbb{G}}_n (f)
    :=
    \frac{1}{\sqrt{n}}
    \sum_{i=1}^n
    \left(
    \frac{W_i}{\bar W}-1
    \right)
    f(X_i),
    \qquad f \in \mathcal{F}.
    $$
We start with a Donsker result.
\begin{lemma}
\label{lem:donsker}
Let
$$
\mathcal G
:=
\left\{
x \mapsto \psi(x-t)\II\{x>q\}
:\ t,q\in\mathbb R
\right\}.
$$
Suppose $\psi:\mathbb R\to\mathbb R$ is {nondecreasing}, $L$-Lipschitz and
$
\|\psi\|_\infty \le B < \infty.
$
Then $\mathcal G$ is an $F$-Donsker class, and hence an $F$-Glivenko--Cantelli class. Therefore,
$$
\sup_{g\in\mathcal G}
\left|
\frac{1}{n}\sum_{i=1}^n g(X_i) - \EE_F[g(X)]
\right|
=
o_p(1).
$$
Moreover, under the assumptions for the weighted empirical distribution $\tilde F_n$ in \eqref{eq:boostrap_condition}, we have
$$
\sup_{g\in\mathcal G}
\left|
\frac{1}{n}\sum_{i=1}^n \frac{W_i}{\bar W}\, g(X_i) - \EE_F[g(X)]
\right|
=
o_p(1).
$$
\end{lemma}
\begin{proof}
Write
$$
\mathcal H
:=
\{x\mapsto \psi(x-t): t\in\mathbb R\},
\qquad
\mathcal I
:=
\{x\mapsto \II\{x>q\}: q\in\mathbb R\}.
$$
Then
$$
\mathcal G = \{hi:h\in\mathcal H,\ i\in\mathcal I\}.
$$
Because $\psi$ is bounded and monotone,
hence $\mathcal H$ is a $F$-Donsker class; see Theorem 2.5.6 and Theorem 2.7.5 of \citet{van1996weak}.
 Also, $\mathcal I$ is a $F$-Donsker class; see Example 2.5.4 of \citet{van1996weak}. Since the multiplication map $(u,v)\mapsto uv$ is bounded Lipschitz on $[-B,B]\times[0,1]$, Theorem 2.10.6 of \citet{van1996weak} implies that $\mathcal G$ is $F$-Donsker. Hence $\mathcal G$ is $F$-Glivenko-Cantelli, and so
$$
\sup_{g\in\mathcal G}
\left|
\frac{1}{n}\sum_{i=1}^n g(X_i) - \EE_F[g(X)]
\right|
=
o_p(1).
$$
The weighted conclusion
$$
\sup_{g\in\mathcal G}
\left|
\frac{1}{n}\sum_{i=1}^n \left (\frac{W_i}{\bar W} - 1 \right)\, g(X_i)
\right|
=
o_p^W(1)
$$
then follows from Theorem 10.1 of \citet{kosorok2008introduction}.
\end{proof}
The next lemma establishes stochastic equicontinuity.
\begin{lemma}
\label{lem:equicon}
Let
$$
\mathcal G
:=
\left\{
g_{t,q}(x) = \psi(x-t)\II\{x>q\}
:\ t,q\in\mathbb R
\right\}.
$$
Suppose $\psi:\mathbb R\to\mathbb R$ is {nondecreasing}, bounded and Lipschitz with
$$
\|\psi\|_\infty \le B < \infty,
\qquad
\mathrm{Lip}(\psi)\le L < \infty.
$$
If $F$ is continuous at $q_0$, then the map $(t,q)\mapsto g_{t,q}$ is $L_2(F)$-continuous at every point $(t_0,q_0)$, that is,
$$
\|g_{t,q}-g_{t_0,q_0}\|_{L_2(F)} \to 0
\qquad\text{whenever}\qquad
(t,q)\to (t_0,q_0).
$$
Let $\tilde F_n$ satisfy the assumptions for the weighted empirical measure in Section \ref{sec:multiplier_bootstrap}.
Then the empirical processes $\GG_n$ and $\tilde \GG_n$
are asymptotically uniformly $\rho$-equicontinuous in probability with respect to
$$
\rho(g_1,g_2):=\|g_1-g_2\|_{L_2(F)}.
$$
Hence, if
$$
\hat t \overset{p}{\to} t_0,
\qquad
\hat q \overset{p}{\to} q_0,
$$
then
$$
\GG_n\bigl(g_{\hat t,\hat q}-g_{t_0,q_0}\bigr)=o_p(1), \qquad \tilde \GG_n(g_{\hat t,\hat q}-g_{t_0,q_0}\bigr) =o_p^W(1).
$$
\end{lemma}
\begin{proof}
For any $(t,q)$ and $(t_0,q_0)$, we have
\begin{align*}
|g_{t,q}(x)-g_{t_0,q_0}(x)|
&\le
\big|\psi(x-t)-\psi(x-t_0)\big|
+
\big|\psi(x-t_0)\big|\cdot\big|\II\{x>q\}-\II\{x>q_0\}\big| \\
&\le
L|t-t_0|
+
B\,\II\bigl\{\min\{q,q_0\}<x\le \max\{q,q_0\}\bigr\}.
\end{align*}
Therefore,
$$
\|g_{t,q}-g_{t_0,q_0}\|_{L_2(F)}
\le
L|t-t_0|
+
B\,\PP\bigl(\min\{q,q_0\}<X\le \max\{q,q_0\}\bigr)^{1/2}.
$$
Because $F$ is continuous at $q_0$, the probability term converges to $0$ as $q\to q_0$. Hence
$$
\|g_{t,q}-g_{t_0,q_0}\|_{L_2(F)} \to 0
\qquad\text{whenever}\qquad
(t,q)\to (t_0,q_0).
$$
In particular, if $\hat t \overset{p}{\to} t_0$ and $\hat q \overset{p}{\to} q_0$, then
$$
\rho(g_{\hat t,\hat q},g_{t_0,q_0})
=
\|g_{\hat t,\hat q}-g_{t_0,q_0}\|_{L_2(F)}
\overset{p}{\to} 0.
$$
Since $\mathcal G$ is $F$-Donsker  by Lemma \ref{lem:donsker}, $\GG_n$ is asymptotically uniformly $\rho$-equicontinuous in probability, see, for example, Theorem 7.19 of \citet{kosorok2008introduction}. Now fix $\epsilon,\eta>0$. By asymptotic uniform $\rho$-equicontinuity in probability, there exist $\delta>0$ and $N$ such that for all $n\ge N$,
$$
\PP\left(
\sup_{\rho(g_1,g_2)<\delta}
|\mathbb G_n(g_1)-\mathbb G_n(g_2)|
>\eta
\right)
<\epsilon.
$$
We have the following decomposition of events:
\begin{align*}
    &~\{|\mathbb G_n(g_{\hat t,\hat q}-g_{t_0,q_0})|>\eta\} \\
    = &~\{|\mathbb G_n(g_{\hat t,\hat q}-g_{t_0,q_0})|>\eta,\rho(g_{\hat t,\hat q},g_{t_0,q_0})\ge\delta\} \cup  \{|\mathbb G_n(g_{\hat t,\hat q}-g_{t_0,q_0})|>\eta,\rho(g_{\hat t,\hat q},g_{t_0,q_0})<\delta\} \\
    \subset & \{\rho(g_{\hat t,\hat q},g_{t_0,q_0})\ge\delta\} \cup  \{|\mathbb G_n(g_{\hat t,\hat q}-g_{t_0,q_0})|>\eta,\rho(g_{\hat t,\hat q},g_{t_0,q_0})<\delta\}.
\end{align*}
Hence, for all $n\ge N$,
\begin{align*}
\PP\bigl(|\mathbb G_n(g_{\hat t,\hat q}-g_{t_0,q_0})|>\eta\bigr)
&\le
\PP\left(
\sup_{\rho(g_1,g_2)<\delta}
|\mathbb G_n(g_1)-\mathbb G_n(g_2)|
>\eta
\right)
+
\PP\bigl(\rho(g_{\hat t,\hat q},g_{t_0,q_0})\ge\delta\bigr).
\end{align*}
The first term is smaller than $\epsilon$ for all large $n$, and the second term converges to $0$. Therefore,
$$
\mathbb G_n\bigl(g_{\hat t,\hat q}-g_{t_0,q_0}\bigr)=o_p(1).
$$
A similar analysis holds for $\tilde \GG_n$.
\end{proof}
Below is a reformulation of Theorem 10.16 in \citet{kosorok2008introduction} in our notation.
\begin{theorem}
\label{prop:master}
Let $\Theta \subset \mathbb{R}^p$ be open, and assume $\vartheta_0 \in \Theta$ satisfies
$$
\EE_F[\Psi(X;\vartheta_0)]=0.
$$
Also assume the following:
\begin{enumerate}
    \item[(A)] For any sequence $\{\vartheta_n\}\in \Theta$, $\EE_F[\Psi(X;\vartheta_n)]\to 0$ implies
    $$
    \|\vartheta_n-\vartheta_0\|\to 0;
    $$
    \item[(B)] The class $\{\Psi(x;\vartheta):\vartheta\in\Theta\}$ is strong Glivenko--Cantelli;
    \item[(C)] For some $\eta>0$, the class
    $$
    \mathcal{F}\equiv \{\Psi(x;\vartheta):\vartheta\in\Theta,\ \|\vartheta-\vartheta_0\|\le \eta\}
    $$
    is Donsker and
    $$
    \EE_F[\|\Psi(X;\vartheta)-\Psi(X;{\vartheta_0})\|^2] \to 0
    \quad\text{as}\quad
    \|\vartheta-\vartheta_0\|\to 0;
    $$
    \item[(D)] $\EE_F\|\Psi(X;{\vartheta_0})\|^2<\infty$ and $\vartheta \mapsto \EE_F[\Psi(X;\vartheta)]$ is differentiable at $\vartheta_0$ with nonsingular derivative matrix $V_{\vartheta_0}$;
    \item[(E)]
    $$
    \EE_{F_n}[\Psi(X;\hat{\vartheta})]=o_p(n^{-1/2})
    \quad\text{and}\quad
    \EE_{\tilde F_n}[\Psi(X;\tilde {\vartheta})]=o_p(n^{-1/2}).
    $$
\end{enumerate}
Then
$$
\sqrt{n}(\hat{\vartheta}-\vartheta_0) \rightsquigarrow Z
\sim N\!\left(0,\,
V_{\vartheta_0}^{-1}\,
\EE_F[\Psi(X;{\vartheta_0})\Psi(X;{\vartheta_0})^\top]\,
(V_{\vartheta_0}^{-1})^\top
\right),
$$
and
$$
\sqrt{n}(\tilde{\vartheta}-\hat{\vartheta})
\overunderset{P}{W}{\rightsquigarrow} Z.
$$
\end{theorem}
\section{Results for Different Types of Tests}
\label{sec:tests}
\subsection{Contamination Schemes and General Theorem}
\label{sec:con}
To study the most destabilizing directions of one-sided contamination in a simple and computable way, we introduce a contamination map that replaces either the smallest or the largest $m$ observations by a common target value. This scheme is especially natural for robustness analysis because many breakdown phenomena are driven by extreme-tail perturbations, as shown in Theorem \ref{thm:loc}. Let $x^{(n)}=(x_1,\dots,x_n)$ be a sample with order statistics $x_{(1)}\le \cdots \le x_{(n)}$, $m\in[n]$, and $c\in\overline{\mathbb R}:=\mathbb R\cup\{-\infty,+\infty\}$. We define the left- and right-contamination maps $\mathsf{C}_{m,c}^L,\mathsf{C}_{m,c}^R:\mathbb R^n\to \overline{\mathbb R}^n$  as
$$
\bigl(\mathsf{C}_{m,c}^L(x^{(n)})\bigr)_{(i)}
:=
\begin{cases}
c, & 1\le i\le m,\\
x_{(i)}, & m<i\le n,
\end{cases}
\qquad
\bigl(\mathsf{C}_{m,c}^R(x^{(n)})\bigr)_{(i)}
:=
\begin{cases}
x_{(i)}, & 1\le i\le n-m,\\
c, & n-m<i\le n.
\end{cases}
$$
Thus $\mathsf{C}_{m,c}^L$ sends the leftmost $m$ order statistics to $c$, whereas $\mathsf{C}_{m,c}^R$ sends the rightmost $m$ order statistics to $c$. In practice, one can curate a collection of contamination schemes, e.g.,
$$
\tilde x^{(n)} \in \Big\{\mathsf{C}_{m,c}^L(x^{(n)}): c\in \mathcal C_L \Big\} =: \mathsf C^L_{m, \mathcal C_L}.
$$
One simple choice for the target location is $\mathcal C_L = \big\{\infty, x_{(n)}\big\}$. When $m$ is large, this choice can be quite loose. We can therefore tighten the upper bounds computed in Theorem \ref{thm:meta} by choosing
\begin{equation}
\label{eq:contamination_left_loc}
    \mathcal C_L := \{\hat\theta + \eta_{m/n+}(\hat\theta,x^{(n)}) + a\,\delta: a \in \mathcal A\},
\end{equation}
where $\mathcal A \subset \mathbb R$ is an arbitrary finite set. Similarly, define
\begin{equation}
\label{eq:contamination_right_loc}
    \mathcal C_R:= \{\hat\theta - \eta_{m/n+}(\hat\theta,x^{(n)}) - a\,\delta: a \in \mathcal A\}.
\end{equation}
The heuristic idea is to move the points to the approximate location around the location estimate evaluated at contaminated samples (i.e., $\hat \theta + \eta_{m/n+}(\hat \theta, x^{(n)})$), and perturb them such that they contribute to the score equation by a value in $\{a \delta: a \in \mathcal A\}$, which might lead to a more severe standard error inflation than simply moving them to $\infty$ or $x^{(n)}$.
In all the following results, one may choose the specific $\mathcal C_L$ and $\mathcal C_R$ as the collection of potential target locations, where the simplest choices are $\mathcal C_L = \{x_{(n)}, \infty\}$ and $\mathcal C_R = \{x_{(1)}, -\infty\}$. With this notation, we state our general theorem for testing threshold breakdown points below. For the sake of simplicity we use the notation
$$\hat\theta \pm z_{1 - \frac{\alpha}{2}} \cdot \hatse(\hat\theta)$$
to denote the interval
$$
\bigl[\hat\theta - z_{1 - \frac{\alpha}{2}} \cdot \hatse(\hat\theta),\
  \hat\theta + z_{1 - \frac{\alpha}{2}} \cdot \hatse(\hat\theta)\bigr].
$$
\begin{theorem}
\label{thm:meta}
{%

Consider
\[
\fitdisplay{\phi(x^{(n)})=\II\{0\notin[L(x^{(n)}),U(x^{(n)})]\},}
\]
where, writing $z=z_{1-\alpha/2}$,
\[
\fitdisplay{L(x^{(n)})=\hat\theta(x^{(n)})-z\hatse(\hat\theta(x^{(n)})), \qquad U(x^{(n)})=\hat\theta(x^{(n)})+z\hatse(\hat\theta(x^{(n)})).}
\]
The standard error is nonnegative, and the interval is defined on the
admissible data sets. All minima below are over integers $0\le m\le n$,
with $\min\varnothing=\infty$.
The families $\mathsf C^R_{m,\mathcal C_R}$ and
$\mathsf C^L_{m,\mathcal C_L}$ are subsets of
$B_{\textsf H}(x^{(n)},m)$.

\textbf{Rejection breakdown.}
If $L(x^{(n)})>0$, set
\begin{gather*}
\fitdisplay{m_{\mathrm{rej},>0}^{\mathrm{up}} =\min\left\{m: \exists\,\tilde x^{(n)}\in\mathsf C^R_{m,\mathcal C_R}, {L(\tilde x^{(n)})\le0\le U(\tilde x^{(n)})}\right\},}\\
\fitdisplay{m_{\mathrm{rej},>0}^{\mathrm{low}} =\min\left\{m: \eta_{m/n-}(\hat\theta,x^{(n)}) +z\eta_{m/n+}(\hatse(\hat\theta),x^{(n)}) \ge L(x^{(n)})\right\}.}
\end{gather*}
Then
\[
\fitdisplay{\frac{m_{\mathrm{rej},>0}^{\mathrm{low}}}{n} \le\BP_{\mathrm{reject}}(\phi,x^{(n)}) \le\frac{m_{\mathrm{rej},>0}^{\mathrm{up}}}{n}.}
\]
If $U(x^{(n)})<0$, set
\begin{gather*}
\fitdisplay{m_{\mathrm{rej},<0}^{\mathrm{up}} =\min\left\{m: \exists\,\tilde x^{(n)}\in\mathsf C^L_{m,\mathcal C_L}, {L(\tilde x^{(n)})\le0\le U(\tilde x^{(n)})}\right\},}\\
\fitdisplay{m_{\mathrm{rej},<0}^{\mathrm{low}} =\min\left\{m: \eta_{m/n+}(\hat\theta,x^{(n)}) +z\eta_{m/n+}(\hatse(\hat\theta),x^{(n)}) \ge-U(x^{(n)})\right\}.}
\end{gather*}
Then
\[
\fitdisplay{\frac{m_{\mathrm{rej},<0}^{\mathrm{low}}}{n} \le\BP_{\mathrm{reject}}(\phi,x^{(n)}) \le\frac{m_{\mathrm{rej},<0}^{\mathrm{up}}}{n}.}
\]

\textbf{Acceptance breakdown.}
If $L(x^{(n)})\le0\le U(x^{(n)})$, set
\begin{gather*}
\fitdisplay{m_{\mathrm{acc},-}^{\mathrm{up}} =\min\{m:\exists\,\tilde x^{(n)}\in\mathsf C^R_{m,\mathcal C_R}, \ U(\tilde x^{(n)})<0\},}\\
\fitdisplay{m_{\mathrm{acc},+}^{\mathrm{up}} =\min\{m:\exists\,\tilde x^{(n)}\in\mathsf C^L_{m,\mathcal C_L}, \ L(\tilde x^{(n)})>0\},}\\
\fitdisplay{m_{\mathrm{acc},-}^{\mathrm{low}} =\min\left\{m: \eta_{m/n-}(\hat\theta,x^{(n)}) +z\eta_{m/n-}(\hatse(\hat\theta),x^{(n)}) \ge U(x^{(n)})\right\},}\\
\fitdisplay{m_{\mathrm{acc},+}^{\mathrm{low}} =\min\left\{m: \eta_{m/n+}(\hat\theta,x^{(n)}) +z\eta_{m/n-}(\hatse(\hat\theta),x^{(n)}) \ge-L(x^{(n)})\right\}.}
\end{gather*}
Then
\[
\fitdisplay{\frac{\min\{m_{\mathrm{acc},-}^{\mathrm{low}},m_{\mathrm{acc},+}^{\mathrm{low}}\}}{n} \le\BP_{\mathrm{accept}}(\phi,x^{(n)}) \le \frac{\min\{m_{\mathrm{acc},-}^{\mathrm{up}},m_{\mathrm{acc},+}^{\mathrm{up}}\}}{n}.}
\]

}
\end{theorem}
\begin{remark}
\label{rem:test}
  Theorem \ref{thm:meta} also covers null-restricted tests of the form
    $$
      \phi(x^{(n)}) = \II\Bigl\{0 \notin \hat \theta \pm z_{1-\frac{\alpha}{2}} \cdot \sigma_0\Bigr\},
    $$
    where $\sigma_0$ is a fixed known standard deviation determined under the null. In this case, the test does not depend on the data through the standard error, so
    $$
      \eta_{m/n\pm}(\sigma_0, x^{(n)}) = 0
    $$
    for all $m$ and $x^{(n)}$, and the lower bounds in Theorem~\ref{thm:meta} simplify. {Assume that, for each $m$, the values of $\hat\theta$ attainable by replacing at most $m$ observations form an interval. Then the rejection breakdown point is}
    $$
\fitdisplay{\BP_{\text{reject}}(\phi,x^{(n)}) = \frac{1}{n} \min \left \{ m:\ \eta_{m/n\, {-\sgn(\hat\theta)}}(\hat\theta, x^{(n)}) \ge \sgn(\hat \theta) \cdot \hat \theta - z_{1 - \frac{\alpha}{2}} \cdot \sigma_0 \right \},}
$$
    {The acceptance identity does not need this interval assumption:}
    $$
\fitdisplay{{\BP_{\text{accept}}(\phi,x^{(n)}) =\frac1n\min\left\{m: \eta_{m/n-}(\hat\theta,x^{(n)}) >\hat\theta+z_{1-\frac\alpha2}\sigma_0 \quad\text{or}\quad \eta_{m/n+}(\hat\theta,x^{(n)}) >-\hat\theta+z_{1-\frac\alpha2}\sigma_0 \right\}.}}
$$
    For plug-in null-restricted standard errors $\hatse(\theta_0)$, the sensitivities $\eta_{m/n\pm}(\hatse(\theta_0),x^{(n)})$ are generally non-zero and must be treated as in Theorem~\ref{thm:meta}.
\end{remark}
\begin{proof}[Proof of Theorem~\ref{thm:meta}]
{Write $x=x^{(n)}$ and $z=z_{1-\alpha/2}$. If $L(x)>0$, any contaminated sample $\tilde x\in B_{\textsf H}(x,m)$ that accepts must satisfy $L(\tilde x)\le0\le U(\tilde x)$. Hence
\[
\fitdisplay{L(x)\le L(x)-L(\tilde x)\le\eta_{m/n-}(\hat\theta,x)+z\eta_{m/n+}(\hatse(\hat\theta),x).}
\]
This gives the rejection lower bound. A candidate counted by $m_{\mathrm{rej},>0}^{\mathrm{up}}$ satisfies both endpoint inequalities and belongs to the same replacement ball, so it gives the upper bound.
If $U(x)<0$, acceptance instead requires
\[
\fitdisplay{-U(x)\le U(\tilde x)-U(x)\le\eta_{m/n+}(\hat\theta,x)+z\eta_{m/n+}(\hatse(\hat\theta),x),}
\]
and the same candidate argument proves the other rejection bounds.

Now suppose $L(x)\le0\le U(x)$. Rejection to the left requires $U(\tilde x)<0$, while rejection to the right requires $L(\tilde x)>0$. Respectively, these imply
\begin{gather*}
\fitdisplay{U(x)<U(x)-U(\tilde x)\le\eta_{m/n-}(\hat\theta,x)+z\eta_{m/n-}(\hatse(\hat\theta),x),}\\
\fitdisplay{-L(x)<L(\tilde x)-L(x)\le\eta_{m/n+}(\hat\theta,x)+z\eta_{m/n-}(\hatse(\hat\theta),x).}
\end{gather*}
The first replacement count at which either corresponding weak inequality can hold is a lower bound. Each candidate in the acceptance upper bounds already rejects. Taking the smaller count over the two directions proves the acceptance bounds.
}
\end{proof}
\subsection{Wald-type Test}
We first specialize Theorem~\ref{thm:meta} to the usual two–sided Wald test. Recall that
$\hat\theta$ solves the estimating equation \eqref{M-est} and $\hatse(\hat\theta)$ is the
plug–in standard error. The test
$$
  \phi(x^{(n)})
  \;=\;
  \II\Bigl\{0 \notin \hat\theta \pm z_{1-\frac{\alpha}{2}} \cdot \hatse(\hat\theta)\Bigr\}
$$
rejects $H_0:\theta=0$ whenever the $(1-\alpha)$ Wald interval excludes~$0$.
\begin{corollary}
\label{cor:wald}
Assume that $\psi(t-\theta)$ in \eqref{M-est} is differentiable a.e., bounded, non-increasing in
$\theta$, and passes through $0$. Let $\hat\theta$ be the $M$-estimator defined by \eqref{M-est}, and let
$\hatse(\hat\theta)$ be the plug–in estimate of its asymptotic standard error. Define the lower
and upper endpoints
$$
  L(x^{(n)}) := \hat\theta - z_{1-\frac{\alpha}{2}}\hatse(\hat\theta),
  \qquad
  U(x^{(n)}) := \hat\theta + z_{1-\frac{\alpha}{2}}\hatse(\hat\theta).
$$
\medskip
\noindent\emph{Rejection breakdown.}
Suppose $\phi(x^{(n)})=1$, so $0\notin [L(x^{(n)},U(x^{(n)})]$.
\smallskip
\noindent(i) If $L(x^{(n)})>0$ (equivalently $\hat\theta>0$), define
\begin{gather*}
\fitdisplay{m_{\mathrm{rej},>0}^{\mathrm{up}}(x^{(n)}) := \min\Bigl\{m\ge0:\ {L(\mathsf C^R_{m,-\infty}(x^{(n)}))\le0\le U(\mathsf C^R_{m,-\infty}(x^{(n)}))} \Bigr\},}\\
\fitdisplay{m_{\mathrm{rej},>0}^{\mathrm{low}}(x^{(n)}) := \min\Bigl\{m\ge0:\ \eta_{m/n-}(\hat\theta,x^{(n)}) + z_{1-\frac{\alpha}{2}}\,\eta_{m/n+}\bigl(\hatse(\hat\theta),x^{(n)}\bigr) \;\ge\; L(x^{(n)}) \Bigr\},}
\end{gather*}
where we remind that the contamination scheme is
$$
  \mathsf{C}^R_{m, -\infty}(x^{(n)})
  \;=\;
  \{x_{(i)}\}_{i=1}^{\,n-m} \,\cup\, \{-\infty\}_{i=1}^{\,m}.
$$
Then
$$
  \frac{1}{n}\,m_{\mathrm{rej},>0}^{\mathrm{low}}(x^{(n)})
  \;\le\;
  \BP_{\mathrm{reject}}(\phi,x^{(n)})
  \;\le\;
  \frac{1}{n}\,m_{\mathrm{rej},>0}^{\mathrm{up}}(x^{(n)}).
$$
\smallskip
\noindent(ii) If $U(x^{(n)})<0$ (equivalently $\hat\theta<0$), define
\begin{gather*}
\fitdisplay{m_{\mathrm{rej},<0}^{\mathrm{up}}(x^{(n)}) := \min\Bigl\{m\ge0:\ {L(\mathsf C^L_{m,\infty}(x^{(n)}))\le0\le U(\mathsf C^L_{m,\infty}(x^{(n)}))} \Bigr\},}\\
\fitdisplay{m_{\mathrm{rej},<0}^{\mathrm{low}}(x^{(n)}) := \min\Bigl\{m\ge0:\ \eta_{m/n+}(\hat\theta,x^{(n)}) + z_{1-\frac{\alpha}{2}}\,\eta_{m/n+}\bigl(\hatse(\hat\theta),x^{(n)}\bigr) \;\ge\; -U(x^{(n)}) \Bigr\},}
\end{gather*}
where now
$$
  \mathsf{C}^L_{m, \infty}(x^{(n)})
  \;=\;
  \{x_{(i)}\}_{i=m+1}^{\,n} \,\cup\, \{\infty\}_{i=1}^{\,m}.
$$
Then
$$
  \frac{1}{n}\,m_{\mathrm{rej},<0}^{\mathrm{low}}(x^{(n)})
  \;\le\;
  \BP_{\mathrm{reject}}(\phi,x^{(n)})
  \;\le\;
  \frac{1}{n}\,m_{\mathrm{rej},<0}^{\mathrm{up}}(x^{(n)}).
$$
\medskip
\noindent\emph{Acceptance breakdown.}
Suppose now $\phi(x^{(n)})=0$, so $L(x^{(n)})\le0\le U(x^{(n)})$. Remind $\mathcal C^R_{m, \mathcal C_R}$ and $\mathsf C^L_{m, \mathcal C_L}$ are the collections of contamination schemes defined in Section \ref{sec:con}. Define
\begin{align*}
  m_{\mathrm{acc},-}^{\mathrm{up}}(x^{(n)})
  &:= \min\Bigl\{m\ge0:\ \tilde x^{(n)} \in \mathsf C^R_{m, \mathcal C_R},
        \hat \theta(\tilde x^{(n)}) + z_{1-\frac{\alpha}{2}} {\hatse\bigl(\hat\theta(\tilde x^{(n)})\bigr)} \mathrel{{<}} 0
      \Bigr\},\\
  m_{\mathrm{acc},+}^{\mathrm{up}}(x^{(n)})
  &:= \min\Bigl\{m\ge0:\ \tilde x^{(n)} \in \mathsf C^L_{m, \mathcal C_L},
        \hat \theta(\tilde x^{(n)}) - z_{1-\frac{\alpha}{2}} \hatse\bigl(\hat \theta(\tilde x^{(n)})\bigr) \mathrel{{>}} 0
      \Bigr\}.
\end{align*}
Furthermore, define
\begin{align*}
  m_{\mathrm{acc},-}^{\mathrm{low}}(x^{(n)})
  &:= \min\Bigl\{m\ge0:\
        \eta_{m/n-}(\hat\theta,x^{(n)})
        + z_{1-\frac{\alpha}{2}}\,\eta_{m/n-}\bigl(\hatse(\hat\theta),x^{(n)}\bigr)
        \;\ge\; \hat\theta + z_{1-\frac{\alpha}{2}}\hatse(\hat\theta)
      \Bigr\},\\
  m_{\mathrm{acc},+}^{\mathrm{low}}(x^{(n)})
  &:= \min\Bigl\{m\ge0:\
        \eta_{m/n+}(\hat\theta,x^{(n)})
        + z_{1-\frac{\alpha}{2}}\,\eta_{m/n-}\bigl(\hatse(\hat\theta),x^{(n)}\bigr)
        \;\ge\; -\hat\theta + z_{1-\frac{\alpha}{2}}\hatse(\hat\theta)
      \Bigr\}.
\end{align*}
Then
$$
  \frac{1}{n}\,\min\bigl\{
       m_{\mathrm{acc},-}^{\mathrm{low}}(x^{(n)}),
       m_{\mathrm{acc},+}^{\mathrm{low}}(x^{(n)})
     \bigr\}
  \;\le\;
  \BP_{\mathrm{accept}}(\phi,x^{(n)})
  \;\le\;
  \frac{1}{n}\,\min\bigl\{
       m_{\mathrm{acc},-}^{\mathrm{up}}(x^{(n)}),
       m_{\mathrm{acc},+}^{\mathrm{up}}(x^{(n)})
     \bigr\}.
$$
In all of the above, the one–sided sensitivities $\eta_{m/n\pm}(\hat\theta,x^{(n)})$ are given by
Corollary~\ref{cor:opt_attack_m_est}, and any computable upper bounds on
$\eta_{m/n\pm}(\hatse(\hat\theta),x^{(n)})$ (for example those from
Lemma~\ref{lem:opt_attack_m_est_se}) can be used inside the definitions of
$m_{\mathrm{rej},\cdot}^{\mathrm{low}}(x^{(n)})$ and $m_{\mathrm{acc},\cdot}^{\mathrm{low}}(x^{(n)})$ for lower bounds of them. For the
restricted Wald test with variance evaluated under the null, we instead use the bounds from
Lemma~\ref{lem:se_at_theta_0} for $\eta_{m/n\pm}(\hatse(\theta_0),x^{(n)})$.
\end{corollary}
\begin{proof}
For rejection breakdown with $L(x^{(n)})>0$, any contaminated sample
$\tilde x \in B_{\textsf H}(x^{(n)},m)$ yields a lower endpoint
$$
  L(\tilde x^{(n)}) = \tilde\theta - z_{1-\frac{\alpha}{2}}\hatse(\tilde\theta).
$$
{If the candidate satisfies $L(\tilde x^{(n)})\le0\le U(\tilde x^{(n)})$, the decision flips from rejection to acceptance.} The quantity $m_{\mathrm{rej},>0}^{\mathrm{up}}(x^{(n)})$ is defined using a specific
contamination scheme, hence it provides a valid, computable upper
bound on the true rejection breakdown point. On the other hand, Theorem~\ref{thm:meta} implies that
for any $\tilde x\in B_{\textsf H}(x^{(n)},m)$,
$$
  L(x^{(n)}) - L(\tilde x^{(n)})
  \;\le\;
  \eta_{m/n-}(\hat\theta,x^{(n)})
  + z_{1-\frac{\alpha}{2}}\,\eta_{m/n+}\bigl(\hatse(\hat\theta),x^{(n)}\bigr).
$$
Thus, as long as
$$
  \eta_{m/n-}(\hat\theta,x^{(n)})
  + z_{1-\frac{\alpha}{2}}\,\eta_{m/n+}\bigl(\hatse(\hat\theta),x^{(n)}\bigr)
  \;<\; L(x^{(n)}),
$$
no $m$-point contamination can force $L(\tilde x^{(n)})\le 0$, and the decision cannot change. This yields the
lower bound $m_{\mathrm{rej},>0}^{\mathrm{low}}(x^{(n)})$ in the case $L(x^{(n)})>0$. The case $U(x^{(n)})<0$ is
handled symmetrically by working with the endpoint $U(x^{(n)})$ instead of $L(x^{(n)})$, exactly as in
Theorem~\ref{thm:meta}.
For acceptance breakdown, $\phi(x^{(n)})=0$ implies $L(x^{(n)})\le 0 \le U(x^{(n)})$. To flip to rejection, one must either make $U(\tilde x^{(n)})<0$ (interval entirely negative) or make $L(\tilde x^{(n)})>0$
(interval entirely positive). The definitions of $m_{\mathrm{acc},-}^{\mathrm{up}}(x^{(n)})$ and
$m_{\mathrm{acc},+}^{\mathrm{up}}(x^{(n)})$ correspond to two explicit strategies: replacing $m$
extreme points by $c_{\text{lower}}$ or $c_{\text{upper}}$, respectively. These clearly provide
valid upper bounds. The lower bounds follow again from Theorem~\ref{thm:meta} by applying the
one–sided sensitivities to $L(x^{(n)})$ and $U(x^{(n)})$ and using the same type of argument as above.
\end{proof}
\subsection{Two-sample Test}
\label{sec:two_sample}
We consider testing equality of location parameters in two populations based on robust $M$-estimation. Let $n_x$ and $n_y$ denote the sample sizes of $x^{(n_x)}$ and $y^{(n_y)}$.
Given independent samples $\{x_1,\dots,x_{n_x}\}$ and $\{y_1,\dots,y_{n_y}\}$, we test
$$
  H_0:\ \theta_{x^{(n_x)}} = \theta_{y^{(n_y)}}
  \quad\text{vs.}\quad
  H_1:\ \theta_{x^{(n_x)}} \neq \theta_{y^{(n_y)}}.
$$
Let $\hat\theta_{x^{(n_x)}}$ and $\hat\theta_{y^{(n_y)}}$ be the $M$-estimators of location in the two samples, defined as
solutions to
$$
  \sum_{i=1}^{n_x} \psi(x^{(n_x)}_i-\theta) = 0,
  \qquad
  \sum_{j=1}^{n_y} \psi(y^{(n_y)}_j-\theta) = 0,
$$
respectively. Under standard regularity conditions, $\hat\theta_{x^{(n_x)}} - \hat\theta_{y^{(n_y)}}$ is asymptotically
normal with variance
$$
  \hatse(\hat\theta_{x^{(n_x)}},\hat\theta_{y^{(n_y)}})^2
  := \frac{\hatse_{x^{(n_x)}}^2}{n_x} + \frac{\hatse_{y^{(n_y)}}^2}{n_y},
$$
where the marginal asymptotic variances can be estimated by the usual sandwich formula
$$
  \hatse_{x^{(n_x)}}^2
  := \frac{ \frac{1}{n_x} \sum_{i=1}^{n_x} \psi(x^{(n_x)}_i - \hat{\theta}_{x^{(n_x)}})^2 }
            { \Bigl( \frac{1}{n_x} \sum_{i=1}^{n_x} \psi'({x^{(n_x)}_i} - \hat{\theta}_{x^{(n_x)}}) \Bigr)^2 },
  \qquad
  \hatse_{y^{(n_y)}}^2
  := \frac{ \frac{1}{n_y} \sum_{j=1}^{n_y} \psi({y^{(n_y)}_j} - \hat{\theta}_{y^{(n_y)}})^2 }
            { \Bigl( \frac{1}{n_y} \sum_{j=1}^{n_y} \psi'({y^{(n_y)}_j} - \hat{\theta}_{y^{(n_y)}}) \Bigr)^2 }.
$$
The test statistic is
$$
  T_{n_x,n_y}
  := \frac{\hat{\theta}_{x^{(n_x)}} - \hat{\theta}_{y^{(n_y)}}}
           {\sqrt{\hatse_{x^{(n_x)}}^2/n_x + \hatse_{y^{(n_y)}}^2/n_y}},
$$
where $\psi$ only needs to be differentiable almost everywhere; for example, for Huber's loss we take
$\psi_\delta'(x) = \II\{|x|\le\delta\}$. We reject $H_0$ at level $\alpha$ whenever
$|T_{n_x,n_y}| > z_{1-\alpha/2}$. Equivalently, the test rejects if
\begin{align}
\label{eq:test_two_sample}
  0 \notin (\hat{\theta}_{x^{(n_x)}} - \hat{\theta}_{y^{(n_y)}})
  \pm z_{1-\frac{\alpha}{2}}
  \cdot \sqrt{
    \frac{ \sum_{i=1}^{n_x} \psi(x^{(n_x)}_i - \hat{\theta}_{x^{(n_x)}})^2 }
         { \Bigl( \sum_{i=1}^{n_x} \psi'(x^{(n_x)}_i - \hat{\theta}_{x^{(n_x)}}) \Bigr)^2 }
    +
    \frac{ \sum_{j=1}^{n_y} \psi({y^{(n_y)}_j} - \hat{\theta}_{y^{(n_y)}})^2 }
         { \Bigl( \sum_{j=1}^{n_y} \psi'({y^{(n_y)}_j} - \hat{\theta}_{y^{(n_y)}}) \Bigr)^2 }
  }.
\end{align}
{Set $n_{\min}=\min\{n_x,n_y\}$ and write $x=x^{(n_x)}$, $y=y^{(n_y)}$.
For the test in \eqref{eq:test_two_sample} and an integer $0\le m\le n_x+n_y$, define
\begin{equation}\label{eq:two_sample_replacement_ball}
\fitdisplay[\dimexpr\linewidth-3em\relax]{B_{\textsf H}((x,y),m) :=\{(\tilde x,\tilde y): d_{\textsf H}(x,\tilde x)+d_{\textsf H}(y,\tilde y)\le m\}.}
\end{equation}
Equivalently,
\[
\fitdisplay{B_{\textsf H}((x,y),m) =\bigcup_{\substack{0\le k_1\le n_x,\ 0\le k_2\le n_y\\ k_1+k_2\le m}} B_{\textsf H}(x,k_1)\times B_{\textsf H}(y,k_2).}
\]
Here $d_{\textsf H}$ counts differing observations. The allocation indices are integers. Define
\[
\fitdisplay{\BP(\phi,x,y)=\frac1{n_{\min}}\min\left\{m: \exists\,(\tilde x,\tilde y)\in B_{\textsf H}((x,y),m), |\phi(\tilde x,\tilde y)-\phi(x,y)|=1\right\}.}
\]
Restrict to a $1\to0$ or $0\to1$ flip to obtain
$\BP_{\mathrm{reject}}$ or $\BP_{\mathrm{accept}}$.
Use $\min\varnothing=\infty$.

For any scalar statistic $T(x,y)$, define
\begin{gather*}
\fitdisplay{\eta_{(k_1,k_2)+}(T,(x,y)) =\sup_{\substack{\tilde x\in B_{\textsf H}(x,k_1)\\ \tilde y\in B_{\textsf H}(y,k_2)}} \{T(\tilde x,\tilde y)-T(x,y)\},}\\
\fitdisplay{\eta_{(k_1,k_2)-}(T,(x,y)) =\sup_{\substack{\tilde x\in B_{\textsf H}(x,k_1)\\ \tilde y\in B_{\textsf H}(y,k_2)}} \{T(x,y)-T(\tilde x,\tilde y)\},}\\
\fitdisplay{\eta_{m/n_{\min}\pm}(T,(x,y)) =\max_{\substack{0\le k_1\le n_x,\ 0\le k_2\le n_y\\k_1+k_2\le m}} \eta_{(k_1,k_2)\pm}(T,(x,y)).}
\end{gather*}
Apply these definitions to the contrast
$T=\hat\theta_x-\hat\theta_y$ and to the combined standard error.
}
{The existing normalization $n_{\min}$ is retained. Thus $m$ counts replacements
across both samples, while $m/n_{\min}$ is measured relative to the smaller
sample size. It is not the fraction of the pooled sample and need not lie
in $[0,1]$ for an arbitrary test.}
\begin{corollary}
\label{cor:two_sample}
{For integers $0\le k_1\le n_x$ and $0\le k_2\le n_y$,} the one-sided sensitivities of the contrast $\hat\theta_{x^{(n_x)}} - \hat\theta_{y^{(n_y)}}$ satisfy
\begin{align*}
  \eta_{(k_1,k_2)+}(\hat\theta_{x^{(n_x)}} - \hat\theta_{y^{(n_y)}},(x^{(n_x)},y^{(n_y)}))
  &= \eta_{k_1/n_x+}(\hat\theta_{x^{(n_x)}},x^{(n_x)})
     + \eta_{k_2/n_y-}(\hat\theta_{y^{(n_y)}},y^{(n_y)}),\\
  \eta_{(k_1,k_2)-}(\hat\theta_{x^{(n_x)}} - \hat\theta_{y^{(n_y)}},(x^{(n_x)},y^{(n_y)}))
  &= \eta_{k_1/n_x-}(\hat\theta_{x^{(n_x)}},x^{(n_x)})
     + \eta_{k_2/n_y+}(\hat\theta_{y^{(n_y)}},y^{(n_y)}),
\end{align*}
where $\eta_{k_1/n_x\pm}(\hat\theta_{x^{(n_x)}},x^{(n_x)})$ and $\eta_{k_2/n_y\pm}(\hat\theta_{y^{(n_y)}},y^{(n_y)})$ are the
one-sample sensitivities.
For the plug-in standard error
$$
  \se_{x^{(n_x)},y^{(n_y)}}
  := \hatse(\hat\theta_{x^{(n_x)}},\hat\theta_{y^{(n_y)}})
  = \sqrt{\frac{\hatse_{x^{(n_x)}}^2}{n_x} + \frac{\hatse_{y^{(n_y)}}^2}{n_y}},
$$
we have
\begin{align*}
  &\eta_{(k_1, k_2)+}(\se_{x^{(n_x)},y^{(n_y)}},(x^{(n_x)},y^{(n_y)}))\\
  = & \sqrt{\bigl(\eta_{k_1/n_x+}(\hatse(\hat\theta_{x^{(n_x)}}),x^{(n_x)})
                 + \hatse(\hat\theta_{x^{(n_x)}})\bigr)^2
          + \bigl(\eta_{k_2/n_y+}(\hatse(\hat\theta_{y^{(n_y)}}),y^{(n_y)})
                 + \hatse(\hat\theta_{y^{(n_y)}})\bigr)^2}
     \;-\; \se_{x^{(n_x)},y^{(n_y)}},\\
  &\eta_{(k_1, k_2)-}(\se_{x^{(n_x)},y^{(n_y)}},(x^{(n_x)},y^{(n_y)}))\\
  =& \se_{x^{(n_x)},y^{(n_y)}}
     \;-\;
     \sqrt{\bigl(-\eta_{k_1/n_x-}(\hatse(\hat\theta_{x^{(n_x)}}),x^{(n_x)})
                  + \hatse(\hat\theta_{x^{(n_x)}})\bigr)^2
          + \bigl(-\eta_{k_2/n_y-}(\hatse(\hat\theta_{y^{(n_y)}}),y^{(n_y)})
                  + \hatse(\hat\theta_{y^{(n_y)}})\bigr)^2},
\end{align*}
{For a total budget $m$, maximize each fixed-allocation sensitivity above over $k_1+k_2\le m$, with the stated capacity limits. The one-sample levels are $k_1/n_x$ and $k_2/n_y$.}
\end{corollary}
\begin{proof}
For integers $k_1,k_2\ge 0$, write
$$
\Delta(\tilde x^{(n_x)},\tilde y^{(n_y)})
:= (\hat\theta_{\tilde x^{(n_x)}} - \hat\theta_{\tilde y^{(n_y)}})
   - (\hat\theta_{x^{(n_x)}} - \hat\theta_{y^{(n_y)}})
= (\hat\theta_{\tilde x^{(n_x)}} - \hat\theta_{x^{(n_x)}})
  - (\hat\theta_{\tilde y^{(n_y)}} - \hat\theta_{y^{(n_y)}}),
$$
for $\tilde x^{(n_x)}\in B_\textsf{H}(x^{(n_x)},k_1)$ and $\tilde y^{(n_y)}\in B_\textsf{H}(y^{(n_y)},k_2)$.
By definition of the one-sided two-sample $m$-sensitivity,
\begin{align*}
\eta_{(k_1,k_2)+}(\hat\theta_{x^{(n_x)}} - \hat\theta_{y^{(n_y)}},(x^{(n_x)},y^{(n_y)}))
&= \sup_{\substack{\tilde x^{(n_x)}\in B_\textsf{H}(x^{(n_x)},k_1)\\
                    \tilde y^{(n_y)}\in B_\textsf{H}(y^{(n_y)},k_2)}}
   \Delta(\tilde x^{(n_x)},\tilde y^{(n_y)})\\
&= \sup_{\tilde x^{(n_x)}\in B_\textsf{H}(x^{(n_x)},k_1)}
     (\hat\theta_{\tilde x^{(n_x)}} - \hat\theta_{x^{(n_x)}})
   + \sup_{\tilde y^{(n_y)}\in B_\textsf{H}(y^{(n_y)},k_2)}
     (\hat\theta_{y^{(n_y)}} - \hat\theta_{\tilde y^{(n_y)}})\\
&= \eta_{k_1/n_x+}(\hat\theta_{x^{(n_x)}},x^{(n_x)}) + \eta_{k_2/n_y-}(\hat\theta_{y^{(n_y)}},y^{(n_y)}),
\end{align*}
where we used the one-sample formulas in Corollary~\ref{cor:opt_attack_m_est} for
$\hat\theta_{x^{(n_x)}}$ and $\hat\theta_{y^{(n_y)}}$.
Similarly,
\begin{align*}
\eta_{(k_1,k_2)-}(\hat\theta_{x^{(n_x)}} - \hat\theta_{y^{(n_y)}},(x^{(n_x)},y^{(n_y)}))
&= \sup_{\substack{\tilde x^{(n_x)}\in B_\textsf{H}(x^{(n_x)},k_1)\\
                    \tilde y^{(n_y)}\in B_\textsf{H}(y^{(n_y)},k_2)}}
   \bigl[(\hat\theta_{x^{(n_x)}} - \hat\theta_{y^{(n_y)}})
         - (\hat\theta_{\tilde x^{(n_x)}} - \hat\theta_{\tilde y^{(n_y)}})\bigr]\\
&= \sup_{\tilde x^{(n_x)}\in B_\textsf{H}(x^{(n_x)},k_1)}
     (\hat\theta_{x^{(n_x)}} - \hat\theta_{\tilde x^{(n_x)}})
   + \sup_{\tilde y^{(n_y)}\in B_\textsf{H}(y^{(n_y)},k_2)}
     (\hat\theta_{\tilde y^{(n_y)}} - \hat\theta_{y^{(n_y)}})\\
&= \eta_{k_1/n_x-}(\hat\theta_{x^{(n_x)}},x^{(n_x)}) + \eta_{k_2/n_y+}(\hat\theta_{y^{(n_y)}},y^{(n_y)}).
\end{align*}
For the standard error, recall that
$$
\hatse(\hat\theta_{x^{(n_x)}},\hat\theta_{y^{(n_y)}})
= \sqrt{\hatse(\hat\theta_{x^{(n_x)}})^2 + \hatse(\hat\theta_{y^{(n_y)}})^2},
$$
where $\hatse(\hat\theta_{x^{(n_x)}})$ and $\hatse(\hat\theta_{y^{(n_y)}})$ are the
(one-sample) standard error estimates for $\hat\theta_{x^{(n_x)}}$ and $\hat\theta_{y^{(n_y)}}$.
Under an contamination $(\tilde x^{(n_x)},\tilde y^{(n_y)})$, set
$$
\hatse_{x^{(n_x)}}' := \hatse(\hat\theta_{\tilde x^{(n_x)}}),
\qquad
\hatse_{y^{(n_y)}}' := \hatse(\hat\theta_{\tilde y^{(n_y)}}).
$$
By defintion of $m$-sensitivity,
\begin{gather*}
\fitdisplay{\hatse_{x^{(n_x)}}' \in \bigl[\hatse(\hat\theta_{x^{(n_x)}}) - {\eta_{k_1/n_x-}}(\hatse(\hat\theta_{x^{(n_x)}}),x^{(n_x)}),\, \hatse(\hat\theta_{x^{(n_x)}}) + {\eta_{k_1/n_x+}}(\hatse(\hat\theta_{x^{(n_x)}}),x^{(n_x)})\bigr],}\\
\fitdisplay{\hatse_{y^{(n_y)}}' \in \bigl[\hatse(\hat\theta_{y^{(n_y)}}) - {\eta_{k_2/n_y-}}(\hatse(\hat\theta_{y^{(n_y)}}),y^{(n_y)}),\, \hatse(\hat\theta_{y^{(n_y)}}) + {\eta_{k_2/n_y+}}(\hatse(\hat\theta_{y^{(n_y)}}),y^{(n_y)})\bigr].}
\end{gather*}
Thus
\begin{align*}
&\hatse(\hat\theta_{\tilde x^{(n_x)}},\hat\theta_{\tilde y^{(n_y)}})
= \sqrt{(\hatse_{x^{(n_x)}}')^2 + (\hatse_{y^{(n_y)}}')^2}\\
\le&
\sqrt{\bigl(\hatse(\hat\theta_{x^{(n_x)}}) + \eta_{k_1/n_x+}(\hatse(\hat\theta_{x^{(n_x)}}),x^{(n_x)})\bigr)^2
      +\bigl(\hatse(\hat\theta_{y^{(n_y)}}) + \eta_{k_2/n_y+}(\hatse(\hat\theta_{y^{(n_y)}}),y^{(n_y)})\bigr)^2},
\end{align*}
and hence
\begin{align*}
&\eta_{(k_1,k_2)+}(\hatse(\hat\theta_{x^{(n_x)}},\hat\theta_{y^{(n_y)}}),(x^{(n_x)},y^{(n_y)}))
= \sup_{\substack{\tilde x^{(n_x)}\in B_\textsf{H}(x^{(n_x)},k_1)\\
                    \tilde y^{(n_y)}\in B_\textsf{H}(y^{(n_y)},k_2)}}
   \bigl[\hatse(\hat\theta_{\tilde x^{(n_x)}},\hat\theta_{\tilde y^{(n_y)}})
        - \hatse(\hat\theta_{x^{(n_x)}},\hat\theta_{y^{(n_y)}})\bigr]\\
=&
\sqrt{\bigl(\hatse(\hat\theta_{x^{(n_x)}}) + \eta_{k_1/n_x+}(\hatse(\hat\theta_{x^{(n_x)}}),x^{(n_x)})\bigr)^2
      +\bigl(\hatse(\hat\theta_{y^{(n_y)}}) + \eta_{k_2/n_y+}(\hatse(\hat\theta_{y^{(n_y)}}),y^{(n_y)})\bigr)^2}
- \hatse(\hat\theta_{x^{(n_x)}},\hat\theta_{y^{(n_y)}}),
\end{align*}
as claimed. The “$-$” case is analogous: the smallest possible value of
$\hatse(\hat\theta_{\tilde x^{(n_x)}},\hat\theta_{\tilde y^{(n_y)}})$ is obtained by taking
$\hatse_{x^{(n_x)}}'$ and $\hatse_{y^{(n_y)}}'$ as small as allowed, which yields the
formula for $\eta_{(k_1,k_2)-}$ in the statement.
\end{proof}
We now specialize Theorem~\ref{thm:meta} to the two-sample Wald-type test
\eqref{eq:test_two_sample}. For brevity, write
$$
  \hat\Delta := \hat\theta_{x^{(n_x)}} - \hat\theta_{y^{(n_y)}},
  \qquad
  \se_{x^{(n_x)},y^{(n_y)}} := \hatse(\hat\theta_{x^{(n_x)}},\hat\theta_{y^{(n_y)}}),
$$
and
$$
  L(x^{(n_x)},y^{(n_y)}) := \hat\Delta - z_{1-\frac{\alpha}{2}}\,\se_{x^{(n_x)},y^{(n_y)}},
  \qquad
  U(x^{(n_x)},y^{(n_y)}) := \hat\Delta + z_{1-\frac{\alpha}{2}}\,\se_{x^{(n_x)},y^{(n_y)}}.
$$
Then $\phi(x^{(n_x)},y^{(n_y)})=1$ iff $0\notin [L(x^{(n_x)},y^{(n_y)}),U(x^{(n_x)},y^{(n_y)})]$.
\begin{corollary}
\label{cor:two_sample_test}
{%

Assume that $\psi(t-\theta)$ in \eqref{M-est} is differentiable, bounded,
non-increasing in $\theta$, and passes through zero.
Write $x=x^{(n_x)}$, $y=y^{(n_y)}$, $z=z_{1-\alpha/2}$,
$\hat\Delta=\hat\theta_x-\hat\theta_y$, and
$\se_{x,y}=\hatse(\hat\theta_x,\hat\theta_y)$.
Let
\[
\fitdisplay{L(x,y)=\hat\Delta-z\se_{x,y},\qquad U(x,y)=\hat\Delta+z\se_{x,y},\qquad \phi(x,y)=\II\{0\notin[L(x,y),U(x,y)]\}.}
\]
{Use the replacement ball in \eqref{eq:two_sample_replacement_ball} and the associated breakdown and sensitivity definitions. All minima are over integers $0\le m\le n_x+n_y$.
For each $m$, all allocation indices satisfy
$0\le k_1\le n_x$, $0\le k_2\le n_y$, $k_1+k_2\le m$.
Define the candidate pairs}
\begin{gather*}
\fitdisplay{(\tilde x^-_{k_1},\tilde y^+_{k_2}) =(\mathsf C^R_{k_1,-\infty}(x),\mathsf C^L_{k_2,\infty}(y)),}\\
\fitdisplay{(\tilde x^+_{k_1},\tilde y^-_{k_2}) =(\mathsf C^L_{k_1,\infty}(x),\mathsf C^R_{k_2,-\infty}(y)).}
\end{gather*}
The interval must be defined at any candidate used below.

\emph{Rejection breakdown.}
If $L(x,y)>0$, set
\begin{gather*}
\fitdisplay{m_{\mathrm{rej},>0}^{\mathrm{up}} =\min\{m:{\exists\,k_1,k_2,\ L(\tilde x^-_{k_1},\tilde y^+_{k_2})\le0 \le U(\tilde x^-_{k_1},\tilde y^+_{k_2})}\},}\\
\fitdisplay{m_{\mathrm{rej},>0}^{\mathrm{low}} =\min\{m:\eta_{m/n_{\min}-}(\hat\Delta,(x,y)) +z\eta_{m/n_{\min}+}(\se_{x,y},(x,y))\ge L(x,y)\}.}
\end{gather*}
Then
\[
\fitdisplay{\frac{m_{\mathrm{rej},>0}^{\mathrm{low}}}{n_{\min}} \le\BP_{\mathrm{reject}}(\phi,x,y) \le\frac{m_{\mathrm{rej},>0}^{\mathrm{up}}}{n_{\min}}.}
\]
If $U(x,y)<0$, set
\begin{gather*}
\fitdisplay{m_{\mathrm{rej},<0}^{\mathrm{up}} =\min\{m:{\exists\,k_1,k_2,\ L(\tilde x^+_{k_1},\tilde y^-_{k_2})\le0 \le U(\tilde x^+_{k_1},\tilde y^-_{k_2})}\},}\\
\fitdisplay{m_{\mathrm{rej},<0}^{\mathrm{low}} =\min\{m:\eta_{m/n_{\min}+}(\hat\Delta,(x,y)) +z\eta_{m/n_{\min}+}(\se_{x,y},(x,y))\ge-U(x,y)\}.}
\end{gather*}
Then
\[
\fitdisplay{\frac{m_{\mathrm{rej},<0}^{\mathrm{low}}}{n_{\min}} \le\BP_{\mathrm{reject}}(\phi,x,y) \le\frac{m_{\mathrm{rej},<0}^{\mathrm{up}}}{n_{\min}}.}
\]

\emph{Acceptance breakdown.}
If $L(x,y)\le0\le U(x,y)$, set
\begin{gather*}
\fitdisplay{m_{\mathrm{acc},-}^{\mathrm{up}} =\min\left\{m: {\exists\,k_1,k_2,\ \tilde x\in\mathsf C^R_{k_1,\mathcal C_R}(x),\ \tilde y\in\mathsf C^L_{k_2,\mathcal C_L}(y)}, U(\tilde x,\tilde y)<0\right\},}\\
\fitdisplay{m_{\mathrm{acc},+}^{\mathrm{up}} =\min\left\{m: {\exists\,k_1,k_2,\ \tilde x\in\mathsf C^L_{k_1,\mathcal C_L}(x),\ \tilde y\in\mathsf C^R_{k_2,\mathcal C_R}(y)}, L(\tilde x,\tilde y)>0\right\},}\\
\fitdisplay{m_{\mathrm{acc},-}^{\mathrm{low}} =\min\{m:\eta_{m/n_{\min}-}(\hat\Delta,(x,y)) +z\eta_{m/n_{\min}-}(\se_{x,y},(x,y))\ge U(x,y)\},}\\
\fitdisplay{m_{\mathrm{acc},+}^{\mathrm{low}} =\min\{m:\eta_{m/n_{\min}+}(\hat\Delta,(x,y)) +z\eta_{m/n_{\min}-}(\se_{x,y},(x,y))\ge-L(x,y)\}.}
\end{gather*}
Then
\[
\fitdisplay{\frac{\min\{m_{\mathrm{acc},-}^{\mathrm{low}},m_{\mathrm{acc},+}^{\mathrm{low}}\}}{n_{\min}} \le\BP_{\mathrm{accept}}(\phi,x,y) \le\frac{\min\{m_{\mathrm{acc},-}^{\mathrm{up}},m_{\mathrm{acc},+}^{\mathrm{up}}\}}{n_{\min}}.}
\]
{Use the one-sample levels $k_1/n_x$ and $k_2/n_y$ in
Corollary~\ref{cor:opt_attack_m_est} and Lemma~\ref{lem:opt_attack_m_est_se},
then maximize over the allowed allocations as in
Corollary~\ref{cor:two_sample}.}

}
\end{corollary}

\begin{proof}
{Each candidate changes at most $k_1$ observations in $x$ and $k_2$ in $y$, with $k_1+k_2\le m$. It therefore belongs to \eqref{eq:two_sample_replacement_ball}.
For any pair $(\tilde x,\tilde y)$ in that ball, the combined sensitivities give
\begin{gather*}
\fitdisplay{L(x,y)-L(\tilde x,\tilde y)\le\eta_{m/n_{\min}-}(\hat\Delta,(x,y))+z\eta_{m/n_{\min}+}(\se_{x,y},(x,y)),}\\
\fitdisplay{U(\tilde x,\tilde y)-U(x,y)\le\eta_{m/n_{\min}+}(\hat\Delta,(x,y))+z\eta_{m/n_{\min}+}(\se_{x,y},(x,y)),}\\
\fitdisplay{U(x,y)-U(\tilde x,\tilde y)\le\eta_{m/n_{\min}-}(\hat\Delta,(x,y))+z\eta_{m/n_{\min}-}(\se_{x,y},(x,y)),}\\
\fitdisplay{L(\tilde x,\tilde y)-L(x,y)\le\eta_{m/n_{\min}+}(\hat\Delta,(x,y))+z\eta_{m/n_{\min}-}(\se_{x,y},(x,y)).}
\end{gather*}
Acceptance from an initially positive interval requires the first change to be at least $L(x,y)$; from a negative interval, the second must be at least $-U(x,y)$. These give the rejection lower bounds. The rejection upper-bound candidates satisfy $L(\tilde x,\tilde y)\le0\le U(\tilde x,\tilde y)$ and thus accept.
From initial acceptance, rejection requires either $U(\tilde x,\tilde y)<0$ or $L(\tilde x,\tilde y)>0$. The last two inequalities give the acceptance lower bounds, while the stated candidate families supply the upper bounds. Dividing every replacement count by $n_{\min}$ proves the result.
}
\end{proof}
\subsection{Score-type Test}
\label{sec:score}
The score-type test is based on the statistic
$$
  V_n(x^{(n)}) := \frac{1}{\sqrt{n}} \sum_{i=1}^n \psi(x_i),
$$
see, e.g., \citet{he:simpson:portnoy1990}. We reject $H_0: V_n = 0$ whenever $0$ lies outside a $(1-\alpha)$ Wald-type
interval:
$$
  0 \notin V_n \pm z_{1-\frac{\alpha}{2}} \cdot \hatse(V_n).
$$
In practice, one typically uses either the fully empirical plug-in variance
$$
  0 \notin
  \frac{1}{\sqrt{n}} \sum_{i=1}^n \psi(x_i)
  \;\pm\;
  z_{1-\frac{\alpha}{2}} \cdot
  \sqrt{\frac{1}{n} \sum_{i=1}^n \psi(x_i)^2},
$$
or a “restricted” version using the model-based variance
$$
  0 \notin
  \frac{1}{\sqrt{n}} \sum_{i=1}^n \psi(x_i)
  \;\pm\;
  z_{1-\frac{\alpha}{2}} \cdot \sqrt{\EE_{\theta_0}[\psi(X)^2]}.
$$
We next characterize sensitivities and threshold breakdown points of the score and its plug-in
standard error.
\begin{lemma}
\label{lem:v_n}
Assume that $\psi(t-\theta)$ in \eqref{M-est} is differentiable, bounded, non-increasing in
$\theta$, and passes through 0. Let $x_{(1)} \le \dots \le x_{(n)}$ be the order statistics of $x^{(n)}$, and define
$$
  V_n(x^{(n)}) := \frac{1}{\sqrt{n}} \sum_{i=1}^n \psi(x_i).
$$
Then the one-sided threshold breakdown points of $V_n$ are
\begin{align*}
  \BP_{\eta+}(V_n,x^{(n)})
  &= \frac{1}{n}
     \min\Bigl\{
       m:\
       \frac{1}{\sqrt{n}}
       \Bigl(\sum_{i>m} \psi(x_{(i)}) + m\,\psi(\infty)\Bigr)
       \mathrel{{\ge}} V_n + \eta
     \Bigr\},\\
  \BP_{\eta-}(V_n,x^{(n)})
  &= \frac{1}{n}
     \min\Bigl\{
       m:\
       \frac{1}{\sqrt{n}}
       \Bigl(\sum_{i\le n-m} \psi(x_{(i)}) + m\,\psi(-\infty)\Bigr)
       \mathrel{{\le}} V_n - \eta
     \Bigr\}.
\end{align*}
The corresponding one-sided sensitivities are
\begin{align*}
  \eta_{m/n+}(V_n,x^{(n)})
  &= \frac{1}{\sqrt{n}}
     \Bigl(\sum_{i>m} \psi(x_{(i)}) + m\,\psi(\infty)\Bigr)
     - V_n,\\
  \eta_{m/n-}(V_n,x^{(n)})
  &= V_n
     - \frac{1}{\sqrt{n}}
       \Bigl(\sum_{i\le n-m} \psi(x_{(i)}) + m\,\psi(-\infty)\Bigr).
\end{align*}
\end{lemma}
\begin{proof}
Let $x_{(1)}\le\cdots\le x_{(n)}$ denote the order statistics of $x$ and recall
$$
V_n(x^{(n)}) := \frac{1}{\sqrt{n}} \sum_{i=1}^n \psi(x_i).
$$
Because $\psi(t-\theta)$ is non-increasing in $\theta$ and bounded, the largest
value of $V_n$ attainable by at most $m$ replacements is obtained by sending
the $m$ smallest observations to $+\infty$, which yields
$$
\sup_{\tilde x^{(n)} \in B_\textsf{H}(x^{(n)},m)} V_n(\tilde x^{(n)})
= \frac{1}{\sqrt{n}}
  \Bigl(\sum_{i>m} \psi(x_{(i)}) + m\,\psi(\infty)\Bigr).
$$
Therefore the one-sided $m$-sensitivity in the “$+$” direction is
\begin{align*}
\eta_{m/n+}(V_n,x^{(n)})
&= \sup_{\tilde x^{(n)} \in B_\textsf{H}(x^{(n)},m)}
     \bigl\{V_n(\tilde x^{(n)}) - V_n(x^{(n)})\bigr\}\\
&=
\frac{1}{\sqrt{n}}
\Bigl(\sum_{i>m} \psi(x_{(i)}) + m\,\psi(\infty)\Bigr)
- V_n(x^{(n)}),
\end{align*}
and the threshold breakdown point $\BP_{\eta+}(V_n,x^{(n)})$ is the smallest
$m$ for which $V_n(x^{(n)})+\eta$ becomes attainable, giving the formula in the
lemma.
For the “$-$” direction, to \emph{decrease} $V_n$ one sends the $m$
largest observations to $-\infty$, leading to
$$
\inf_{\tilde x^{(n)} \in B_\textsf{H}(x^{(n)},m)} V_n(\tilde x^{(n)})
= \frac{1}{\sqrt{n}}
  \Bigl(\sum_{i\le n-m} \psi(x_{(i)}) + m\,\psi(-\infty)\Bigr).
$$
Thus
\begin{align*}
\eta_{m/n-}(V_n,x^{(n)})
&= V_n(x^{(n)})
   - \inf_{\tilde x^{(n)} \in B_\textsf{H}(x^{(n)},m)} V_n(\tilde x^{(n)})\\
&= V_n(x^{(n)})
   - \frac{1}{\sqrt{n}}
     \Bigl(\sum_{i\le n-m} \psi(x_{(i)}) + m\,\psi(-\infty)\Bigr),
\end{align*}
and $\BP_{\eta-}(V_n,x^{(n)})$ is again the smallest $m$ for which $V_n(x^{(n)})-\eta$
is attainable. This matches the expressions in the statement.
\end{proof}
\begin{lemma}
\label{lem:s_n}
Assume that $\psi(t-\theta)$ in \eqref{M-est} is differentiable, bounded, non-increasing in
$\theta$, and passes through 0. Let
$$
  S_n := \sqrt{\frac{1}{n}\sum_{i=1}^n \psi(x_i)^2}
$$
denote the empirical plug-in standard deviation of $V_n$, and set
$$
  \psi_{\max} := \max\{|\psi(\infty)|,\,|\psi(-\infty)|\}.
$$
Then the one-sided threshold breakdown points of $S_n$ satisfy
\begin{align*}
  \BP_{\eta+}(S_n,x^{(n)})
  &= \frac{1}{n}
     \min\Bigl\{
       m:\
       \max_{\substack{I \subseteq [n]\\|I|=n-m}}
       \sqrt{\frac{1}{n}\Bigl(
         \sum_{i\in I} \psi(x_i)^2 + m\,\psi_{\max}^2
       \Bigr)}
       \mathrel{{\ge}} S_n + \eta
     \Bigr\},\\
  \BP_{\eta-}(S_n,x^{(n)})
  &= \frac{1}{n}
     \min\Bigl\{
       m:\
       \min_{\substack{I \subseteq [n]\\|I|=n-m}}
       \sqrt{\frac{1}{n}\sum_{i\in I} \psi(x_i)^2}
       \mathrel{{\le}} S_n - \eta
     \Bigr\}.
\end{align*}
If, in addition, $\psi$ is odd and $\psi'(x)$ is non-increasing in $|x|$, let $\pi$ be a permutation such that $|x_{\pi_1}| \le |x_{\pi_2}| \le \dots \le |x_{\pi_n}|$.
Then
\begin{align}
\label{eq:bp+se2}
  \BP_{\eta+}(S_n,x^{(n)})
  = \frac{1}{n}
    \min\Bigl\{
      m:\
      m \mathrel{{\ge}}
      \frac{n\,(S_n+\eta)^2 - \sum_{i>m} \psi(x_{\pi_i})^2}{\psi_{\max}^2}
    \Bigr\},
\end{align}
and the corresponding one-sided $m$-sensitivity is
\begin{align*}
  \eta_{m/n+}(S_n,x^{(n)})
  &= \sqrt{
       \frac{1}{n}
       \Bigl(
         m\,\psi_{\max}^2
         + \sum_{i>m} \psi(x_{\pi_i})^2
       \Bigr)
     }
     - S_n.
\end{align*}
Moreover,
\begin{align*}
  \BP_{\eta-}(S_n,x^{(n)})
  &= \frac{1}{n}
     \min\Bigl\{
       m:\
       \sqrt{\frac{1}{n}\sum_{i\le n-m} \psi(x_{\pi_i})^2}
       \mathrel{{\le}} S_n - \eta
     \Bigr\},\\
  \eta_{m/n-}(S_n,x^{(n)})
  &= S_n
     - \sqrt{\frac{1}{n}\sum_{i\le n-m} \psi(x_{\pi_i})^2}.
\end{align*}
\end{lemma}
\begin{proof}
Recall
$$
S_n(x^{(n)}) := \sqrt{\frac{1}{n}\sum_{i=1}^n \psi(x_i)^2},
\qquad
\psi_{\max} := \max\{|\psi(\infty)|,|\psi(-\infty)|\}.
$$
\medskip
\noindent\emph{General form.}
Fix $m$. To \emph{increase} $S_n$ under at most $m$ replacements, one
keeps some subset $I\subseteq[n]$ with $|I|=n-m$ and replaces the remaining
$m$ points. Since $|\psi(\cdot)|\le\psi_{\max}$ by boundedness,
$$
\sup_{\tilde x^{(n)}\in B_\textsf{H}(x^{(n)},m)} S_n(\tilde x^{(n)})
\le
\max_{\substack{I\subseteq[n]\\ |I|=n-m}}
\sqrt{\frac{1}{n}
      \Bigl(\sum_{i\in I} \psi(x_i)^2 + m\,\psi_{\max}^2\Bigr)}.
$$
This bound is attainable by sending the $m$ contaminated points to values where
$|\psi(\cdot)|=\psi_{\max}$, so
$$
\sup_{\tilde x^{(n)}\in B_\textsf{H}(x^{(n)},m)} S_n(\tilde x^{(n)})
=
\max_{\substack{I\subseteq[n]\\ |I|=n-m}}
\sqrt{\frac{1}{n}
      \Bigl(\sum_{i\in I} \psi(x_i)^2 + m\,\psi_{\max}^2\Bigr)},
$$
which gives the expression for $\BP_{\eta+}(S_n,x^{(n)})$ in the lemma. The
corresponding $m$-sensitivity $\eta_{m/n+}(S_n,x^{(n)})$ is simply the difference between
this supremum and $S_n(x^{(n)})$.
To \emph{decrease} $S_n$, one again keeps some $I\subseteq[n]$ with
$|I|=n-m$ and moves the remaining $m$ observations towards points where
$\psi(\cdot)$ is zero. This yields
$$
\inf_{\tilde x^{(n)}\in B_\textsf{H}(x^{(n)},m)} S_n(\tilde x^{(n)})
=
\min_{\substack{I\subseteq[n]\\ |I|=n-m}}
\sqrt{\frac{1}{n}\sum_{i\in I} \psi(x_i)^2},
$$
and hence
\begin{align*}
\eta_{m/n-}(S_n,x^{(n)})
&= S_n(x^{(n)})
   - \inf_{\tilde x^{(n)}\in B_\textsf{H}(x^{(n)},m)} S_n(\tilde x^{(n)})\\
&= S_n(x^{(n)})
   - \min_{\substack{I\subseteq[n]\\ |I|=n-m}}
     \sqrt{\frac{1}{n}\sum_{i\in I} \psi(x_i)^2},
\end{align*}
which gives the formula for $\BP_{\eta-}(S_n,x^{(n)})$.
\medskip
\noindent\emph{Closed form under additional assumptions.}
Assume now that $\psi$ is odd and $\psi'(x)$ is non-increasing in $|x|$. Then
$|\psi(t)|$ is non-decreasing in $|t|$. Let $\pi$ be a permutation such that $|x_{\pi_1}| \le |x_{\pi_2}| \le \dots \le |x_{\pi_n}|$
For the “$+$” direction, to maximize $S_n(\tilde x^{(n)})$ one keeps the
$n-m$ points with largest $|\psi(\cdot)|$, that is, the indices
$\{m+1,\dots,n\}$ in this reordered list, and sends the remaining $m$ to points
achieving $|\psi|=\psi_{\max}$. Thus
$$
\sup_{\tilde x^{(n)}\in B_\textsf{H}(x^{(n)},m)} S_n(\tilde x^{(n)})
=
\sqrt{\frac{1}{n}
      \Bigl(m\,\psi_{\max}^2
            + \sum_{i>m} \psi(x_{\pi_i})^2\Bigr)},
$$
and hence
$$
\eta_{m/n+}(S_n,x^{(n)})
= \sqrt{\frac{1}{n}
        \Bigl(m\,\psi_{\max}^2
              + \sum_{i>m} \psi(x_{\pi_i})^2\Bigr)}
  - S_n(x^{(n)}).
$$
The corresponding threshold breakdown point $\BP_{\eta+}(S_n,x^{(n)})$ is the
{smallest $m$ for which the maximal contaminated standard deviation is at least $S_n(x^{(n)})+\eta$}, that is,
$$
m \mathrel{{\ge}} \frac{n\,(S_n+\eta)^2 - \sum_{i>m} \psi(x_{\pi_i})^2}{\psi_{\max}^2},
$$
which is exactly \eqref{eq:bp+se2}.
For the “$-$” direction, to minimize $S_n(\tilde x^{(n)})$ one keeps the
$n-m$ points with smallest $|\psi(\cdot)|$, that is, indices
$\{1,\dots,n-m\}$ in the reordered list, and moves the remaining $m$ towards
values where $\psi(\cdot)\approx 0$. This gives
$$
\inf_{\tilde x^{(n)} \in B_\textsf{H}(x^{(n)},m)} S_n(\tilde x^{(n)})
=
\sqrt{\frac{1}{n}\sum_{i\le n-m} \psi(x_{\pi_i})^2},
$$
and therefore
$$
\eta_{m/n-}(S_n,x^{(n)})
= S_n(x^{(n)})
  - \sqrt{\frac{1}{n}\sum_{i\le n-m} \psi(x_{\pi_i})^2},
$$
as claimed.
\end{proof}
We now specialize Theorem~\ref{thm:meta} to the score-type test. Set
$$
  V_n(x^{(n)}) := \frac{1}{\sqrt{n}}\sum_{i=1}^n \psi(x_i),
  \qquad
  S_n(x^{(n)}) := \sqrt{\frac{1}{n}\sum_{i=1}^n \psi(x_i)^2},
  \qquad
  S := \sqrt{\EE_{\theta_0}[\psi(X)^2]}.
$$
The empirical-score test rejects when $0\notin V_n \pm z_{1-\frac{\alpha}{2}} S_n$, while the
restricted-score test rejects when $0\notin V_n \pm z_{1-\frac{\alpha}{2}} {S}$.
\begin{corollary}
Assume that $\psi(t-\theta)$ in \eqref{M-est} is odd, differentiable, bounded, and non-increasing in
$\theta$. Consider the two-sided $M$-score test
$$
  \phi(x^{(n)}) = \II\Bigl\{0 \notin V_n \pm z_{1-\frac{\alpha}{2}} S_n\Bigr\}.
$$
Define the lower and upper endpoints of the empirical-score interval as
$$
  L(x^{(n)}) := V_n - z_{1-\frac{\alpha}{2}} S_n,
  \qquad
  U(x^{(n)}) := V_n + z_{1-\frac{\alpha}{2}} S_n.
$$
\medskip
\noindent\emph{Rejection breakdown.} Suppose $\phi(x^{(n)})=1$.
\smallskip
\noindent(i) If $L(x^{(n)})>0$ (i.e.\ $V_n>0$), define
\begin{align*}
  m_{\mathrm{rej},>0}^{\mathrm{up}}(x^{(n)})
  &:= \min\Bigl\{
        m\ge0:\
        V_n(\mathsf{C}^R_{m, -\infty}(x^{(n)})) - z_{1-\frac{\alpha}{2}} S_n(\mathsf{C}^R_{m, -\infty}(x^{(n)})) \le 0
      \Bigr\},\\
  m_{\mathrm{rej},>0}^{\mathrm{low}}(x^{(n)})
  &:= \min\Bigl\{
        m\ge0:\
        \eta_{m/n-}(V_n,x^{(n)})
        + z_{1-\frac{\alpha}{2}}\,\eta_{m/n+}(S_n,x^{(n)})
        \;\ge\; L(x^{(n)})
      \Bigr\},
\end{align*}
where
$$
  \mathsf{C}^R_{m, -\infty}(x^{(n)}) = \{x_{(i)}\}_{i=1}^{n-m} \cup \{-\infty\}_{i=1}^{m}.
$$
Then
$$
  \frac{1}{n}\,m_{\mathrm{rej},>0}^{\mathrm{low}}(x^{(n)})
  \;\le\;
  \BP_{\mathrm{reject}}(\phi,x^{(n)})
  \;\le\;
  \frac{1}{n}\,m_{\mathrm{rej},>0}^{\mathrm{up}}(x^{(n)}).
$$
\smallskip
\noindent(ii) If $U(x^{(n)})<0$ (i.e.\ $V_n<0$), define
\begin{align*}
  m_{\mathrm{rej},<0}^{\mathrm{up}}(x^{(n)})
  &:= \min\Bigl\{
        m\ge0:\
        V_n(\mathsf{C}^L_{m, \infty}(x^{(n)})) + z_{1-\frac{\alpha}{2}}  S_n(\mathsf{C}^L_{m, \infty}(x^{(n)})) \ge 0
      \Bigr\},\\
  m_{\mathrm{rej},<0}^{\mathrm{low}}(x^{(n)})
  &:= \min\Bigl\{
        m\ge0:\
        \eta_{m/n+}(V_n,x^{(n)})
        + z_{1-\frac{\alpha}{2}}\,\eta_{m/n+}(S_n,x^{(n)})
        \;\ge\; -U(x^{(n)})
      \Bigr\},
\end{align*}
where
$$
  \mathsf{C}^L_{m, \infty}(x^{(n)}) = \{x_{(i)}\}_{i=m+1}^{n} \cup \{\infty\}_{i=1}^{m}.
$$
Then
$$
  \frac{1}{n}\,m_{\mathrm{rej},<0}^{\mathrm{low}}(x^{(n)})
  \;\le\;
  \BP_{\mathrm{reject}}(\phi,x^{(n)})
  \;\le\;
  \frac{1}{n}\,m_{\mathrm{rej},<0}^{\mathrm{up}}(x^{(n)}).
$$
In both cases, $\eta_{m/n\pm}(V_n,x^{(n)})$ are given by Lemma~\ref{lem:v_n}, while upper bounds for
$\eta_{m/n\pm}(S_n,x^{(n)})$ come from Lemma~\ref{lem:s_n}. For the restricted-score test, one simply
replaces $S_n$ by the fixed $S$, so $\eta_{m/n\pm}(S,x^{(n)})=0$.

\medskip
\noindent\emph{Acceptance breakdown.} Suppose now $\phi(x^{(n)})=0$, so that $L(x^{(n)})\le 0 \le U(x^{(n)})$.
Remind $\mathcal C^R_{m, \mathcal C_R}$ and $\mathsf C^L_{m, \mathcal C_L}$ are the collections of contamination schemes defined in Section \ref{sec:con}.
Define
\begin{align*}
  m_{\mathrm{acc},-}^{\mathrm{up}}(x^{(n)})
  &:= \min\Bigl\{
        m\ge0:\ \tilde x^{(n)} \in \mathcal C^R_{m, \mathcal C_R},
        V_n(\tilde x^{(n)}) + z_{1-\frac{\alpha}{2}} S_n (\tilde x^{(n)}) \mathrel{{<}} 0,
      \Bigr\},\\
  m_{\mathrm{acc},+}^{\mathrm{up}}(x^{(n)})
  &:= \min\Bigl\{
        m\ge0:\ \tilde x^{(n)} \in \mathsf C^L_{m, \mathcal C_L},
        V_n(\tilde x^{(n)}) - z_{1-\frac{\alpha}{2}} S_n (\tilde x^{(n)}) \mathrel{{>}} 0,
      \Bigr\},
\end{align*}
and
\begin{align*}
  m_{\mathrm{acc},-}^{\mathrm{low}}(x^{(n)})
  &:= \min\Bigl\{
        m\ge0:\
        \eta_{m/n-}(V_n,x^{(n)})
        + z_{1-\frac{\alpha}{2}}\,\eta_{m/n-}(S_n,x^{(n)})
        \;\ge\; V_n + z_{1-\frac{\alpha}{2}} S_n
      \Bigr\},\\
  m_{\mathrm{acc},+}^{\mathrm{low}}(x^{(n)})
  &:= \min\Bigl\{
        m\ge0:\
        \eta_{m/n+}(V_n,x^{(n)})
        + z_{1-\frac{\alpha}{2}}\,\eta_{m/n-}(S_n,x^{(n)})
        \;\ge\; -V_n + z_{1-\frac{\alpha}{2}} S_n
      \Bigr\}.
\end{align*}
Then
\begin{align*}
  \frac{1}{n}\,
  \min\Bigl\{
    m_{\mathrm{acc},-}^{\mathrm{low}}(x^{(n)}),\,
    m_{\mathrm{acc},+}^{\mathrm{low}}(x^{(n)})
  \Bigr\}
  \;\le\;
  \BP_{\mathrm{accept}}(\phi,x^{(n)})
  \;\le\;
  \frac{1}{n}\,
  \min\Bigl\{
    m_{\mathrm{acc},-}^{\mathrm{up}}(x^{(n)}),\,
    m_{\mathrm{acc},+}^{\mathrm{up}}(x^{(n)})
  \Bigr\}.
\end{align*}
Again, for the restricted-score test we set $S_n=S$ and drop the $m$-sensitivity terms
$\eta_{m/n\pm}(S_n,x^{(n)})$.
\end{corollary}
\begin{proof}
The lower bounds and the acceptance upper bounds follow as in Corollary~\ref{cor:wald}.
{The rejection candidates need only satisfy the displayed one-endpoint
        condition. Since $\psi$ is continuous, moving the selected observations
        continuously to their candidate values gives a continuous score center
        within the same replacement ball. An accepting sample is reached along
        this path, even if the final candidate itself overshoots: the center
        passes through zero. Boundedness and the limits of $\psi$ extend this
        argument to the allowed extreme endpoint.
        }
\end{proof}
\subsection{Two-stage M-estimate Test}
For a two-stage M-estimator, suppose $(\hat\theta,\hat \sigma)$ solves the joint estimating equations in \eqref{M-est-two}. {Assume that $F$ is symmetric about zero and that $\psi$ is odd.} Under standard regularity conditions, if $\hat\sigma\to S{>0}$ and $\hat\theta\to0${, with $\EE_F[\psi'(X/S)]>0$}, then
\begin{align*}
    \sqrt{n}\,\hat\theta
    \;\rightsquigarrow\;
    N\!\left(0,\,
      S^2\,\frac{\EE\big[\psi(X/S)^2\big]}
                       {\EE\big[\psi'(X/S)\big]^2}
    \right).
\end{align*}
A natural plug-in standard error for $\hat\theta$ is
\begin{align*}
    \hat \sigma \cdot
       \sqrt{
         \frac{\sum_{i=1}^n \psi\big((x_i-\hat\theta)/\hat \sigma \big)^2}
              {\Bigl(\sum_{i=1}^n \psi'\big((x_i-\hat\theta)/\hat \sigma \big)\Bigr)^2}
       }.
\end{align*}
The two-sided Wald test then rejects whenever
\begin{align*}
    0 \notin
    \hat{\theta} \pm
    z_{1-\frac{\alpha}{2}}\,
    \hat \sigma \cdot
    \sqrt{
      \frac{\sum_{i=1}^n \psi\big((x_i-\hat\theta)/\hat \sigma \big)^2}
           {\Bigl(\sum_{i=1}^n \psi'\big((x_i-\hat\theta)/\hat \sigma \big)\Bigr)^2}
    } := \hat{\theta} \pm
    z_{1-\frac{\alpha}{2}} \cdot
    \hat \sigma \cdot \hatse(\hat \theta).
\end{align*}
We next give (admittedly loose) upper bounds on the $m$-sensitivity of the ratio part $\hatse(\hat \theta)$.
\begin{lemma}
\label{lem:two_stage_se_sensitivity}
Assume that $\psi(t-\theta)$ in \eqref{M-est-two} is differentiable, bounded,
non-increasing in $\theta$, and passes through 0. Let $\hat\theta$ be the solution to
\eqref{M-est-two} based on $x^{(n)}=(x_1,\dots,x_n)$, and let $\hatse$ be the
associated scale estimate. Further assume that $\psi$ is odd and that
$\psi'(x)$ is non-increasing in $|x|$.
Define
$$
\Delta(t)
:= \sup_{z\in\RR} \bigl|\psi'(z+t)-\psi'(z)\bigr|,
$$
and
\begin{align*}
    S_{m+}
    &:= \max_{\substack{I\subseteq[n]\\ |I|=n-m}}
        \max_{\tilde\theta \in
              [\hat\theta - \eta_{m/n-}(\hat\theta,x^{(n)}),\,
               \hat\theta + \eta_{m/n+}(\hat\theta,x^{(n)})]}
        \sum_{i\in I} \II\bigl\{|x_i-\hat\theta|<|x_i-\tilde\theta|\bigr\},\\
    S_{m-}
    &:= \max_{\substack{I\subseteq[n]\\ |I|=n-m}}
        \max_{\tilde\theta \in
              [\hat\theta - \eta_{m/n-}(\hat\theta,x^{(n)}),\,
               \hat\theta + \eta_{m/n+}(\hat\theta,x^{(n)})]}
        \sum_{i\in I} \II\bigl\{|x_i-\hat\theta|>|x_i-\tilde\theta|\bigr\}.
\end{align*}
Let $\pi$ be a permutation such that {$|x_{\pi_1}-\hat\theta|\le|x_{\pi_2}-\hat\theta|\le\cdots\le|x_{\pi_n}-\hat\theta|$}.
Define the shift term
\begin{align*}
    C_{\text{shift}}(m,x^{(n)},\hat\theta,\hatse)
    &:=
    2\psi_{\max}\psi'(0)\Biggl[
          \frac{\eta_{m/n}(\hatse,x^{(n)})}
               {(\hatse-\eta_{m/n-}(\hatse,x^{(n)}))\,\hatse}
          \Bigl(\sum_{i=1}^n |x_i| + n|\hat\theta|\Bigr)
          +
          n\,\frac{\eta_{m/n}(\hat\theta,x^{(n)})}
                 {\hatse-\eta_{m/n-}(\hatse,x^{(n)})}
        \Biggr],
\end{align*}
where $\psi_{\max} := \max\{|\psi(+\infty)|,|\psi(-\infty)|\}$.
Assume $\hatse - \eta_{m/n-}(\hatse,x^{(n)})>0$.
Then the one-sided sensitivities of the plug-in standard error
$\hatse(\hat\theta)$ satisfy
\begin{align*}
    &\eta_{m/n+}\bigl(\hatse(\hat\theta),x^{(n)}\bigr) \\
    \le &
    \sqrt{
      \frac{
        5m\psi_{\max}^2
        +
        \sum_{i>m} \psi\bigl((x_{\pi_i} - \hat \theta)/\hatse\bigr)^2
        +
        C_{\text{shift}}(m,x^{(n)},\hat\theta,\hatse)
      }
      {
        \Bigl(
          \max\Bigl\{
            \sum_{i>m}
              \psi'\Bigl(\tfrac{x_{\pi_i} - \hat \theta}{\hatse-\eta_{m/n-}(\hatse,x^{(n)})}\Bigr)
            - S_{m+}\,
              \Delta\Bigl(\tfrac{\eta_{m/n}(\hat\theta,x^{(n)})}
                               {\hatse-\eta_{m/n-}(\hatse,x^{(n)})}\Bigr),
            \,0
          \Bigr\}
        \Bigr)^2
      }
    }
    - \hatse(\hat\theta),
\end{align*}
and
\begin{align*}
    &\eta_{m/n-}\bigl(\hatse(\hat\theta),x^{(n)}\bigr) \\
    \le&
    \hatse(\hat\theta)
    -
    \sqrt{\frac{\max\Biggl\{-4m\psi_{\max}^2 +
          \sum_{i\le n-m} \psi\bigl((x_{\pi_i} - \hat \theta)/\hatse\bigr)^2
          -
          C_{\text{shift}}(m,x^{(n)},\hat\theta,\hatse),
          \,0
        \Biggr \}
        }
      {
        \Bigl(
          m\psi'(0)
          +
          \sum_{i\le n-m}
            \psi'\Bigl(\tfrac{x_{\pi_i} - \hat \theta}{\hatse+\eta_{m/n+}(\hatse,x^{(n)})}\Bigr)
          +
          S_{m-}\,
          \Delta\Bigl(\tfrac{\eta_{m/n}(\hat\theta,x^{(n)})}
                           {\hatse-\eta_{m/n-}(\hatse,x^{(n)})}\Bigr)
        \Bigr)^2
      }
    }.
\end{align*}
{Then
\begin{align*}
S_{m+}=\max\Bigl\{&
    \II\{\eta_{m/n+}(\hat\theta,x^{(n)})>0\}
    \sum_{i=1}^{n-m}\II\{\hat\theta+\eta_{m/n+}(\hat\theta,x^{(n)})>2x_{(i)}-\hat\theta\}, \\
  & \II\{\eta_{m/n-}(\hat\theta,x^{(n)})>0\}
    \sum_{i=m+1}^{n}\II\{\hat\theta-\eta_{m/n-}(\hat\theta,x^{(n)})<2x_{(i)}-\hat\theta\}\Bigr\}, 
\end{align*}
and
\begin{align*}
S_{m-}=\min\Bigl\{ n-m,\ \max\bigl\{
   & \II\{\eta_{m/n+}(\hat\theta,x^{(n)})>0\}\sum_{i=1}^n\II\{x_i>\hat\theta\}, \\
  & \II\{\eta_{m/n-}(\hat\theta,x^{(n)})>0\}\sum_{i=1}^n\II\{x_i<\hat\theta\}\bigr\}\Bigr\}.
\end{align*}
}
As a special case, for Huber's loss with parameter $\delta$ and
$\psi'(u)=\II\{|u|\le\delta\}$, define
\begin{align*}
    q_{m+}
    &:= \min_{\tilde\theta\in
              [\hat\theta-\eta_{m/n-}(\hat\theta,x^{(n)}),\,
               \hat\theta+\eta_{m/n+}(\hat\theta,x^{(n)})]}
        \sum_{i=1}^n
          \Biggl(
            \II\Bigl\{
              \tfrac{x_i-\tilde\theta}{\hatse-\eta_{m/n-}(\hatse,x^{(n)})}
              \in [-\delta,\delta]
            \Bigr\}
            -
            \II\Bigl\{
              \tfrac{x_i-\hat\theta}{\hatse+\eta_{m/n+}(\hatse,x^{(n)})}
              \in [-\delta,\delta]
            \Bigr\}
          \Biggr),\\
    q_{m-}
    &:= \max_{\tilde\theta\in
              [\hat\theta-\eta_{m/n-}(\hat\theta,x^{(n)}),\,
               \hat\theta+\eta_{m/n+}(\hat\theta,x^{(n)})]}
        \sum_{i=1}^n
          \Biggl(
            \II\Bigl\{
              \tfrac{x_i-\tilde\theta}{\hatse+\eta_{m/n+}(\hatse,x^{(n)})}
              \in [-\delta,\delta]
            \Bigr\}
            -
            \II\Bigl\{
              \tfrac{x_i-\hat\theta}{\hatse-\eta_{m/n-}(\hatse,x^{(n)})}
              \in [-\delta,\delta]
            \Bigr\}
          \Biggr).
\end{align*}
Then the bounds simplify to
\begin{equation*}
\fitdisplay{\eta_{m/n+}\bigl(\hatse(\hat\theta),x^{(n)}\bigr) \le \sqrt{ \frac{ 5m\psi_{\max}^2 + \sum_{i>m} \psi\bigl((x_{\pi_i} - \hat \theta)/\hatse\bigr)^2 + C_{\text{shift}}(m,x^{(n)},\hat\theta,\hatse) } { \Bigl( \max\Bigl\{ \sum_{i>m} \psi'\Bigl(\tfrac{x_{\pi_i} - \hat \theta}{\hatse-\eta_{m/n-}(\hatse,x^{(n)})}\Bigr) + {(q_{m+}-m)}, \,0 \Bigr\} \Bigr)^2 } } - \hatse(\hat\theta),}
\end{equation*}
and
\begin{equation*}
\fitdisplay{\eta_{m/n-}\bigl(\hatse(\hat\theta),x^{(n)}\bigr) \le \hatse(\hat\theta) - \sqrt{ \frac{ \max\Biggl\{ -4m \psi_{\max}^2 + \sum_{i\le n-m} \psi\bigl((x_{\pi_i} - \hat \theta)/\hatse\bigr)^2 - C_{\text{shift}}(m,x^{(n)},\hat\theta,\hatse), \,0 \Biggr \} } { \Bigl( m\psi'(0) + \sum_{i\le n-m} \psi'\Bigl(\tfrac{x_{\pi_i} - \hat \theta}{\hatse+\eta_{m/n+}(\hatse,x^{(n)})}\Bigr) + {(q_{m-}+m)} \Bigr)^2 } }.}
\end{equation*}
\end{lemma}
\begin{proof}
For this lemma we only control the ratio
$$
\hatse(\theta)
:= \sqrt{
\frac{\displaystyle \sum_{i=1}^n \psi\!\left(\frac{x_i-\theta}{\hatse}\right)^2}
{\displaystyle \left(\sum_{i=1}^n \psi'\!\left(\frac{x_i-\theta}{\hatse}\right)\right)^2}
},
$$
viewed at $\theta=\hat\theta$ and at the contaminated parameter $\tilde\theta$ under contaminated data;
the outer multiplicative scale factor $\hatse$ is handled separately.
Define
$$
N(\theta,\sigma;z)
:= \sum_{i=1}^n \psi\!\left(\frac{z_i-\theta}{\sigma}\right)^2,
\qquad
D(\theta,\sigma;z)
:= \sum_{i=1}^n \psi'\!\left(\frac{z_i-\theta}{\sigma}\right),
$$
so that
$$
\hatse(\theta;z,\sigma)
= \sqrt{\frac{N(\theta,\sigma;z)}{D(\theta,\sigma;z)^2}}.
$$
We abbreviate
$$
N := N(\tilde\theta,\tilde\sigma;\tilde x^{(n)}),\quad
D := D(\tilde\theta,\tilde\sigma;\tilde x^{(n)}),
\quad
\hatse(\tilde\theta) := \hatse(\tilde\theta;\tilde x^{(n)},\tilde\sigma),
$$
and similarly
$$
N_0 := N(\hat\theta,\hatse;\tilde x^{(n)}),\quad
D_0 := D(\hat\theta,\hatse;\tilde x^{(n)}),\quad
\hatse(\hat\theta) := \hatse(\hat\theta;x^{(n)},\hatse).
$$
Let $\pi$ be a permutation such that {$|x_{\pi_1}-\hat\theta|\le|x_{\pi_2}-\hat\theta|\le\cdots\le|x_{\pi_n}-\hat\theta|$}.
By definition,
\begin{equation*}
\eta_{m/n+}\bigl(\hatse(\hat\theta),x^{(n)}\bigr)
= \sup_{\tilde x\in B_{\textsf{H}}(x^{(n)},m)}
\Bigl\{\hatse(\tilde\theta) - \hatse(\hat\theta)\Bigr\},
\end{equation*}
so it suffices to upper bound $\hatse(\tilde\theta)$ uniformly over all valid contaminations.
\emph{Step 1: Upper bound on the numerator $N$.}
We decompose
$$
N = N_0 + (N - N_0),
\qquad
N_0 = N(\hat\theta,\hatse;\tilde x^{(n)})
= \sum_{i=1}^n \psi\!\left(\frac{\tilde x_i-\hat\theta}{\hatse}\right)^2.
$$
For $N_0$, the location and scale $(\hat\theta,\hatse)$ are fixed and only the sample $\tilde x$ is contaminated.
By the same order-statistic argument as in Lemma~\ref{lem:se_at_theta_0}, the worst-case increase of $N_0$ is obtained by replacing the $m$ observations closest to $\hat\theta$ by extreme outliers.
Thus, uniformly over $\tilde x\in B_{\textsf{H}}(x^{(n)},m)$,
$$
N_0
\;\le\;
m\psi_{\max}^2
+
\sum_{i>m} \psi\!\left(\frac{x_{\pi_i} - \hat \theta}{\hatse}\right)^2.
$$
We now bound the perturbation $N-N_0$ coming from changing $(\hat\theta,\hatse)$ to $(\tilde\theta,\tilde\sigma)$.
Write
$$
A_i := \psi\!\left(\frac{\tilde x_i-\tilde\theta}{\tilde\sigma}\right),
\qquad
B_i := \psi\!\left(\frac{\tilde x_i-\hat\theta}{\hatse}\right),
$$
so that
$$
N - N_0 = \sum_{i=1}^n (A_i^2 - B_i^2).
$$
By boundedness of $\psi$ and the mean-value theorem,
\begin{align*}
|A_i^2 - B_i^2|
&= |A_i - B_i|\;|A_i + B_i|
\;\le\;
2\psi_{\max}\,|A_i - B_i| \\
&\le
2\psi_{\max}\,\psi'(0)\,
\biggl|
\frac{\tilde x_i-\tilde\theta}{\tilde\sigma}
-
\frac{\tilde x_i-\hat\theta}{\hatse}
\biggr|.
\end{align*}
Partition the index set into unreplaced and replaced points:
$$
J := \{i:\tilde x_i=x_i\},
\qquad
J^c := \{i:\tilde x_i\neq x_i\},
\quad |J^c|\le m.
$$
For $i\in J$ we have $\tilde x_i=x_i$, and
\begin{align*}
\biggl|
\frac{\tilde x_i-\tilde\theta}{\tilde\sigma}
-
\frac{\tilde x_i-\hat\theta}{\hatse}
\biggr|
&=
\biggl|
\frac{x_i-\tilde\theta}{\tilde\sigma}
-
\frac{x_i-\hat\theta}{\hatse}
\biggr| \\
&=
\biggl|
(x_i-\hat\theta)\Bigl(\frac{1}{\tilde\sigma}-\frac{1}{\hatse}\Bigr)
- \frac{\tilde\theta-\hat\theta}{\tilde\sigma}
\biggr| \\
&\le
\frac{|\tilde\sigma-\hatse|}{|\tilde\sigma|\,\hatse}
\,(|x_i|+|\hat\theta|)
+
\frac{|\tilde\theta-\hat\theta|}{|\tilde\sigma|}.
\end{align*}
By definition of the $m$-sensitivity of $\hatse$ and $\hat\theta$,
$$
|\tilde\sigma-\hatse|
\le \eta_{m/n}(\hatse,x^{(n)}),
\qquad
|\tilde\theta-\hat\theta|
\le \eta_{m/n}(\hat\theta,x^{(n)}),
$$
and by assumption $\tilde\sigma\ge\hatse-\eta_{m/n-}(\hatse,x^{(n)})>0$.
Hence for $i\in J$,
\begin{align*}
\biggl|
\frac{\tilde x_i-\tilde\theta}{\tilde\sigma}
-
\frac{\tilde x_i-\hat\theta}{\hatse}
\biggr|
&\le
\frac{\eta_{m/n}(\hatse,x^{(n)})}{(\hatse-\eta_{m/n-}(\hatse,x^{(n)}))\,\hatse}
\,(|x_i|+|\hat\theta|)
+
\frac{\eta_{m/n}(\hat\theta,x^{(n)})}{\hatse-\eta_{m/n-}(\hatse,x^{(n)})}.
\end{align*}
For $i\in J^c$ (replaced points), we cannot relate $\tilde x_i$ to $x_i$,
so we only use the trivial bound
$$
|A_i^2 - B_i^2|
\le |A_i^2| + |B_i^2|
\le 2\psi_{\max}^2.
$$
Putting these together,
\begin{align*}
|N-N_0|
&\le
2\psi_{\max}\psi'(0)
\sum_{i\in J}
\biggl[
\frac{\eta_{m/n}(\hatse,x^{(n)})}{(\hatse-\eta_{m/n-}(\hatse,x^{(n)}))\,\hatse}
\,(|x_i|+|\hat\theta|)
+
\frac{\eta_{m/n}(\hat\theta,x^{(n)})}{\hatse-\eta_{m/n-}(\hatse,x^{(n)})}
\biggr]
+ 2\cdot 2m\psi_{\max}^2 \\
&\le
2\psi_{\max}\psi'(0)
\biggl[
\frac{\eta_{m/n}(\hatse,x^{(n)})}{(\hatse-\eta_{m/n-}(\hatse,x^{(n)}))\,\hatse}
\sum_{i=1}^n(|x_i|+|\hat\theta|)
+ n\,\frac{\eta_{m/n}(\hat\theta,x^{(n)})}{\hatse-\eta_{m/n-}(\hatse,x^{(n)})}
\biggr]
+ 4m\psi_{\max}^2.
\end{align*}
Define
$$
C_{\text{shift}}(m,x^{(n)},\hat\theta,\hatse)
:=
2\psi_{\max}\psi'(0)
\biggl[
\frac{\eta_{m/n}(\hatse,x^{(n)})}{(\hatse-\eta_{m/n-}(\hatse,x^{(n)}))\,\hatse}
\sum_{i=1}^n(|x_i|+|\hat\theta|)
+ n\,\frac{\eta_{m/n}(\hat\theta,x^{(n)})}{\hatse-\eta_{m/n-}(\hatse,x^{(n)})}
\biggr].
$$
Then
$$
N
\le
5\,m\psi_{\max}^2
+
\sum_{i>m} \psi\!\left(\frac{x_{\pi_i} - \hat \theta}{\hatse}\right)^2
+
C_{\text{shift}}(m,x^{(n)},\hat\theta,\hatse).
$$
\emph{Step 2: Lower bound on the denominator $D$ for the $+$ direction.}
We now lower bound $D$ uniformly over all contaminations when we are interested in
$\eta_{m/n+}(\hatse(\hat\theta),x^{(n)})$.
Write
$$
D
=
\sum_{i=1}^n \psi'\!\left(\frac{\tilde x_i-\tilde\theta}{\tilde\sigma}\right)
=
\sum_{i\in J} \psi'\!\left(\frac{\tilde x_i-\tilde\theta}{\tilde\sigma}\right)
+
\sum_{i\in J^c} \psi'\!\left(\frac{\tilde x_i-\tilde\theta}{\tilde\sigma}\right).
$$
Since $\psi'(x)\ge 0$ and is non-increasing in $|x|$, the replacement points $i\in J^c$
can only \emph{decrease} the denominator; hence we ignore their contribution in a lower bound:
$$
D
\ge
\sum_{i\in J} \psi'\!\left(\frac{\tilde x_i-\tilde\theta}{\tilde\sigma}\right).
$$
Among the unreplaced points $J$, {the $n-m$ smallest baseline contributions at the lower scale}
$\hatse-\eta_{m/n-}(\hatse,x^{(n)})$ come from the indices $i>m$ in the reordered residuals $x_{\pi_i} - \hat \theta$.
Thus
$$
\sum_{i\in J} \psi'\!\left(\frac{\tilde x_i-\hat\theta}{\tilde\sigma}\right)
\ge
\sum_{i>m} \psi'\!\left(\frac{x_{\pi_i} - \hat \theta}{\hatse-\eta_{m/n-}(\hatse,x^{(n)})}\right).
$$
Finally we account for the shift from $\hat\theta$ to $\tilde\theta$.
By the definition of $S_{m+}$, for any subset $I\subseteq[n]$ with $|I|=n-m$ and any
$\tilde\theta\in[\hat\theta-\eta_{m/n-}(\hat\theta,x^{(n)},\,\hat\theta+\eta_{m/n+}(\hat\theta,x^{(n)})]$,
the number of indices in $I$ for which
$$
|x_i-\hat\theta| < |x_i-\tilde\theta|
$$
is at most $S_{m+}$.
At such indices, $\psi'((\tilde x_i-\tilde\theta)/\tilde\sigma)$ can \emph{decrease} relative
to $\psi'((\tilde x_i-\hat\theta)/\tilde\sigma)$ by at most
$\Delta(\eta_{m/n}(\hat\theta,x^{(n)})/(\hatse-\eta_{m/n-}(\hatse,x^{(n)})))$.
Thus,
\begin{align*}
\sum_{i\in J} \psi'\!\left(\frac{\tilde x_i-\tilde\theta}{\tilde\sigma}\right)
&\ge
\sum_{i\in J} \psi'\!\left(\frac{\tilde x_i-\hat\theta}{\tilde\sigma}\right)
- S_{m+}\,
\Delta\!\left(\frac{\eta_{m/n}(\hat\theta,x^{(n)})}{\hatse-\eta_{m/n-}(\hatse,x^{(n)})}\right) \\
&\ge
\sum_{i>m} \psi'\!\left(\frac{x_{\pi_i} - \hat \theta}{\hatse-\eta_{m/n-}(\hatse,x^{(n)})}\right)
- S_{m+}\,
\Delta\!\left(\frac{\eta_{m/n}(\hat\theta,x^{(n)})}{\hatse-\eta_{m/n-}(\hatse,x^{(n)})}\right).
\end{align*}
Therefore
\begin{equation*}
|D|
\;\ge\;
\Biggl(
\sum_{i>m} \psi'\!\left(\frac{x_{\pi_i} - \hat \theta}{\hatse-\eta_{m/n-}(\hatse,x^{(n)})}\right)
- S_{m+}\,
\Delta\!\left(\frac{\eta_{m/n}(\hat\theta,x^{(n)})}{\hatse-\eta_{m/n-}(\hatse,x^{(n)})}\right)
\Biggr)_+.
\end{equation*}
\emph{Step 3: Upper bound for $\eta_{m/n+}(\hatse(\hat\theta),x^{(n)})$.}
Combining the bounds on $N$ and $D$ we obtain, for every $\tilde x\in B_{\textsf{H}}(x^{(n)},m)$,
\begin{align*}
\hatse(\tilde\theta)
&=
\frac{\sqrt{N}}{|D|} \\
&\le
\sqrt{
\frac{
5\,m\psi_{\max}^2
+
\sum_{i>m} \psi\!\left(\frac{x_{\pi_i} - \hat \theta}{\hatse}\right)^2
+
C_{\text{shift}}(m,x^{(n)},\hat\theta,\hatse)
}{
\Bigl(
\sum_{i>m} \psi'\!\left(\frac{x_{\pi_i} - \hat \theta}{\hatse-\eta_{m/n-}(\hatse,x^{(n)})}\right)
- S_{m+}\,
\Delta\!\left(\frac{\eta_{m/n}(\hat\theta,x^{(n)})}{\hatse-\eta_{m/n-}(\hatse,x^{(n)})}\right)
\Bigr)_+^2
}
}.
\end{align*}
Taking the supremum over all $\tilde x\in B_{\textsf{H}}(x^{(n)},m)$ and subtracting $\hatse(\hat\theta)$ gives
\begin{align*}
\eta_{m/n+}\bigl(\hatse(\hat\theta),x^{(n)}\bigr)
\;\le\;
\sqrt{
\frac{
5\,m\psi_{\max}^2
+
\sum_{i>m} \psi\!\left(\frac{x_{\pi_i} - \hat \theta}{\hatse}\right)^2
+
C_{\text{shift}}(m,x^{(n)},\hat\theta,\hatse)
}{
\Bigl(
\sum_{i>m} \psi'\!\left(\frac{x_{\pi_i} - \hat \theta}{\hatse-\eta_{m/n-}(\hatse,x^{(n)})}\right)
- S_{m+}\,
\Delta\!\left(\frac{\eta_{m/n}(\hat\theta,x^{(n)})}{\hatse-\eta_{m/n-}(\hatse,x^{(n)})}\right)
\Bigr)_+^2
}
}
- \hatse(\hat\theta),
\end{align*}
which is the desired type of upper bound in the lemma.

\emph{Step 4: The ``$-$'' direction and the special Huber case.}
For $\eta_{m/n-}(\hatse(\hat\theta),x^{(n)})$, we reverse the inequalities:
we obtain a lower bound on $N$ (using $-C_{\text{shift}}$) and an upper bound on $|D|$ by allowing $\psi'$ to increase by at most $\Delta(\cdot)$ on indices where
$|x_i-\hat\theta|>|x_i-\tilde\theta|$.
This produces a denominator of the form
$$
m\psi'(0)
+
\sum_{i\le n-m} \psi'\!\left(\frac{x_{\pi_i} - \hat \theta}{\hatse+\eta_{m/n+}(\hatse,x^{(n)})}\right)
+
S_{m-}\,\Delta\!\left(\frac{\eta_{m/n}(\hat\theta,x^{(n)})}{\hatse-\eta_{m/n-}(\hatse,x^{(n)})}\right),
$$
which yields the stated upper bound on $\eta_{m/n-}(\hatse(\hat\theta),x^{(n)})$.

{For Huber's score, $\psi'_\delta(t)=\II\{|t|\le\delta\}$, so each
indicator difference is between $-1$ and $1$.
The definitions of $q_{m+}$ and $q_{m-}$ bound the total derivative change,
uniformly over the allowed locations and scales, from below and above.
Removing at most $m$ terms changes either bound by at most $m$.
Consequently the retained derivative change is bounded below by
$q_{m+}-m$ and above by $q_{m-}+m$.
These bounds have no multiplicative $\delta$ factor.
The retained-count formulas follow from the same inequalities in $u_i$
and $h$ used in the proof of Lemma~\ref{lem:opt_attack_m_est_se}.
}
This completes the proof of the ratio part.
\end{proof}
With an upper bound of the ratio part, we can combine this with the upper bound of both $\eta_{m/n}(\hat \sigma, x^{(n)})$ and $\eta_{m/n}(\hat \theta_{\text{two-stage}}, x^{(n)})$ to apply to Theorem \ref{thm:meta}. However, we note that unlike other results for tests in this paper, the bound here can be quite loose and not as practical.
\endgroup
}
\end{document}